\documentclass[11pt,reqno]{amsart}

\usepackage[utf8]{inputenc}
\usepackage{amsmath}
\usepackage{mathtools}
\usepackage{amsthm}
\usepackage{amsfonts}
\usepackage{amsopn}
\usepackage{amssymb}
\usepackage{esint}
\usepackage{amsopn}
\usepackage{amssymb}
\usepackage{esint}
\usepackage{fullpage}
\usepackage{esint,comment}
\usepackage{stmaryrd,mathrsfs}
\usepackage{dutchcal}
\usepackage{upgreek}
\usepackage[mathscr]{euscript}
\usepackage[breaklinks,pdfstartview=FitH]{hyperref}
\usepackage{esint}
\usepackage{tikz}

\DeclareMathOperator{\NM}{NM}

\newcommand{\vol}{{\cH^{2k+2}}}
\DeclareMathOperator{\diam}{diam}
\DeclareMathOperator{\Lip}{Lip}
\DeclareMathOperator{\Imag}{Im}

\DeclareMathOperator{\length}{length}
\DeclareMathOperator{\nbhd}{nbhd}
\DeclareMathOperator{\supp}{supp}

\DeclareMathOperator{\Per}{Per}

\newcommand{\cc}{{\mathsf{c}}}
\newcommand{\area}{{\cH^{2k+1}}}
\newcommand{\from}{\colon\thinspace}
\newcommand{\N}{\mathbb{N}}
\newcommand{\Bword}{B^{\square}}
\newcommand{\C}{\mathbb{C}}
\newcommand{\R}{\mathbb{R}}
\newcommand{\Z}{\mathbb{Z}}
\renewcommand{\H}{\mathbb{H}}
\newcommand{\cB}{\mathcal{B}}
\newcommand{\cE}{\mathcal{E}}
\newcommand{\cH}{\mathcal{H}}
\newcommand{\cI}{\mathcal{I}}
\newcommand{\cJ}{\mathscr{J}}
\newcommand{\cL}{\mathcal{L}}
\newcommand{\cN}{\mathcal{N}}

\newcommand{\cF}{\mathcal{F}}
\newcommand{\cG}{\mathcal{G}}
\newcommand{\cS}{\mathcal{S}}
\newcommand{\cT}{\mathcal{T}}

\newcommand{\dEuc}{d_{\text{Euc}}}

\newcommand{\one}{\mathbf{1}}
\newcommand{\symdiff}{\mathop{\triangle}}
\newcommand{\vpf}{{\overline{\mathsf{v}}}}
\newcommand{\vpfl}[1]{{\overline{\mathsf{v}}_{\!#1}}}

\newcommand{\vpfprime}{{\overline{\mathsf{v}}'}}
\newcommand{\G}{\mathbb{G}}
\newcommand{\cM}{\mathcal{M}}
\newcommand{\BB}{\boldsymbol{\mathscr{B}}}

\newcommand{\vv}{\mathsf{v}}
\newcommand{\hh}{\mathsf{h}}
\renewcommand{\mid}{:\ }
\renewcommand{\supset}{\supseteq}
\renewcommand{\subset}{\subseteq}
\renewcommand{\gamma}{\upgamma}
\renewcommand{\beta}{\upbeta}
\renewcommand{\alpha}{\upalpha}
\renewcommand{\kappa}{\upkappa}
\renewcommand{\psi}{\uppsi}
\renewcommand{\rho}{\uprho}
\renewcommand{\delta}{\updelta}
\renewcommand{\pi}{\uppi}
\renewcommand{\omega}{\upomega}
\renewcommand{\sigma}{\upsigma}
\renewcommand{\emptyset}{\varnothing}
\newcommand{\e}{\varepsilon}
\renewcommand{\theta}{\uptheta}
\newcommand{\f}{\phi}
\newcommand{\ud}[0]{\,\mathrm{d}}
\renewcommand{\epsilon}{\varepsilon}
\renewcommand{\k}{\{1,\ldots,k\}}
\newcommand{\s}{\mathfrak{s}}
\renewcommand{\xi}{\upxi}
\renewcommand{\tau}{\uptau}

\newtheorem{thm}{Theorem}

\newtheorem{question}[thm]{Question}
\newtheorem{lemma}[thm]{Lemma}
\newtheorem{prop}[thm]{Proposition}
\newtheorem{cor}[thm]{Corollary}

\newtheorem{defn}[thm]{Definition}
\theoremstyle{remark}
\newtheorem{remark}[thm]{Remark}

\renewcommand{\le}{\leqslant}
\renewcommand{\ge}{\geqslant}

\renewcommand{\setminus}{\smallsetminus}
\newcommand{\eqdef}{\stackrel{\mathrm{def}}{=}}
\newcommand{\1}{\mathbf 1}
\newcommand{\cb}{\mathcal{b}}

\newcommand{\ms}{\scriptscriptstyle}

\renewcommand{\k}{\{1,\ldots,k\}}
\newcommand{\n}{\{1,\ldots,n\}}
\newcommand{\m}{\{1,\ldots,m\}}
\renewcommand{\H}{\mathbb{H}}
\newcommand{\GL}{\mathsf{GL}}
\newcommand{\CH}{\mathcal{H}}

\renewcommand{\gamma}{\upgamma}
\renewcommand{\beta}{\upbeta}
\renewcommand{\alpha}{\upalpha}
\renewcommand{\kappa}{\upkappa}
\renewcommand{\psi}{\uppsi}
\renewcommand{\rho}{\uprho}
\renewcommand{\delta}{\updelta}
\renewcommand{\pi}{\uppi}
\renewcommand{\omega}{\upomega}
\renewcommand{\lambda}{\uplambda}
\renewcommand{\eta}{\upeta}
\renewcommand{\chi}{\upchi}
\renewcommand{\nu}{\upnu}
\renewcommand{\mu}{\upmu}
\renewcommand{\phi}{\upphi}

\newcommand{\0}{\mathbf{0}}

\begin{document}

\title{Vertical perimeter versus horizontal perimeter}

\author[Assaf Naor]{Assaf Naor}
\address{Princeton University,
Department of Mathematics,
Fine Hall, Washington Road,
Princeton, NJ 08544-1000, USA}
\email{naor@math.princeton.edu}

\author[Robert Young]{Robert Young}
\address{New York University, Courant Institute of Mathematical Sciences, 251 Mercer Street, New York, NY 10012, USA}
\email{ryoung@cims.nyu.edu}

\thanks{A.N.~was supported by BSF grant 2010021, the Packard Foundation and the Simons Foundation.  R.Y.~was supported by NSF grant 1612061, the Sloan Foundation, and the Fall 2016 program on Geometric Group Theory at the Mathematical Sciences Research Institute. The research that is presented here was conducted under the auspices of the Simons Algorithms and Geometry (A\&G) Think Tank.}

\thanks{An extended abstract~\cite{NY17} announcing parts of this work and titled ``The integrality gap of the Goemans--Linial SDP relaxation for sparsest cut is at least a constant multiple of $\sqrt{\log n}$" appeared in the Proceedings of the 49th Annual ACM SIGACT Symposium on the Theory of Computing (STOC'17).}

\date{\today}



\maketitle


\begin{abstract}
Given $k\in \N$, the $k$'th discrete Heisenberg group, denoted $\H_{\ms{\Z}}^{2k+1}$, is the group generated by the elements $a_1,b_1,\ldots,a_k,b_k,c$, subject to the commutator relations  $[a_1,b_1]=\ldots=[a_k,b_k]=c$, while all the other pairs of elements from this generating set are required to commute, i.e., for every distinct $i,j\in \k$ we have $[a_i,a_j]=[b_i,b_j]=[a_i,b_j]=[a_i,c]=[b_i,c]=1$ (in particular, this implies that $c$ is in the center of $\H_{\ms{\Z}}^{2k+1}$). Denote $\mathfrak{S}_k=\{a_1,b_1,a_1^{-1},b_1^{-1},\ldots,a_k,b_k,a_k^{-1},b_k^{-1}\}$. The {\em horizontal boundary} of $\Omega\subset \H_{\ms{\Z}}^{2k+1}$, denoted  $\partial_{\hh}\Omega$, is the set of all those pairs $(x,y)\in \Omega\times (\H_{\ms{\Z}}^{2k+1}\setminus \Omega)$ such that $x^{-1}y\in \mathfrak{S}_k$. The {\em horizontal perimeter} of $\Omega$ is the cardinality $|\partial_{\hh}\Omega|$ of $\partial_{\hh}\Omega$, i.e., it is the total number of edges incident to $\Omega$ in the Cayley graph induced by  $\mathfrak{S}_k$. For $t\in \N$, define $\partial^t_{\vv} \Omega$ to be the set of all those pairs $(x,y)\in \Omega\times (\H_{\ms{\Z}}^{2k+1}\setminus \Omega)$ such that $x^{-1}y\in \{c^t,c^{-t}\}$. Thus, $|\partial^t_{\vv}\Omega|$ is the total number of edges incident to $\Omega$ in the (disconnected) Cayley graph induced by $\{c^t,c^{-t}\}\subset \H_{\ms{\Z}}^{2k+1}$. The {\em vertical perimeter}  of $\Omega$ is defined by $|\partial_{\vv}\Omega|= \sqrt{\sum_{t=1}^\infty |\partial^t_{\vv}\Omega|^2/t^2}$. It is shown here that if $k\ge 2$,  then  $|\partial_{\vv}\Omega|\lesssim \frac{1}{k} |\partial_{\hh}\Omega|$. The proof of this ``vertical versus horizontal isoperimetric inequality" uses a new structural result that decomposes sets of finite perimeter in the Heisenberg group into pieces that admit an ``intrinsic corona decomposition.''  This allows one to deduce an endpoint  $W^{1,1}\to L_2(L_1)$ boundedness of a certain singular integral operator from a corresponding lower-dimensional $W^{1,2}\to L_2(L_2)$ boundedness. Apart from its intrinsic geometric interest, the above (sharp) isoperimetric-type inequality has several (sharp) applications, including that for every $n\in \N$, any embedding into an $L_1(\mu)$ space of a ball of radius $n$ in the word metric on $\H_{\ms{\Z}}^{5}$ that is induced by the generating set $\mathfrak{S}_2$ incurs bi-Lipschitz distortion that is   at least a universal constant multiple of $\sqrt{\log n}$. As an application to approximation algorithms, it follows that for every $n\in \N$  the integrality gap of the Goemans--Linial semidefinite program for the Sparsest Cut Problem on inputs of size $n$ is at least a universal constant multiple of $\sqrt{\log n}$.
\end{abstract}

\section{Introduction}

We shall start by formulating the main result that is obtained here in the discrete setting, since this is the setting which is needed for applications to  approximation algorithms and metric embeddings. However, we shall soon thereafter proceed to describe the corresponding (and equivalent) continuous setting, which is where it is most natural to carry out the bulk of the ensuing proofs.

For every $k\in \N$ the $k$'th discrete Heisenberg group, denoted $\H_{\ms{\Z}}^{2k+1}$, is defined via generators and relations as follows. The $2k+1$ generators are denoted $a_1,b_1,\ldots,a_k,b_k,c$. Recalling that the commutator of $g,h\in \H_{\ms{\Z}}^{2k+1}$ is denoted $[g,h]=ghg^{-1}h^{-1}$, the defining relations of $\H_{\ms{\Z}}^{2k+1}$ are the requirements that $c=[a_1,b_1]=\ldots=[a_k,b_k]$ and  $[a_i,a_j]=[b_i,b_j]=[a_i,b_j]=[a_i,c]=[b_i,c]=1$ for every distinct $i,j\in \k$. In particular, since  $c$ commutes with all of the generators,  it belongs to the center of $\H_{\ms{\Z}}^{2k+1}$.  The set $\mathfrak{S}_k=\{a_1,b_1,a_1^{-1},b_1^{-1},\ldots,a_k,b_k,a_k^{-1},b_k^{-1}\}$ is a symmetric generating subset of $\H_{\ms{\Z}}^{2k+1}$, and it therefore induces a left-invariant word metric on $\H_{\ms{\Z}}^{2k+1}$ which we denote by $d_W:\H_{\ms{\Z}}^{2k+1}\times \H_{\ms{\Z}}^{2k+1}\to [0,\infty)$. A standard concrete realization of $\H_{\ms{\Z}}^{2k+1}$ is the group of all $k+2$ by $k+2$ matrices with integer entries of the following form (and usual matrix multiplication).
\begin{equation}\label{eq:def discrete H}
\begin{pmatrix} 1 & x_1 &x_2 & \dots &x_k& z\\
  0 & 1&0 &\dots& 0 & y_1\\
  \vdots & \ddots & \ddots  & \ddots & \vdots&\vdots\\
            \vdots & \ddots & \ddots& \ddots &0&y_{k-1}\\
              \vdots & \ddots & \ddots &\ddots&1&y_k
              \\
              0 & \dots & \dots &\dots&0&1
                       \end{pmatrix}=I_{k+2}+\sum_{i=1}^kx_iE_{1,i+1}+\sum_{j=1}^ky_jE_{j+1,k+2}+zE_{1,k+2}.
                       \end{equation}
Here $I_{k+2}$ is the $(k+2)\times (k+2)$ identity matrix and $\{E_{i,j}:\ i,j\in \{1,\ldots,k+2\}\}$ is the standard basis of $M_{k+2}(\R)$. While it is beneficial to keep in mind the above concrete realization of $\H_{\ms{\Z}}^{2k+1}$, it will later be convenient to work with a different  realization of $\H_{\ms{\Z}}^{2k+1}$, and moreover the initial part of the ensuing discussion (including all of the main results and applications) can be understood while only referring to the above abstract definition of $\H_{\ms{\Z}}^{2k+1}$ through generators and relations.

\begin{defn}[Discrete boundaries]\label{def:discrete perimeters}
{\em Given $\Omega\subset \H_{\ms{\Z}}^{2k+1}$, the {\em horizontal boundary} of $\Omega$ is defined by
\begin{equation}\label{eq:def horizontal boundary}
\partial_{\hh}\Omega\eqdef \left\{(x,y)\in \Omega\times \big(\H_{\ms{\Z}}^{2k+1}\setminus \Omega\big): x^{-1}y\in \mathfrak{S}_k\right\}.
\end{equation}
Given also  $t\in \N$, the $t${\em-vertical boundary} of $\Omega$ is defined by
\begin{equation}\label{eq:def vertical t boundary}
\partial^t_{\vv} \Omega\eqdef \left\{(x,y)\in \Omega\times \big(\H_{\ms{\Z}}^{2k+1}\setminus \Omega\big): x^{-1}y\in \left\{c^t,c^{-t}\right\}\right\}.
\end{equation}
The {\em horizontal perimeter} of $\Omega$ is defined to be the cardinality $|\partial_{\hh}\Omega|$ of its horizontal boundary. The {\em vertical perimeter} of $\Omega$ is defined to be the quantity
\begin{equation}\label{eq:def discrete vertical perimeter}
|\partial_{\vv}\Omega|\eqdef \bigg(\sum_{t=1}^\infty \frac{|\partial^t_{\vv}\Omega|^2}{t^2}\bigg)^{\frac12}.
\end{equation}
}
\end{defn}

\begin{remark}\label{rem:combinatotrial restate} It may be instructive to restate some of the concepts that were introduced in Definition~\ref{def:discrete perimeters} through the following graph-theoretic descriptions.
The Cayley graph that is induced by a symmetric subset $\Sigma=\Sigma^{-1}$ of a group $G$ will be  denoted below by $\mathcal{X}_{\Sigma}(G)$. Recall that this   is the graph whose vertex set is $G$ and whose edge set is $E_\Sigma(G)=\big\{\{g,g\sigma\}:\ g\in G\ \wedge \ \sigma\in \Sigma\big\}$. In particular, $\mathcal{X}_{\Sigma}(G)$ is connected if and only if $\Sigma$ generates $G$. With this (standard) terminology, the horizontal boundary $\partial_{\hh}\Omega$ is the edge boundary of $\Omega$ in the Cayley graph $\mathcal{X}_{\mathfrak{S}_k}(\H_{\ms{\Z}}^{2k+1})$, i.e., those pairs $(x,y)\in \H_{\ms{\Z}}^{2k+1}\times \H_{\ms{\Z}}^{2k+1}$ such that  $\{x,y\}\in E_{\mathfrak{S}_k}(\H_{\ms{\Z}}^{2k+1})$, $x\in \Omega$ and $y\notin \Omega$. In the same vein, the $t$-vertical boundary  $\partial^t_{\vv} \Omega$ is the edge boundary of $\Omega$ in the Cayley graph $\mathcal{X}_{\{c^t,c^{-t}\}}(\H_{\ms{\Z}}^{2k+1})$.
\end{remark}

The vertical perimeter as defined in~\eqref{eq:def discrete vertical perimeter} is a more subtle concept, and in particular it does not have a combinatorial description that is analogous to those of Remark~\ref{rem:combinatotrial restate}. The definition~\eqref{eq:def discrete vertical perimeter} was first published in~\cite[Section~4]{LafforgueNaor}, where the isoperimetric-type conjecture that we resolve here also appeared for the first time. These were formulated by the first named author and were circulating for several years before~\cite{LafforgueNaor} appeared, intended as a possible route towards the algorithmic application that we indeed succeed to obtain here. The basic idea is that because $c$ is equal to each of the commutators $\{[a_i,b_i]\}_{i=1}^k$, if the horizontal boundary of $\Omega$ is small, i.e., it is  ``difficult" to leave the set $\Omega$ in the ``horizontal directions" $\{a_i^{\pm 1},b_i^{\pm 1}\}_{i=1}^k$, then this should be reflected by the ``smallness" of the sequence $\{|\partial_\vv^t\Omega|\}_{t=1}^\infty$ that measures how  difficult it is to leave $\Omega$ in the ``vertical directions" $\{c^{\pm t}\}_{t=1}^\infty$. That this smallness should be measured through the quantity $|\partial_{\vv}\Omega|$, i.e.,  the $\ell_2$ norm of the sequence $\{|\partial_\vv^t\Omega|/t\}_{t=1}^\infty$, was arrived at through trial and error, inspired by  functional inequalities that were obtained in~\cite{AusNaoTes,LafforgueNaor}, as explained in~\cite[Section~4]{LafforgueNaor}; see also Section~\ref{sec:endpoint}  below.

\begin{thm}
\label{thm:isoperimetric discrete}  For every integer $k\ge 2$ and every finite subset $\Omega\subset  \H_{\ms{\Z}}^{2k+1}$ we have $|\partial_{\vv}\Omega|\lesssim \frac{1}{k}|\partial_{\hh}\Omega|$.
\end{thm}

\subsubsection*{Asymptotic notation} In Theorem~\ref{thm:isoperimetric discrete} as well as in what follows  we use the following standard conventions for asymptotic notation. Given $\alpha,\beta\in (0,\infty)$, the notations
$\alpha\lesssim \beta$ and $\beta\gtrsim \alpha$ mean that $\alpha\le \gamma \beta$ for some
universal constant $\gamma\in (0,\infty)$. The notation $\alpha\asymp \beta$
stands for $(\alpha\lesssim \beta) \wedge  (\beta\lesssim \alpha)$. If we need to allow for dependence on parameters, we indicate this by subscripts. For example, in the presence of an auxiliary parameter $\psi$, the notation $\alpha\lesssim_\uppsi \beta$ means that $\alpha\le \gamma(\psi)\beta $, where $\gamma(\psi)\in (0,\infty)$ is allowed to depend only on $\uppsi$, and similarly for the notations $\alpha\gtrsim_\psi \beta$ and $\alpha\asymp_\psi \beta$.

The ``vertical versus horizontal isoperimetric inequality" of Theorem~\ref{thm:isoperimetric discrete} is sharp up to the implicit universal constant, as exhibited by considering the case when $\Omega$ is a singleton (see also Remark~\ref{rem:box} below). In fact, this inequality has  qualitatively different types of sets $\Omega$ that saturate it. Namely, for every $m\in \N$ there is a finite  $\Omega\subset \H^{5}_{\ms{\Z}}$ for which $|\{j\in \N:\ \sum_{t=2^{j-1}}^{2^j-1}|\partial^t_\vv\Omega|^2/t^2\asymp |\partial_\hh\Omega|^2/m\}|\asymp m$; this follows from an examination of the sub-level sets of the embeddings that are discussed in the paragraph following Theorem~\ref{thm:distortion R} below. We shall later see that by a simple argument the case $k=2$ of Theorem~\ref{thm:isoperimetric discrete} implies its statement for general $k\ge 2$. The case $k=2$ is therefore the heart of the matter, and moreover all of our applications of Theorem~\ref{thm:isoperimetric discrete} use only this special case. For these reasons, the discussion in most of the Introduction will focus on $k=2$, but our ensuing proofs will be carried out for general $k$ because they derive structural results that are of interest in any dimension and do not reduce directly  to a fixed dimension.

\begin{remark}\label{eq:finite p}  Note the restriction $k\ge 2$ in Theorem~\ref{thm:isoperimetric discrete}. The case $k=1$ remains an intriguing mystery and a subject of our ongoing research that will be published elsewhere. This ongoing work shows that  Theorem~\ref{thm:isoperimetric discrete} {\em fails} for $k=1$, but that there exists  $q\in (2,\infty)$ such that for every $\Omega\subset \H_{\ms{\Z}}^3$ we have
\begin{equation}\label{eq:p version}
\bigg(\sum_{t=1}^\infty \frac{|\partial^t_{\vv}\Omega|^q}{t^{1+\frac{q}{2}}}\bigg)^{\frac{1}{q}}\lesssim |\partial_{\hh}\Omega|.
\end{equation}
A simple argument shows that $\sup_{s\in \N} |\partial^s_{\vv}\Omega|/\sqrt{s}\le  \gamma |\partial_{\hh}\Omega|$ for some universal constant $\gamma>0$. Hence, for every $t\in \N$ we have 
$$\frac{|\partial^t_{\vv}\Omega|^q}{t^{1+\frac{q}{2}}}\le\biggl(\frac{|\partial^t_{\vv}\Omega|}{t}\biggr)^2 \sup_{s\in \N} \biggl(\frac{|\partial^s_{\vv}\Omega|}{\sqrt{s}}\biggr)^{q-2}\le \biggl(\frac{|\partial^t_{\vv}\Omega|}{t}\biggr)^2 (\gamma|\partial_{\hh}\Omega|)^{q-2}.$$
This implies that the left hand side of~\eqref{eq:p version} is bounded from above by a universal  constant multiple of  $|\partial_{\vv}\Omega|^{2/q}|\partial_{\hh}\Omega|^{1-2/q}$. Therefore~\eqref{eq:p version} is formally weaker than the estimate  $|\partial_{\vv}\Omega|\lesssim |\partial_{\hh}\Omega|$ of Theorem~\ref{thm:isoperimetric discrete}. It would be interesting to determine the infimum over those $q$ for which~\eqref{eq:p version} holds true for every $\Omega\subset \H_{\ms{\Z}}^3$, with ongoing work indicating that it is at least $4$. It should be stressed, however, that all the applications of Theorem~\ref{thm:isoperimetric discrete} that are obtained here (i.e., the algorithmic application in Section~\ref{sec:sparsest} and the geometric applications in Section~\ref{sec:embedding})    use the case $k=2$ of Theorem~\ref{thm:isoperimetric discrete}, and understanding the case $k=1$ would not yield any further improvements. So,  while the case $k=1$ is geometrically  interesting in its own right, it is not needed for the applications that we currently have in mind.
\end{remark}

\subsection{Sparsest Cut}\label{sec:sparsest} The {\em Sparsest Cut Problem} is a central open question in approximation algorithms that has attracted major research efforts over the past decades. A remarkable aspect of this natural and versatile optimization problem is that it admits a simple to describe  algorithm that was proposed in the mid-1990s by Goemans and Linial, yet while  this specific algorithm received close scrutiny by many researchers and to date it   is still the best-known approximation algorithm for the Sparsest Cut Problem, despite major efforts it remained unknown for many years how well this algorithm can perform in general. Here we settle the longstanding open question of determining (up to lower order factors) the approximation ratio of the Goemans--Linial algorithm. Modulo previously published reductions, this is a quick consequence of Theorem~\ref{thm:isoperimetric discrete}. In fact, the statement of Theorem~\ref{thm:isoperimetric discrete}  was initially conjectured as a possible route towards this algorithmic application.

Fix $n\in \N$. The input of the Sparsest Cut Problem consists of two $n$ by $n$ symmetric matrices with nonnegative entries $C=(C_{ij}),D=(D_{ij})\in M_n([0,\infty))$, which are often called capacities and demands, respectively. The goal is to design a polynomial-time algorithm to evaluate the quantity
\begin{equation}\label{eq:def opt}
\mathsf{OPT}(C,D)\eqdef \min_{\emptyset\subsetneq A\subsetneq \n}\frac{\sum_{(i,j)\in A\times (\n\setminus A)}C_{ij}}{\sum_{(i,j)\in A\times (\n\setminus A)}D_{ij}}.
\end{equation}

In view of the extensive literature on the Sparsest Cut Problem, it would be needlessly repetitive to recount here the rich and multifaceted impact of this optimization problem on computer science and mathematics; see instead the articles~\cite{AKRR90,LR99}, the surveys~\cite{Shm95,Lin02,Chawla08,Nao10}, Chapter~10 of the monograph~\cite{DL97}, Chapter~15 of the monograph~\cite{Mat02}, Chapter~1 of the monograph~\cite{Ost13}, and the references therein. It suffices to say  that by tuning the choice of matrices $C,D$ to the  problem at hand, the  minimization in~\eqref{eq:def opt} finds a partition of the ``universe" $\n$ into two parts, namely the sets $A$ and $\n\setminus A$, whose appropriately weighted interface is as small as possible, thus allowing for inductive solutions of various algorithmic tasks, a procedure known as {\em divide and conquer}. (Not all of the uses of the  Sparsest Cut Problem fit into this framework. A recent algorithmic application of a different nature can be found in~\cite{MMV14}.)

It is $\mathsf{NP}$-hard to compute $\mathsf{OPT}(C,D)$~\cite{SM90}. By~\cite{CK09-hardness} there exists $\e_0>0$ such that it is even $\mathsf{NP}$-hard to compute $\mathsf{OPT}(C,D)$ within a multiplicative factor of less than $1+\e_0$. If one assumes Khot's Unique Games Conjecture~\cite{Kho02,Kho10,Tre12} then by~\cite{CKKRS06,KV15} there does not exist a polynomial-time algorithm that can compute $\mathsf{OPT}(C,D)$ within any universal constant factor.

Due to the above hardness results, a much more realistic goal would be to design a polynomial-time algorithm that takes as input the capacity and demand matrices $C,D\in M_n([0,\infty))$ and outputs a number $\mathsf{ALG}(C,D)$ that is guaranteed to satisfy $\mathsf{ALG}(C,D)\le \mathsf{OPT}(C,D)\le \rho(n)\mathsf{ALG}(C,D)$, with (hopefully) the quantity $\rho(n)$ growing to $\infty$ slowly as $n\to \infty$. Determining the best possible asymptotic behaviour of $\rho(n)$ (assuming $\mathsf{P}\neq \mathsf{NP}$) is an  open problem of major importance.

In~\cite{LLR95,AR98}  an algorithm was designed, based on linear programming (through the connection to multicommodity flows) and Bourgain's embedding theorem~\cite{Bou85}, which yields $\rho(n)=O(\log n)$. An algorithm based on semidefinite programming (to be described precisely below) was proposed by Goemans and Linial in the mid-1990s.  To the best of our knowledge this idea first appeared in the literature in~\cite[page~158]{Goe97}, where it was speculated that it might even yield a constant factor approximation for the  Sparsest Cut Problem (see also~\cite{Lin02,Lin-open}).

The Goemans--Linial algorithm is simple to describe. It takes as input the symmetric matrices $C=(C_{ij}),D=(D_{ij})\in M_n([0,\infty))$ and proceeds to compute the following quantity.
\begin{equation}\label{eq:def sdp}
\mathsf{SDP}(C,D)\eqdef \inf_{(v_1,\ldots,v_n)\in \mathsf{NEG}_n} \frac{\sum_{i=1}^n\sum_{j=1}^nC_{ij}\|v_i-v_j\|_{\ell_2^n}^2}{\sum_{i=1}^n\sum_{j=1}^nD_{ij}\|v_i-v_j\|_{\ell_2^n}^2},
\end{equation}
where
$$
\mathsf{NEG}_n\eqdef \Big\{(v_1,\ldots v_n)\in (\R^n)^n:\ \|v_i-v_j\|_{\ell_2^n}^2\le \|v_i-v_k\|_{\ell_2^n}^2+\|v_k-v_j\|_{\ell_2^n}^2\ \mathrm{for\ all\ } i,j,k\in \n\Big\}.
$$
Thus $\mathsf{NEG}_n$  is the set of $n$-tuples $(v_1,\ldots v_n)$ of vectors in $\R^n$ such that $(\{v_1,\ldots,v_n\},\upnu_n)$ is a semi-metric space, where  $\upnu_n:\R^n\times \R^n\to [0,\infty)$ is defined by
$$\forall\, x=(x_1,\ldots,x_n),y=(y_1,\ldots,y_n)\in \R^n,\qquad \upnu_n(x,y)\eqdef \sum_{j=1}^n (x_j-y_j)^2=\|x-y\|_{\ell_2^n}^2.$$
 A semi-metric space $(\cM,d)$ is said (see e.g.~\cite{DL97}) to be of {\em negative type}  if $(\cM,\sqrt{d})$ embeds isometrically into a Hilbert space. So, $\mathsf{NEG}_n$ can be  described as the set of all (ordered) negative type semi-metrics of size $n$. The evaluation of the quantity $\mathsf{SDP}(C,D)$ in~\eqref{eq:def sdp} can be cast as a semidefinite program (SDP), so it can be achieved (up to $o(1)$ precision) in polynomial time~\cite{GLS93}. One has $\mathsf{SDP}(C,D)\le  \mathsf{OPT}(C,D)$ for all symmetric matrices $C,D\in M_n([0,\infty))$. See  e.g.~\cite[Section~15.9]{Mat02-book} or~\cite[Section~4.3]{Nao10} for an explanation of the above assertions about $\mathsf{SDP}(C,D)$, as well as additional background and motivation. The pertinent question is therefore to evaluate the asymptotic behavior as $n\to \infty$ of the following quantity, which is commonly known as the  {\em integrality gap} of the Goemans--Linial semidefinite programming relaxation~\eqref{eq:def sdp} for the Sparsest Cut Problem.
$$
\uprho_{\mathsf{GL}}(n)\eqdef \sup_{\substack{C,D\in M_n([0,\infty))\\ C,D\ \mathrm{symmetric}}} \frac{\mathsf{OPT}(C,D)}{\mathsf{SDP}(C,D)}.
$$

Since $\rho_{\GL}(n)$ can be computed in polynomial time, we have $\rho(n)\le \rho_{\GL}(n)$; the algorithmic output $\mathsf{ALG}(C,D)$ in this case is simply $\mathsf{SDP}(C,D)$. The above mentioned {\em Goemans--Linial conjecture} was that  $\rho_{\GL}(n)=O(1)$. This hope was dashed in the remarkable work~\cite{KV15}, where the lower bound $\rho_{\mathsf{GL}}(n)\gtrsim \sqrt[6]{\log\log n}$ was proven. An improved analysis of the ideas of~\cite{KV15} was conducted in~\cite{KR09}, yielding the estimate $\uprho_{\mathsf{GL}}(n)\gtrsim \log\log n$. An entirely different approach based on the geometry of the Heisenberg group was introduced in~\cite{LN06}.  In combination with the important works~\cite{CK10,CheegerKleinerMetricDiff} it gives a different proof that $\lim_{n\to \infty}\rho_{\GL}(n)=\infty$. In~\cite{CKN09,CKN} the previously best-known lower bound $\rho_{\mathsf{GL}}(n)\gtrsim (\log n)^\delta$ was obtained for an effective (but small) positive universal constant $\delta$.

 Despite these lower bounds, the Goemans--Linial algorithm yields an approximation ratio of $o(\log n)$, so it is asymptotically more accurate than the linear program of~\cite{LLR95,AR98}.  Specifically, in~\cite{CGR08} it was shown that $\rho_{\GL}(n)\lesssim (\log n)^{\frac34}$ and this was improved in~\cite{ALN08} to $\rho_{\GL}(n)\lesssim (\log n)^{\frac12+o(1)}$. See Section~\ref{sec:previous} below for additional background on the  results quoted above. No other polynomial-time algorithm for the Sparsest Cut problem is known (or conjectured) to have an approximation ratio that is asymptotically better than that of the Goemans--Linial algorithm. However, despite major scrutiny by researchers in approximation algorithms, the asymptotic behavior of $\rho_{\GL}(n)$ as $n\to \infty$ remained unknown. Theorem~\ref{thm:main GL lower intro} below resolves this question up to lower-order factors.

 \begin{thm}\label{thm:main GL lower intro} For every integer $n\ge 2$ we have
$\sqrt{\log n}\lesssim \rho_{\mathsf{GL}}(n)\lesssim (\log n)^{\frac12+o(1)}$.
\end{thm}

Since the upper bound on $\uprho_{\mathsf{GL}}(n)$ in Theorem~\ref{thm:main GL lower intro} is due to~\cite{ALN08}, our contribution  is the corresponding lower bound, thus determining the asymptotic behavior of $\uprho_{\mathsf{GL}}(n)$ up to lower order factors.

\subsection{Embeddings}\label{sec:embedding} For $p\in [1,\infty]$, the $L_p$ (bi-Lipschitz) distortion of a separable metric space $(\cM,d)$, denoted $c_p(\cM,d)\in [1,\infty]$, is the infimum over those $D\in [1,\infty]$ for which there exists an embedding $f:\cM\to L_p(\R)$ such that $d(x,y)\le \|f(x)-f(y)\|_{L_p(\R)}\le Dd(x,y)$ for every $x,y\in \cM$.

 Recall that $d_W:\H^5_{\ms{\Z}}\times \H^5_{\ms{\Z}}\to \N\cup\{0\}$  denotes the left-invariant word metric that is induced by the generating set $\mathfrak{S}_2=\{a_1,a_1^{-1},b_1,b_1^{- 1},a_2,a_2^{- 1},b_2,b_2^{- 1}\}$. For every $r\in [0,\infty)$ denote the corresponding (closed) ball of radius $r$ centered at the identity element by $\BB_r=\{h\in \H_{\ms{\Z}}^5:\ d_W(h,1)\le r\}$.  It is well-known (see e.g.~\cite{Bas72}) that $|\BB_r|\asymp r^6$ and $d_W(1,c^r)\asymp\sqrt{r}$ for every $r\in \N$. By~\cite[Theorem~2.2]{LN06}, the metric $d_W$ is bi-Lipschitz equivalent to a metric on $\H^5_{\ms{\Z}}$ that is of negative type. We remark that~\cite{LN06} makes this assertion for a different metric on a larger continuous group that contains $\H_{\ms{\Z}}^5$ as a discrete co-compact subgroup, but by a simple general result (see e.g.~\cite[Theorem~8.3.19]{BBI01}) the word metric $d_W$ is bi-Lipschitz equivalent to the metric considered in~\cite{LN06}. A well-known duality argument (see e.g.~\cite[Lemma~4.5]{Nao10} or~\cite[Section~1]{CKN09}), which to the best of our knowledge was first derived by Rabinovich in the 1990s, establishes that  for every $n\in \N$ the integrality gap $\uprho_{\mathsf{GL}}(n)$ is equal to the supremum of the $L_1$ distortion  $c_1(X,d)$ over all $n$-point metric spaces $(X,d)$ of negative type. Therefore, in order to establish Theorem~\ref{thm:main GL lower intro} it suffices to prove the following theorem.

\begin{thm}\label{thm:distortion R}
For every $r\ge 2$ we have $c_1(\BB_{r},d_W)\asymp \sqrt{\log r}\asymp \sqrt{\log |\BB_r|}$.
\end{thm}
The new content of Theorem~\ref{thm:distortion R} is the lower bound $c_1(\BB_r,d_W)\gtrsim \sqrt{\log r}$. The matching upper bound $c_1(\BB_r,d_W)\lesssim \sqrt{\log r}$ has several proofs in the literature; see e.g.~the discussion immediately following Corollary~1.3 in~\cite{LafforgueNaor} or Remark~\ref{sec:embed intro}  below.  The previously best known estimate~\cite{CKN} was that there exists a universal constant $\updelta>0$ such that  $c_1(\BB_r,d_W)\ge (\log r)^\updelta$.

\begin{remark}\label{sec:embed intro} Theorem~\ref{thm:distortion R} also yields a sharp result for the general problem of finding the asymptotically largest-possible $L_1$ distortion of a finite doubling metric space with $n$ points. A metric space $(\cM,d)$ is said to be $K$-doubling for some $K\in \N$ if every ball in $\cM$ (centered anywhere and of any radius) can be covered by $K$ balls of half its radius. By~\cite{KLMN05},
\begin{equation}\label{eq:descent}
c_1(\cM,d)\lesssim \sqrt{(\log K)\log |\cM|}.
\end{equation}
As noted in~\cite{GKL03}, the dependence on $|\cM|$ in~\eqref{eq:descent}, but with a worse dependence on $K$, follows by combining~\cite{Ass83} and~\cite{Rao99} (this dependence on $K$  was improved  in~\cite{GKL03}).  The metric space $(\H_{\ms{\Z}}^5,d_W)$ is $O(1)$-doubling because $|\BB_{r}|\asymp r^6$ for every $r\ge 1$. Theorem~\ref{thm:distortion R} shows that~\eqref{eq:descent} is sharp  when $K=O(1)$, thus improving  over the previously best-known construction~\cite{LS11} of arbitrarily large $O(1)$-doubling finite metric spaces $\{(\cM_i,d_i)\}_{i=1}^\infty$ for which $c_1(\cM_i,d_i)\gtrsim \sqrt{(\log |\cM_i|)/\log\log |\cM_i|}$. Probably~\eqref{eq:descent} is sharp for all $K\le |\cM|$; conceivably this could be proven by incorporating Theorem~\ref{thm:distortion R}  into the argument of~\cite{JLM11}, but we shall not pursue this here. Theorem~\ref{thm:distortion R} establishes for the first time the existence of a metric space that simultaneously has several useful geometric properties yet poor (indeed, worst-possible) embeddability properties into $L_1$. By virtue of being $O(1)$-doubling, $(\H_{\ms{\Z}}^5,d_W)$ also has Markov type $2$ due to~\cite{DLP13} (which improves over~\cite{NPSS06}, where the conclusion that it has Markov type $p$ for every $p<2$ was obtained). For more on the bi-Lipschitz invariant Markov type and its applications, see~\cite{Bal92,Nao12}. The property of having Markov type $2$ is  shared by the construction of~\cite{LS11}, which is also $O(1)$-doubling, but $(\H_{\ms{\Z}}^5,d_W)$ has additional features that  the example of~\cite{LS11} fails to have. For one, it is a group; for another, by~\cite{Li14,Li16} we know that $(\H_{\ms{\Z}}^5,d_W)$ has Markov convexity $4$ (and no less). (See~\cite{LNP09,MN13} for background on the bi-Lipschitz invariant Markov convexity and its consequences.) By~\cite[Section~3]{MN13} the example of~\cite{LS11} does not have Markov convexity $p$ for any finite $p$. No examples of arbitrarily large finite metric spaces $\{(\cM_i,d_i)\}_{i=1}^\infty$ with bounded Markov convexity (and with the associated Markov convexity constants uniformly bounded) such that $c_1(\cM_i,d_i)\gtrsim\sqrt{\log |\cM_i|}$ were previously known to exist. Analogous statements are known to be impossible for Banach spaces~\cite{MW78}, so it is natural in the context of the Ribe program (see the surveys~\cite{Nao12,Bal13} for more on this research program) to ask whether there is a potential metric version of~\cite{MW78}; the above discussion shows that there is not.
\end{remark}

\begin{remark}\label{rem:dim reduction} It was proved in~\cite{LN14} that for every $p>2$ there is a doubling subset $D_p$ of $L_p(\R)$ that does not admit a bi-Lipschitz embedding into $L_q(\R)$ for $q\in (1,p)$, and furthermore there is $p_0>2$ such that $D_p$ does not even admit a bi-Lipschitz embedding into $L_1(\R)$ for $p>p_0$. By substituting Theorem~\ref{thm:distortion R} into the proof of this statement in~\cite{LN14}, we see that it actually holds true with $p_0=2$.
\end{remark}

The following precise theorem about $L_1$ embeddings that need not be bi-Lipschitz implies Theorem~\ref{thm:distortion R} by considering the special case of the modulus $\upomega(t)=t/D$ for $D\ge 1$ and $t\in [0,\infty)$.
\begin{thm}\label{thm:integral criterion} Fix $r\ge 2$ and a nondecreasing function $\upomega:[0,\infty)\to [0,\infty)$ such that $\omega(s)\le s$ for all $s\ge 0$. Then there exists a mapping $\upphi:\BB_{r}\to L_1(\R)$ that satisfies
\begin{equation}\label{eq:compression omega on ball}
\forall\,x,y\in \BB_{r},\qquad \upomega\big(d_W(x,y)\big)\lesssim \|\upphi(x)-\upphi(y)\|_{L_1(\R)}\le d_W(x,y),
\end{equation}
{\bf \em if and only if}
\begin{equation}\label{eq:integral criterion}
\int_{1}^{ 2r} \frac{\upomega(s)^2}{s^3}\ud s\lesssim 1.
\end{equation}
\end{thm}
The fact that the integrability requirement~\eqref{eq:integral criterion} implies the existence of the desired embedding $\upphi$ is
 due to~\cite[Corollary~5]{Tes08}. The new content of Theorem~\ref{thm:integral criterion} is  that the existence of the embedding $\upphi$ implies~\eqref{eq:integral criterion}. By letting $r\to \infty$ in Theorem~\ref{thm:integral criterion} we see that there is  $\upphi:\H_{\ms{\Z}}^5\to L_1(\R)$ that satisfies
 \begin{equation}\label{eq:compression assumption}
\forall\, x,y\in \Z^5,\qquad  \upomega\big(d_W(x,y)\big)\lesssim \|\upphi(x)-\upphi(y)\|_{L_1(\R)}\le d_W(x,y),
\end{equation}
{\bf \em if and only if}
\begin{equation}\label{eq:1 to infty integral}
\int_{1}^\infty \frac{\upomega(s)^2}{s^3}\ud s\lesssim 1.
\end{equation}

In~\cite{CKN} it was shown that if $\upphi:\H_{\ms{\Z}}^5\to L_1(\R)$ satisfies~\eqref{eq:compression assumption},  then there must exist arbitrarily large $t\ge 2$ for which $\upomega(t)\lesssim t/(\log t)^\updelta$, where $\updelta>0$ is a universal constant. This follows from~\eqref{eq:1 to infty integral} with $\updelta=\frac12$, which is the largest possible constant for which this conclusion holds true.  This positively answers a question that was asked in~\cite[Remark~1.7]{CKN}.  In fact, it provides an even better conclusion, because~\eqref{eq:1 to infty integral} implies that, say, there must exist arbitrarily large $t\ge 4$ for which $$\upomega(t)\lesssim \frac{t}{\sqrt{(\log t)\log\log t}}.$$  (The precise criterion is determined by the integrability condition~\eqref{eq:1 to infty integral}.) Finally, by considering $\upomega(t)=t^{1-\e}/D$ for $\e\in (0,1)$ and $D\ge 1$, we obtain the following noteworthy corollary.

\begin{cor}[$L_1$ distortion of snowflakes]\label{coro:snoflake}
For every $\e\in (0,1)$ we have $c_1\big(\H_{\ms{\Z}}^5,d_W^{1-\e}\big)\asymp \frac{1}{\sqrt{\e}}$.
\end{cor}
The fact that for every $O(1)$-doubling metric space $(X,d)$ we have $c_1(X,d^{1-\e})\lesssim 1/\sqrt{\e}$ follows from an argument of~\cite{LMN05} (see also~\cite[Theorem~5.2]{NS11}). Corollary~\ref{coro:snoflake} shows that this is sharp. More generally, it follows from Theorem~\ref{thm:integral criterion} that for every $r\ge 2$ and $\e\in (0,1)$ we have
$$
c_1\big(\BB_{r},d_W^{1-\e}\big)\asymp\min\left\{\frac{1}{\sqrt{\e}},\sqrt{\log r}\right\}.
$$

\subsection{An endpoint estimate}\label{sec:endpoint}  An equivalent formulation of  the case $k=2$ of Theorem~\ref{thm:isoperimetric discrete} is  that every finitely supported function $\upphi:\H_{\ms{\Z}}^5\to \R$ satisfies the following Poincar\'e-type inequality.
\begin{multline}\label{eq:discrete global intro}
\Bigg(\sum_{t=1}^\infty \frac{1}{t^2}\bigg(\sum_{h\in \H_{\ms{\Z}}^5} \big|\upphi\big(hc^t\big)-\upphi(h)\big|\bigg)^2\Bigg)^{\frac12}\lesssim \sum_{h\in \H_{\ms{\Z}}^5}\sum_{\sigma\in \mathfrak{S}_2}\big|\upphi(h\sigma)-\upphi(h)\big|\\= 2\sum_{h\in \H_{\ms{\Z}}^5} \Big( \big|\upphi(h a_1)-\upphi(h)\big|+\big|\upphi(h b_1)-\upphi(h)\big|+\big|\upphi(ha_2)-\upphi(h)\big|+\big|\upphi(h b_2)-\upphi(h)\big|\Big).
\end{multline}
See Lemma~\ref{lem:functional form discrete} below for a simple proof of this equivalence; One direction is immediate, because Theorem~\ref{thm:isoperimetric discrete} is nothing more than the special case $\upphi=\1_\Omega$ of~\eqref{eq:discrete global intro}. Furthermore, by a straightforward convexity argument that appears in Lemma~\ref{lem:convexity to vector valued}  below, the estimate~\eqref{eq:discrete global intro} has a vector-valued  version which asserts that for every finitely supported function $\phi:\H_\Z^5\to L_1(\R)$ we have
\begin{equation}\label{eq:discrete global intro L1 valued}
\Bigg(\sum_{t=1}^\infty \frac{1}{t^2}\bigg(\sum_{h\in \H_{\ms{\Z}}^5} \big\|\upphi\big(hc^t\big)-\upphi(h)\big\|_{L_1(\R)}\bigg)^2\Bigg)^{\frac12}\lesssim \sum_{h\in \H_{\ms{\Z}}^5}\sum_{\sigma\in \mathfrak{S}_2}\big\|\upphi(h\sigma)-\upphi(h)\big\|_{L_1(\R)}.
\end{equation}
Next, as explained in Lemma~\ref{eq:to loclaize wt}   below (mimicking an argument that appears in Section~3.2 of~\cite{LafforgueNaor}), the vector-valued inequality~\eqref{eq:discrete global intro L1 valued} formally implies its local counterpart, which asserts that there exists a universal constant $\alpha\ge 1$ such that for every $n\in \N$ and every $\upphi:\H_{\ms{\Z}}^5\to L_1(\R)$ we have
\begin{equation}\label{eq:discrete local intro}
\Bigg(\sum_{t=1}^{n^2} \frac{1}{t^2}\bigg(\sum_{h\in \BB_{n}} \big\|\upphi\big(hc^t\big)-\upphi(h)\big\|_{L_1(\R)}\bigg)^2\Bigg)^{\frac12}\lesssim \sum_{h\in \BB_{\alpha n}}\sum_{\sigma\in \mathfrak{S}_2}\big\|\upphi(h\sigma)-\upphi(h)\big\|_{L_1(\R)}.
\end{equation}

To deduce Theorem~\ref{thm:integral criterion} from~\eqref{eq:discrete local intro}, suppose that $r\ge 2$, that $\upomega:[0,\infty)\to [0,\infty)$ is a nondecreasing function satisfying $\omega(s)\le s$ for all $s\ge 0$, and that the mapping $\upphi:\BB_{r}\to L_1(\R)$ satisfies~\eqref{eq:compression omega on ball}. For notational convenience, fix universal constants $\beta\in (0,1)$ and $\gamma\in (1,\infty)$ such that for every $t\in \N$ we have $\beta\sqrt{t}\le d_W(c^t,1)\le \gamma\sqrt{t}$.  Since $\omega(s)\le s$ for all $s\ge 0$, the left hand side of~\eqref{eq:integral criterion} is at most $\log 2r$.  Hence, it suffices to prove Theorem~\ref{thm:integral criterion} for $r\ge 1+\max\{\alpha,\gamma\}$, where $\alpha$ is the universal constant in~\eqref{eq:discrete local intro}. Denote $n=\lfloor \min\{r/(1+\gamma),(r-1)/\alpha\}\rfloor$. If $t\in \{1,\ldots,n^2\}$ and $h\in \BB_n$ then $d_W(hc^t,1)\le n+\gamma \sqrt{t}\le (1+\gamma)n\le r$, and therefore we may apply~\eqref{eq:compression omega on ball} with $x=hc^t$ and $y=h$ to deduce that $\|\upphi(hc^t)-\upphi(h)\|_{L_1(\R)}\gtrsim \omega(d_W(c^t,1))\ge \omega(\beta\sqrt{t})$. Consequently,
\begin{multline}\label{eq:pass to int}
\sum_{t=1}^{n^2} \frac{1}{t^2}\bigg(\sum_{h\in \BB_{n}} \big\|\upphi\big(hc^t\big)-\upphi(h)\big\|_{L_1(\R)}\bigg)^2\gtrsim \sum_{t=1}^{n^2} \frac{|\BB_n|^2\omega\big(\beta\sqrt{t}\big)^2}{t^2}\gtrsim n^{12}\sum_{t=1}^{n^2} \int_t^{t+1}\frac{\omega\big(\beta\sqrt{u/2}\big)^2}{u^2}\ud u\\=\beta^2n^{12}\int_{\frac{\beta}{\sqrt{2}}}^{\frac{\beta\sqrt{n^2+1}}{\sqrt{2}}}\frac{\omega(s)^2}{s^3}\ud s\ge \frac{\beta^2(r/2)^{12}}{\max\{(1+\gamma)^{12},\alpha^{12}\}} \int_1^{\frac{\beta r }{2\max\{1+\gamma,\alpha\}}}\frac{\omega(s)^2}{s^3}\ud s,
\end{multline}
where the second inequality in~\eqref{eq:pass to int} uses the fact that $\omega$ is non-decreasing, the penultimate step of~\eqref{eq:pass to int} uses the change of variable $s=\beta\sqrt{u/2}$, and for the final step of~\eqref{eq:pass to int} recall that $\beta<1$ and the definition of $n$. Recalling that $\beta,\gamma>0$ are universal constants, it follows that
\begin{multline}\label{eq:added to 1}
\int_1^{2r} \frac{\omega(s)^2}{s^3}\ud s=\int_{\frac{\beta r }{2\max\{1+\gamma,\alpha\}}}^{2r}\frac{\omega(s)^2}{s^3}\ud s+ \int_1^{\frac{\beta r }{2\max\{1+\gamma,\alpha\}}}\frac{\omega(s)^2}{s^3}\ud s\\\lesssim  1+\frac{1}{r^{12}} \sum_{t=1}^{n^2} \frac{1}{t^2}\bigg(\sum_{h\in \BB_{n}} \big\|\upphi\big(hc^t\big)-\upphi(h)\big\|_{L_1(\R)}\bigg)^2,
\end{multline}
where the final step of~\eqref{eq:added to 1} uses~\eqref{eq:pass to int} and  that $\omega(s)\le s$ for all $s\ge 0$. By our choice of $n$ we have $h\sigma\in \BB_{\alpha n+1}\subset \BB_r$ for  $h\in \BB_{\alpha n}$ and $\sigma\in S$. So, $\|\upphi(h\sigma)-\upphi(h)\|_{L_1(\R)}\le d_W(h\sigma,h)=1$, by~\eqref{eq:compression omega on ball}. The right hand side of~\eqref{eq:discrete local intro} is therefore at most a universal constant multiple of $|\BB_{\alpha n}|\cdot|\mathfrak{S}_2|\lesssim (\alpha n)^6\lesssim r^6$. Hence, by combining~\eqref{eq:added to 1} with~\eqref{eq:discrete local intro} we obtain that the desired estimate~\eqref{eq:integral criterion} indeed holds true.

 We have thus shown that all of the new  results that were stated above follow from Theorem~\ref{thm:isoperimetric discrete}. The bulk of the ensuing discussion will therefore be devoted to the proof of Theorem~\ref{thm:isoperimetric discrete}. Prior to doing so, we shall now conclude the Introduction by explaining the analytic context of inequality~\eqref{eq:discrete global intro} and describing some interesting (and likely quite challenging) questions that remain open.

 Fix $q\in [2,\infty)$. A Banach space $(X,\|\cdot\|_X)$ is said to be uniformly convex if for every $\e\in (0,1)$ there exists $\delta\in (0,1)$  such that $\|x+y\|_X\le 2(1-\delta)$ for every $x,y\in X$ that satisfy $\|x\|_X=\|y\|_X=1$ and $\|x-y\|_X\ge \e$. If one can take here $\delta\gtrsim_X \e^q$ then  $(X,\|\cdot\|_X)$ is said to have a modulus of uniform convexity of power-type $q$. An important theorem of Pisier~\cite{Pis75} asserts that in the setting of uniform convexity (of Banach spaces), power-type behavior is automatic, i.e., every uniformly convex space admits an equivalent norm whose modulus of uniform convexity is of power-type $q$ for some $q\in [2,\infty)$. By~\cite{LafforgueNaor}, for every $k\in \N$, $p>1$ and $q\ge 2$, if $(X,\|\cdot\|_X)$ has a modulus of uniform convexity of power-type $q$ then every finitely supported mapping $\phi:\H_{\ms{\Z}}^{2k+1}\to X$ satisfies
\begin{equation}\label{eq:discrete global intro X valued}
\Bigg(\sum_{t=1}^\infty \frac{1}{t^{1+\frac{q}{2}}}\bigg(\sum_{h\in \H_{\ms{\Z}}^{2k+1}} \big\|\upphi\big(hc^t\big)-\upphi(h)\big\|_{X}^p\bigg)^{\frac{q}{p}}\Bigg)^{\frac{1}{q}}\lesssim_{X,p,q,k} \bigg(\sum_{h\in \H_{\ms{\Z}}^{2k+1}}\sum_{\sigma\in \mathfrak{S}_k}\big\|\upphi(h\sigma)-\upphi(h)\big\|_{X}^p\bigg)^{\frac{1}{p}}.
\end{equation}

The special case of~\eqref{eq:discrete global intro X valued} when $X=\R$ and $p=q=2$ is due to~\cite{AusNaoTes}, where the quadratic nature of the inequality allows for its proof using representation theory. The general case of~\eqref{eq:discrete global intro X valued} was proved in~\cite{LafforgueNaor} using a  semigroup argument that relies on vector-valued Littlewood--Paley--Stein theory. We remark that both~\cite{AusNaoTes} and~\cite{LafforgueNaor} treat only the case $k=1$ of~\eqref{eq:discrete global intro X valued} but the proofs in~\cite{AusNaoTes,LafforgueNaor}  carry over effortlessly to general $k\in \N$; alternatively, a simple argument shows that one can formally deduce~\eqref{eq:discrete global intro X valued} for any $k\in \N$ from its validity for $k=1$, as explained in   Lemma~\ref{lem:from small dim to bigger} below.

If $X=\R$ and $q=2$, then~\eqref{eq:discrete global intro X valued}  says that every finitely supported $\phi:\H^{2k+1}_{\ms{\Z}}\to \R$ satisfies
\begin{equation}\label{eq:discrete global intro R valued}
\forall\, p>1,\qquad \Bigg(\sum_{t=1}^\infty \frac{1}{t^2}\bigg(\sum_{h\in \H_{\ms{\Z}}^{2k+1}} \big|\upphi\big(hc^t\big)-\upphi(h)\big|^p\bigg)^{\frac{2}{p}}\Bigg)^{\frac{1}{2}}\lesssim_{p,k} \bigg(\sum_{h\in \H_{\ms{\Z}}^{2k+1}}\sum_{\sigma\in \mathfrak{S}_k}\big|\upphi(h\sigma)-\upphi(h)\big|^p\bigg)^{\frac{1}{p}}.
\end{equation}
Thus, the equivalent formulation~\eqref{eq:discrete global intro} of Theorem~\ref{thm:isoperimetric discrete} (when $k=2$) is the endpoint case $p=1$ of~\eqref{eq:discrete global intro R valued}.

\eqref{eq:discrete global intro X valued} asserts that the following operator is bounded from $W^{1,p}(\H_{\ms{\Z}}^{2k+1};X)$ to $\ell_q(\ell_p(\H_{\ms{\Z}}^{2k+1};X))$.
\begin{equation}\label{eq:def T}
\forall\, \phi\in W^{1,p}(\H_{\ms{\Z}}^{2k+1};X),\ \forall (h,t)\in \H_{\ms{\Z}}^{2k+1}\times \N, \qquad T\phi(h,t)\eqdef \frac{1}{t^{\frac12+\frac{1}{q}}} \big(\phi(hc^t)-\phi(h)\big).
\end{equation}
Here, $W^{1,p}(\H_{\ms{\Z}}^{2k+1};X)$ denotes the (discrete) $X$-valued Sobolev space on $\H_{\ms{\Z}}^{2k+1}$ that is induced by the generators $\mathfrak{S}_k$, i.e., the space of all $\phi:\H^{2k+1}_{\ms{\Z}}\to X$ for which the following semi-norm is finite.
$$
\|\phi\|_{W^{1,p}(\H_{\ms{\Z}}^{2k+1};X)}\eqdef \bigg(\sum_{h\in \H_{\ms{\Z}}^{2k+1}}\sum_{\sigma\in \mathfrak{S}_k}\big\|\upphi(h\sigma)-\upphi(h)\big\|_X^p\bigg)^{\frac{1}{p}}.
$$
Also,  $\ell_q(\ell_p(\H_{\ms{\Z}}^{2k+1};X))$ is the space of all $\mathfrak{a}:\H_{\ms{\Z}}^{2k+1}\times \N\to X$ for which the following norm is finite.
$$
\|\mathfrak{a}\|_{\ell_q(\ell_p(\H_{\ms{\Z}}^{2k+1};X))}\eqdef \Bigg(\sum_{t=1}^\infty \bigg(\sum_{h\in \H_{\ms{\Z}}^{2k+1}} \big\|\mathfrak{a}(h,t)\big\|_{X}^p\bigg)^{\frac{q}{p}}\Bigg)^{\frac{1}{q}}.
$$
Our result is that  $T$ is  actually bounded from  $W^{1,1}(\H_{\ms{\Z}}^{2k+1};\R)$ to $\ell_2(\ell_1(\H_{\ms{\Z}}^{2k+1};\R))$ when $k\ge 2$. This assertion suffices for  the geometric and algorithmic applications that are established here, but since our proof relies heavily on the fact that we are dealing with real-valued functions, the availability of such an endpoint estimate in the vector-valued setting  remains an elusive open question that is stated explicitly below. This question is important because its positive solution would probably involve the introduction of a markedly new approach that is likely to be valuable elsewhere.

 \begin{question}\label{Q:endpoint vector valued} Fix $q\ge 2$. Suppose that $X$ is a Banach space  whose modulus of uniform convexity has power-type $q$. Does there exist $k=k(X)\in \N$ for which the operator $T$ that is defined in~\eqref{eq:def T}  is bounded from $W^{1,1}(\H_{\ms{\Z}}^{2k+1};X)$ to $\ell_q(\ell_1(\H_{\ms{\Z}}^{2k+1};X))$? Does $k=2$ suffice here?
 \end{question}

 \begin{remark} In Remark~\ref{eq:finite p}, we described our ongoing work on the case $k=1$ of Theorem~\ref{thm:isoperimetric discrete}, which can be rephrased as follows in terms of the operator $T$  that is defined in~\eqref{eq:def T}. By Lemma~\ref{lem:functional form discrete} below, the isoperimetric-type inequality~\eqref{eq:p version} is equivalent to the assertion that there exists $q>2$ for which  $T$ is bounded from $W^{1,1}(\H_{\ms{\Z}}^{3};\R)$ to $\ell_q(\ell_1(\H_{\ms{\Z}}^{3};\R))$. Also, as we described in Remark~\ref{eq:finite p}, if $T$ is bounded from $W^{1,1}(\H_{\ms{\Z}}^{3};\R)$ to $\ell_q(\ell_1(\H_{\ms{\Z}}^{3};\R))$, then necessarily $q\ge 4$. Thus, even though by~\cite{LafforgueNaor} we know that $T$ is bounded from $W^{1,p}(\H_{\ms{\Z}}^{3};\R)$ to $\ell_2(\ell_p(\H_{\ms{\Z}}^{3};\R))$ for every $p>1$, at the endpoint $p=1$ the operator $T$ is unbounded from $W^{1,1}(\H_{\ms{\Z}}^{3};\R)$ to $\ell_2(\ell_1(\H_{\ms{\Z}}^{3};\R))$, but the infimum over those $q>2$ for which $T$ is bounded from $W^{1,1}(\H_{\ms{\Z}}^{3};\R)$ to $\ell_q(\ell_1(\H_{\ms{\Z}}^{3};\R))$ is finite (and is at least $4$). We have evidence that suggests that this infimum is equal to $4$, but at present we do not have a proof of this statement. Regardless, the above results already establish that in $3$ dimensions the boundedness of $T$ exhibits a somewhat curious jump discontinuity as $p\to 1$. This phenomenon has some geometric applications of independent interest (to dimensionality reduction) that we will present in forthcoming work.
 \end{remark}

If one could somehow construct a Banach space $X$ for which Question~\ref{Q:endpoint vector valued} has a negative answer, then given that we establish here that this question has a positive answer when, say, $X=\ell_p$ for $p\in [1,2]$ and $q=2$ (this follows from~\eqref{eq:discrete global intro L1 valued}, since $\ell_p$ is isometric to a subspace of $L_1(\R)$, by~\cite{Kad58}), it would then be interesting to characterize intrinsically the class of Banach spaces for which the answer to Question~\ref{Q:endpoint vector valued} is positive (at present, the answer to Question~\ref{Q:endpoint vector valued} is unknown  when $X=\ell_p$ for some $p\in (2,\infty)$ and $q=p$). In a related vein, we ask the following intriguing question.

 \begin{question}\label{Q:FH} Let $\mathscr{F}_{\ms{\H}}$ denote the class of all those Banach spaces $X$ into which $\H_{\ms{\Z}}^3$  does not admit a bi-Lipschitz embedding.\footnote{One could also consider classes of Banach spaces that are defined the same way but with $\H_{\ms{\Z}}^3$ replaced with $\H_{\ms{\Z}}^{2k+1}$ for each $k\in \N$. The currently available evidence suggests that the resulting classes of Banach spaces do not actually depend on $k$, but this has not been proven. For simplicity and concreteness, we restrict the discussion here to $k=1$, but all of the questions could be studied for general $k$ as well.  } Could $\mathscr{F}_{\ms{\H}}$ be described using intrinsic geometric properties?
 \end{question}

 By~\cite{CK06,LN06} (see Remark~\ref{rem:embeddings of continuous group} below), $\mathscr{F}_{\ms{\H}}$  contains all the Banach spaces that admit an equivalent uniformly convex norm. By~\cite{CK10},    $\mathscr{F}_{\ms{\H}}$  also contains all the $L_1(\mu)$ spaces. Due to these results, a natural guess for an answer to Question~\ref{Q:FH} would be that $\mathscr{F}_{\ms{\H}}$  coincides with the Banach spaces that have finite cotype (see e.g.~\cite{Mau03}). If true, this would be a remarkable geometric result, since (by the Maurey--Pisier theorem~\cite{MP76}) it would mean that the existence of a bi-Lipschitz embedding of  $\H_{\ms{\Z}}^3$ into a Banach space $X$ implies that {\em every} finite metric space embeds into $X$ with bi-Lipschitz distortion $1+\e$ for every $\e>0$. As a concrete example of a classical Banach space for which it is unknown whether or not it belongs to $\mathscr{F}_{\ms{\H}}$, consider the Schatten--von Neumann trace class $\mathsf{S}_1$ (the space of all operators on $\ell_2$, equipped with the nuclear norm; see e.g.~\cite[\S III.G]{Woj91}). We suspect that $\mathsf{S}_1\in \mathscr{F}_{\ms{\H}}$, and that moreover any embedding of $\BB_n$ into $\mathsf{S}_1$ incurs  bi-Lipschitz distortion that is at least a universal constant multiple of $\sqrt{\log n}$. However, for this strengthening of Theorem~\ref{thm:distortion R} to hold true one would need to find a way to prove it without relying on an isoperimetric-type inequality as we do here. The proofs in~\cite{CK10,CheegerKleinerMetricDiff,CKN} that $\H^3$ does not admit a bi-Lipschitz embedding into $L_1(\R)$ also rely on special properties of real-valued functions through the reduction to questions about subsets of finite perimeter in the continuous Heisenberg group (see Section~\ref{sec:continuous} below). So, perhaps even as a step towards Question~\ref{Q:FH} in its full generality, it would be of great interest to devise an approach that applies to mappings that take value in $\mathsf{S}_1$ rather than $L_1(\R)$.

 By a  fundamental theorem of Ostrovskii~\cite{Ost12}, $X\in \mathscr{F}_{\ms{\H}}$ if and only if the bi-Lipschitz distortion of any embedding of   $\BB_n$ into $X$ tends to $\infty$ as $n\to \infty$. The following question asks for a quantitative refinement of this assertion,  motivated by similar dichotomic phenomena that occur in metric embedding (see~\cite{Men09,MN11-arxiv,Nao12,MN13,ANN15}); by~\cite{AusNaoTes} (in combination with Pisier's renorming theorem~\cite{Pis75}) its answer is positive when $X$ admits an equivalent uniformly convex norm (see ~\cite{Li14,LafforgueNaor} for different proofs of this fact), and by~\cite{CKN} its answer is also positive when $X=L_1(\R)$.

 \begin{question}\label{Q:rate} Suppose that $X\in \mathscr{F}_{\ms{\H}}$. Does this imply that there exists $\theta=\theta(X)>0$ such that any embedding of $\BB_n$ into $X$ incurs bi-Lipschitz distortion at least $(\log n)^\theta$?
 \end{question}

\subsection*{Roadmap} The Introduction contained a description of the new results that are established here, but we did not yet present an overview of the ideas that go into our proof of  Theorem~\ref{thm:isoperimetric discrete}. In particular, we did not yet explain how the fact that the underlying Heisenberg group is of dimension at least $5$ becomes relevant to the proof of  Theorem~\ref{thm:isoperimetric discrete}, despite the fact that this dimension had no role  whatsoever in the precursor~\cite{LafforgueNaor} of this result, i.e., its $\ell_p$-version~\eqref{eq:discrete global intro R valued} for $p>1$. The reason why we are postponing these (important) explanations is that, even though the discrete setting that was described above is needed for applications, our proof of Theorem~\ref{thm:isoperimetric discrete} actually takes place in a continuous setting  that requires the presentation of additional concepts and basic facts. We therefore postpone the overview of the steps of the proof of Theorem~\ref{thm:isoperimetric discrete} to Section~\ref{sec:overview}, which follows Section~\ref{sec:continuous}, where various concepts related to the continuous Heisenberg group are presented, and Section~\ref{sec:reductions}, where initial reductions are performed, including a reduction of Theorem~\ref{thm:isoperimetric discrete} to its continuous counterpart. With this groundwork in place, it becomes more natural to explain the ideas of our proof in Section~\ref{sec:overview}. Since many readers may be familiar with the (standard) continuous setting, those who wish to see the proof overview can read Section~\ref{sec:vert}, which defines vertical perimeter in the continuous setting, then skip to Section~\ref{sec:overview} on first reading, though prior to doing so we recommend familiarization with the notation in Section~\ref{sec:intrinsic Lipschitz} since it treats the (somewhat less-standard)  notion of {\em intrinsic Lipschitz graph}, which has a central role in our proof.

In Section~\ref{sec:intrinsic graphs} we derive, as a crucial new ingredient of our proof of Theorem~\ref{thm:isoperimetric discrete}, a special case of the (continuous counterpart of) the vertical-versus-horizontal isoperimetric inequality of Theorem~\ref{thm:isoperimetric discrete}. It is important to stress that this argument of Section~\ref{sec:intrinsic graphs}  is the {\em only} place where the assumption $k>1$ is used in our proof of Theorem~\ref{thm:isoperimetric discrete}. The remaining steps of the proof work for Heisenberg groups of any dimension, yielding structural information that will be used in future work also when $k=1$.

Section~\ref{sec:corona decompositions} contains definitions and basic facts related to {\em (local) intrinsic  corona decompositions}, which are  structural properties of sets in the continuous Heisenberg group (existence of certain well-behaved multi-scale covers) that will be used to deduce our isoperimetric-type inequality in its full generality from the special case established in Section~\ref{sec:intrinsic graphs}. In Section~\ref{sec:isoperimetry of corona} we prove that sets that admit an intrinsic corona decomposition satisfy the desired isoperimetric-type inequality.

In Section~\ref{sec:decomposing cellular} we state Theorem~\ref{thm:Ahlfors admits corona}, which is a technical structural result asserting that any subset $E$ of the continuous Heisenberg group such that $E$, the complement of $E$ and the boundary of $E$ are all locally Ahlfors regular, admits a suitable local intrinsic corona decomposition.  Thus, by the results of Section~\ref{sec:isoperimetry of corona}, such sets satisfy the desired isoperimetric-type inequality.  To apply this fact, in Section~\ref{sec:decomposing cellular}  we also show how to decompose a cellular set (see Section~\ref{sec:cellularSets})  of finite perimeter in the continuous Heisenberg group into parts that are locally Ahlfors regular as above, in such a way that if we sum up the isoperimetric-type inequalities that follow from Theorem~\ref{thm:Ahlfors admits corona} for each of these parts, then we obtain the desired inequality for the initial cellular set.

As part of the  basic reductions that are contained in Section~\ref{sec:reductions}, a simple approximation argument shows that it suffices to treat cellular sets (see Lemma~\ref{le:to cellular q} below). Therefore, Section~\ref{sec:decomposing cellular} implies  that it remains to prove Theorem~\ref{thm:Ahlfors admits corona}, i.e., to construct a local intrinsic corona decomposition for every  set that satisfies the above local Ahlfors-regularity (i.e., for the set itself, its complement, and its boundary). This construction is performed in Section~\ref{sec:constructing corona}, thus completing the proof of  Theorem~\ref{thm:isoperimetric discrete}.

We conclude this article with Section~\ref{sec:previous}, which contains further historical background on the Sparsest Cut Problem, as well as descriptions of directions for future research (some of which we will pursue in forthcoming works).

\subsection*{Acknowledgments} A.~N. is grateful to Mike Christ and Vincent Lafforgue for helpful discussions.

\section{The continuous setting}\label{sec:continuous}

In what follows, it will be beneficial to allow only one exception to the convention for asymptotic notation that was described after Theorem~\ref{thm:isoperimetric discrete}. Specifically,  since throughout   we will fix an integer $k\in \N$ and many constant factors do depend on $k$, in order to not overburden the notation we shall allow the notations $\lesssim,\gtrsim,\asymp$ to coincide with $\lesssim_k,\gtrsim_k,\asymp_k$, respectively. Nevertheless, our main application is when $k=2$, so in this case the implicit constant factors are in fact universal constants.

\subsection{Definition of continuous Heisenberg group} Fix $k\in \N$. In light of the matrix realization~\eqref{eq:def discrete H} that we  chose in the Introduction for the discrete Heisenberg group $\H^{2k+1}_{\ms{\Z}}$, an obvious way to define the continuous  Heisenberg group $\H^{2k+1}$  is that $\H^{2k+1}$ consists of all those matrices as in~\eqref{eq:def discrete H}, but now with the entries  $x_1,\ldots,x_k,y_1,\ldots,y_k,z$ allowed to be arbitrary real numbers. This would work just fine in the ensuing discussion, but for notational convenience we prefer to consider a different realization of the Heisenberg group that arises by identifying the above real matrix group with its Lie algebra via the exponential map and taking the Baker--Campbell--Hausdorff formula as the definition of the group product. We shall now describe the continuous Heisenberg group in this (standard) way, and proceed to adhere exclusively to this specific realization in what follows.

Fix a Hilbertian norm $\|\cdot\|$ on $\R^{2k+1}$. Fix also an orthonormal basis $X_1,\ldots,X_k,Y_1,\ldots,Y_k,Z$ of $\R^{2k+1}$, so that every $h\in \R^{2k+1}$ can be written as $h=\sum_{i=1}^k\alpha_iX_i+\sum_{i=1}^k\beta_iY_i+\gamma Z$ for some $\alpha_1,\ldots,\alpha_k,\beta_1,\ldots,\beta_k,\gamma\in \R$. Denote $x_i(h)=\alpha_i$, $y_i(h)=\beta_i$ for all $i\in \k$ and $z(h)=\gamma$, i.e., $x_1,\ldots,x_k,y_1,\ldots,y_k,z:\R^{2k+1}\to \R$ are the coordinate functions corresponding to the basis $X_1,\ldots,X_k,Y_1,\ldots,Y_k,Z$. We shall also write $x(h)=\sum_{i=1}^kx_i(h)X_i$, $y(h)=\sum_{i=1}^k y_i(h)Y_i$ and $\pi(h)=x(h)+y(h)$ (thus $\pi(Z)=0$). In what follows, we shall canonically identify the span of $\{X_1,\ldots,X_k,Y_1,\ldots,Y_k\}$ with $\R^{2k}$, so that the mapping $\pi:\R^{2k+1}\to \R^{2k}$ is the  orthogonal projection onto the first $2k$ coordinates.

 For every $u,v\in \R^{2k+1}$ write $\omega(u,v)=\sum_{i=1}^k \big(x_i(u)y_i(v)-y_i(u)x_i(v)\big)$, i.e., $\omega(u,v)$ is the standard symplectic form on $\R^{2k}$ applied to the vectors $\pi(u),\pi(v)$. In particular, $\omega(X_i,Y_i)=-\omega(Y_i,X_i)=1$ and  for every two basis elements $u,v\in \{X_1,\ldots,X_k,Y_1,\ldots,Y_k,Z\}$ such that $\{u,v\}\neq \{X_i,Y_i\}$ for all $i\in \k$ we have $\omega(u,v)=0$. The (continuous) Heisenberg group $\H^{2k+1}$ is defined to be the group whose underlying set is $\R^{2k+1}$, equipped with the following product:
  \begin{equation}\label{eq:def group product omega}
 \forall\, u,v\in \H^{2k+1},\qquad u v\eqdef u+v+\frac{\omega(u,v)}{2} Z.
 \end{equation}
 It is straightforward to check that this turns $\H^{2k+1}$ into a group whose identity element is the all-$0$ vector $\0\in \H^{2k+1}$ and the inverse of $h\in \H^{2k+1}$ is $-h$. The resulting group is isomorphic to the real matrix group (with usual matrix multiplication) that we described above. Indeed,  consider the isomorphism that assigns to every $h\in \R^{2k+1}$ the following upper triangular $k+2$ by $k+2$ matrix.
\begin{equation*}
\psi(h)\eqdef  \begin{pmatrix} 1 & x_1(h) &x_2(h) & \dots &x_k(h)& w(h)\\
  0 & 1&0 &\dots& 0 & y_1(h)\\
  \vdots & \ddots & \ddots  & \ddots & \vdots&\vdots\\
            \vdots & \ddots & \ddots& \ddots &0&y_{k-1}(h)\\
              \vdots & \ddots & \ddots &\ddots&1&y_k(h)
              \\
              0 & \dots & \dots &\dots&0&1
                       \end{pmatrix}, \quad\mathrm{where}\quad w(h)\eqdef z(h)+\frac12\sum_{i=1}^k x_i(h)y_i(y).
\end{equation*}
The discrete Heisenberg group $\H_{\ms{\Z}}^{2k+1}$ can now be realized as the subgroup of $\H^{2k+1}$ that is generated by $\{X_1,\ldots,X_k,Y_1,\ldots,Y_k\}$. Note that the $z$-coordinate of an element in $\H_{\ms{\Z}}^{2k+1}$ is either an integer or a half-integer. From now on, we shall work exclusively with this specific realization of $\H_{\ms{\Z}}^{2k+1}$.

In order to avoid confusing multiplication by scalars with the group law of $\H^{2k+1}$, for every $h=\sum_{i=1}^k\alpha_iX_i+\sum_{i=1}^k\beta_iY_i+\gamma Z\in \H^{2k+1}$ and $t\in \R$ it will be convenient to use the exponential notation  $h^t=\sum_{i=1}^kt\alpha_iX_i+\sum_{i=1}^kt\beta_iY_i+t\gamma Z$. One directly computes from~\eqref{eq:def group product omega} that for every $i\in \k$ and $s,t\in \R$ we have $[X_i^s,Y_i^t]=Z^{st}$, where we use the standard commutator notation $[u,v]=uvu^{-1}v^{-1}$ for every $u,v\in \H^{2k+1}$. Given a subset $A\subset \H^{2k+1}$ we shall denote its linear span by $\langle A\rangle$. Thus,  $\langle h\rangle =\langle \{h\}\rangle=\{h^t:t\in \R\}$ is the one-parameter subgroup generated by $h\in \H^{2k+1}$.

 \subsubsection{Horizontal derivatives} If $f:\H^{2k+1}\to \R$ is smooth, then for every $i\in \k$ its {\em horizontal derivatives} in directions $X_i,Y_i$ are defined by
$$
\forall\, h\in \H^{2k+1},\qquad X_if(h)\eqdef \frac{\partial f}{\partial x_i}(h)-\frac12 y_i(h)\frac{\partial f}{\partial z}(h) \quad \mathrm{and}\quad Y_if(h)\eqdef \frac{\partial f}{\partial y_i}(h)+\frac12 x_i(h)\frac{\partial f}{\partial z}(h).
$$
These are differential operators corresponding to left-invariant horizontal vector fields on $\H^{2k+1}$, so they are left-invariant.  That is, if we let $\rho_g$ be the left action $\rho_gf(h)=f(g^{-1}h)$ for all $g\in \H^{2k+1}$, then $\{X_i\}_{i=1}^k$ and $\{Y_i\}_{i=1}^k$ commute with $\rho_g$.

The {\em horizontal gradient} of $f$ is then defined as follows.
$$
\forall\, h\in \H^{2k+1},\qquad \nabla_\H f(h)\eqdef \big(X_1f(h),\ldots,X_kf(h),Y_1f(h),\ldots,Y_kf(h)\big)\in \R^{2k}.
$$
Thus, for $p\in (0,\infty]$ the $p$-norm of the horizontal gradient of $f$ is given by.
\begin{align*}
\|\nabla_\H f(h)\|_{\ell_p^{2k}}&=\bigg(\sum_{i=1}^k \big(|X_if(h)|^p+|Y_if(h)|^p\big)\bigg)^{\frac{1}{p}}\\&=\bigg(\sum_{i=1}^k\Big|\frac{\partial f}{\partial x_i}(h)-\frac12 y_i(h)\frac{\partial f}{\partial z}(h)\Big|^p+\sum_{i=1}^k\Big|\frac{\partial f}{\partial y_i}(h)+\frac12 x_i(h)\frac{\partial f}{\partial z}(h)\Big|^p \bigg)^{\frac{1}{p}}.
\end{align*}

\subsubsection{Lines and planes} Denote from now on $\mathsf{H}\eqdef \langle X_1,\dots, X_k,Y_1,\dots, Y_k\rangle=\pi(\H^{2k+1})$. This $2k$--dimensional linear subspace $\mathsf{H}$ is called the subspace of \emph{horizontal vectors}.  If $U\subset \mathsf{H}$ is an affine hyperplane in $\mathsf{H}$ (thus $U$ is a translate of a linear subspace of $\mathsf{H}$ of dimension $2k-1$), then we call $V=\pi^{-1}(U)\subset \H^{2k+1}$ the \emph{vertical plane} lying over $U$.

A {\em horizontal line} in $\H^{2k+1}$ is a coset of the form $w  \langle h \rangle$ for some $w\in\H^{2k+1}$ and $h\in \mathsf{H}$. If $u\in \H^{2k+1}$, then $P_u=u \mathsf{H}$ is called the \emph{horizontal plane} centered at $u$.  Equivalently, $P_u$ is the union of all of the horizontal lines that pass through $u$.  We can write
\begin{equation}\label{eq:plane equation}
  P_u=\left\{hZ^{z(u)+\frac{\omega(\pi(u),\pi(h))}{2}}\mid h\in \mathsf{H}\right\}.
\end{equation}
(See Section~\ref{sec:planes and angles}.)

The term \emph{plane} in $\H^{2k+1}$ will always stand for either a vertical plane or a horizontal plane.  By \eqref{eq:plane equation}, any plane in $\H^{2k+1}$ is an affine hyperplane in $\R^{2k+1}$.  Since $\omega$ is nondegenerate on $\mathsf{H}$, for any affine hyperplane $Q\subset \R^{2k+1}$, either $Q$ is vertical or there is a unique $u\in \H^{2k+1}$ such that $Q=P_u$.  Every plane $P\subset \H^{2k+1}$ separates $\H^{2k+1}$ into two \emph{half-spaces}, which we denote by $P^+$ and $P^-$.

\subsubsection{The Carnot--Carath\'eodory metric}
Suppose that $a,b\in \R$ satisfy $a<b$. If $\gamma\from [a,b]\to \H^{2k+1}$ is a curve such that the coordinate functions $\{x_i\circ\gamma\}_{i=1}^k$, $\{y_i\circ \gamma\}_{i=1}^k$, $z\circ \gamma$ are all Lipschitz, then the tangent vector $\gamma'(t)\in \H^{2k+1}$ is defined for almost all $t\in [a,b]$.  We say that $\gamma$ is a \emph{horizontal curve} if $\gamma'(t)\in P_{\gamma(t)}$ for almost every $t\in [a,b]$. Equivalently, for $\gamma$ to be horizontal we require that $\gamma^{-1}\gamma'\in\mathsf{H}$ almost everywhere. The length of a horizontal curve $\gamma\from [a,b]\to \H^{2k+1}$ is defined to be $\ell(\gamma)=\int_a^b\|\pi(\gamma'(t))\|\ud t=\ell(\pi\circ \gamma)$. Note that segments of horizontal lines are horizontal curves.  Since $[X_i^s,Y_i^t]=Z^{st}$ for every $i\in \k$ and $s,t\in \R$, any two points in $\H^{2k+1}$ can be connected by a horizontal curve consisting of segments of horizontal lines. We can therefore define the Carnot--Carath\'eodory distance between any two points $v,w\in \H^{2k+1}$ by
$$d(v,w)\eqdef \inf \big\{\ell(\gamma)\mid \gamma:[0,1]\to \H^{2k+1} \text{ is\ horizontal,}\  \gamma(0)=v,\ \gamma(1)=w\big\}.$$
See e.g.~\cite{Gro96,Mon02} for more on this metric; it suffices to recall that there exists $C_k\in [1,\infty)$ such that
\begin{equation}
  \label{eq:metric approximation}
 \forall\, h\in \H^{2k+1},\qquad   d(\0,h)\le \sum_{i=1}^k |x_i(h)|+\sum_{i=1}^k |y_i(h)|+4\sqrt{|z(h)|}\le C_kd(\0,h),
  \end{equation}
and
\begin{equation} \label{eq:metric lower bound}
 \forall\, h\in \H^{2k+1},\qquad  d(\0,h) \ge \|\pi(h)\|.
\end{equation}

In what follows, for every $r\in (0,\infty)$ we shall denote by $B_r\subset \H^{2k+1}$ the open ball in the metric $d$ that is centered at $\0$, i.e., $B_r=\{h\in \H^{2k+1}:\ d(h,\0)<r\}$. For every  $\Omega\subset \H^{2k+1}$, the Lipschitz constant of a mapping $f:\Omega\to \R$ relative to the metric $d$ will be denoted by $\|f\|_{\Lip(\Omega)}$. For every $s\in (0,\infty)$ the $s$-dimensional Hausdorff measure that is induced by the metric $d$ on $\H^{2k+1}$ will be denoted below by $\CH^s$. Note that it follows from~\eqref{eq:metric approximation} that  $\H^{2k+1}$ is a metric space with Hausdorff dimension $2k+2$. One checks that $d$ is a left-invariant metric on $\H^{2k+1}$ and that the Lebesgue measure  is a Haar measure of $\H^{2k+1}$. Thus $\CH^{2k+2}$ is proportional to the Lebesgue measure. Also, for every $h\in \H^{2k+1}$ and $r\in [0,\infty)$ the open ball of radius  $r$ in the metric $d$ is $hB_r=B_r(h)$.

If $h\in \H^{2k+1}$ and $v\in \mathsf{H}\setminus \{\0\}$, then the horizontal line $h\langle v\rangle$ is a geodesic.  Consequently
\begin{equation}\label{eq:horizontal geodesics}
  d(h,hv)=\dEuc\big(\pi(h),\pi(hv)\big)=\|v\|,
\end{equation}
where $\dEuc(\cdot,\cdot)$ denotes the Euclidean metric that is induced by the Hilbertian norm $\|\cdot\|$. Furthermore, by the classical isoperimetric inequality for curves in $\R^2$, one can show (e.g.~\cite{Mon02}) that
$$
\forall\, r\in (0,\infty),\qquad d\big(\0,Z^{\pi r^2}\big)=2\pi r.
$$

\subsubsection{Isometries and scaling automorphisms}
The isometry group of $\H^{2k+1}$ acts transitively on $\H^{2k+1}$ by left-multiplication.  There are also many isometries that fix the identity element $\0$.  Any element $v\in \H^{2k+1}$ can be written uniquely as $v=hZ^t$ for some $h\in \mathsf{H}$ and $t\in \R$. If $f\from \R^{2k}\to \R^{2k}$ is a Euclidean isometry that leaves the symplectic form $\omega$ invariant, then the map $hZ^t\mapsto f(h)Z^t$ is an isometry of $\H^{2k+1}$.  The group consisting of isometries of this type is isomorphic to the unitary group $U(k)$, and it acts transitively on the unit sphere in $\mathsf{H}$.

Another important class of automorphisms of $\H^{2k+1}$ is the following family of {\em scalings}.  For every $t\in (0,\infty)$ define $\s_t:\H^{2k+1}\to \H^{2k+1}$ by
$$
\s_t\bigg(\sum_{i=1}^k \alpha_iX_1+\sum_{i=1}^k\beta_i Y_i+\gamma Z\bigg)\eqdef \sum_{i=1}^k t\alpha_iX_1+\sum_{i=1}^kt\beta_i Y_i+t^2\gamma Z.
$$
It is straightforward to check  that $\s_t$ is an automorphism of $\H^{2k+1}$, that it sends horizontal curves to horizontal curves, and that if $\gamma$ is a horizontal curve, then $\ell(\s_t\circ \gamma)=t\ell(\gamma)$.  It follows that $\s_t$ acts as a scaling of the Carnot--Carath\'eodory metric, i.e., $d(\s_t(u),\s_t(v))=td(u,v)$ for all $u,v\in \H^{2k+1}$.

\begin{remark}\label{rem:embeddings of continuous group} Analogously to Question~\ref{Q:FH}, one can  consider the class $\mathscr{F}_{\ms{\H}}^{\ms{\R}}$ of those Banach spaces $X$ into which the continuous Heisenberg group $\H^3$ does not admit a bi-Lipschitz embedding. By~\cite{LN06,CK06}, $\mathscr{F}_{\ms{\H}}^{\ms{\R}}$  contains  all the Banach spaces that have the Radon--Nikod\'ym property (RNP), hence in particular all the reflexive Banach spaces and all the separable dual Banach spaces (see e.g.~\cite[Chapter~5]{BL00} for these assertions, as well as much more on the RNP). By~\cite{CK10},    $\mathscr{F}_{\ms{\H}}^{\ms{\R}}$  also contains all the $L_1(\mu)$ spaces. We ask whether $X\in \mathscr{F}_{\ms{\H}}^{\ms{\R}}$ implies that also $L_1(\mu;X)\in \mathscr{F}_{\ms{\H}}^{\ms{\R}}$; if true, this would be a beautiful strengthening of~\cite{CK10}. We also do not know whether  $\mathscr{F}_{\ms{\H}}^{\ms{\R}}$  contains all the noncommutative $L_1$ spaces (see e.g.~\cite{PX03}), but  since the Schatten--von Neumann trace class $\mathsf{S}_1$  has the RNP (since $\mathsf{S}_1$ is separable and it is the dual of the space of all the compact operators on $\ell_2$, equipped with the operator norm), by~\cite{LN06,CK06} we do know that  $\mathsf{S}_1\in \mathscr{F}_{\ms{\H}}^{\ms{\R}}$ (recall that it is unknown whether or not $\mathsf{S}_1$ belongs to $\mathscr{F}_{\ms{\H}}$). Since $\H^3_{\ms{\Z}}$ is bi-Lipschitz equivalent to a subset of $\H^3$, we trivially have $\mathscr{F}_{\ms{\H}}\subset \mathscr{F}_{\ms{\H}}^{\ms{\R}}$. This inclusion is strict because by~\cite{BL08} we know that $\H^3_{\ms{\Z}}$ (indeed, any locally-finite metric space) admits a bi-Lipschitz embedding into a Banach space $Z$ with the RNP (in fact, one can take $Z$ to be the $\ell_2$ direct sum of the finite dimensional spaces $\{\ell_\infty^n\}_{n=1}^\infty$, so $Z$ can even be reflexive). By considering the re-scaled copies $\s_t(\H^3_{\ms{\Z}})\subset \H^3$ as $t\to 0^+$, one sees that if $X$ belongs to $\mathscr{F}_{\ms{\H}}$, then for any non-principal ultrafilter $\mathcal{U}$, the ultrapower $X^\mathcal{U}$ belongs to $\mathscr{F}_{\ms{\H}}^{\ms{\R}}$ (see e.g.~\cite{Hein80} for  ultrapowers of Banach spaces). By~\cite{BS72},  $X^\mathcal{U}$ has the RNP if and only if $X$ admits an equivalent uniformly convex norm. Hence, it follows from~\cite{LN06,CK06} that  $\mathscr{F}_{\ms{\H}}$ contains all the Banach spaces that admit an equivalent uniformly convex norm (different proofs of this fact were obtained in ~\cite{AusNaoTes,Li14,LafforgueNaor}). As in Question~\ref{Q:FH}, it is natural to ask for an intrinsic geometric characterization of the class $\mathscr{F}_{\ms{\H}}^{\ms{\R}}$. We chose to highlight Question~\ref{Q:FH} rather than this continuous variant due to its local nature, i.e., because by~\cite{Ost12} the balls $\BB_n$ embed with distortion $O_X(1)$ (as $n\to \infty$) into a Banach space $X$ if and only if $X\in \mathscr{F}_{\ms{\H}}$. In other words, membership in $\mathscr{F}_{\ms{\H}}$ is determined by the geometry of finite subsets of the Banach space in question while membership in $\mathscr{F}_{\ms{\H}}^{\ms{\R}}$ is an inherently infinite dimensional property.
\end{remark}

\subsubsection{Measure and perimeter}\label{sec:measure perimeter} If $E\subset \H^{2k+1}$ has smooth or piecewise smooth boundary,  then $\partial E$ has Hausdorff dimension $2k+1$.  In this paper we will primarily deal with sets $E\subset \H^{2k+1}$ such that $\cH^{2k+1}(\partial E)<\infty$ or surfaces $A\subset \H^{2k+1}$ of topological dimension $2k+1$ such that $\cH^{2k+1}(A)<\infty$; such surfaces include the boundaries of the cellular sets that are introduced in Section~\ref{sec:cellularSets} below and the intrinsic Lipschitz graphs that are introduced in Section~\ref{sec:intrinsic Lipschitz} below. In Section~\ref{sec:delta monotone is close to plane}, however, we will need to work with limits of cellular sets in order to apply a compactness argument. The boundaries of such limits can be more complicated, and therefore we will need in Section~\ref{sec:delta monotone is close to plane} the more sophisticated notion of {\em Heisenberg perimeter}.  See~\cite{FSSCRectifiability} for the definition of this notion, which is also recalled in~\cite{CK10,CKN}. We will not need to work here with this definition directly, so it suffices to state that as in~\cite{FSSCRectifiability} one associates to every measurable subset $E\subset \H^{2k+1}$ a measure $\Per_E$ that is called the \emph{perimeter measure} of $E$ and has the properties that are stated in Propositions~\ref{prop:per cellular}, \ref{prop:per scale invariant}, \ref{prop:FSSC Cacciopoli} and \ref{prop:compact lower semicontinuous} below, all of which are proven in~\cite{FSSCRectifiability}.

\begin{prop}\label{prop:per cellular}
  There exists $c_k\in (0,\infty)$ such that if $E\subset \H^{2k+1}$ has piecewise smooth boundary and $U\subset \H^{2k+1}$ is an open set, then  $\Per_{ E}(U)=c_k\cH^{2k+1}(U\cap \partial E)$.
\end{prop}

\begin{prop}\label{prop:per scale invariant}
 The  perimeter measure is scale-invariant in the sense that for any $E\subset \H^{2k+1}$, any open set $U$ and any $t\in (0,\infty)$, we have $\Per_{\s_t(E)}(\s_t(U))=t^{2k+1}\Per_{E}(U)$.
\end{prop}
If $\Per_{E}(B)<\infty$ for every ball $B\subset \H^{2k+1}$ we shall say that $E$ has locally finite perimeter.
\begin{prop}\label{prop:FSSC Cacciopoli}
  If $E\subset \H^{2k+1}$ has locally finite perimeter, then $\Per_E(U)\lesssim\cH^{2k+1}(U\cap \partial E)$ for every open set $U\subset \H^{2k+1}$.
\end{prop}
Finally, the following compactness result appears in~\cite[Theorem~1.28]{GNIsoSob} and~\cite[Theorem~2.10]{FSSCRectifiability}.
\begin{prop}\label{prop:compact lower semicontinuous} Let $\{E_i\}_{i=1}^\infty$ be a sequence of measurable subsets of $\H^{2k+1}$  such that for every ball $B\subset \H^{2k+1}$ we have $\sup_{i\in \N} \Per_{E_i}(B)<\infty$. Then there exists a subsequence $\{i_n\}_{n=1}^\infty\subset \N$ and a measurable  $E\subset \H^{2k+1}$ such that $\lim_{n\to \infty} \one_{E_{i_n}}=\one_E$ in $L_1^{\mathrm{loc}}(\CH^{2k+2})$ and $\Per_{E}(U)\le \sup_{n\in \N} \Per_{E_{n_i}}(U)$ for every open $U\subset \H^{2k+1}$.
\end{prop}

 \subsubsection{Cellular sets}\label{sec:cellularSets} Recall that the discrete Heisenberg group  $\H^{2k+1}_{\ms{\Z}}$ is the subgroup of $\H^{2k+1}$ generated by $\{X_1,Y_1,\dots,X_k,Y_k\}$.  The unit cube $[-\frac{1}{2},\frac{1}{2}]^{2k+1}$ is a fundamental domain for the left action of $\H^{2k+1}_{\ms{\Z}}$ on $\H^{2k+1}$.  Its translates tile $\H^{2k+1}$ by parallelepipeds and give $\H^{2k+1}$ the structure of a polyhedral complex $\tau^{2k+1}$. In what follows, we call this complex the \emph{unit lattice} and say that a subset of $\H^{2k+1}$ is \emph{cellular} if it is the union of closed $(2k+1)$--cells of the unit lattice.

If $E\subset \H^{2k+1}$ is measurable and satisfies $\cH^{2k+1}(\partial E)<\infty$, then $E$ can be approximated by scalings of cellular sets.  To justify this (well-known, but not easily quotable) fact, we need to recall the following relative version of the isoperimetric inequality for $\H^{2k+1}$. Below, and in what follows, given $E\subset \H^{2k+1}$ we use the notation $E^\cc=\H^{2k+1}\setminus E$.  Also, it is convenient to use throughout the ensuing discussion the following standard notation for normalized integrals. If $\mu$ is a measure on $\H^{2k+1}$ and $U\subset \H^{2k+1}$ is $\mu$-measurable with $\mu(U)>0$, then for every $f\in L_1(U)$ write $$\fint_U f\ud\mu\eqdef \frac{1}{\mu(U)}\int_U f\ud\mu.$$
\begin{lemma}\label{lem:relativeIsoperimetric}
  There is $\Lambda_k\in (0,\infty)$ such that if $E\subset \H^{2k+1}$ is measurable and $r\in (0,\infty)$, then
  \begin{equation}\label{eq:local iso}
    \left(\frac{\vol(E\cap B_r)\cdot \vol(E^{\cc}\cap B_r)}{\vol(B_r)^2}\right)^{\frac{2k+1}{2k+2}} \lesssim \frac{r \area(\partial E\cap B_{\Lambda_kr})}{\vol(B_{\Lambda_kr})}.
  \end{equation}
\end{lemma}
\begin{proof}
  The  Poincaré--Sobolev inequality on $\H^{2k+1}$ (see~\cite{HajKos}  or equation~(2.5) in~\cite{CKN}) asserts that there exists $\Lambda_k\in (0,\infty)$ such that every smooth $f:\H^{2k+1}\to \R$ satisfies
    \begin{equation}\label{eq:PoincareSobolev}
    \left(\fint_{B_r\times B_r} |f(x)-f(y)|^{\frac{2k+2}{2k+1}} \ud \CH^{2k+2}(x)\ud
      \CH^{2k+2}(y)\right)^{\frac{2k+1}{2k+2}}\lesssim r \fint_{B_{\Lambda_k r}} \big\|\nabla_\H f(x)\|_{\ell_1^{2k}}\ud\CH^{2k+2}(x),
  \end{equation}
The desired estimate~\eqref{eq:local iso} follows by applying~\eqref{eq:PoincareSobolev} when $f$ is a smooth
  approximation to $\one_E$.
\end{proof}

\begin{lemma}\label{lem:cellular approx}
  Let $E\subset \H^{2k+1}$ be a measurable set and suppose that $\cH^{2k+1}(\partial E)<\infty$.  Then for every $\rho\in (0,\infty)$, there is a set $F=F_\rho$ with the following properties.
\begin{itemize}
\item $\s_{1/\rho}(F_\rho)$ is a cellular set,
\item $\cH^{2k+1}(\partial F_\rho)\lesssim \cH^{2k+1}(\partial E)$,
\item $\cH^{2k+2}(E\symdiff F_\rho)\lesssim \rho \cH^{2k+1}(\partial E).$
\end{itemize}
Here, $A\symdiff B\eqdef (A\setminus B)\cup (B\setminus A)$ is (as usual) the symmetric difference of $A,B\subset \H^{2k+1}$.
\end{lemma}
\begin{proof} It suffices to prove Lemma~\ref{lem:cellular approx} when $\rho=1$. Indeed, for general $\rho\in (0,\infty)$ apply the case $\rho=1$ of  Lemma~\ref{lem:cellular approx} with $E$ replaced by $\s_{1/\rho}(E)$ to construct a cellular set $F_0$, and then take $F=\s_\rho(F_0)$.  Consider the following set of cells in the unit lattice $\tau^{2k+1}$.
  $$\mathcal{C}\eqdef \left\{C \in \tau^{2k+1}\mid \cH^{2k+2}(C\cap E) > \frac{\cH^{2k+2}(C)}{2}\right\}.$$
Write $F=\bigcup_{C\in \mathcal{C}} C$.  Then $F$ is cellular by design, and  we shall next show that  $F$ satisfies the remaining two desired properties.

  Let $A\subset \tau^{2k+1}$ be the set of cells that intersect $\partial F$. Then $\cH^{2k+1}(\partial F)\asymp |A|$.  For each $C\in \tau^{2k+1}$ let $h_C\in C$ be the center of $C$.  If $C\in A$, then $C$ is adjacent to a cell $D\in \tau^{2k+1}$ with $|\{C,D\}\cap \mathcal{C}|=1$.  Let $r=\diam (C)=\diam ([-\frac{1}{2},\frac{1}{2}]^{2k+1})$ and let $B=B_{2r}(h_C)=h_CB_{2r}$, so that $C\cup D\subset B$.  Then $\cH^{2k+2}(E\cap B)\ge \cH^{2k+2}(C)/2\asymp \CH^{2k+2}(B)$ and $\cH^{2k+2}(E^\cc\cap B)\ge \cH^{2k+2}(D)/2\asymp \CH^{2k+2}(B)$. Lemma~\ref{lem:relativeIsoperimetric} therefore implies that $\cH^{2k+1}(\partial E\cap B_{2\Lambda_k r}(h_C))\gtrsim 1.$ Since the balls $\{B_{2\Lambda_k r}(h_C)\}_{C\in \tau^{2k+1}}$ cover $\H^{2k+1}$ with finite multiplicity, it follows that
  $$\cH^{2k+1}(\partial E)\gtrsim \sum_{C\in A} \cH^{2k+1}\big(\partial E\cap B_{2\Lambda_k r}(h_C)\big)\gtrsim |A|\asymp \cH^{2k+1}(\partial F).$$

  Next, by the definition of $F$ for every $C\in \tau^{2k+1}$ we have
 $$
    \min\big\{\cH^{2k+2}(E\cap C),\cH^{2k+2}(E^{\cc} \cap C)\big\}= \cH^{2k+2}\big((E\symdiff F)\cap C\big),
 $$
 and also we trivially have
 $$
    \max\big\{\cH^{2k+2}(E\cap C),\cH^{2k+2}(E^{\cc} \cap C)\big\}\ge \frac{\cH^{2k+2}(C)}{2}.
    $$
 Consequently, if we denote $B=B_r(h_C)$, then
  \begin{align*}
    \cH^{2k+2}(E\cap B)\cdot \cH^{2k+2}(E^{\cc}\cap B)&\ge \cH^{2k+2}(E\cap C)\cdot  \cH^{2k+2}(E^\cc\cap C)\\
                                                  &\ge \cH^{2k+2}\big((E\symdiff F)\cap C\big)\cdot  \frac{\cH^{2k+2}(C)}{2},
  \end{align*}
  and, using Lemma~\ref{lem:relativeIsoperimetric} (and recalling that $r=\diam ([-\frac{1}{2},\frac{1}{2}]^{2k+1})\asymp 1$),
  $$\cH^{2k+1}(\partial E\cap B_{\Lambda_k r})\gtrsim \cH^{2k+2}\big((E\symdiff F)\cap C\big)^{\frac{2k+1}{2k+2}}\gtrsim \cH^{2k+2}((E\symdiff F)\cap C).$$
  These estimates imply that
  \begin{multline*}\cH^{2k+1}(\partial E)\gtrsim \sum_{C\in \tau^{2k+1}} \cH^{2k+1}\big(\partial E\cap B_{2\Lambda_kr}(h_C)\big)\\\gtrsim \sum_{C\in \tau^{2k+1}}\cH^{2k+2}\big((E\symdiff F)\cap C\big)=\cH^{2k+2}(E\symdiff F).\tag*{\qedhere}\end{multline*}
\end{proof}

We shall record for future use the following lemma, which is a consequence of Lemma~\ref{lem:cellular approx}.

\begin{lemma}\label{lem:finite Hausdorff finite perimeter}
  Suppose that $E\subset \H^{2k+1}$ is a measurable set such that $\cH^{2k+1}(B_r\cap \partial E)<\infty$ for all $r\in (0,\infty)$. Then for any open set $U\subset \H^{2k+1}$ we have $\Per_E(U)\lesssim \cH^{2k+1}(U\cap \partial E)$.
\end{lemma}
\begin{proof}
  First, consider the case $\cH^{2k+1}(\partial E)<\infty$.  By Lemma~\ref{lem:cellular approx}, there is a sequence $\{F_{2^{-i}}\}_{i=1}^\infty$ of scalings of cellular sets such that $\lim_{i\to \infty} \one_{F_{2^{-i}}}= \one_E$ in $L_1^{\mathrm{loc}}(\CH^{2k+2})$ and $ \cH^{2k+1}(\partial F_i)\lesssim \cH^{2k+1}(\partial E)$ for every $i\in \N$.
  Since by Proposition~\ref{prop:per cellular} we have $\Per_{F_{2^{-i}}}(\H^{2k+1})\asymp \cH^{2k+1}(\partial F_{2^{-i}})$, it follows that $\Per_{F_{2^{-i}}}(\H^{2k+1})\lesssim \cH^{2k+1}(\partial E)$. So, by Proposition~\ref{prop:compact lower semicontinuous} we have $\Per_{E}(\H^{2k+1})\lesssim  \cH^{2k+1}(\partial E)<\infty$.  By Proposition~\ref{prop:FSSC Cacciopoli} we therefore have $\Per_E(U)\lesssim \cH^{2k+1}(U\cap \partial E)$ for every open set $U\subset \H^{2k+1}$.

  For the general case, for every $i\in \N$ let $E_i=B_i\cap E$. So $\lim_{i\to \infty} \one_{E_{i}}=\one_E$ in $L_1^{\mathrm{loc}}(\CH^{2k+2})$.  By the validity of Lemma~\ref{lem:finite Hausdorff finite perimeter} in the special case $\cH^{2k+1}(\partial E)<\infty$ that we just established,  Proposition~\ref{prop:compact lower semicontinuous} implies  $E$ has locally finite perimeter. So, by Proposition~\ref{prop:FSSC Cacciopoli} we have $\Per_E(U)\lesssim\cH^{2k+1}(U\cap \partial E)$ for every open set $U\subset \H^{2k+1}$.
\end{proof}

\subsection{Vertical perimeter}\label{sec:vert}
The vertical perimeter of a subset of the Heisenberg group was introduced in~\cite{LafforgueNaor}.  If $E \subset \H^{2k+1}$ and $s\in \R$, then for every $\rho\in \R$ denote
$$\mathsf{D}_sE\eqdef E\symdiff EZ^{2^{2s}}\subset \H^{2k+1}.$$
If $U,E\subset \H^{2k+1}$ are measurable, then define $\vpfl{U}(E)\from \R\to \R$ by
\begin{equation}\label{eq:vert per on U}
\forall\, \rho\in \R,\qquad \vpfl{U}(E)({\rho})\eqdef\frac{\vol(U\cap \mathsf{D}_{\rho}E)}{2^{{\rho}}}=\frac{1}{2^\rho}\int_{U}\big|\1_E(x)-\1_E\big(xZ^{-2^{2\rho}}\big)\big|\ud \CH^{2k+2}(x).
\end{equation}
Note that $\vpfl{U}(E)$ is left-invariant in the sense that $\vpfl{gU}(gE)=\vpfl{U}(E)$ for every $g\in \H^{2k+1}$. When $U=\H^{2k+1}$ it will be convenient to use the simpler notation $\vpfl{\H^{2k+1}}(E)=\vpf(E)$. The \emph{vertical perimeter} of $E$ is then defined to be the quantity
\begin{equation*}\label{eq:def v bar}
\|\vpf(E)\|_{L_2(\R)} =\bigg(\int_{-\infty}^\infty \vpf(E)(\rho)^2\ud \rho\bigg)^{\frac12}.
\end{equation*}

\begin{remark}\label{rem:box}
  The following inequality and example may shed some light on the behavior of the vertical perimeter.  If $E\subset \H^{2k+1}$ is measurable and $s\in \R$, then there is a geodesic $\gamma_s$ from $\0$ to $Z^{2^{2s}}$ with $\ell(\gamma_s)\asymp 2^s$.  We can connect each point $h$ of $E$ to a point of $EZ^{2^{2s}}$ by a translate of $\gamma_s$, and if $hZ^{2^{2s}}\in E\symdiff EZ^{2^{2s}}$, then $h\gamma_s$ crosses $\partial E$.  It follows from these observations that
 \begin{equation}\label{eq:vertical at most area of boundary}
  \vpf(E)(s)\lesssim \frac{\cH^{2k+1}(\partial E) d(\0,Z^{2^{2s}})}{2^s}\lesssim \cH^{2k+1}(\partial E).
  \end{equation}
  Thus, $\|\vpf(E)(s)\|_{L_\infty(\R)}\lesssim \cH^{2k+1}(\partial E)$. For an example where this  bound is sharp, consider the box $$C_r\eqdef [-r,r]^{2k}\times [-r^2,r^2]=\Big\{h:\in \H^{2k+1}:\ \max\big\{|x_1(h)|,|y_1(h)|,\ldots,|x_k(h)|,|y_k(h)|, \sqrt{|z(h)|}\big\}\le r\Big\},$$
  for some $r\in (0,\infty)$. A straightforward computation shows that
  $$
  \forall\, s\in \R,\qquad \vpf(C_r)(s)\asymp \begin{cases}
      2^s r^{2k} & \text{ if } 2^s<r,\\
      \frac{r^{2k+2}}{2^s} & \text{ if } 2^s\ge r.
    \end{cases}
  $$
  So, $\vpf(C_r)(s)$ has a maximum of roughly $r^{2k+1}\asymp \cH^{2k+1}(\partial C_r)$ near $s=\log_2 r$ and exponentially decaying tails.  In general, $\vpf(E)(s)$ measures the ``bumpiness'' of $\partial E$ at scale $2^s$, so the vertical perimeter of $E$ can include contributions from many different scales.
\end{remark}

We record for ease of future use some elementary properties of the vertical perimeter function.
\begin{lemma}\label{lem:vPerProps}
  Suppose that $A,B,U\subset \H^{2k+1}$ are measurable sets of finite measure. Then
  \begin{enumerate}
  \item $\vpf(A)=\vpf(A^\cc)$,
  \item $|\vpf(A)-\vpf(B)|\le \vpf(A\symdiff B)$,
  \item $\vpf(A\cap B)\le \vpfl{B}(A)+\vpf(B)$,
  \item $\vpf(\s_t(A))(\rho)=t^{2k+1}\vpf(A)(\rho-\log_2 t)$ for every $t\in (0,\infty)$ and every $\rho\in \R$.
  \end{enumerate}
\end{lemma}
\begin{proof}  The first property follows from the fact that $\mathsf{D}_\rho A=\mathsf{D}_\rho A^\cc$.  For the second property, note that since the Boolean operation $\symdiff$ is associative and commutative, we have
  \begin{multline*}
    2^{\rho}\big|\vpf(A)(\rho)-\vpf(B)(\rho)\big|\stackrel{\eqref{eq:vert per on U}}{=} \Big|\big\|\1_{\mathsf{D}_\rho A}\big\|_{L_1(\CH^{2k+2})}-\big\|\1_{\mathsf{D}_\rho B}\big\|_{L_1(\CH^{2k+2})}\Big|
    \le \cH^{2k+2}(\mathsf{D}_\rho A\symdiff \mathsf{D}_\rho B)\\
    = \cH^{2k+2}\big((A\symdiff B)\symdiff (AZ^{2^{2\rho}}\symdiff BZ^{2^{2\rho}})\big)
    = \cH^{2k+2}(\mathsf{D}_\rho(A\symdiff B))=2^\rho \vpf(A\symdiff B)(\rho).
  \end{multline*}
  For the third property, it suffices to show that $\mathsf{D}_\rho(A\cap B)\subset (B\cap \mathsf{D}_\rho A) \cup \mathsf{D}_\rho(B)$. To this end, fix $h\in \mathsf{D}_\rho(A\cap B)$. If $|\{hZ^{-2^{2\rho}},h\}\cap B|=1$, then    $h\in \mathsf{D}_\rho(B)$, as required. Otherwise either $\{hZ^{-2^{2\rho}},h\}\subset B$ or $\{hZ^{-2^{2\rho}},h\}\cap B=\emptyset$. The latter case cannot occur because it implies that $hZ^{-2^{2\rho}},h\notin A\cap B$, and therefore  $h\not \in \mathsf{D}_\rho(A\cap B)$ in contradiction to our choice of $h$. It remains to consider the case $\{hZ^{-2^{2\rho}},h\}\subset  B$, but then the assumption $h\in \mathsf{D}_\rho(A\cap B)$ implies that also $h\in \mathsf{D}_\rho A$, as required. Finally, the fourth property holds since for every $(t,\rho)\in (0,\infty)\times \R$ we have
  \begin{multline*}
  \vpf\big(\s_t(A)\big)(\rho)=\frac{\cH^{2k+2}\Big(\s_t(A)\symdiff \big(\s_t(A)Z^{2^{2\rho}}\big)\Big)}{2^{\rho}}=\frac{\CH^{2k+2}\Big(\s_t(A)\symdiff \s_t\big(AZ^{\frac{1}{t^2}2^{2^{\rho}}}\big)\Big)}{2^\rho}
  \\=\frac{\cH^{2k+2}(\s_t(\mathsf{D}_{\rho-\log_2 t}(A)))}{2^{\rho}}=\frac{t^{2k+2}\CH^{2k+2}\big(\mathsf{D}_{\rho-\log_2 t}A\big)}{2^{\rho}}
  =t^{2k+1}\vpf(A)(\rho-\log_2 t).\tag*{\qedhere}
  \end{multline*}
\end{proof}

\subsection{Intrinsic Lipschitz graphs}\label{sec:intrinsic Lipschitz}
One of the key ideas in the present work is that surfaces in the Heisenberg group can be approximated by \emph{intrinsic Lipschitz graphs}.  These were introduced by Franchi, Serapioni, and Serra Cassano~\cite{FSSC06} as a natural class of surfaces in the Heisenberg group that are analogues of Lipschitz graphs in Euclidean space.  David and Semmes proved~\cite{DavidSemmesSingular}  that a set in Euclidean space is uniformly rectifiable if it has a certain decomposition (a corona decomposition; see Section~\ref{sec:corona decompositions} below) into pieces that can be approximated by Lipschitz graphs.  In Section~\ref{sec:corona decompositions} below we will define a similar intrinsic version of this decomposition for surfaces in $\H^{2k+1}$.

We recall the definitions of intrinsic graphs and intrinsic Lipschitz graphs, which are due to~\cite{FSSC06}. Let $V\subset \H^{2k+1}$ be any vertical plane; typically, we take $\{h\in \H^{2k+1}\mid x_k(h)=0\}$. Let $W=V^{\perp}$ be the orthogonal complement of $V$, i.e., the horizontal line through the origin that is perpendicular to $V$. Suppose that $U\subset V$ and that $f:U\to W$ is any function. Then the {\em intrinsic graph} of $f$ is defined to be the following set.
$$
\Gamma_f\eqdef \{u f(u)\mid u\in U\}.
$$
Any such set is called below an {\em intrinsic graph over} $U$. It is often convenient to discuss intrinsic graphs of real-valued functions, which is achieved by identifying $W$ with $\R$. Specifically, if $w\in W$ is a vector of unit Euclidean length and $f:U\to \R$, then we slightly abuse notation by denoting
\begin{equation}\label{eq:gamma f}
\Gamma_f\eqdef \left\{u w^{f(u)} \mid u\in U\right\}.
\end{equation}

Since the unitary group $U(k)$ acts transitively on the unit sphere of $\R^{2k}$, if $\Gamma$ is any intrinsic graph, then there exists an isometric automorphism of $\H^{2k+1}$ that takes $\Gamma$ to an intrinsic graph over $U\subset \{h\in \H^{2k+1}\mid x_k(h)=0\}$ of the form $\Gamma=\{v X_k^{f(v)}\mid v\in U\}$. If in addition we have $U=V=\{h\in \H^{2k+1}\mid x_k(h)=0\}$, then $\Gamma$ bounds two half spaces $\{v X_k^{t}\mid v\in V\ \wedge\ t\ge f(v)\}$ and $\{v X_k^{t}\mid v\in V\ \wedge\ t\le f(v)\}$, which we refer to below as $\Gamma^+$ and $\Gamma^-$ (the superscripts $+$ and $-$ will be used interchangeably and are not intended to indicate the sign of the inequality).

If $V\subset \H^{2k+1}$ is a vertical plane and $W=V^{\perp}$, then denote the orthogonal projection onto $W$  by $\mathsf{Proj}_W:\R^{2k+1}\to W$.  For every $\lambda\in (0,1)$ define the following {\em double cone}.
$$
\mathrm{Cone}_\lambda(V)\eqdef \left\{h\in \H^{2k+1}\mid \|\mathsf{Proj}_W(h)\|>\lambda d(\0,h)\right\}.
$$
When $V=\{h\in \H^{2k+1}\mid x_k(h)=0\}$  we shall use the simpler notation
$$
\mathrm{Cone}_\lambda=\mathrm{Cone}_\lambda\big(\{h\in \H^{2k+1}\mid x_k(h)=0\}\big)=\left\{h\in \H^{2k+1}\mid |x_k(h)|>\lambda d(\0,h)\right\}.
$$
If $U\subset V$ and $\Gamma\subset \H^{2k+1}$ is an intrinsic graph over $U$, then we say that $\Gamma$ is an {\em intrinsic Lipschitz graph over} $U$ if there exists $\lambda\in (0,1)$ such that for every $h\in \Gamma$ we have $(h \mathrm{Cone}_\lambda(V))\cap \Gamma=\emptyset$. In this case we say that $\Gamma$ is an intrinsic $\lambda$-Lipschitz graph. Correspondingly, $\Gamma^+$ and $\Gamma^-$ are then called intrinsic $\lambda$-Lipschitz half-spaces.

We warn that this definition is different from the definition of the Lipschitz constant of an intrinsic Lipschitz graph that is given in~\cite{FSSCDifferentiability}, but we believe that the above modification is more natural in the present context. One reason for this is that the above definition  leads to Theorem~\ref{thm:intrinsic nonlinear hahn banach} below, which is a sharp Lipschitz extension statement that is an intrinsic version of the classical nonlinear Hahn--Banach theorem~\cite{McS34} (see also~\cite{WW75} or~\cite[Lemma~1.1]{BL00}).

In the concrete setting $V=\{h\in \H^{2k+1}\mid x_k(h)=0\}$, given  $\lambda\in (0,1)$  and an intrinsic graph $\Gamma$, the requirement that $(h \mathrm{Cone}_\lambda)\cap \Gamma=\emptyset$ for all $h\in \Gamma$ is the same as the requirement  that for every $w_1,w_2\in \Gamma$ we have $|x_k(w_1)-x_k(w_2)|\le \uplambda d(w_1,w_2)$.

\begin{lemma}\label{lem:translate graph}
  If $h\in \H^{2k+1}$ and $\Gamma$ is an intrinsic $\lambda$-Lipschitz graph over $U\subset V$, then $h\Gamma$ is an intrinsic $\lambda$-Lipschitz graph over some set $U'\subset V$.
\end{lemma}
\begin{proof}
  Let $f\from U\to \R$ be such that $\Gamma=\{v X_k^{f(v)}\mid v\in U\}$ and let $v_0\in V$, $t_0\in \R$ be the unique elements such that $h=v_0 X_k^{t_0}$.  Then
  $$h\Gamma=\big\{hvX_k^{f(v)}\mid v\in U\big\} = \big\{v_0X_k^{t_0}vX_k^{f(v)}\mid v\in U\big\} = \big\{v_0vZ^{t_0y_k(v)}X_k^{f(v)+t_0}\mid v\in U\big\}.$$
  If $v'=v_0vZ^{t_0y_k(v)}$, then $v=v_0^{-1} v' Z^{-t_0 y_k(v_0^{-1}v')}$. So, by setting $U'\eqdef \{v_0vZ^{t_0y_k(v)}\mid v\in U\}$ and defining $f'\from U'\to \R$ by
  $$\forall\, v'\in U',\qquad f'(v')\eqdef f\left(v_0^{-1} v' Z^{-t_0 y_k(v_0^{-1}v')}\right)+t_0,$$
 we conclude that $h\Gamma$ is the intrinsic graph of $f'$ over $U'$.  Since the metric on $\H^{2k+1}$ is invariant under translation, $h\Gamma$ is also $\lambda$-Lipschitz.
\end{proof}

\begin{remark}
  If $\lambda\in (0,1)$ and $\Gamma=\Gamma_f$ is an intrinsic graph of a function $f:V\to W$ as above, then there is no relationship between the requirement that $\Gamma$ is an intrinsic Lipschitz graph and the requirement that $f$ is a Lipschitz function between the corresponding subsets of $\H^{2k+1}$ (equipped with the Carnot--Carath\'{e}odory metric); see Remark~3.13 in~\cite{FSSCDifferentiability}.
\end{remark}

A variant of the following extension result is Theorem~4.25 of~\cite{FSSCDifferentiability}, with the difference being that in~\cite{FSSCDifferentiability} this is obtained with a weaker bound on the Lipschitz constant of the extended graph. Control on the Lipschitz constant of the extension will be important for us in Section~\ref{sec:STRLipschitzGraphs} below, so we will include a proof of Theorem~\ref{thm:intrinsic nonlinear hahn banach} here, while also taking the opportunity to state it in a sharp form.  The important property for us will turn out to be the weaker assertion that if the Lipschitz constant of $\Gamma$ is small, then the Lipschitz constant of the extended graph $\widetilde{\Gamma}$ is also small.

\begin{thm}\label{thm:intrinsic nonlinear hahn banach}
Fix $\lambda\in (0,1)$ and a vertical plane $V\subset \H^{2k+1}$. Write $W=V^\perp$. Suppose that $\emptyset\neq U\subset V$ and that $f:U\to W$ is a function such that $\Gamma=\Gamma_f$ is an intrinsic $\lambda$-Lipschitz graph over $U$. Then there exists an intrinsic $\lambda$-Lipschitz graph  over $V$, denoted $\widetilde{\Gamma}$, such that $\Gamma\subset \widetilde{\Gamma}$.
\end{thm}

\begin{proof}
We may assume without loss of generality that $V=\{h\in \H^{2k+1}\mid x_k(h)=0\}$ and $W=\langle X_k\rangle$. Due to Lemma~\ref{lem:translate graph}, by applying a translation we may suppose that $\0\in \Gamma$.

The double cone $\mathrm{Cone}_\lambda$ consists of two halves, which we denote by $\mathrm{Cone}_\lambda^+$ and $\mathrm{Cone}_\lambda^-$. Namely,
$$
\mathrm{Cone}_\lambda^+\eqdef\left\{h:\in \H^{2k+1}\mid x_1(h)>\lambda d(\0,h)\right\}\quad\mathrm{and}\quad \mathrm{Cone}_\lambda^-\eqdef\left\{h:\in \H^{2k+1}\mid x_1(h)<-\lambda d(\0,h)\right\}.
$$
Thus $\mathrm{Cone}_\lambda^-=(\mathrm{Cone}_\lambda^+)^{-1}$, and for all $g,h\in \H^{2k+1}$, we have $g\in h\mathrm{Cone}_\lambda^+$ if and only if $h\in g\mathrm{Cone}_\lambda^-$.  Define
$$
\widetilde{\Gamma}^+\eqdef \bigcup_{h\in \Gamma} h \mathrm{Cone}_\lambda^+.
$$
We claim that $\widetilde{\Gamma}\eqdef \partial \widetilde{\Gamma}^+$ satisfies the desired conditions.

Observe first we have
\begin{equation}\label{eq:cone product}
\mathrm{Cone}_\lambda^+\cdot \mathrm{Cone}_\lambda^+\subset \mathrm{Cone}_\lambda^+.
\end{equation}
Indeed,  if $h_1, h_2\in \mathrm{Cone}_\lambda^+$, then $x_1(h_1)> \lambda d(\0,h_1)$ and $x_1(h_2)>\lambda d(\0,h_2)$. Hence,
\begin{multline}\label{eq:check cone product}
x_1(h_1h_2)=x_1(h_1)+x_1(h_2)>\lambda\big(d(\0,h_1)+d(\0,h_2)\big)\\=\lambda\big(d(\0,h_1^{-1})+d(\0,h_2)\big)\ge \lambda d(h_1^{-1},h_2)=\lambda d(\0,h_1h_2).
\end{multline}
By the definition of $\mathrm{Cone}_\lambda^+$,  \eqref{eq:check cone product} means that $h_1h_2\in \mathrm{Cone}_\lambda^+$, thus completing the verification of~\eqref{eq:cone product}.  A key consequence is that \begin{equation}\label{eq:Gamma cone equals Gamma}
  \widetilde{\Gamma}^+\cdot \mathrm{Cone}_\lambda^+\subset \widetilde{\Gamma}^+.
\end{equation}

We claim that $\widetilde{\Gamma}$ is an intrinsic graph, that is, that for any $v\in V$,
$$
\Phi_v\eqdef\left\{t\in \R\mid vX_k^t\in \widetilde{\Gamma}^+\right\}\subset \R
$$
is a half-line.  First, we show that $\Phi_v\ne\emptyset$.  Since $\0\in \Gamma$, it suffices to show that $vX_k^t\in \mathrm{Cone}_\lambda^+$ when $t$ is sufficiently large.  For every $t\in \R$ we have
\begin{equation}\label{eq:u0 computation 1}
d\big(vX_k^t,\0\big)\le d\big(\0, v\big)+d\big(v,vX_k^t\big)\stackrel{\eqref{eq:horizontal geodesics}}{=}  d(\0,v)+|t|.
\end{equation}
Since $x_k(vX_k^t)=t$, it follows from~\eqref{eq:u0 computation 1} that
$$
t>\frac{\lambda}{1-\lambda} d(\0,v) \implies x_k\big(vX_k^t\big)> \lambda d\big(vX_k^t,\0\big).
$$
Thus there exists $t>0$ for which $vX_k^t\in \mathrm{Cone}_\lambda^+\subset \widetilde{\Gamma}^+$. Hence $\Phi_v\neq\emptyset$.  Similarly,
\begin{equation}\label{eq:intersect negative cone}
  \exists\, t\in (-\infty,0), \qquad vX_k^t\in \mathrm{Cone}_\lambda^-.
\end{equation}

Next, we claim that $\Phi_v\neq \R$.  In fact, we show that for all $h\in \Gamma$,
\begin{equation}\label{eq:extension disjoint from cone}
  h \mathrm{Cone}_\lambda^- \cap \widetilde{\Gamma}^+=\emptyset.
\end{equation}
By \eqref{eq:intersect negative cone}, this implies the claim.

Suppose that $w\in h \mathrm{Cone}_\lambda^- \cap \widetilde{\Gamma}^+$; then there is an $h'\in \Gamma$ such that $w\in h' \mathrm{Cone}_\lambda^+$.  We have $h^{-1}w\in \mathrm{Cone}_\lambda^-$, so $w^{-1}h\in \mathrm{Cone}_\lambda^+$ and therefore
$$h=w \cdot w^{-1}h\in h' \mathrm{Cone}_\lambda^+ \cdot  \mathrm{Cone}_\lambda^+\subset h' \mathrm{Cone}_\lambda^+.$$
Since $h,h'\in \Gamma$, this contradicts the underlying assumption (since $f$ is an intrinsic $\lambda$-Lipschitz function) that $(h' \mathrm{Cone}_\lambda)\cap \Gamma=\emptyset$.

Since $X_k^s\in \mathrm{Cone}_\lambda^+$ for $s>0$, by~\eqref{eq:Gamma cone equals Gamma} we see that if $t\in \Phi_v$, then $vX_k^{t+s}=(vX_k^t) X_k^s\in \widetilde{\Gamma}^+\cdot \mathrm{Cone}_\lambda^+\subset \widetilde{\Gamma}^+$, i.e., $t+s\in \Phi_v$. Thus $t\in \Phi_v \implies [t,\infty)\subset \Phi_v$. Since $\Phi_v$ is an open nonempty proper subset of $\R$, this shows that $\Phi_v=(\upphi(v),\infty)$ for some $\upphi(v)\in \R$.  Consequently, $\widetilde{\Gamma}=\Gamma_\upphi$, i.e., $\widetilde{\Gamma}$ is an intrinsic graph over $V$.

If $h=uX_k^{f(u)}\in \Gamma$, then $uX_k^{f(u)+s}\in h \mathrm{Cone}_\lambda^+$ for all $s>0$ and $uX_k^{f(u)+s}\in h \mathrm{Cone}_\lambda^-$ for all $s<0$.  Equation \eqref{eq:extension disjoint from cone} then implies that $h\in \widetilde{\Gamma}$, so $\widetilde{\Gamma}\supset \Gamma$.

It remains to prove that $\widetilde{\Gamma}$ is an intrinsic $\lambda$-Lipschitz graph. To this end, the goal is to show that if $h_1,h_2\in \widetilde{\Gamma}$, then $h_2\notin h_1\mathrm{Cone}_\lambda$.  Because $h_1\in \widetilde{\Gamma}$, for every $\epsilon>0$, there is a $p\in \widetilde{\Gamma}^+$ such that $d(p,h_1)<\epsilon$.  Let $g\in \Gamma$ be such that $p\in g \mathrm{Cone}_\lambda^+$. Using~\eqref{eq:cone product} we therefore have
\begin{equation}\label{eq:g plus inclusions}
p \mathrm{Cone}_\lambda^+\subset g\mathrm{Cone}_\lambda^+ \cdot \mathrm{Cone}_\lambda^+\subset g\mathrm{Cone}_\lambda^+\subset \widetilde{\Gamma}^+.
\end{equation}
Since this holds for every $\epsilon>0$, we have $h_1 \mathrm{Cone}_\lambda^+\subset \widetilde{\Gamma}^+$.  Since $\widetilde{\Gamma}^+$ is an open subset of $\H^{2k+1}$ and $h_2$ is assumed to be in its boundary $\widetilde{\Gamma}$, it follows from~\eqref{eq:g plus inclusions} that $h_2\notin h_1 \mathrm{Cone}_\lambda^+$. By the same reasoning with the roles of $h_1$ and $h_2$ reversed, we also see that
$$
h_1\notin h_2 \mathrm{Cone}_\lambda^+= h_2\left( \mathrm{Cone}_\lambda^-\right)^{-1} \implies h_2\notin h_1 \mathrm{Cone}_\lambda^-.
$$
Since $h_1 \mathrm{Cone}_\lambda=(h_1 \mathrm{Cone}_\lambda^+)\cup (h_1 \mathrm{Cone}_\lambda^-)$, this concludes the proof that $h_2\notin h_1\mathrm{Cone}_\lambda$.
\end{proof}

\section{Initial reductions}\label{sec:reductions}

The purpose of this section is to present simple reductions between various questions so as to set the stage for the main steps of the proof of Theorem~\ref{thm:isoperimetric discrete}. In Section~\ref{sec:reductions discrete} below we shall relate various inequalities for functions on the discrete Heisenberg group. Some of these reductions were already quoted in Section~\ref{sec:endpoint} to deduce Theorem~\ref{thm:integral criterion} (and hence also all of the new results that were stated in the Introduction) from Theorem~\ref{thm:isoperimetric discrete}. The arguments in Section~\ref{sec:reductions discrete} are for the most part due to~\cite{LafforgueNaor}, except that in~\cite{LafforgueNaor} they were carried out only in the special case $k=1$. This was done because when~\cite{LafforgueNaor} was written the relevance of the assumption $k\ge 2$ in Theorem~\ref{thm:isoperimetric discrete} was not known, and it was therefore  believed that all of the Heisenberg groups $\{\H_{\ms{\Z}}^{2k+1}\}_{k=1}^\infty$  have identical roles in the present context. Since it turns out that the underlying dimension does play a role, for completeness we include in Section~\ref{sec:reductions discrete} below an explanation of the straightforward modifications of the arguments in~\cite{LafforgueNaor} so as to obtain the desired statements for general $k\in \N$. Section~\ref{sec:continutous to discrete} below shows via a partition of unity argument that in order to prove Theorem~\ref{thm:isoperimetric discrete} it suffices to prove a certain (singular) Sobolev-type inequality on the continuous Heisenberg group. Again, the argument is included for completeness, but it follows steps that were carried out in~\cite{LafforgueNaor} for the special case $k=1$. Section~\ref{sec:continuous reductions} below is devoted to reductions that take place entirely in the continuous setting.  We first show that in order to establish the Sobolev-type inequality  of Section~\ref{sec:continutous to discrete}  it suffices to prove an isoperimetric-type inequality  on the continuous Heisenberg group that is analogous to its discrete counterpart that we stated in Theorem~\ref{thm:isoperimetric discrete}; this step is a quick application of the coarea formula. We conclude Section~\ref{sec:continuous reductions} by establishing the technical (but convenient) statement  that it suffices to prove the desired continuous isoperimetric-type inequality for cellular sets.

\subsection{Reductions in the discrete setting}\label{sec:reductions discrete} We shall start by presenting simple arguments that allow one to relate various inequalities on the discrete Heisenberg group. For ease of reference in future work, we shall present our statements  in a  form that is more general than what is needed here (but the proofs of the more general case are identical to the special case that we will use).

\subsubsection{From sets to functions} The following lemma explains how certain isoperimetric-type inequalities, as in Theorem~\ref{thm:isoperimetric discrete} or in~\eqref{eq:p version}, are equivalent to certain functional inequalities. This step is a standard and simple application of the classical co-area formula, in combination with convexity.
\begin{lemma}\label{lem:functional form discrete}
Fix $k\in \N$, a sequence $\{w_t\}_{t=1}^\infty\subset [0,\infty)$, $C\in (0,\infty)$ and $q\in [1,\infty)$. The inequality
\begin{equation}\label{eq:set version p}
\bigg(\sum_{t=1}^\infty w_t|\partial^t_{\vv}\Omega|^q\bigg)^{\frac{1}{q}}\le C|\partial_{\hh}\Omega|
\end{equation}
holds true for every finite $\Omega\subset \H_{\ms{\Z}}^{2k+1}$  if and only if  every finitely supported  $\f:\H_{\ms{\Z}}^{2k+1}\to \R$ satisfies
\begin{equation}\label{eq:cheeger version p}
\Bigg(\sum_{t=1}^\infty w_t\bigg(\sum_{h\in \H_{\ms{\Z}}^{2k+1}} \big|\upphi\big(hZ^t\big)-\upphi(h)\big|\bigg)^q\Bigg)^{\frac{1}{q}}\le \frac{C}{2}\sum_{h\in \H_{\ms{\Z}}^{2k+1}}\sum_{\sigma\in \mathfrak{S}_k}\big|\upphi(h\sigma)-\upphi(h)\big|.
\end{equation}
\end{lemma}

\begin{proof} The estimate~\eqref{eq:set version p} is nothing more than the special case $\f=\1_\Omega$ of~\eqref{eq:cheeger version p}.  Conversely,  to show that~\eqref{eq:cheeger version p} follows from~\eqref{eq:set version p}, suppose that $\f:\H_{\ms{\Z}}^{2k+1}\to \R$ is finitely supported and for every $u\in \R$ denote $\Omega_u=\{h\in \H_{\ms{\Z}}^{2k+1}:\ \f(h)<u\}$. Since $\f$ is finitely supported, $\Omega_u$ is finite if $u\le 0$ and $\H_{\ms{\Z}}^{2k+1}\setminus \Omega_u$ is finite when $u>0$. In both cases~\eqref{eq:set version p} holds true with $\Omega$ replaced by $\Omega_u$, and therefore~\eqref{eq:cheeger version p} holds true with $\f$ replaced by $\1_{\Omega_u}$ for every $u\in \R$, i.e.,
\begin{multline}\label{eq:to integrate u discrete}
\Bigg(\sum_{t=1}^\infty w_t\bigg(\sum_{h\in \H_{\ms{\Z}}^{2k+1}} \big|\1_{\{\upphi(hZ^t)<u\}}-\1_{\{\upphi(h)<u\}}\big|\bigg)^q\Bigg)^{\frac{1}{q}} \le \frac{C}{2}\sum_{h\in \H_{\ms{\Z}}^{2k+1}}\sum_{\sigma\in \mathfrak{S}_k}\big|\1_{\{\upphi(h\sigma)<u\}}-\1_{\{\upphi(h)<u\}}\big|.
\end{multline}
The right hand side of~\eqref{eq:cheeger version p} is equal to the integral of the right hand side of~\eqref{eq:to integrate u discrete} with respect to $u$. Hence, it suffices to use the triangle inequality in $\ell_q$ to conclude as follows.
\begin{align*}
\frac{C}{2}\sum_{h\in \H_{\ms{\Z}}^{2k+1}}\sum_{\sigma\in \mathfrak{S}_k}\big|\upphi(h\sigma)-\upphi(h)\big|&\stackrel{\eqref{eq:to integrate u discrete}}{\ge} \!\int_{-\infty}^\infty \Bigg(w_t\bigg(\sum_{h\in \H_{\ms{\Z}}^{2k+1}} \big|\1_{\{\upphi(hZ^t)<u\}}-\1_{\{\upphi(h)<u\}}\big|\bigg)^q\Bigg)^{\frac{1}{q}} \ud u\\&\, \ge \Bigg(\sum_{t=1}^\infty w_t\bigg(\sum_{h\in \H_{\ms{\Z}}^{2k+1}} \int_{-\infty}^\infty \big|\1_{\{\upphi(hZ^t)<u\}}-\1_{\{\upphi(h)<u\}}\big|\ud u\bigg)^q\Bigg)^{\frac{1}{q}}\\&\, =\Bigg(\sum_{t=1}^\infty w_t\bigg(\sum_{h\in \H_{\ms{\Z}}^{2k+1}} \big|\upphi\big(hZ^t\big)-\upphi(h)\big|\bigg)^q\Bigg)^{\frac{1}{q}}. \tag*{\qedhere}
\end{align*}
\end{proof}

\subsubsection{From real-valued to vector-valued} The following lemma is another simple use of convexity which shows that certain real-valued functional inequalities are equivalent to the analogous inequalities for functions that take value in an $L_p(\mu)$ space.

\begin{lemma}\label{lem:convexity to vector valued} Fix $k\in \N$, a sequence $\{w_t\}_{t=1}^\infty\subset [0,\infty)$, and $p,q\in (0,\infty)$ with $p\le q$. Suppose also that $\eta\in (0,\infty)$ is such that for every finitely supported $\f:\H_{\ms{\Z}}^{2k+1}\to \R$ we have
\begin{equation}\label{eq:real case}
\eta\Bigg(\sum_{t=1}^\infty w_t\bigg(\sum_{h\in \H_{\ms{\Z}}^{2k+1}} \big|\upphi\big(hZ^t\big)-\upphi(h)\big|^p\bigg)^\frac{q}{p}\Bigg)^{\frac{1}{q}}\le \bigg(\sum_{h\in \H_{\ms{\Z}}^{2k+1}}\sum_{\sigma\in \mathfrak{S}_k}\big|\upphi(h\sigma)-\upphi(h)\big|^p\bigg)^{\frac{1}{p}}.
\end{equation}
Then also for every measure space $(\mathscr{S},\mu)$ and every finitely supported $\Phi:\H_{\ms{\Z}}^{2k+1}\to L_p(\mu)$ we have
\begin{equation}\label{eq:vectors case}
\eta\Bigg(\sum_{t=1}^\infty w_t\bigg(\sum_{h\in \H_{\ms{\Z}}^{2k+1}} \big\|\Phi\big(hZ^t\big)-\Phi(h)\big\|_{L_p(\mu)}^p\bigg)^\frac{q}{p}\Bigg)^{\frac{1}{q}}\le \bigg(\sum_{h\in \H_{\ms{\Z}}^{2k+1}}\sum_{\sigma\in \mathfrak{S}_k}\big\|\Phi(h\sigma)-\Phi(h)\big\|_{L_p(\mu)}^p\bigg)^{\frac{1}{p}}.
\end{equation}
\end{lemma}

\begin{proof}
For $\mu$-almost every $\mathcal{s}\in \mathscr{S}$ the function $\f_\mathcal{s}:\H_{\Z}^{2k+1}\to \R$ given by $\f_\mathcal{s}(h)=\Phi(h)(\mathcal{s})$ is finitely supported, and therefore by~\eqref{eq:real case} we have

\begin{align*}
\sum_{h\in \H_{\ms{\Z}}^{2k+1}}\sum_{\sigma\in \mathfrak{S}_k}\big\|\Phi(h\sigma)-\Phi(h)\big\|_{L_p(\mu)}^p&=\int_{\mathscr{S}}\sum_{h\in \H_{\ms{\Z}}^{2k+1}} \sum_{\sigma\in \mathfrak{S}_k} \big|\f_\mathcal{s}(h\sigma)-\f_\mathcal{s}(h)\big|^p \ud \mu(\mathcal{s})\\
&\stackrel{\mathclap{\eqref{eq:real case}} }{\ge}\,\eta\int_{\mathscr{S}}  \Bigg(\sum_{t=1}^\infty w_t\bigg(\sum_{h\in \H_{\ms{\Z}}^{2k+1}} \big|\upphi_\mathcal{s}\big(hZ^t\big)-\upphi_\mathcal{s}(h)\big|^p\bigg)^\frac{q}{p}\Bigg)^{\frac{p}{q}}\ud \mu(\mathcal{s})  \\
&\ge \eta\Bigg(\sum_{t=1}^\infty w_t\bigg(\sum_{h\in \H_{\ms{\Z}}^{2k+1}}  \int_{\mathscr{S}} \big|\upphi_\mathcal{s}\big(hZ^t\big)-\upphi_\mathcal{s}(h)\big|^p\ud \mu(\mathcal{s}) \bigg)^\frac{q}{p}\Bigg)^{\frac{p}{q}}
\\
&= \eta\Bigg(\sum_{t=1}^\infty w_t\bigg(\sum_{h\in \H_{\ms{\Z}}^{2k+1}} \big\|\Phi\big(hZ^t\big)-\Phi(h)\big\|_{L_p(\mu)}^p\bigg)^\frac{q}{p}\Bigg)^{\frac{p}{q}},
\end{align*}
where in the penultimate step we used the triangle inequality in $\ell_{q/p}$ (recall that $q\ge p$).
\end{proof}

\subsubsection{From global to local}\label{sec:local global} The proof of the following lemma follows the steps of a proof that appears in Section~3.2 of~\cite{LafforgueNaor}. We include the details because~\cite{LafforgueNaor} makes  an analogous statement only when $X$ is uniformly convex, $k=1$ and $p>1$, while we need to use here the case $X=L_1(\R)$, $k\ge 2$ and $p=1$. The argument below is a straightforward adaptation of the argument in~\cite{LafforgueNaor}.

\begin{lemma}\label{lem:localization norm} Fix $k\in \N$, a sequence $\{w_t\}_{t=1}^\infty\subset [0,\infty)$, $p,q\in [1,\infty)$ and $K\in (0,\infty)$. Let $(X,\|\cdot\|_X)$ be a normed space such that every finitely supported $\f:\H_{\ms{\Z}}^{2k+1}\to X$ satisfies
\begin{equation}\label{eq:to loclaize wt}
\Bigg(\sum_{t=1}^\infty w_t\bigg(\sum_{h\in \H_{\ms{\Z}}^{2k+1}} \big\|\f\big(hZ^t\big)-\f(h)\big\|_{X}^p\bigg)^\frac{q}{p}\Bigg)^{\frac{1}{q}}\le K\bigg(\sum_{h\in \H_{\ms{\Z}}^{2k+1}}\sum_{\sigma\in \mathfrak{S}_k}\big\|\f(h\sigma)-\f(h)\big\|_{X}^p\bigg)^{\frac{1}{p}}.
\end{equation}
Then there is a universal constant $c\in \N$ ($c=21$ works here) such that for every $f:\H^{2k+1}_{\ms{\Z}}\to X$ and every $n\in \N$ we have
\begin{equation}\label{eq:localized wt}
\Bigg(\sum_{t=1}^{n^2}w_t\bigg(\sum_{h\in \BB_n} \big\|f\big(hZ^t\big)-f(h)\big\|_{X}^p\bigg)^\frac{q}{p}\Bigg)^{\frac{1}{q}}\lesssim_k K\bigg(\sum_{h\in \BB_{cn}}\sum_{\sigma\in \mathfrak{S}_k}\big\|f(h\sigma)-f(h)\big\|_{X}^p\bigg)^{\frac{1}{p}}.
\end{equation}
\end{lemma}

Before proving Lemma~\ref{lem:localization norm}, we record for ease of later reference the following general estimates.

\begin{lemma} Let $\Gamma$ be a finitely generated group and fix a finite symmetric generating set $\Sigma\subset \Gamma$. Let $\rho_\Sigma:\Gamma\times \Gamma\to [0,\infty)$ denote the left-invariant word metric that is induced by $\Sigma$ on $\Gamma$ and for every $n\in \N$ denote by $B_\Sigma(n)=\{\gamma\in \Gamma:\ \rho_\Sigma(g,1_\Gamma)\le n\}$ the $\rho_\Sigma$-ball of radius $n$ centered at the identity element $1_\Gamma\in \Gamma$.  Then, for every $p\in [1,\infty)$, every $n\in \N$, every metric space $(\cM,d_\cM)$, and every finitely supported function $\f:\Gamma\to \cM$, we have
\begin{equation}\label{eq:gamma times ball}
\bigg(\sum_{x\in \Gamma}\frac{1}{|B_\Sigma(n)|}\sum_{y\in B_\Sigma(n)} d_\cM\big(\f(xy),\f(x)\big)^p\bigg)^{\frac{1}{p}}\le n\bigg(\sum_{x\in \Gamma}\max_{\sigma\in \Sigma}d_\cM\big(\f(x\sigma),\f(x)\big)^p\bigg)^{\frac{1}{p}}.
\end{equation}
Also, if $\Gamma$ has polynomial growth, then, fixing $r\in \N$  and $\alpha,\beta>0$ such that $\frac{1}{\alpha}m^r\le |B_\Sigma(m)|\le \beta m^r$ for every $m\in \N$ (such $\alpha,\beta,r$ exist by Gromov's theorem~\cite{Gro81}), we have
\begin{equation}\label{eq:general bruce}
\bigg(\frac{1}{|B_\Sigma(n)|^2}\sum_{x,y\in B_\Sigma(n)} d_\cM\big(\f(x),\f(y)\big)^p\bigg)^{\frac{1}{p}}\le 2n
\bigg(\frac{6^r\alpha^2\beta^2}{|B_\Sigma(3n)|}\max_{\sigma\in \Sigma} \sum_{x\in B_\Sigma(3n)} d_\cM\big(\f(x\sigma),\f(x)\big)^p\bigg)^{\frac{1}{p}}.
\end{equation}
\end{lemma}

\begin{proof} The proof of~\eqref{eq:gamma times ball} is essentially contained  in the proof of Lemma~3.4 in~\cite{LafforgueNaor}, even though~\cite[Lemma~3.2]{LafforgueNaor} is stated for the special case $\Gamma=\H_{\ms{\Z}}^{3}$. We briefly recall the simple argument. For every $y\in B_\Sigma(n)$ fix $\sigma_1(y),\ldots, \sigma_n(y)\in \Sigma\cup\{1_\Gamma\}$ such that $y=\sigma_1(y)\cdots\sigma_n(y)$. Then, denoting $\gamma_j(y)=\sigma_1(y)\cdots\sigma_{j-1}(y)$ for every $j\in \n$, by the triangle inequality we have
$$
\forall(x,y)\in \Gamma\times B_\Sigma(n),\qquad d_\cM\big(\f(xy),\f(x)\big)\le \sum_{j=1}^n d_\cM\Big(\f\big(x\gamma_j(y)\sigma_{j}(y)\big),\f\big(x\gamma_j(y)\big)\Big).
$$
By the triangle inequality in $\ell_p(\Gamma)$, this implies that
\begin{multline*}
\bigg(\sum_{x\in \Gamma}\sum_{y\in B_\Sigma(n)} d_\cM\big(\f(xy),\f(x)\big)^p\bigg)^{\frac{1}{p}}\le \sum_{j=1}^n \bigg(\sum_{y\in B_\Sigma(n)}\sum_{x\in \Gamma}d_\cM\Big(\f\big(x\gamma_j(y)\sigma_{j}(y)\big),\f\big(x\gamma_j(y)\big)\Big)^p\bigg)^{\frac{1}{p}}\\=\sum_{j=1}^n \bigg(\sum_{y\in B_\Sigma(n)}\sum_{z\in \Gamma}d_\cM\Big(\f\big(z\sigma_{j}(y)\big),\f\big(z\big)\Big)^p\bigg)^{\frac{1}{p}}\le n\bigg(|B_\Sigma(n)|\sum_{z\in \Gamma}\max_{\sigma\in \Sigma} d_{\cM}\big(\f(z\sigma),\f(z)\big)^p\bigg)^{\frac{1}{p}},
\end{multline*}
thus concluding the proof of~\eqref{eq:gamma times ball}.

The proof of~\eqref{eq:general bruce} is essentially  contained in the proof of Lemma~3.2 of~\cite{LafforgueNaor}, which is itself a generalization of an inequality of Kleiner~\cite[Theorem~2.2]{Kle10} which treats real-valued mappings and $p=2$. Indeed, even though~\cite[Lemma~3.2]{LafforgueNaor} is stated for the special case $\Gamma=\H_{\ms{\Z}}^{3}$, up to the final sentence of the proof of~\cite[Lemma~3.2]{LafforgueNaor}  the argument is valid in {\em any} finitely generated group $\Gamma$, and the penultimate step of the final displayed equation in the proof of~\cite[Lemma~3.2]{LafforgueNaor} gives that
\begin{multline*}
\frac{1}{|B_\Sigma(n)|^2}\sum_{x,y\in B_\Sigma(n)} d_\cM\big(\f(x),\f(y)\big)^p\\\le (2n)^p\cdot \frac{|B_\Sigma(2n)|\cdot|B_\Sigma(3n)|}{|B_\Sigma(n)|^2}\cdot  \max_{\sigma\in \Sigma} \frac{1}{|B_\Sigma(3n)|}\sum_{x\in B_\Sigma(3n)} d_\cM\big(\f(x),\f(x\sigma)\big)^p.\tag*{\qedhere}
\end{multline*}
\end{proof}

\begin{proof}[Proof of Lemma~\ref{lem:localization norm}] Fix $n\in \N$. By translating $f$ we may assume without loss of generality that $\sum_{x\in \BB_{7n}}f(x)=0$. Let $\xi_n:\H^{2k+1}_{\ms{\Z}}\to [0,1]$ be a cutoff function that is $\frac{1}{n}$-Lipschitz (with respect to the word metric $d_W$), is equal to $1$ on $\BB_{5n}$, and vanishes outside $\BB_{6n}$. Consider the finitely supported mapping $\f_n=\xi_n f:\H_{\ms{\Z}}^{2k+1}\to X$. We claim that the following two estimates hold true:
\begin{equation}\label{eq:cutoff1}
\Bigg(\sum_{t=1}^{n^2} w_t\bigg(\sum_{h\in \BB_n} \big\|f\big(hZ^t\big)-f(h)\big\|_{X}^p\bigg)^\frac{q}{p}\Bigg)^{\frac{1}{q}}\le \Bigg(\sum_{t=1}^\infty w_t\bigg(\sum_{h\in \H_{\ms{\Z}}^{2k+1}} \big\|\f_n\big(hZ^t\big)-\f_n(h)\big\|_{X}^p\bigg)^\frac{q}{p}\Bigg)^{\frac{1}{q}}
\end{equation}
and
\begin{equation}\label{eq:21}
\bigg(\sum_{h\in \H_{\Z}^{2k+1}}\sum_{\sigma\in \mathfrak{S}_k} \big\|\f_n(h\sigma)-\f_n(h)\big\|_X^p\bigg)^{\frac{1}{p}}\lesssim_k \bigg(\sum_{h\in \BB_{21n}}\sum_{\sigma\in \mathfrak{S}_k} \big\|f(h\sigma)-f(h)\big\|_X^p\bigg)^{\frac{1}{p}}.
\end{equation}
The estimates~\eqref{eq:cutoff1} and~\eqref{eq:21} imply the desired inequality~\eqref{eq:localized wt} by an application of~\eqref{eq:to loclaize wt} to $\f_n$.

Suppose that $h\in \BB_n$  and $t\in \{1,\ldots,n^2\}$. Since $d_W(\0,Z^t)\le 4n$ (see e.g.~\cite{Bla03}), we have $h,hZ^t\in \BB_{5n}$, and therefore $\f(h)=f(h)$ and $\f(hZ^t)= f(hZ^t)$. This immediately implies~\eqref{eq:cutoff1}.  Next, since $\xi_n$ is $\frac{1}{n}$-Lipschitz and takes values in $[0,1]$, for every $\sigma\in \mathfrak{S}_k$ and $h\in \H_{\ms{\Z}}^{2k+1}$ we have
\begin{multline*}
\big\|\f_n(h\sigma)-\f_n(h)\big\|_X= \big\|\xi_n(h\sigma)\big(f(h\sigma)-f(h)\big)+\big(\xi_n(h\sigma)-\xi_n(h)\big)f(h)\big\|_X\\\le \xi_n(h\sigma)\big\|f(h\sigma)-f(h)\big\|_X+ \big|\xi_n(h\sigma)-\xi_n(h)\big|\cdot\big\|f(h)\big\|_X\le  \big\|f(h\sigma)-f(h)\big\|_X+\frac{1}{n}\big\|f(h)\big\|_X.
\end{multline*}
Therefore, by the triangle inequality in $\ell_p(\BB_{7n}\times \mathfrak{S}_k)$ we have
\begin{multline}\label{eq:gradiant on 7n}
\bigg(\sum_{h\in \H_{\Z}^{2k+1}}\sum_{\sigma\in \mathfrak{S}_k} \big\|\f_n(h\sigma)-\f_n(h)\big\|_X^p\bigg)^{\frac{1}{p}}=\bigg(\sum_{h\in \BB_{7n}}\sum_{\sigma\in \mathfrak{S}_k} \big\|\f_n(h\sigma)-\f_n(h)\big\|_X^p\bigg)^{\frac{1}{p}}\\
\le \bigg(\sum_{h\in \BB_{7n}}\sum_{\sigma\in \mathfrak{S}_k} \big\|f(h\sigma)-f(h)\big\|_X^p\bigg)^{\frac{1}{p}}+\frac{(2k)^{\frac{1}{p}}}{n}\bigg(\sum_{h\in \BB_{7n}}\big\|f(h)\big\|_X^p\bigg)^{\frac{1}{p}},
\end{multline}
where the first step of~\eqref{eq:gradiant on 7n} holds true because $\f_n$ is supported on $\BB_{6n}$. To deduce~\eqref{eq:21} from~\eqref{eq:gradiant on 7n}, and hence also to conclude the proof of Lemma~\ref{lem:localization norm},  it suffices to note that
$$
\frac{1}{n}\bigg(\sum_{h\in \BB_{7n}}\|f(h)\|_X^p\bigg)^{\frac{1}{p}}\le \frac{1}{n}\bigg(\frac{1}{|\BB_{7n}|}\sum_{g,h\in \BB_{7n}}\|f(g)-f(h)\|_X^p\bigg)^{\frac{1}{p}}\lesssim_k\bigg(\sum_{h\in \BB_{21n}}\sum_{\sigma\in \mathfrak{S}_k} \|f(h\sigma)-f(h)\|_X^p\bigg)^{\frac{1}{p}},
$$
where the first step follows from Jensen's inequality since $\sum_{x\in \BB_{7n}}f(x)=0$, and the final step is an application of~\eqref{eq:general bruce}, using the fact~\cite{Bas72} that $|\BB_m|\asymp_k m^{2k+2}$ for every $m\in \N$.
\end{proof}

\subsubsection{From low dimension to high dimension} Lemma~\ref{lem:from small dim to bigger} below is of lesser importance for the present purposes. We shall use it only to show that the case $k=2$ of Theorem~\ref{thm:isoperimetric discrete} implies the general case $k\in \{2,3,4,\ldots,\}$ with the stated dependence on $k$ (otherwise the ensuing proof yields a worse asymptotic dependence on $k$ as $k\to \infty$). Below, a mapping $\f:X\to Y$ between two abstract sets $X$ and $Y$ is said to be finitely supported if there exists a point $y\in Y$ such that $\f(x)=y$ for all but finitely many $x\in X$, i.e., $|\f^{-1}(Y\setminus\{y\})|<\infty$. This is a slight (and harmless) departure from the more common use of this term when $Y$ is a Banach space, in which case one requires that $y=0$.

\begin{lemma}\label{lem:from small dim to bigger}
Fix $k,m\in \N$ with $m\le k$, a sequence $\{w_t\}_{t=1}^\infty\subset [0,\infty)$, $q\in [1,\infty)$ and $\beta\in (0,\infty)$. Let $X$ be a set and $\mathsf{K},\mathsf{L}:X\times X\to [0,\infty)$ be two  symmetric functions  such that $\mathsf{K}(x,x)=\mathsf{L}(x,x)=0$ for all $x\in X$. Suppose that for every finitely supported mapping $\f:\H_{\ms{\Z}}^{2m+1}\to X$ we have
\begin{equation}\label{eq:assume m ineq}
\beta \Bigg(\sum_{t=1}^{\infty}w_t\bigg(\sum_{h\in \H_{\ms{\Z}}^{2m+1}} \mathsf{K}\big(\phi(hZ^t),\phi(h)\big)\bigg)^q\Bigg)^{\frac{1}{q}}\le    \sum_{h\in \H_{\ms{\Z}}^{2m+1}}\sum_{\sigma\in \mathfrak{S}_m} \mathsf{L}\big(\phi(h\sigma),\phi(h)\big).
\end{equation}
Then for every finitely supported mapping $\Phi:\H_{\ms{\Z}}^{2k+1}\to X$ we have
\begin{equation}\label{eq:deduce m ineq}
 \frac{ \beta k}{m}\Bigg(\sum_{t=1}^{\infty}w_t\bigg(\sum_{h\in \H_{\ms{\Z}}^{2k+1}} \mathsf{K}\big(\Phi(hZ^t),\Phi(h)\big)\bigg)^q\Bigg)^{\frac{1}{q}}\le   \sum_{h\in \H_{\ms{\Z}}^{2k+1}}\sum_{\sigma\in \mathfrak{S}_k} \mathsf{L}\big(\Phi(h\sigma),\Phi(h)\big).
\end{equation}
\end{lemma}

\begin{proof} For every $A\subset \k$ denote $\mathfrak{S}_A=\{X_i^{\pm 1},Y_i^{\pm 1}\}_{i\in A}\subset \mathfrak{S}_k$. Let  $\G_A\triangleleft\H_{\ms{\Z}}^{2k+1}$ be the (normal)  subgroup of $\H_{\ms{\Z}}^{2k+1}$ that is generated by $\mathfrak{S}_A$ and let $\{g_j(A)\}_{j=1}^\infty\subset \H_{\ms{\Z}}^{2k+1}$ be representatives of the distinct left-cosets of $\G_A$ in $\H_{\ms{\Z}}^{2k+1}$. If $|A|=m$, then $\G_A$ is isomorphic to $\H^{2m+1}_{\ms{\Z}}$. Therefore, for every $A\subset \k$ with $|A|=m$ and every $j\in \N$, an application of~\eqref{eq:assume m ineq} to the mapping $(h\in \G_A)\mapsto \Phi(g_j(A)h)$ shows that
\begin{equation}\label{eq:apply m ineq each j}
 \beta \Bigg(\sum_{t=1}^{\infty}w_t\bigg(\sum_{h\in \G_A} \mathsf{K}\big(\Phi(g_j(A)hZ^t),\Phi(g_j(A)h)\big)\bigg)^q\Bigg)^{\frac{1}{q}} \le    \sum_{h\in \G_A}\sum_{\sigma\in \mathfrak{S}_A} \mathsf{L}\big(\Phi(g_j(A)h\sigma),\Phi(g_j(A)h)\big).
\end{equation}
Observe that
\begin{equation}\label{eq:sum cosets 1}
\sum_{j=1}^\infty\sum_{h\in \G_A}\sum_{\sigma\in \mathfrak{S}_A} \mathsf{L}\big(\Phi(g_j(A)h\sigma),\Phi(g_j(A)h)\big)=\sum_{h\in \H_{\ms{\Z}}^{2k+1}}\sum_{\sigma\in \mathfrak{S}_A} \mathsf{L}\big(\Phi(h\sigma),\Phi(h)\big),
\end{equation}
and, by the triangle inequality in $\ell_q$,
\begin{align}\label{eq:sum cosets 2}
\nonumber \sum_{j=1}^\infty &\Bigg(\sum_{t=1}^{\infty}w_t\bigg(\sum_{h\in \G_A} \mathsf{K}\big(\Phi(g_j(A)hZ^t),\Phi(g_j(A)h)\big)\bigg)^q\Bigg)^{\frac{1}{q}} \\ \nonumber &\ge \Bigg(\sum_{t=1}^{\infty}w_t\bigg(\sum_{j=1}^\infty\sum_{h\in \G_A} \mathsf{K}\big(\Phi(g_j(A)hZ^t),\Phi(g_j(A)h)\big)\bigg)^q\Bigg)^{\frac{1}{q}}\\&=\Bigg(\sum_{t=1}^{\infty}w_t\bigg(\sum_{h\in \H_{\ms{\Z}}^{2k+1}} \mathsf{K}\big(\Phi(hZ^t),\Phi(h)\big)\bigg)^q\Bigg)^{\frac{1}{q}}.
\end{align}
By summing~\eqref{eq:apply m ineq each j} over $j\in \N$ and using~\eqref{eq:sum cosets 1} and~\eqref{eq:sum cosets 2} it follows that
\begin{equation}\label{eq:to sum over A}
\beta \Bigg(\sum_{t=1}^{\infty}w_t\bigg(\sum_{h\in \H_{\ms{\Z}}^{2k+1}} \mathsf{K}\big(\Phi(hZ^t),\Phi(h)\big)\bigg)^q\Bigg)^{\frac{1}{q}}\le   \sum_{h\in \H_{\ms{\Z}}^{2k+1}}\sum_{\sigma\in \mathfrak{S}_A} \mathsf{L}\big(\Phi(h\sigma),\Phi(h)\big).
\end{equation}
By summing~\eqref{eq:to sum over A} over all those $A\subset \k$ with $|A|=m$, we obtain the estimate
\begin{equation}\label{eq:k-1 m-1}
\beta \binom{k}{m}\Bigg(\sum_{t=1}^{\infty}w_t\bigg(\sum_{h\in \H_{\ms{\Z}}^{2k+1}} \mathsf{K}\big(\Phi(hZ^t),\Phi(h)\big)\bigg)^q\Bigg)^{\frac{1}{q}}\le   \binom{k-1}{m-1}\sum_{h\in \H_{\ms{\Z}}^{2k+1}}\sum_{\sigma\in \mathfrak{S}_k} \mathsf{L}\big(\Phi(h\sigma),\Phi(h)\big),
\end{equation}
where we used the fact that for each $\sigma\in \mathfrak{S}_k$, the number of $A\subset \k$ with $|A|=m$ that contain $\sigma$ is equal to $\binom{k-1}{m-1}$. The desired estimate~\eqref{eq:deduce m ineq} is the same as~\eqref{eq:k-1 m-1}.
\end{proof}

Suppose that we have already proved Theorem~\ref{thm:isoperimetric discrete} when $k=2$, i.e., that $|\partial_\vv \Omega|\lesssim |\partial_\hh\Omega|$ for every finite $\Omega\subset \H_{\ms{\Z}}^5$. By Lemma~\ref{lem:functional form discrete} this implies that for every finitely supported $\f:\H_{\ms{\Z}}^5\to \R$,
\begin{equation*}
\Bigg(\sum_{t=1}^\infty \frac{1}{t^2}\bigg(\sum_{h\in \H_{\ms{\Z}}^{5}} \big|\upphi\big(hZ^t\big)-\upphi(h)\big|\bigg)^2\Bigg)^{\frac{1}{2}}\lesssim \sum_{h\in \H_{\ms{\Z}}^{5}}\sum_{\sigma\in \mathfrak{S}_2}\big|\upphi(h\sigma)-\upphi(h)\big|.
\end{equation*}
Due to Lemma~\ref{lem:from small dim to bigger}, it follows that for every integer $k\ge 2$ and every finitely supported  $\Phi:\H_{\ms{\Z}}^{2k+1}\to \R$,
\begin{equation*}
\Bigg(\sum_{t=1}^\infty \frac{1}{t^2}\bigg(\sum_{h\in \H_{\ms{\Z}}^{2k+1}} \big|\Phi\big(hZ^t\big)-\Phi(h)\big|\bigg)^2\Bigg)^{\frac{1}{2}}\lesssim \frac{1}{k}\sum_{h\in \H_{\ms{\Z}}^{2k+1}}\sum_{\sigma\in \mathfrak{S}_k}\big|\Phi(h\sigma)-\Phi(h)\big|.
\end{equation*}
By (the trivial direction of) Lemma~\ref{lem:functional form discrete}, it follows that  $|\partial_\vv \Omega|\lesssim \frac{1}{k}|\partial_\hh\Omega|$ for every finite $\Omega\subset \H_{\ms{\Z}}^{2k+1}$, i.e., the conclusion of Theorem~\ref{thm:isoperimetric discrete} holds true with the stated asymptotic dependence on $k$ as $k\to \infty$. This shows that in order to prove Theorem~\ref{thm:isoperimetric discrete} we can from now on ignore the dependence on $k$.

\subsection{From continuous to discrete}\label{sec:continutous to discrete}  Our goal here is to show that Theorem~\ref{thm:isoperimetric discrete} follows from a certain (singular) Sobolev-type inequality   on the continuous Heisenberg group that is stated in Corollary~\ref{cor:continuous version} below. Once this assertion will be established, we will henceforth focus our attention entirely on considerations that take place in the continuous setting. The proof of the following lemma is a simple adaptation of the steps of a proof that appears in Section~3.3 of~\cite{LafforgueNaor}.

\begin{lemma}\label{lem:discretization w} Fix $k\in \N$ and  $\mathscr{W}:[1,\infty)\to [0,\infty)$ that satisfies $\|\mathscr{W}\|_{L_1(1,\infty)}=\int_1^\infty\mathscr{W}(s)\ud s<\infty$. Define a sequence $\{w_t\}_{t=1}^\infty$ by setting
$$\forall\, t\in \N,\qquad w_t\eqdef \int_{s}^{s+1} \mathscr{W}(s)\ud s.
$$
Fix also $C\in (0,\infty)$ and  $p,q\in (0,\infty)$ with $p\ge 1$. Let $(X,\|\cdot\|_X)$ be a normed space such that every smooth and compactly supported function $\Phi:\H_{\ms{\R}}^{2k+1}\to X$ satisfies
\begin{multline}\label{eq:assumption cont for discretization}
\Bigg(\int_1^\infty \mathscr{W}(t)\bigg(\int_{\H^{2k+1}} \big\|\Phi(hZ^t)-\Phi(h)\big\|_X^p\ud \cH^{2k+2}(h)\bigg)^{\frac{q}{p}}\ud t\Bigg)^{\frac{1}{q}}\\ \le C \bigg(\int_{\H^{2k+1}}\big\|\nabla_\H \Phi(h)\big\|_{\ell_p^{2k}(X)}^p \ud \cH^{2k+2}(h)\bigg)^{\frac{1}{p}}.
\end{multline}
Then every finitely supported function $\f:\H_{\ms{\Z}}^{2k+1}\to X$  satisfies
\begin{multline}\label{eq:integrated omega const}
\Bigg(\sum_{t=1}^\infty w_t\bigg(\sum_{h\in \H_{\ms{\Z}}^{2k+1}} \big\|\f\big(hZ^t\big)-\f(h)\big\|_{X}^p\bigg)^\frac{q}{p}\Bigg)^{\frac{1}{q}}\\\lesssim \Big(C+\|\mathscr{W}\|_{L_1(1,\infty)}^{\frac{1}{q}}\Big)\bigg(\sum_{h\in \H^{2k+1}_{\ms{\Z}}} \sum_{\sigma\in \mathfrak{S}_k} \big\|\f(h\sigma)-\f(h)\big\|_X^p\bigg)^{\frac{1}{p}}.
\end{multline}
\end{lemma}

\begin{proof}[Proof of Lemma~\ref{lem:discretization w}] Suppose that $\f:\H_{\ms{\Z}}^{2k+1}\to X$ is finitely supported. We shall show that there exists a compactly supported smooth function $\Phi:\H^{2k+1}\to X$ that satisfies the following two estimates, which immediately imply the desired inequality~\eqref{eq:integrated omega const} via an application of~\eqref{eq:assumption cont for discretization}.
\begin{equation}\label{eq:grad of Phi is less than discrete}
\bigg(\int_{\H^{2k+1}}\big\|\nabla_\H \Phi(h)\big\|_{\ell_p^{2k}(X)}^p \ud \cH^{2k+2}(h)\bigg)^{\frac{1}{p}}\lesssim \bigg(\sum_{h\in \H^{2k+1}_{\ms{\Z}}} \sum_{\sigma\in \mathfrak{S}_k} \big\|\f(h\sigma)-\f(h)\big\|_X^p\bigg)^{\frac{1}{p}},
\end{equation}
and
\begin{align}\label{eq:vert of phi less than vet of Phi}
\nonumber \Bigg(\sum_{t=1}^\infty &w_t\bigg(\sum_{h\in \H_{\ms{\Z}}^{2k+1}} \big\|\f\big(hZ^t\big)-\f(h)\big\|_{X}^p\bigg)^\frac{q}{p}\Bigg)^{\frac{1}{q}}\\ \nonumber
&\lesssim
\Bigg(\int_1^\infty \mathscr{W}(t)\bigg(\int_{\H^{2k+1}} \big\|\Phi(hZ^t)-\Phi(h)\big\|_X^p\ud \cH^{2k+2}(h)\bigg)^{\frac{q}{p}}\ud t\Bigg)^{\frac{1}{q}}\\&\qquad \quad+\|\mathscr{W}\|_{L_1(1,\infty)}^{\frac{1}{q}}\bigg(\sum_{h\in \H^{2k+1}_{\ms{\Z}}} \sum_{\sigma\in \mathfrak{S}_k} \big\|\f(h\sigma)-\f(h)\big\|_X^p\bigg)^{\frac{1}{p}}.
\end{align}

Since $\H_{\ms{\Z}}^{2k+1}$ is a co-compact lattice of $\H^{2k+1}$, we can fix from now on a compactly supported smooth function $\chi:\H^{2k+1}\to [0,1]$ such that $\sum_{g\in \H^{2k+1}_{\ms{\Z}}} \chi(g^{-1}h)=1$ for all $h\in \H^{2k+1}$.  Define
$$
\forall\, h\in \H^{2k+1},\qquad \Phi(h)\eqdef\sum_{g\in \H_{\ms{\Z}}^{2k+1}}\chi(g^{-1}h)f(g).
$$
If we denote the support of $\f$ by $S=\supp(\f)\subset \H_{\ms{\Z}}^{2k+1}\subset \H^{2k+1}$ and the support of $\chi$ by $T=\supp(\chi)\subset \H^{2k+1}$ then the support of $\Phi$ is contained in $ST$. Hence $\Phi$ is compactly supported.

Fix $m\in \N$ such that $T\cap \H^{2k+1}_{\ms{\Z}}\subset \BB_m$. Note that $m$ depends only on $k$ (through the choice of the bump function $\chi$, which is fixed once and for all). Arguing identically to the proof of equation (56) of~\cite{LafforgueNaor} (which uses the triangle inequality in $\ell_p^{2k}(X)$, i.e., the assumption $p\ge 1$), we see that
\begin{equation}\label{eq:quote 56}
\forall\, g\in \H_{\ms{\Z}}^{2k+1},\qquad \sup_{h\in gT}\big\|\nabla_\H\Phi(h)\big\|_{\ell_p^{2k}}^p\lesssim  \sum_{z\in \BB_{2m}} \big\|\f(gz)-\f(g)\big\|_X^p.
\end{equation}
Hence, since the assumption on $\chi$ implies that $\bigcup_{g\in \H_{\ms{\Z}}^{2k+1}}gT=\H^{2k+1}$, we deduce that
\begin{multline*}
\bigg(\int_{\H^{2k+1}}\big\|\nabla_\H \Phi(h)\big\|_{\ell_p^{2k}(X)}^p\ud \cH^{2k+2}(h)\bigg)^{\frac{1}{p}} \le \bigg(\sum_{g\in \H_{\ms{\Z}}^{2k+1}}\int_{gT}\big\|\nabla_\H \Phi(h)\big\|_{\ell_p^{2k}(X)}^p\ud \cH^{2k+2}(h)\bigg)^{\frac{1}{p}}\\ \stackrel{\eqref{eq:quote 56}}{\lesssim} \bigg(\sum_{g\in \H_{\ms{\Z}}^{2k+1}} \sum_{z\in \BB_{2m}} \big\|\f(gz)-\f(g)\big\|_X^p\bigg)^{\frac{1}{p}}\stackrel{\eqref{eq:gamma times ball}}{\lesssim} \bigg(\sum_{h\in \H^{2k+1}_{\ms{\Z}}} \sum_{\sigma\in \mathfrak{S}_k} \big\|\f(h\sigma)-\f(h)\big\|_X^p\bigg)^{\frac{1}{p}}.
\end{multline*}
This completes the justification of~\eqref{eq:grad of Phi is less than discrete}.

To prove~\eqref{eq:vert of phi less than vet of Phi}, by arguing identically to the proof of equation (64) in~\cite{LafforgueNaor}, while using~\eqref{eq:gamma times ball} in place of the use of~\cite[Lemma~3.4]{LafforgueNaor}, we see that for every $t\in \N$ and $s\in [t,t+1]$ we have
\begin{multline}\label{eq:quote 64}
\bigg(\sum_{h\in \H^{2k+1}_{\ms{\Z}}}\big\|\f(hZ^t)-\f(h)\big\|_X^p\bigg)^{\frac{1}{p}}\\\lesssim \bigg(\int_{\H^{2k+1}} \big\|\Phi(hZ^s)-\Phi(h)\big\|_X^p\ud\cH^{2k+2}(h)\bigg)^{\frac{1}{p}}+\bigg(\sum_{h\in \H_{\ms{\Z}}^{2k+1}}\sum_{\sigma\in \mathfrak{S}_k}\big\|\f(h\sigma)-\f(h)\big\|_X^p\bigg)^{\frac{1}{p}}.
\end{multline}
The desired estimate~\eqref{eq:vert of phi less than vet of Phi} now follows by raising~\eqref{eq:quote 64} to the power $q$, multiplying the resulting estimate by $\mathscr{W}(s)$, integrating over $s\in [t,t+1]$ and summing over $t\in \N$.
\end{proof}

The following corollary is the special case $p=1$, $q=2$, $X=\R$ and  $\mathscr{W}(s)=1/s^2$ of Lemma~\ref{lem:discretization w} (combined with Lemma~\ref{lem:from small dim to bigger} so as to obtain the stated dependence on $k$).

\begin{cor}\label{cor:continuous version} In order to establish Theorem~\ref{thm:isoperimetric discrete} it suffices to prove that any $k\in \{2,3,\ldots,\}$, every finitely supported smooth function $f:\H^{2k+1}\to \R$ satisfies
  \[\bigg(\int_0^\infty \left(\int_{\H^{2k+1}} |f(hZ^t)-f(h)| \ud\cH^{2k+2}(h)\right)^{2} \frac{\ud t}{t^{2}}\bigg)^{\frac12} \lesssim \int_{\H^{2k+1}} \left\|\nabla_\H f(h)\right\|_{\ell_1^{2k}} \ud\cH^{2k+2}(h).\]
\end{cor}

\subsection{Reductions in the continuous setting}\label{sec:continuous reductions} By Corollary~\ref{cor:continuous version}, all of the new results that we stated in the Introduction follow from Theorem~\ref{thm:mainThmSobolev} below, which answers Question~4.1 in~\cite{LafforgueNaor} positively when the dimension of the underlying Heisenberg group is at least $5$. Note, however, that~\cite[Question~4.1]{LafforgueNaor} ignored the role of this dimension, so the forthcoming work that we discussed in Remark~\ref{eq:finite p}  shows that~\cite[Question~4.1]{LafforgueNaor} actually has a negative answer when $k=1$.
\begin{thm}\label{thm:mainThmSobolev}
  Fix $k\in \{2,3,\ldots,\}$.  If $f\from \H^{2k+1}\to \R$ is smooth and compactly supported, then
  \begin{equation}\label{eq:L1 functional in theorem}
  \bigg(\int_0^\infty \left(\int_{\H^{2k+1}} |f(hZ^t)-f(h)| \ud\cH^{2k+2}(h)\right)^{2} \frac{\ud t}{t^{2}}\bigg)^{\frac12} \lesssim \int_{\H^{2k+1}} \left\|\nabla_\H f(h)\right\|_{\ell_1^{2k}} \ud\cH^{2k+2}(h).
  \end{equation}
\end{thm}

Theorem~\ref{thm:mainThmSobolev} follows from an isoperimetric-type inequality that is stated in Theorem~\ref{thm:isoperimetric continuous} below. It relates the vertical perimeter (recall Section~\ref{sec:vert}) of a set to the measure of its boundary.  Suppose that $E\subset \H^{2k+1}$ is measurable and $\cH^{2k+1}(\partial E)<\infty$. It is not clear \emph{a priori} that $\|\vpf(E)\|_{L_2(\R)}$ is necessarily finite, but the following theorem shows that it is at most a constant multiple of $\cH^{2k+1}(\partial E)$ when the dimension of the underlying Heisenberg group is at least 5.
\begin{thm}
  \label{thm:isoperimetric continuous}
  If $k\in \{2,3,\ldots,\}$ and $E\subset \H^{2k+1}$ is measurable, then
    $\|\vpf(E)\|_{L_2(\R)}\lesssim \cH^{2k+1}(\partial E)$.
\end{thm}

We shall now explain how Theorem~\ref{thm:isoperimetric continuous} implies Theorem~\ref{thm:mainThmSobolev}.  To do so, we will need the following well-known coarea formula; see e.g.~\cite[Theorem~2.3.5]{FSSCSerrinMeyer} or~\cite[Theorem~3.3]{AmbrosioFineProperties}.
\begin{thm}
  Let $f\from \H^{2k+1}\to \R$ be a smooth, compactly supported function, and for $u\in \R$, let $E_{u}=\{h\in \H^{2k+1}\mid f(h)<u\}$.  Then the following inequality holds true.
  \begin{equation}\label{eq:coarea}
    \int_{-\infty}^\infty \cH^{2k+1}(\partial E_{u})\ud u\lesssim \int_{\H^{2k+1}} \left\|\nabla_\H f(h)\right\|_{\ell_1^{2k}} \ud\cH^{2k+2}(h).
  \end{equation}
\end{thm}

When $k\ge 2$ and $q=2$, the following lemma shows that Theorem~\ref{thm:isoperimetric continuous} implies Theorem~\ref{thm:mainThmSobolev}.

\begin{lemma}\label{lem:isoperimetric to functional continuous} Fix $k\in \N$, $q\in [1,\infty)$ and $C\in (0,\infty)$. Suppose that $\|\vpf(E)\|_{L_q(\R)}\le C\cH^{2k+1}(\partial E)$ for every measurable  $E\subset \H^{2k+1}$. Then every  smooth and compactly supported $f\from \H^{2k+1}\to \R$ satisfies
  \[\bigg(\int_0^\infty \left(\int_{\H^{2k+1}} |f(hZ^t)-f(h)| \ud\cH^{2k+2}(h)\right)^{q} \frac{\ud t}{t^{1+\frac{q}{2}}}\bigg)^{\frac{1}{q}} \lesssim C\int_{\H^{2k+1}} \left\|\nabla_\H f(h)\right\|_{\ell_1^{2k}} \ud\cH^{2k+2}(h).\]
\end{lemma}

\begin{proof} If $f\from \H^{2k+1}\to \R$ is measurable, then
  $$\int_{\H^{2k+1}} |f(hZ^t)-f(h)| \ud\cH^{2k+2}(h)= \int_{-\infty}^\infty \cH^{2k+2}(E_{u}\symdiff E_{u} Z^t) \ud u.$$
  Hence, by the triangle inequality in $L_q(\R)$ and the substitution $t=2^{2s}$, we have
  \begin{multline}\label{eq:apply iso to level sets}
    \left(\int_0^\infty \left(\int_{\H^{2k+1}} |f(hZ^t)-f(h)| \ud \cH^{2k+2}(h) \right)^{q} \frac{\ud t}{t^{1+\frac{q}{2}}}\right)^{\frac{1}{q}}
    \le \int_{-\infty}^\infty \left(\int_0^\infty \cH^{2k+2}(E_{u}\symdiff E_{u} Z^t)^{q} \frac{\ud t}{t^{1+\frac{q}{2}}}\right)^{\frac{1}{q}}\\
    \asymp \int_{-\infty}^\infty \left(\int_{-\infty}^\infty \frac{\cH^{2k+2}(E_{u}\symdiff E_{u} Z^{2^{2s}})^{q}}{2^{qs}} \ud s\right)^{\frac{1}{q}}\ud u
    \stackrel{\eqref{eq:vert per on U}}{=} \int_{-\infty}^\infty \|\vpf(E_u)\|_{L_q(\R)} \ud u.
  \end{multline}
  Now, by applying the assumption of Lemma~\ref{lem:isoperimetric to functional continuous}  to each $\{E_u\}_{u\in \R}$, we conclude the proof of Lemma~\ref{lem:isoperimetric to functional continuous}  by combining~\eqref{eq:apply iso to level sets} with the following estimate.
  \begin{equation*}\int_{-\infty}^\infty \|\vpf(E_u)\|_{L_q(\R)} \ud u\lesssim C\int_{-\infty}^\infty \cH^{2k+1}(\partial E_u) \ud u\stackrel{\eqref{eq:coarea}}{\lesssim} C\int_{\H^{2k+1}} \left\|\nabla_\H f(h)\right\|_{\ell_1^{2k}} \ud\cH^{2k+2}(h).\qedhere\end{equation*}
\end{proof}

Theorem~\ref{thm:isoperimetric cellular} below is nothing more than Theorem~\ref{thm:isoperimetric continuous} in the special case of cellular sets (recall Section~\ref{sec:cellularSets}). We prefer to state Theorem~\ref{thm:isoperimetric cellular} separately because we shall next prove that it implies Theorem~\ref{thm:isoperimetric continuous} in full generality. Once this is done, we can devote the rest of the discussion to the proof of Theorem~\ref{thm:isoperimetric cellular}.

\begin{thm}
  \label{thm:isoperimetric cellular}
  If $k\in \{2,3,\ldots\}$ and $F\subset \H^{2k+1}$ is a cellular set, then
  $\|\vpf(F)\|_{L_2(\R)}\lesssim \cH^{2k+1}(\partial F).$
\end{thm}

The special case $k\in \{2,3,\ldots\}$ and $q=2$ of the following lemma shows that Theorem~\ref{thm:isoperimetric continuous} in full generality is a consequence of (its special case) Theorem~\ref{thm:isoperimetric cellular}.

\begin{lemma}\label{le:to cellular q} Fix $k\in \N$, $q\in [1,\infty)$ and $C\in (0,\infty)$. Suppose that $\|\vpf(F)\|_{L_q(\R)} \le C\cH^{2k+1}(\partial F)$ for any cellular set $F\subset \H^{2k+1}$. Then $\|\vpf(E)\|_{L_q(\R)} \lesssim (C+1)\cH^{2k+1}(\partial E)$ for any measurable  $E\subset \H^{2k+1}$.
\end{lemma}

\begin{proof} Suppose that $E\subset \H^{2k+1}$ is a  measurable set satisfying $\cH^{2k+1}(\partial E)<\infty$.  Fix $i\in \Z$ and denote $\rho=2^i$.  By Lemma~\ref{lem:cellular approx}, there is a set $F_\rho\subset \H^{2k+1}$ such that its re-scaling $\s_{1/\rho^{}}(F_\rho)$ is cellular, $\cH^{2k+1}(\partial F_\rho)\lesssim \cH^{2k+1}(\partial E)$, and $\cH^{2k+2}(E\symdiff F_\rho)\lesssim \rho \cH^{2k+1}(\partial E)=2^i \cH^{2k+1}(\partial E).$  Using Lemma~\ref{lem:vPerProps} and the assumption of Lemma~\ref{le:to cellular q}  applied to $\s_{1/\rho^{}}(F_\rho)$, we see that
  \begin{multline}\label{eq:F rho per}
  \|\vpf(F_\rho)\|_{L_q(\R)}=\rho^{2k+1}\|\vpf(\s_{1/\rho^{}}(F_\rho))\|_{L_q(\R)}\\\le C\rho^{2k+1}\cH^{2k+1}(\partial \s_{1/\rho^{}}(F_\rho))=C\cH^{2k+1}(\partial F_\rho)\lesssim C\cH^{2k+1}(\partial E).
  \end{multline}
If $t\in \R$ satisfies $t\ge i$, then by Lemma~\ref{lem:vPerProps}.(2), we have
  \begin{equation}\label{eq:t>i}|\vpf(E)(t)-\vpf(F_\rho)(t)|\le \vpf(E\symdiff F_\rho)(t)\stackrel{\eqref{eq:vert per on U}}{\le} \frac{2 \cH^{2k+2}(E\symdiff F_\rho)}{2^t} \lesssim 2^{i-t}\cH^{2k+1}(\partial E).\end{equation}
Hence,
  $$\|\vpf(E)\|_{L_q(i,\infty)}\le \|\vpf(F_\rho)\|_{L_q(i,\infty)}+\|\vpf(E)-\vpf(F_\rho) \|_{L_q(i,\infty)}\stackrel{\eqref{eq:F rho per}\wedge\eqref{eq:t>i}}{\lesssim} (C+1)\cH^{2k+1}(\partial E).$$
  Since this is true for any $i\in \Z$, we have
  \begin{equation*}
  \|\bar{v}(E)\|_{L_q(\R)}=\lim_{i\to -\infty}\|\vpf(E)\|_{L_q(i,\infty)}\lesssim (C+1)\cH^{2k+1}(\partial E).\qedhere
  \end{equation*}
\end{proof}

\section{Proof overview}\label{sec:overview}

By the reductions in Section~\ref{sec:reductions}, we know that in order to prove Theorem~\ref{thm:isoperimetric discrete} it suffices to prove Theorem~\ref{thm:isoperimetric cellular}. Here we shall sketch in broad strokes the steps of the proof of Theorem~\ref{thm:isoperimetric cellular}.  The  proof uses ``intrinsic corona decompositions'' in the Heisenberg group to reduce the desired bound to a bound on pieces of intrinsic Lipschitz half-spaces (recall Section~\ref{sec:intrinsic Lipschitz}).  Each of these intrinsic Lipschitz half-spaces is bounded by the intrinsic graph of a function defined on a vertical plane. We slice the vertical plane into cosets of $\H^3$ and bound the vertical perimeter of the intrinsic half-space using an $L_2$ inequality for functions on these cosets.  In essence, our argument ``lifts'' an $L_2$ inequality for functions on $\H^{3}$ to a formally stronger endpoint $L_1$ inequality on $\H^{2k+1}$. The assumption that $k>1$ arises from this lifting process. When $k=1$ the vertical plane is $2$-dimensional, so we cannot relate the question to a functional inequality on $\H^3$ and therefore this strategy does not work. However, the bound on the vertical perimeter of intrinsic Lipschitz half-spaces (Section~\ref{sec:intrinsic graphs}) is the {\em only step} of our proof that fails for the $3$-dimensional Heisenberg group; when $k=1$, we can still use intrinsic corona decompositions to reduce to the case of intrinsic Lipschitz half-spaces, but the desired inequality fails for these half-spaces.

\subsection*{Intrinsic Lipschitz half-spaces}
To prove Theorem~\ref{thm:isoperimetric cellular}, we first establish a version of Theorem~\ref{thm:isoperimetric continuous} for intrinsic Lipschitz half-spaces.  Specifically, in Section~\ref{sec:intrinsic graphs} we prove that, provided $k\in \{2,3,\ldots,\}$, for every $\lambda\in (0,1)$, if $E\subset \H^{2k+1}$ is an intrinsic $\lambda$-Lipschitz half-space (recall Section~\ref{sec:intrinsic Lipschitz}), then for every point $p\in \H^{2k+1}$ and every radius $r>0$ we have
  \begin{equation}\label{eq:lip graph ineq overview}
    \left\| \vpfl{B_r(p)}(E)\right\|_{L_2(\R)}\lesssim\frac{r^{2k+1}}{1-\uplambda}.
  \end{equation}

When, say, $\lambda\in (0,\frac12)$, the estimate~\eqref{eq:lip graph ineq overview} is in essence the special case of Theorem~\ref{thm:isoperimetric continuous} for (pieces of) intrinsic Lipschitz half-spaces. This is so because, due to the isoperimetric inequality for the Heisenberg group~\cite{Pan82}, the right-hand side of~\eqref{eq:lip graph ineq overview} is at most a constant (depending on $k$) multiple of $\CH^{2k+1}(\partial(B_r\cap E))$ whenever $\CH^{2k+2}(B_r\cap E)\gtrsim r^{2k+2}$, i.e., provided that the intrinsic Lipschitz half-space $E$ occupies a constant fraction of the volume of the ball $B_r$. The estimate~\eqref{eq:lip graph ineq overview} will be used below only in such a non-degenerate situation.

Our proof of~\eqref{eq:lip graph ineq overview} relies crucially on a local $L_2$-variant of~\eqref{eq:L1 functional in theorem} for $\H^3$ that was proven in~\cite{AusNaoTes} using representation theory; see Theorem~\ref{thm:quote ANT} below for a precise formulation.  We leverage this lower-dimensional functional inequality to deduce~\eqref{eq:lip graph ineq overview} as follows.  Suppose that $\partial E$ is an intrinsic Lipschitz graph over a $2k$-dimensional vertical plane $V\subset \H^{2k+1}$; that is, $\partial E=\Gamma_f$ for some function $f:V\to \R$ as in~\eqref{eq:gamma f}.  The function $f$ does not have to be Lipschitz with respect to the Carnot--Carath\'eodory metric, but a computation that appears in Section~\ref{sec:intrinsic graphs} shows that the restrictions of $f$ to cosets of copies of $\H^3$ in $V$ are in fact Lipschitz with respect to the Carnot--Carath\'eodory metric.  This allows us to use the quadratic inequality of~\cite{AusNaoTes} for each of these ``fibers" of $f$, and by combining the resulting inequalities (using an application of Cauchy--Schwarz) one arrives at~\eqref{eq:lip graph ineq overview}.  It is important to stress that this proof does not work for functions on $\H^3$ because it relies on slicing $V$ into copies of $\H^3$.  When $V\subset \H^3$, $V$ is two-dimensional and can be sliced into vertical lines, but there is no analogue of the quadratic inequality of~\cite{AusNaoTes} that we use here for vertical lines in $\H^3$.

\subsection*{Intrinsic corona decompositions}
In order to apply \eqref{eq:lip graph ineq overview} to more general sets, we use intrinsic corona decompositions in the Heisenberg group.  The full definitions will appear in Section~\ref{sec:corona decompositions}, and a qualitative description appears below. Corona decompositions are an established tool in analysis for reducing the study of certain singular integrals on $\R^n$ to the case of Lipschitz graphs, starting with seminal works of David~\cite{Dav84,Dav91} and Jones~\cite{Jon89,Jon90} on the Cauchy integral and culminating with the David--Semmes theory of quantitative rectifiability~\cite{DavidSemmesSingular,DSAnalysis}. Our adaptation
of this technique is mostly technical, but it will also involve a conceptually new ingredient, namely the use of quantitative monotonicity to govern the decomposition.

A corona decomposition covers $\partial E$ by two types of sets, called \emph{stopping-time regions} and \emph{bad cubes}.  Stopping-time regions correspond to parts of $\partial E$ that are close to intrinsic Lipschitz graphs, and bad cubes correspond to parts of $\partial E$, like sharp corners, that are not.  The multiplicity of this cover depends on the shape of $\partial E$ at different scales.  For example, $\partial E$ might look smooth on a large neighborhood of a point $x$, jagged at a medium scale, then smooth again at a small scale.  If so, then $x$ is contained in a large stopping-time region, a medium-sized bad cube, and a second small stopping-time region.  A cover like this is a corona decomposition if it satisfies a {\em Carleson packing condition} (see Section~\ref{sec:corona decompositions}) that bounds its average multiplicity on any ball.

Section~\ref{sec:isoperimetry of corona} establishes that Theorem~\ref{thm:isoperimetric continuous} holds when $\partial E$ admits a corona decomposition.  We proceed as follows.  The vertical perimeter of $\partial E$ comes from three sources: the bad cubes, the graphs that approximate the stopping-time regions, and the error incurred by approximating a stopping-time region by an intrinsic Lipschitz graph.  By the Carleson packing condition, there are few bad cubes, and they contribute vertical perimeter on the order of $\cH^{2k+1}(\partial E)$.  Using~\eqref{eq:lip graph ineq overview}, we argue that the intrinsic Lipschitz graphs also contribute vertical perimeter on the order of $\cH^{2k+1}(\partial E)$.  Finally, the difference between a stopping-time region and an intrinsic Lipschitz graph is bounded by the size of the stopping-time region.  The stopping-time regions also satisfy a Carleson packing condition, so these errors also contribute vertical perimeter on the order of $\cH^{2k+1}(\partial E)$.  Summing these contributions, we obtain the desired bound.

Not every set with finite perimeter admits a corona decomposition, even in $\R^n$, but we show that the boundaries of cellular sets can be built out of pieces that admit such decompositions.  (A similar technique for surfaces in $\R^n$ was used in \cite{You16}.)  We start by showing that in order to establish Theorem~\ref{thm:isoperimetric cellular} it suffices to prove that for every $r>0$ we have $\|\overline{\mathsf{v}}_{B_r}(E)\|_{L_2(\R)}\lesssim r^{2k+1}$ under the additional assumption  that the sets $E,E^\cc,\partial E$ are  $r$-locally Ahlfors-regular.  That is, $\CH^{2k+2}(uB_\rho \cap E)\asymp \rho^{2k+2}\asymp \CH^{2k+2}(vB_\rho \setminus E)$ and $\CH^{2k+1}(wB_\rho\cap \partial E)\asymp \rho^{2k+1}$  for all $\rho\in (0,r)$ and $(u,v,w)\in E\times E^\cc\times \partial E$.  To see this reduction, recall that in Theorem~\ref{thm:isoperimetric cellular} we are given a cellular set $F\subset \H^{2k+1}$. Any such set is Ahlfors-regular on sufficiently small balls. We argue that $F$ can be decomposed into sets that satisfy the desired local Ahlfors-regularity.  The full construction of this decomposition is carried out in Section~\ref{sec:decomposing cellular}, but we remark briefly that it amounts to the following natural ``greedy" iterative procedure. If one of the sets $F, F^\cc, \partial F$ were not locally Ahlfors-regular then there would be some smallest ball $B$ such that the density of $F$, $F^\cc$ or $\partial F$ is either too low or too high on $B$ (it is actually convenient to work here with ``cellular balls" here rather than  Carnot--Carath\'eodory balls so as to maintain the inductive hypothesis that $F$ is cellular, but we ignore this technical point for the purpose of the overview; see Section~\ref{sec:decomposing cellular}). By replacing $F$ by either $F\cup B$ or $F\setminus B$, we cut off a piece of $\partial F$ and decrease $\CH^{2k+1}(\partial F)$.  Since $B$ was the smallest ball where Ahlfors-regularity fails, $F, F^\cc, \partial F$ are Ahlfors-regular on balls smaller than $B$.  Repeating this process eventually reduces $F$ to the empty set.  We arrive at the conclusion of Theorem~\ref{thm:isoperimetric cellular} for the initial set $F$ by proving the (local version of) the theorem for each piece of this decomposition, then summing the resulting inequalities.

Finally, in Section~\ref{sec:constructing corona}, we arrive at  the main technical step of the decomposition procedure: showing that if $E,E^\cc$, and $\partial E$ are  Ahlfors regular, then $\partial E$ admits a corona decomposition (Theorem~\ref{thm:Ahlfors admits corona}).  This is the longest section of the proof, but it follows the well-established methods of~\cite{DavidSemmesSingular,DSAnalysis}.  We start by constructing a sequence of nested partitions of $\partial E$ into pieces called \emph{cubes}; this is a standard construction due to Christ~\cite{ChristTb} and David~\cite{DavidWavelets} and only uses the Ahlfors regularity of $\partial E$.  These partitions are analogues of the standard tilings of $\R^n$ into dyadic cubes.  Next, we classify the cubes into \emph{good cubes}, which are close to a piece of a hyperplane, and \emph{bad cubes}, which are not.  In order to produce a corona decomposition, there cannot be too many bad cubes, i.e., they must satisfy a Carleson packing condition.  In~\cite{DavidSemmesSingular,DSAnalysis}, this condition follows from \emph{quantitative rectifiability}; the surface in question is assumed to satisfy a condition that bounds the sum of its (appropriately normalized) local deviations from hyperplanes.  These local deviations are higher-dimensional versions of Jones' $\beta$-numbers~\cite{Jon89,Jon90}, and the quantitative rectifiability assumption leads to the desired packing condition.  In the present setting, the packing condition follows instead from {\em quantitative non-monotonicity}.  The concept of quantitative non-monotonicity of a set $E\subset \H^{2k+1}$ (see Section~\ref{sec:monotone and delta monotone}) was defined in~\cite{CKN09,CKN},  where the kinematic formula for the Heisenberg group was used to show that the total non-monotonicity of all of the cubes is at most a constant multiple of $\CH^{2k+1}(\partial E)$.  This means that there cannot be many cubes that have large non-monotonicity. By a result of~\cite{CKN09,CKN}, if a surface has small non-monotonicity, then it is close to a half-space.  Consequently, most cubes are close to a half-space and are therefore good.  (The result in~\cite{CKN09,CKN} is stronger than what we need for this proof; it provides power-type bounds on how closely a nearly-monotone surface approximates a hyperplane.  For our purposes, it is enough to have \emph{some} bound (not necessarily power-type) on the shape of nearly-monotone sets, and we will deduce the bound that we need by applying a quick compactness argument to a result from~\cite{CheegerKleinerMetricDiff} that states that if a set  is {\em precisely monotone} (i.e., every horizontal line intersects the boundary of the set in at most one point), then it is a half-space.)

Next, we partition the good cubes into stopping-time regions by using an iterative construction that corrects overpartitioning that may have occurred when the Christ cubes were constructed.  If $Q$ is a largest good cube that hasn't been treated yet and if $P$ is its approximating half-space, we find all of the descendants of $Q$ with approximating half-spaces that are sufficiently close to $P$.  If we glue these half-spaces together using a partition of unity, the result is an intrinsic Lipschitz half-space that approximates all of these descendants.  By repeating this procedure for each untreated cube, we obtain a collection of stopping-time regions.  These regions satisfy a Carleson packing condition because if a point $x\in \partial E$ is contained in many different stopping-time regions, then either $x$ is contained in many different bad cubes, or $x$ is contained in good cubes whose approximating hyperplanes point in many different directions.  In either case, these cubes generate non-monotonicity, so there can only be a few points with large multiplicity.  In fact, this cover satisfies a Carleson packing condition, so it is an intrinsic corona decomposition.

\section{Isoperimetry of intrinsic Lipschitz graphs}\label{sec:intrinsic graphs}

In this section, we bound the (local) vertical perimeter of intrinsic Lipschitz graphs in  $\H^{2k+1}$, provided that $k\ge 2$.  This is the only place in the proof of Theorem~\ref{thm:isoperimetric cellular} in which the assumption that $k>1$ is used; see Remark~\ref{rem:k ge 2 is here} below.  Specifically, we will establish the following proposition.
\begin{prop}\label{prop:lip graph general statement}
    For every integer $k\ge 2$ and every $\lambda\in (0,1)$, if $E\subset \H^{2k+1}$ is a $\lambda$-Lipschitz half-space then for every point $p\in \H^{2k+1}$ and every radius $r>0$ we have
  \begin{equation}\label{eq:our ineq for intrinsic graph canonical}
    \left\| \vpfl{B_r(p)}(E)\right\|_{L_2(\R)}\lesssim\frac{r^{2k+1}}{1-\uplambda}.
  \end{equation}
\end{prop}

Before proving Proposition~\ref{prop:lip graph general statement}, we shall state Theorem~\ref{thm:quote ANT} below for ease of later reference. It follows from~\cite[Theorem~7.5]{AusNaoTes}, where it is proven using representation theory. A  generalization of this result (namely an $L_p$ version for every $p\in (1,\infty)$  as well as more general target spaces) follows from~\cite[Theorem~2.1]{LafforgueNaor}, where it is proven using different methods than those of~\cite{AusNaoTes}.

\begin{thm}\label{thm:quote ANT}
For every Lipschitz function $f:\H^{3}\to \R$ and every $r\in (0,\infty)$ we have
\begin{equation}\label{eq:quote ANT}
  \int_0^{r^2}\fint_{B_r} \frac{\left|f(h)-f\big(hZ^{-t}\big)\right|^2}{t^2}\ud\mathcal{H}^4(h)\ud t\lesssim \|f\|_{\Lip(\H^3)}^2.
\end{equation}

\end{thm}

\begin{remark}\label{rem:ANT}
The statement of~\cite[Theorem~7.5]{AusNaoTes} actually treats smooth functions, asserting that there exists a universal constant $C\in [1,\infty)$ such that for every smooth $f:\H^3\to \R$  we have
\begin{equation}\label{eq:R ge 1 case}
\forall\, r\in [1,\infty),\qquad \int_1^{r^2} \int_{B_r} \frac{\left|f(h)-f\big(hZ^{-t}\big)\right|^2}{t^2}\ud\mathcal{H}^4(h)\ud t\lesssim \int_{B_{Cr}} \left\|\nabla_\H f(h)\right\|_{\ell_2^2}^2\ud\mathcal{H}^4(h).
\end{equation}
When $f$ is Lipschitz but not necessarily smooth, it follows that
\begin{equation}\label{eq:R ge 1 Lipschitz}
\forall\, r\in [1,\infty),\qquad \int_1^{r^2} \fint_{B_r} \frac{\left|f(h)-f\big(hZ^{-t}\big)\right|^2}{t^2}\ud\mathcal{H}^4(h)\ud t\lesssim \|f\|_{\Lip(\H^3)}^2.
\end{equation}
Indeed, convolve $f$ on the left with a smooth nonnegative function of arbitrarily small support and integral $1$.  The result is a smooth function with the same or smaller Lipschitz constant that approximates $f$ in $L_\infty(B_r)$.  Applying \eqref{eq:R ge 1 case} to the resulting function and using the fact that the integrals appearing in \eqref{eq:R ge 1 case} are over compact sets, we obtain \eqref{eq:R ge 1 Lipschitz}. By scale-invariance, \eqref{eq:R ge 1 Lipschitz} implies \eqref{eq:quote ANT}.  Indeed, if $r>0$ and $\e\in (0,r)$, apply~\eqref{eq:R ge 1 Lipschitz} with $r$ replaced by $r/\e$ and $f$ replaced by $h\mapsto f(\s_\e(h))$, change the variable to $h'=\s_\e(h)$ and let $\e\to 0$ in the resulting inequality.

\end{remark}

\begin{proof}[{Proof of Proposition~\ref{prop:lip graph general statement}}]
  Denote $V=\{h\in \H^{2k+1}\mid x_k(h)=\0\}$.  By applying an isometric automorphism  and using Lemma~\ref{lem:translate graph}, we may assume that $p=\0$  and that $E$ is bounded by an intrinsic Lipschitz graph $\Gamma$ over $V$, say $\Gamma=\{v X_k^{f(v)}\mid v\in V\}$
  for some continuous function $f\from V\to \R$.  The Lipschitz condition implies that $|x_k(w_1)-x_k(w_2)|\le \uplambda d(w_1,w_2)$ for all $w_1,w_2\in \Gamma$.

  We shall identify $\H^3$ with the subgroup $\langle X_1,Y_1,Z\rangle \triangleleft \H^{2k+1}$; importantly, $\H^3$ commutes with $X_k$.  We shall also denote the linear subspace $\langle X_2,\ldots,X_{k-1},Y_2,\ldots,Y_{k} \rangle\subset \mathsf{H}$ by $\mathbb{G}$.  Then $V=\mathbb{G}\H^3$ in the sense that for each $v\in V$, there are unique elements $g\in \mathbb{G}$ and $h\in \H^3$ such that $v=gh$.  For every $g\in \mathbb{G}$, define $f_g:\H^3\to \R$ by setting $f_g(h)=f(gh)$ for every $h\in \H^3$.

  Fixing $g\in \mathbb{G}$, for every $h_1,h_2\in \H^3$ if we denote $w_1=gh_1X_k^{f_g(h_1)}$ and  $w_2= gh_2X_k^{f_g(h_2)}$ then by the definition of $\Gamma$ we have $w_1,w_2\in \Gamma$.  Moreover, $x_k(gh_1)=x_k(gh_2)=0$ since $g,h_1,h_2\in V$, so $x_k(w_1)=f_g(h_1)$ and $x_k(w_2)=f_g(h_2)$. Hence, by our assumption on $f$ we have
  \begin{multline}\label{eq:get lip f on subgroup}
    |f_g(h_1)-f_g(h_2)|=|x_k(w_1)-x_k(w_2)|\le \uplambda d(w_1,w_2) = \uplambda d\left(h_1X_k^{f_g(h_1)},h_2X_k^{f_g(h_2)}\right)\\
    = \uplambda d\left(X_k^{f_g(h_1)}h_1,X_k^{f_g(h_1)}h_2X_k^{f_g(h_2)-f_g(h_1)}\right)
    \stackrel{\eqref{eq:metric approximation} }{\le} \uplambda\left(C_kd(h_1,h_2)+ |f_g(h_1)-f_g(h_2)|\right),
  \end{multline}
where $C_k$ is the constant in~\eqref{eq:metric approximation}.  This simplifies to give that
  \begin{align*}
    |f_g(h_1)-f_g(h_2)|\lesssim \frac{\uplambda}{1-\uplambda}d(h_1,h_2).
  \end{align*}
  Note that this is not quite the condition that $f_g$ is $\frac{\uplambda}{1-\uplambda}$-Lipschitz, because the metric $d$ is the metric on $\H^{2k+1}$ rather than the metric $d_{\H^3}$ on $\H^3$.  By \eqref{eq:metric approximation}, however, $d_{\H^3}(h_1,h_2) \asymp d(h_1,h_2)$ for all $h_1,h_2\in \H^3$, so $\|f_g\|_{\Lip(\H^3)}\lesssim \uplambda/(1-\uplambda)$.  By Theorem~\ref{thm:quote ANT} we therefore have
  \begin{equation}\label{eq:us LN on slice}
    \forall\, r\in (0,\infty),\qquad \int_0^{r^2}\fint_{B_r\cap \H^3}\left|f_g(h)-f_g\big(hZ^{-t}\big)\right|^2\ud \mathcal{H}^4(h)\frac{\ud t}{t^2}\lesssim \frac{\uplambda^2}{(1-\uplambda)^2}.
  \end{equation}

  Every $u\in \H^{2k+1}$ can be written uniquely as $u=ghX_k^t$ with $g\in \mathbb{G}$, $h\in \H^{3}$ and $t\in \R$. Indeed, necessarily $g=\sum_{i=2}^{k-1} x_i(u)X_i+\sum_{j=2}^{k}y_j(u)Y_j$,  $h=x_1(u)X_1+y_1(u)Y_1+(z(u)+x_k(u)y_k(u)/2) Z$ and $t=x_k(u)$.  Since $\mathbb{G}, \H^3, \langle X_k\rangle$ are orthogonal subspaces of $\R^{2k+1}$, for every $s\in \R$ we have
  \begin{align}\label{eq:use cartesian}
    \nonumber \vpfl{B_r}(E)(s)
    &= \frac{\vol\left( B_r\cap\left(E\symdiff E Z^{2^{2s}}\right)\right)}{2^s}\\\nonumber &=\frac{1}{2^s}\int_{\mathbb G} \int_{\H^3} \int_{-\infty}^\infty \left|\1_{E}\big(ghX_k^t\big)-\1_{E}\Big(ghX_k^tZ^{-2^{2s}}\Big)\right| \1_{B_r}\big(ghX_k^t\big)\ud t\ud \mathcal{H}^4(h)\ud\mathcal{H}^{2k-3}(g)\\
    &=\frac{1}{2^s}\int_{\mathbb G} \int_{\H^3} \int_{-\infty}^\infty \left|\1_{\{t\ge f_g(h)\}}-\1_{\left\{t\ge f_g\left(hZ^{-2^{2s}}\right)\right\}}\right| \1_{B_r}\big(ghX_k^t\big)\ud t\ud \mathcal{H}^4(h)\ud\mathcal{H}^{2k-3}(g).
  \end{align}

  For every $(g,h,t)\in \mathbb{G}\times \H^3\times \R$, we have $\pi(ghX_k^t)=g+x_1(h)X_1+y_1(h)Y_1+t X_k$, where $\pi\from \H^{2k+1}\to \R^{2k}$ is the canonical projection.  By~\eqref{eq:metric lower bound}, this implies that if $d(\0,ghX_k^t)< r$, then $\|g\|< r$ and $|t|< r$, so
  $$d(\0,h)=d(g,gh) \le d(g,\0)+d(\0,ghX_k^t)+d(ghX_k^t,gh)\le \|g\|+r+|t|< 3r.$$
  Hence $\1_{B_r}(ghX_k^t)\le \1_{B_{r}}(g)\1_{B_{3r}}(h)$. A substitution of this point-wise inequality into~\eqref{eq:use cartesian} gives
  \begin{align}\label{eq:bound by product of balls}
    \vpfl{B_r}(E)(s)&\le \frac{1}{2^s}\int_{\mathbb G} \int_{\H^3} \int_{-\infty}^\infty \left|\1_{\{t\ge f_g(h)\}}-\1_{\left\{t\ge f_g\left(hZ^{-2^{2s}}\right)\right\}}\right| \1_{B_{r}}(g)\1_{B_{3r}}(h)\ud t\ud \mathcal{H}^4(h)\ud\mathcal{H}^{2k-3}(g)\nonumber\\
                    &=\frac{\alpha r^{2k+1}}{2^s}\fint_{B_{r}\cap\mathbb{G}} \fint_{B_{3r}\cap \H^3}\left|f_g(h)-  f_g\left(hZ^{-2^{2s}}\right)\right|\ud \mathcal{H}^4(h)\ud\mathcal{H}^{2k-3}(g),
  \end{align}
  where $\alpha=\alpha_k=81\mathcal{H}^{2k-3}(B_1\cap\mathbb{G})\mathcal{H}^4(B_1\cap \H^3)$. Now,
  \begin{align}
    \int_{-\infty}^{\log_2 r} &\vpfl{B_r}(E)(s)^2d s\nonumber \\
                              &\lesssim r^{4k+2} \int_0^{r^2} \bigg(\fint_{B_{r}\cap\mathbb{G}} \fint_{B_{3r}\cap \H^3}\left|f_g(h)-  f_g\left(hZ^{-t}\right)\right|\ud \mathcal{H}^4(h)\ud\mathcal{H}^{2k-3}(g)\bigg)^2\frac{\ud t}{t^2}\label{eq:change of varable s t}\\
                              &\le r^{4k+2}\fint_{B_{r}\cap\mathbb{G}} \int_0^{r^2} \fint_{B_{3r}\cap \H^3}\left|f_g(h)-  f_g\left(hZ^{-t}\right)\right|^2\ud \mathcal{H}^4(h)\frac{\ud t}{t^2}\ud\mathcal{H}^{2k-3}(g)\label{eq:use jensen}\\&\lesssim \frac{\uplambda^2 r^{4k+2}}{(1-\uplambda)^2},\label{eq:final use of LN on slice}
  \end{align}
  where in~\eqref{eq:change of varable s t} we used~\eqref{eq:bound by product of balls} and the change of variable $t=2^{2s}$,  in~\eqref{eq:use jensen} we used Jensen's inequality, and in~\eqref{eq:final use of LN on slice} we used~\eqref{eq:us LN on slice}. Also, the trivial bound $\vpfl{B_r}(E)(s)\le \vol(B_r)/2^s\lesssim r^{2k+2}/2^s$ gives
  \begin{equation}\label{eq:use trivial vertical}
    \int_{\log_2 r}^\infty \vpfl{B_r}(E)(s)^2d s\lesssim_k r^{4k+4}\int_{\log_2 r}^\infty\frac{\ud s}{2^{2s}}\asymp r^{4k+2}.
  \end{equation}
  The desired estimate~\eqref{eq:our ineq for intrinsic graph canonical} now follows by combining~\eqref{eq:final use of LN on slice} and~\eqref{eq:use trivial vertical}.  \end{proof}

\begin{remark}\label{rem:k ge 2 is here}
Note that Proposition~\ref{prop:lip graph general statement} used the assumption that $k>1$ crucially. The above proof would not work for $k=1$ because an analogue of Theorem~\ref{thm:quote ANT} does not exist for $1$-dimensional vertical slices of $\H^3$, while for $k\ge 2$ we succeeded by slicing $\H^{2k+1}$ into copies of $\H^3$.
\end{remark}

\section{Corona decompositions}\label{sec:corona decompositions}

In this section we shall formulate a definition of a {\em corona decomposition} of a subset of $\H^{2k+1}$. This definition mimics the analogous notion of a corona
decomposition of a subset of $\R^k$, as defined by David and Semmes~\cite{DavidSemmesSingular} (see also~\cite[Chapter~3]{DSAnalysis}). Here we need a local variant of this definition, but the key conceptual difference with the Euclidean case is the utilization of intrinsic Lipschitz graphs. Other differences occur in the ensuing proofs: Some of them are of a more technical nature, arising from geometric peculiarities of the Heisenberg group that differ from their Euclidean counterparts, but conceptually new issues arise as well,  such as the use of  (quantitative) monotonicity methods~\cite{CheegerKleinerMetricDiff,CKN} for the purpose of constructing the desired decomposition.

\begin{defn}[local Ahlfors regularity]\label{def:local Ahlfors regularity}
  Suppose that $C,r,s\in (0,\infty)$. We shall say that a Borel subset $A\subset \H^{2k+1}$ is $(C,r)$-Ahlfors $s$-regular if for every $x\in A$ and and $\rho\in (0,r]$ we have
  $$
  \frac{r^s}{C}\le \mathcal{H}^s\big(B_\rho(x)\cap A\big)\lesssim Cr^s.
  $$
  This terminology is shorthand for what one would normally refer to as $A$ being $r$-locally Ahlfors $s$-regular with constant $C$, but since we need to state this often in what follows, we prefer to use the above nonstandard shorter nomenclature. In fact, it will be convenient to introduce the following convention for even shorter terminology. Whenever we are given $E\subset \H^{2k+1}$ and we say that $E$ is $(C,r)$-regular it will be understood by default that $E$ is assumed to be $(C,r)$-Ahlfors $s$-regular with $s=2k+2$. When we say that $\partial E$ is $(C,r)$-regular we mean that $\partial E$ is assumed to be $(C,r)$-Ahlfors $s$-regular with $s=2k+1$. When we say that $(E,E^\cc,\partial E)$ is $(C,r)$-regular we mean that $E,  E^\cc$ are both $(C,r)$-Ahlfors $(2k+2)$-regular and also $\partial E$ is $(C,r)$-Ahlfors $(2k+1)$-regular.
\end{defn}

\begin{defn}[local cubical patchwork]\label{def:patchwork} Fix $K,s,r\in (0,\infty)$ and $n\in \Z$ for which $2^n\le r < 2^{n+1}$. Suppose that $A$ is a Borel subset of $\H^{2k+1}$.  A sequence $\{\Delta_i\}_{i=-\infty}^n$ of Borel partitions of $A$ is said to be an $s$-dimensional $(K,r)$-cubical patchwork for $A$ if it has the following properties.
  \begin{enumerate}
  \item For every integer $i\le n$ and every $Q\in \Delta_i$ we have
  \begin{equation}\label{eq:volume and diameter growth}
  \frac{2^i}{K}< \diam(Q)<K2^i\qquad\mathrm{and}\qquad \frac{2^{is}}{K}<\cH^{s}(Q)<K2^{is}.
  \end{equation}
  \item For every two integers $i,j\le n$ with $i\le j$ the partition $\Delta_i$ is a refinement of the partition $\Delta_j$, i.e., for every $Q\in \Delta_i$ and $Q'\in \Delta_j$  either $Q\cap Q'=\emptyset$ or $Q\subset Q'$.
  \item \label{item:bdary} For every integer $i\le n$, $Q\in \Delta_i$ and $t>0$ we have
  \begin{equation}\label{eq:PatchworkSmallBdry}
      \cH^s\big(\partial_{\le t2^i} Q\big)\le Kt^{\frac{1}{K}}2^{is},
    \end{equation}
  where for every $\rho>0$ we denote
    \begin{equation}\label{eq:def partial less r}
    \partial_{\le \rho}Q\eqdef \{x\in Q: d(x, A\setminus Q)\le \rho\}\cup \{x\in  A\setminus Q\mid  d(x, Q)\le \rho\}.\end{equation}
  \end{enumerate}

In the above setting, write
\begin{equation}\label{eq:disjoint union}
\Delta\eqdef \bigsqcup_{i=-\infty}^n \Delta_i.
\end{equation}
The elements of $\Delta$ (i.e., any atom of one of the hierarchical partitions $\{\Delta_i\}_{i=-\infty}^n$) are often called Christ cubes. The disjoint union in~\eqref{eq:disjoint union} indicates that we are slightly abusing notation by considering the elements of $\Delta_i$ and $\Delta_j$ to be distinct for $i\neq j$. In other words, if $Q\in \Delta_i$ for multiple integers $i\le n$ then $Q$ appears in $\Delta$ with that multiplicity, which by~\eqref{eq:volume and diameter growth} can be at most $(2\log_2 K)/s+1$. We impose a natural partial order on $\Delta$ by saying that $Q,Q'\in \Delta$ satisfy $Q\preceq Q'$ if and only if $Q\subset Q'$ and $Q\in \Delta_i$, $Q'\in \Delta_j$ for some $i\le j$. Below we  shall slightly abuse notation by letting $Q\subset Q'$ have the same meaning as $Q\preceq Q'$ for every $Q,Q'\in \Delta$ (note that this can be an actual abuse of notation only if $Q=Q'$). For every integer $i\le n$ and $Q\in \Delta_i$ we shall define $\sigma(Q)=2^i$ (we think of $\sigma(Q)$ as a surrogate for the ``side-length" of the Christ cube $Q$). Thus if $\Delta$ contains multiple copies of $Q$ then each such copy has a different ``$\sigma$-value," corresponding to those indices $i\le n$ for which $Q\in \Delta_i$. Under these conventions, the condition~\eqref{eq:volume and diameter growth} becomes the requirement that for every $Q\in \Delta$ we have
\begin{equation}\label{eq:sigma version diam volume}
  \frac{\sigma(Q)}{K}< \diam(Q)<K\sigma(Q)\qquad\mathrm{and}\qquad \frac{\sigma(Q)^s}{K}<\cH^{s}(Q)<K\sigma(Q)^s,
\end{equation}
and the condition~\eqref{eq:PatchworkSmallBdry} becomes the requirement that for every $Q\in \Delta$ and $t>0$ we have
\begin{equation}\label{eq:PatchworkSmallBdry sigma version}
  \cH^s\big(\partial_{\le t\sigma(Q)} Q\big)\le Kt^{\frac{1}{K}}\sigma(Q)^s.
\end{equation}

While the above definition makes sense for every $s>0$, in the present article we will only consider cubical patchworks of boundaries of open sets of $\H^{2k+1}$, in which case it will be tacitly assumed that $s=2k+1$. Thus given $E\subset \H^{2k+1}$, when we say below that $\{\Delta_i\}_{i=-\infty}^n$ is a $(K,r)$-cubical patchwork for $\partial E\subset \H^{2k+1}$ this  will always be understood to mean that $\{\Delta_i\}_{i=-\infty}^n$ is a $(2k+1)$-dimensional $(K,r)$-cubical patchwork for $\partial E$ in the above sense.
\end{defn}

Christ~\cite{ChristTb} and David~\cite{DavidWavelets} proved that for every $C,r,s>0$ there exists $K=K(C,s)$ such that if a Borel subset $A\subset \H^{2k+1}$ is $(C,r)$-Ahlfors $s$-regular then $A$ also admits an $s$-dimensional $(K,r)$-cubical patchwork. Note that the statements of~\cite{ChristTb, DavidWavelets}  assume global Ahlfors regularity rather than local Ahlfors regularity, i.e., the results are stated in the case $r=\diam(A)$. Nevertheless, the proofs of~\cite{ChristTb, DavidWavelets} work identically for general $r$: Discard the cubes that these proofs construct for larger scales (they are irrelevant for the desired local conclusion), so that the Ahlfors regularity is needed only in the range of scales in which it is hypothesized to hold true.

We record for ease of later reference (specifically, in Section~\ref{sec:constructing corona} below) the following very simple fact that follows by contrasting~\eqref{eq:sigma version diam volume} with~\eqref{eq:PatchworkSmallBdry sigma version}; see also~\cite[Lemma~I.3.5]{DSAnalysis}.

\begin{lemma}[existence of ``approximate centers" for Christ cubes]\label{lem:cubeCenters} For every $K,r,s>0$ there exists $c=c(K)$ such that  if a Borel subset $A\subset \H^{2k+1}$ admits an $s$-dimensional $(K,r)$-cubical patchwork $\{\Delta_i\}_{i=-\infty}^n$ then for every  $Q\in \Delta$ there is a point $x_Q\in Q$ such
  that
  \[d(x_Q,A\setminus Q)> c \sigma(Q).\]
\end{lemma}

\begin{proof} Choose any $c<1/K^{2K}$. Fix $Q\in \Delta$ and suppose for the purpose of obtaining a contradiction that $d(x,A\setminus Q)\le c \sigma(Q)$ for every $x\in Q$. Recalling~\eqref{eq:def partial less r}, this means that $\partial_{\le c\sigma(Q)}Q\supset Q$. Therefore
$$
\frac{\sigma(Q)^s}{K}\stackrel{\eqref{eq:sigma version diam volume}}{\le} \cH^s(Q)\le \cH^s\left(\partial_{\le c\sigma(Q)} Q\right)\stackrel{\eqref{eq:PatchworkSmallBdry sigma version}}{\le} Kc^{\frac{1}{K}}\sigma(Q)^s.
$$
By cancelling $\sigma(Q)^s$ from both sides,  this contradicts our choice of $c$.
\end{proof}

The following very important definition is equivalent to Definition~3.9 in~\cite{DSAnalysis}, though note that the constant $C$ that appears there differs from the constant $C$ that we use below (by a factor that depends only on the parameters of the given cubical patchwork).

\begin{defn}[Carleson packing condition]\label{def:packing} Fix $r,s,K,C\in (0,\infty)$ and suppose that $\{\Delta_i\}_{i=-\infty}^n$ is an  $s$-dimensional $(K,r)$-cubical patchwork of a Borel subset $A\subset \H^{2k+1}$. If $\mathcal{D}\subset \Delta$ is a collection of Christ cubes, we say that $\mathcal{D}$ satisfies the $s$-dimensional Carleson packing condition with constant $C$ if for every $Q\in \Delta$ we have
\begin{equation*}\label{eq:def carleson}
\sum_{\substack{R\in \mathcal{D}\\R\subset Q}} \sigma(R)^s\le C\sigma(Q)^s.
\end{equation*}
When $s=2k+1$ it will be convenient to  use below the shorter terminology ``$\mathcal{D}$ is $C$-Carleson" to indicate that  $\mathcal{D}$ satisfies the $(2k+1)$-dimensional Carleson packing condition with constant $C$.
\end{defn}
  A Carleson packing condition expresses the fact that
$\mathcal{D}$ is a ``small'' set of cubes.  It implies, for instance, that
$\mathcal{H}^s$-almost every point $x\in A$ is contained in only finitely many
elements of $\mathcal{D}$.

\begin{defn}[coherent collection of cubes]\label{def:coherent}
Fix $r,s,K\in (0,\infty)$ and suppose that $\{\Delta_i\}_{i=-\infty}^n$ is an  $s$-dimensional $(K,r)$-cubical patchwork of a Borel subset $A\subset \H^{2k+1}$. If $\mathcal{S}\subset \Delta$ is a sub-collection of Christ cubes then we say that $\mathcal{S}$ is coherent if the following three properties hold true.
\begin{enumerate}
\item $\mathcal{S}$ has a (necessarily unique) maximal element with respect to inclusion, i.e, there exists a unique cube $\mathsf{Q}(\mathcal{S})\in\mathcal{S}$ such that $\mathsf{Q}(\mathcal{S})\supset Q$ for every $Q\in \mathcal{S}$.
\item If $Q\in \mathcal{S}$ and $Q'\in \Delta$  satisfies $Q\subset Q'\subset \mathsf{Q}(\mathcal{S})$ then also $Q'\in\mathcal{S}$.
\item If $Q\in \mathcal{S}$ then either all the children of $Q$ in $\{\Delta_i\}_{i=-\infty}^n$ belong to $\mathcal{S}$ or none of them do, i.e., if $\sigma(Q)=2^i$ then either $\{Q'\in \Delta_{i+1}\mid Q'\subset Q\}\subset \mathcal{S}$ or $\{Q'\in \Delta_{i+1}\mid Q'\subset Q\}\cap \mathcal{S}=\emptyset$.
\end{enumerate}
\end{defn}

Using Definition~\ref{def:packing}  and Definition~\ref{def:coherent}, and following~\cite[Definition~3.13]{DSAnalysis}, we formulate the  notion of a  {\em coronization}, which is a partition of a cubical patchwork $\Delta$ into ``good'' and ``bad'' sets.
\begin{defn}[local coronization]\label{def:coronization} Fix $K,C,r>0$ and  $E\subset \H^{2k+1}$.  A $(K,C,r)$-coronization of $\partial E$ is a triple $(\cB,\cG,\cF)$ with the following properties. There exists   a $(K,r)$-cubical patchwork $\{\Delta_i\}_{i=-\infty}^n$ for $\partial E$ such that $\cB\subset \Delta$ (the set of bad cubes) and $\cG\subset \Delta$
  (the set of good cubes) partition $\Delta$ into two disjoint sets, i.e., $\cB\cup \cG=\Delta$ and $\cB\cap\cG=\emptyset$, and $\cF\subset 2^\cG$ is a collection of subsets of $\cG$, which are called
  below stopping-time regions.  These sets are required to have the following properties.
  \begin{enumerate}
  \item $\cB$ is $C$-Carleson.
  \item The elements of $\cF$ are pairwise disjoint and their union is $\cG$.
  \item Each $\mathcal{S}\in \cF$ is coherent.
  \item The set of maximal cubes $\{\mathsf{Q}(\mathcal{S})\mid \mathcal{S}\in \cF\}$ is $C$-Carleson.
  \end{enumerate}
\end{defn}

\begin{defn}\label{def:local distance}
  If $U,V,W\subset \H^{2k+1}$, we define the \emph{$U$-local distance between $V$ and $W$} by
  $$d_U(V,W)\eqdef \inf \big\{r\mid (V\symdiff W)\cap U\subset \nbhd_r(\partial V) \cap \nbhd_r(\partial W)\big\},$$
  where we denote $\nbhd_r(A)=\big\{h\in \H^{2k+1}\mid d(h,A)<r\big\}$ for every $A\subset \H^{2k+1}$.
\end{defn}
The local distance is not a metric on subsets of $\H^{2k+1}$; it only satisfies the weaker version of the triangle inequality given below.
\begin{lemma}\label{lem:localHausProps}
  If $U,V,W,X\subset \H^{2k+1}$, then the following assertions hold true.
  \begin{enumerate}
  \item $d_U(V,W)=d_U(W,V)$.
  \item If $\H^{2k+1}\supset U'\supset U$, then $d_{U}(V,W)\le d_{U'}(V,W)$.
  \item $U\cap \partial V\subset \overline{\nbhd}_s(\partial W)$, where $s=d_U(V,W)$.
  \item If $r=\max \{d_U(V,W),d_U(W,X)\}$ and $U'=\nbhd_r(U)$, then
    $$d_U(V,X)\le d_{U'}(V,W)+d_{U'}(W,X).$$
  \end{enumerate}
\end{lemma}
\begin{proof}
  The first two properties follow from the definition.  For the third property, note that if $x\in U\cap \partial V$ but $x\not \in \partial W$, then there is a neighborhood of $x$ that contains points of $V$ and of $V^{\cc}$ but is disjoint from $\partial W$.  Consequently $x\in U\cap (\overline{V\symdiff W})$, so $d(x,\partial W)\le s$, as desired. For the fourth property, suppose that $x\in (V\symdiff X)\cap U$.  Then $x\in V\symdiff W$ or $x\in W\symdiff X$.  Without loss of generality, suppose that $x\in V\symdiff W$.  Then $d(x,\partial V)\le d_U(V,W)$ and $d(x,\partial W)\le d_U(V,W)$.  Let $w\in \partial W$ be such that $d(x,w)=d(x,\partial W)<r$.  Then $w\in U'$, so by property (3) above, $d(w,\partial X)\le d_{U'}(W,X)$ and
  \begin{equation*}
  d(x,\partial X)\le d(x,w)+d(w,\partial X)\le d_{U}(V,W)+d_{U'}(W,X)\le d_{U'}(V,W)+d_{U'}(W,X).\qedhere
  \end{equation*}
\end{proof}

A corona decomposition of a set is a coronization where every stopping-time region is close to an intrinsic Lipschitz graph.  In what follows, if $\Delta$ is a cubical patchwork for $A\subset \H^{2k+1}$, then for every $Q\in \Delta$ and $\rho>0$ denote $ N_{\rho}(Q)= \nbhd_{\rho\sigma(Q)}(Q)$ and $\rho Q= A\cap N_\rho(Q).$

\begin{defn}[local corona decomposition]\label{def:coronaDecomp} Fix $K,r,C,\lambda,\theta>0$. Given $E\subset \H^{2k+1}$, we say that the pair $(E,\partial E)$ admits a \emph{$(K,C,\lambda,\theta,r)$-corona decomposition} if there is a  $(K,C,r)$-coronization $(\cB,\cG,\cF)$ of $\partial E$ such that for each $\cS\in \cF$ there is an intrinsic $\lambda$-Lipschitz graph $\Gamma(\cS)$ that bounds a Lipschitz half-space $\Gamma^+(\cS)$, such that for all $Q\in \cS$ we have
  \begin{equation}\label{eq:defCoronaDecomp}
    d_{N_4(Q)}(\Gamma^+(\cS), E)\le \theta \sigma(Q).
  \end{equation}
  We say that the pair $(E,\partial E)$ admits a {$(K,r)$-corona decomposition} if for every $\lambda, \theta>0$ there exists $C=C(\lambda,\theta)$ such that the pair $(E,\partial E)$ admits a $(K,C,\lambda,\theta,r)$-corona decomposition.
\end{defn}
The number 4 in~\eqref{eq:defCoronaDecomp} is arbitrary; we can replace it with an arbitrarily large constant at the cost of increasing the Carleson packing constants of the coronization.

The following covering lemma will be helpful both to construct corona decompositions and to bound the vertical perimeter of a set with a corona decomposition. Fix $K,r>0$ and $E\subset \H^{2k+1}$. Suppose that $\Delta$ is a $(K,r)$-cubical patchwork of $\partial E$.  If $\cS\subset \Delta$ is a coherent set, we define $d_\cS\from \H^{2k+1}\to \R$ by
$$\forall\, y\in \H^{2k+1},\qquad d_{\cS}(y)=\inf_{Q\in \cS} \big\{d(y,Q)+\sigma(Q)\big\}.$$
Denote $\Gamma_0=d_{\cS}^{-1}(0)\subset \H^{2k+1}$. Since $d_{\cS}$ is an infimum of $1$-Lipschitz functions, it is also $1$-Lipschitz.

\begin{lemma}\label{lem:coveringQS mini}
  Let $\cS\subset \Delta$ be a coherent set and suppose that $0<\psi< \frac{1}{10}$.  There is a countable set of points $\{c_i\}_{i\in I} \subset 10\mathsf{Q}(\cS)$ such that the balls $\{B_i=B_{\psi d_{\cS}(c_i)}(c_i)\}_{i\in I}$ are disjoint and the balls $\{3B_i=B_{3\psi d_{\cS}(c_i)}(c_i)\}_{i\in I}$ satisfy
  $$10\mathsf{Q}(\cS)\subset  \bigcup_{i\in I}3B_i\cup \Gamma_0.$$
  Furthermore, if $0<\rho<1/\psi$, then the balls $\{\rho B_i=B_{\rho\psi d_{\cS}(c_i)}(c_i)\}_{i\in I}$ have bounded multiplicity with a bound depending on $\rho$ and $\psi$.
\end{lemma}
\begin{proof}
  Write $G=10\mathsf{Q}(\cS)\setminus \Gamma_0.$ Let $\{c_i\}_{i\in I}\subset G$ be a maximal set of points such that the balls $\{B_i=B_{\psi d_{\cS}(c_i)}(c_i)\}_{i\in I}$ are disjoint; this set is either finite or countably infinite.  For any $v\in G$, there is an $i\in I$ such that $B_{\psi d_{\cS}(v)}(v)$ intersects $B_i$, so $d(v,c_i)\le \psi (d_\cS(v)+d_\cS(c_i))$.  Since the mapping $d_\cS$ is 1-Lipschitz, it follows that the following estimates hold true.
  \begin{align*}
    d(v,c_i)&\le \psi \big(2d_{\cS}(c_i)+d(v,c_i)\big),\\
    \frac{9}{10} d(v,c_i)&\le 2 \psi d_{\cS}(c_i), \\
    d(v,c_i)&\le 3 \psi d_{\cS}(c_i).
  \end{align*}
  Hence $v\in 3B_i$ and thus $G \subset \bigcup_{i\in I} 3B_i.$ Next, suppose that $\rho \psi< 1$.  If $i\in I$ and $v\in \rho B_i$, then
  $$\frac{1}{1+\rho \psi} d_{\cS}(v) \le d_{\cS}(c_i)\le \frac{1}{1-\rho \psi}d_{\cS}(v).$$
  It follows that there is $\beta=\beta(\rho,\psi)>0$ such that $B_i$ is a ball of radius at least $\beta d_{\cS}(v)$ with center in $B_{d_{\cS}(v)/\beta}(v)$.  Since the balls $\{B_i\}_{i\in I}$ are disjoint, the Ahlfors-regularity of $\H^{2k+1}$ implies that there can be only boundedly many such balls, and thus only finitely many $i\in I$ such that $v\in \rho B_i$.
\end{proof}

\section{Isoperimetry of sets with corona decompositions}\label{sec:isoperimetry of corona}

Here we show that Proposition~\ref{prop:lip graph general statement} holds not only for intrinsic Lipschitz graphs, but also for subsets of $\H^{2k+1}$ that have a corona decomposition.  Specifically, we prove the following proposition.
\begin{prop}\label{prop:coronaBoundsVPerL2L1} Fix $K,C,\Lambda,r\in (0,\infty)$, $\theta\in (0,\frac{1}{240})$, $\lambda\in (0,1)$ and $p\in [1,\infty)$. Suppose that for every intrinsic $\lambda$-Lipschitz half-space $\Gamma^+\subset \H^{2k+1}$, every $R\in (0,\infty)$ and every $q\in \H^{2k+1}$ we have
  \begin{equation}\label{eq:assumption p bound on graphs}
   \left\| \vpfl{B_R(q)}(\Gamma^+)\right\|_{L_p(\R)}\le \Lambda R^{2k+1}.
  \end{equation}
Then for every $E\subset \H^{2k+1}$ such that $(E,\partial E)$ admits a $(K,C,\lambda,\theta, r)$-corona decomposition and  any $x\in \H^{2k+1}$ we have
  $$\big\|\vpfl{B_r(x)}(E)\big\|_{L_p(\R)}\lesssim_{k,K} C(\Lambda+1)r^{2k+1}.$$
\end{prop}
Note that if $k\in \{2,3,\ldots\}$, then by Proposition~\ref{prop:lip graph general statement} the assumption of Proposition~\ref{prop:coronaBoundsVPerL2L1} holds true with $p=2$ and  $\Lambda\lesssim 1/(1-\lambda)$. We stated Proposition~\ref{prop:coronaBoundsVPerL2L1} for general $p\in [1,\infty)$ rather than only for $p=2$ in anticipation of forthcoming work that treats the case $k=1$. In the next sections, we will show that surfaces of finite perimeter in $\H^{2k+1}$ can be decomposed into sets with corona decompositions and use Proposition~\ref{prop:coronaBoundsVPerL2L1} to prove Theorem~\ref{thm:isoperimetric continuous}.

The proof of Proposition~\ref{prop:coronaBoundsVPerL2L1} uses the following lemma.
\begin{lemma}\label{lem:covering error}
  Suppose that $\psi\in (0,\frac{1}{10})$, $K,r\in (0,\infty)$, $\lambda\in (0,1)$, $\theta\in (0,\frac{\psi}{12})$.  Let $E\subset \H^{2k+1}$ be a set admitting a $(K,r)$-cubical patchwork $\Delta$. Suppose that $\cS\subset \Delta$ is a coherent set with maximal cube $M$ and let $\Gamma^+$ be a Lipschitz half-space bounded by an intrinsic $\lambda$-Lipschitz graph $\Gamma$.  Let $\{c_i\}_{i\in I}$ and $\{B_i\}_{i\in I}$ be as in Lemma~\ref{lem:coveringQS mini}.  If $\cS$ and $\Gamma$ satisfy \eqref{eq:defCoronaDecomp}, then
  $$(\Gamma^+\symdiff E)\cap N_4(M)\subset  \bigcup_{i\in I} 4B_i\cup \Gamma.$$
\end{lemma}
\begin{proof}
  Let $F=\Gamma^+\symdiff E$ and suppose that $y\in F\cap N_4(M)$.  If $d_{\cS}(y)=0$, then for every $\epsilon\in (0, \sigma(M)]$ there is a $Q\in \cS$ such that $d(y,Q)+\sigma(Q)< \epsilon$.  Let $Q'$ be an ancestor of $Q$ such that $\frac{\epsilon}{2}\le \sigma(Q')\le \epsilon$.  Since $d(y,Q)<\epsilon$, we have $y\in N_4(Q')$ and $d(y,Q')+\sigma(Q')<2\epsilon$.  By \eqref{eq:defCoronaDecomp}, $d(x,\Gamma)\le \theta\sigma(Q)\le \theta\epsilon$.  Since this holds for every $\epsilon$, we have $y\in \Gamma$.

  We can therefore suppose from now on that $y\in F\cap N_4(M)$ and that $0<d_{\cS}(y)\le 5\sigma(M)$.  If $d_{\cS}(y)\ge \frac{\sigma(M)}{4}$, let $A=M$.  Then $y\in N_4(A)$ and $\frac{d_{\cS}(y)}{5}\le \sigma(A)\le 4d_{\cS}(y)$.  Otherwise, let $Q\in \cS$ be a cube such that $d(y,Q)+\sigma(Q)\le 2 d_{\cS}(y)$ and let $A$ be the ancestor of $Q$ such that $\frac{d_{\cS}(y)}{5} \le 2 d_{\cS}(y)\le \sigma(A)< 4 d_{\cS}(y)$.  Then $d(y,A)\le 2d_{\cS}(y)\le \sigma(A)$, so in either case, we have $y\in N_4(A)$ and $\frac{d_{\cS}(y)}{5}\le \sigma(A)\le 4d_{\cS}(y)$.  By \eqref{eq:defCoronaDecomp}, this implies that
  $$d(y,\partial E)\le \theta \sigma(A)\le \frac{\psi d_{\cS}(y)}{3}.$$

  Choose $v\in \partial E$ with $d(v,y)\le \frac{\psi}{2} d_{\cS}(y)$.  Then $d(v,y)\le \sigma(M)$, so $v\in 5M$.  Since $d_{\cS}$ is 1-Lipschitz, we have $d_{\cS}(v)\ge (1-\psi) d_{\cS}(y) >0$.  By Lemma~\ref{lem:coveringQS mini}, there is  $i\in I$ such that $v\in 3B_i$, and
  $$d_{\cS}(y)\le \frac{1}{1-\psi} d_{\cS}(v)\le \frac{1+3\psi}{1-\psi}
  d_{\cS}(c_i)\le 2 d_{\cS}(c_i).$$
  Thus
  \begin{equation*}
    d(c_i,y)\le d(c_i,v)+d(v,y)
            \le 3\psi d_{\cS}(c_i)+\frac{\psi}{2} d_{\cS}(y)
            \le 4\psi d_{\cS}(c_i).
  \end{equation*}
  Hence $y\in 4B_i$, as desired.
\end{proof}

\begin{proof}[{Proof of Proposition~\ref{prop:coronaBoundsVPerL2L1}}]
  By re-scaling and translating $E$ we may assume that $x=\0$ and $r=1$.  Let $\Delta=\{\Delta_i\}_{i=-\infty}^0$ be a $(K,r)$-cubical patchwork for $\partial E$ and let $(\cB, \cG, \cF)$ be a  coronization of $\partial E$ that satisfies Definition~\ref{def:coronaDecomp}.  In particular, $\cB$ and $\{\mathsf{Q}(\cS)\mid \cS\in \cF\}$ are both $C$-Carleson.

  Recall that for every $s\in \R$ we let $\mathsf{D}_s E=E\symdiff EZ^{2^{2s}}$.  We start by constructing a cover of $\mathsf{D}_sE$.  Suppose that $s<0$ and let $i$ be the integer such that $s\in (i-1,i]$.  If $x\in \mathsf{D}_s E$, then any path from $x$ to $x Z^{-2^{2s}}$ crosses $\partial E$. There is a path from $x$ to $x Z^{-2^{2s}}$ of length less than $4\cdot 2^s$, and this path intersects $\partial E$ in a point $y$ such that $d(x,y) < 4\cdot 2^s$. If $Q\in \Delta_i$ is the cube such that $y\in Q$, then $x\in N_4(Q)$.  It follows that
  \begin{equation}\label{eq:coverByN4}
    \mathsf{D}_s E\subset \bigcup_{Q\in \Delta_i} N_4(Q).
  \end{equation}

  Due to~\eqref{eq:coverByN4}, we can decompose $\vpf(E)$ into terms corresponding to cubes.  For $Q\in \Delta$, we let the $Q$-component of $\vpf(E)$ be
  $$\vpf(E;Q)(s)\eqdef \vpfl{N_4(Q)}(E)(s)\one_{\left\{\frac{\sigma(Q)}{2}< 2^s \le \sigma(Q)\right\}}.$$
  Then for each integer $i\le 0$ and $s\in (i-1,i]$, we have
  $$\sum_{Q\in \Delta}\vpf(E;Q)(s)=\sum_{Q\in \Delta_i}\vpf(E;Q)(s)=2^{-s}\sum_{Q\in\Delta_i}\vol(N_4(Q)\cap \mathsf{D}_s E).$$
  As $Q$ ranges over $\Delta_i$, the sets $N_4(Q)$ cover $\mathsf{D}_sE$ with  multiplicity that is bounded by a quantity that depends only on $k$ and $K$, so it follows that
  $$\forall\, s \in (-\infty,0],\qquad \vpf(E)(s)\asymp \sum_{Q\in \Delta} \vpf(E;Q)(s).$$

  For any subset $\cJ \subset \Delta$ and any $F\subset \H^{2k+1}$, denote $W_\cJ(F)\eqdef \sum_{Q\in \cJ} \vpf(F;Q)$.  If $M\in \Delta$, let $\Delta_M=\{Q\in \Delta\mid Q\subset M\}$.  We claim that
  \begin{equation}\label{eq:vert on M}
\forall\, M\in \Delta,\qquad   \left\|W_{\Delta_M}(E)\right\|_{L_p(\R)}\lesssim_{k,K} C(\Lambda+1)\sigma(M)^{2k+1}.
  \end{equation}

 To prove~\eqref{eq:vert on M}, fix $M\in \Delta$. We first separate $W_M(E)$ into contributions of bad cubes and stopping-time regions.  Let $\cB_M=\Delta_M\cap \cB$ and let $\cF_M=\{\cS\cap \Delta_M\mid \cS\in \cF\}$.  Note that $\cF_M$ consists of coherent collections of cubes, so that the maximal element $\mathsf{Q}(\cT)\in \cT$ is well-defined for every $\cT\in \cF_M$. Moreover, the elements of $\cF_M$ still satisfy~\eqref{eq:defCoronaDecomp} and $\{\mathsf{Q}(\cS)\mid \cS\in \cF_M\}$ is still $C$-Carleson.  By design we have
  \begin{equation}\label{eq:barVdecomp}
 \forall\, s\in (-\infty,0],\qquad    W_{\Delta_M}(E)(s)=W_{\cB_M}(E)(s) + \sum_{\cT\in \cF_M} W_{\cT}(E)(s).
  \end{equation}
We will prove~\eqref{eq:vert on M} by bounding each of the terms in~\eqref{eq:barVdecomp} separately.

  First, we bound $\|W_{\cB_M}(E)\|_{L_p(\R)}$.  If $Q\in \Delta_i$ and $s\in (i-1,i]$, we have the trivial bound
  \begin{equation*}
    \vpf(E;Q)(s)\le \frac{\vol(N_4(Q))}{2^s}\asymp_{k,K} \frac{2^{i(2k+2)}}{2^s}\asymp_{k,K} \sigma(Q)^{2k+1},
    \end{equation*}
    and consequently,
    \begin{equation}\label{eq:vert bad cubes on each Q}
    \|\vpf(E;Q)\|_{L_p(\R)}\lesssim \bigg(\int_{i-1}^i \vpf(E;Q)(s)^p\ud s\bigg)^{\frac{1}{p}} \lesssim_{k,K} \sigma(Q)^{2k+1}.
  \end{equation}
Hence,
  \begin{equation}\label{eq:bad cube error estimate}
    \|W_{\cB_M}(E)\|_{L_p(\R)}\le \sum_{Q\in \cB_M} \|\vpf(E;Q)\|_{L_p(\R)}\stackrel{\mathclap{\eqref{eq:vert bad cubes on each Q}}}{\lesssim}_{k,K} \sum_{Q\in \cB_M} \sigma(Q)^{2k+1}\le C \sigma(M)^{2k+1},
  \end{equation}
  where the final inequality in~\eqref{eq:bad cube error estimate} is the $C$-Carleson condition for $\cB$.

Our next goal is to show that
\begin{equation}\label{eq:point for ech cT}
\forall\, \cT\in \cF_M,\qquad \|W_\cT(E)\|_{L_p(\R)}\lesssim_{k,K} (\Lambda+1)\sigma\big(\mathsf{Q}(\cT)\big)^{2k+1}.
\end{equation}
Let $\Gamma^+$ be a $\lambda$-Lipschitz half-space such that $d_{N_4(\mathsf{Q}(\cT))}(\Gamma^+, E)\le \theta \sigma(\mathsf{Q}(\cT)).$ Denote $A=\Gamma^+\symdiff E$.  Then $\mathsf{D}_sE = \mathsf{D}_sA \symdiff \mathsf{D}_s\Gamma^+$, so $W_\cT(E)\le W_\cT(A)+W_\cT(\Gamma^+).$ We shall  bound $W_\cT(\Gamma^+)$ using~\eqref{eq:assumption p bound on graphs}.  Suppose that $q\in \mathsf{Q}(\cT)$ and let $R=\diam \mathsf{Q}(\cT)+4\sigma(\mathsf{Q}(\cT))$. Then $N_4(\mathsf{Q}(\cT))\subset B_R(q)$.  We claim that for all $s\in \R$ we have $W_\cT(\Gamma^+)(s) \lesssim_{k,K} \vpfl{B_R(q)}(\Gamma^+)(s)$. Indeed, if $i=\lceil s\rceil$, then
  \begin{equation}\label{eq:W cT gamma+}
    W_\cT(\Gamma^+)(s) =\sum_{Q\in \cT\cap \Delta_i} 2^{-s}\vol\big(N_4(Q)\cap \mathsf{D}_s\Gamma^+\big).
  \end{equation}
 The sets $\{N_4(Q):\ Q\in \cT\cap \Delta_i\}$ have multiplicity bounded by a quantity that depends only on $k$ and $K$, and they  are all contained in $B_R(q)$.  It therefore follows from~\eqref{eq:W cT gamma+} that
  $$W_\cT(\Gamma^+)(s) \lesssim_{k,K} 2^{-s} \vol\big(B_R(q)\cap \mathsf{D}_s\Gamma^+\big) = \vpfl{B_R(q)}(\Gamma^+)(s),$$
which is the required point-wise estimate. By the assumption of Proposition~\ref{prop:coronaBoundsVPerL2L1} we deduce that
  \begin{equation}\label{eq:WS Gamma}
    \|W_\cT(\Gamma^+)\|_{L_p(\R)}\lesssim_{k,K} \big\|\vpfl{B_R(q)}(\Gamma^+)\big\|_{L_p(\R)} \stackrel{\mathclap{\eqref{eq:assumption p bound on graphs}}}{\lesssim}_{k,K} \Lambda R^{2k+1}\asymp_{k,K} \Lambda \sigma\big(\mathsf{Q}(\cT)\big)^{2k+1}.
  \end{equation}

Due to~\eqref{eq:WS Gamma}, to establish~\eqref{eq:point for ech cT}  it suffices to show that $\|W_\cT(A)\|_{L_p(\R)}\lesssim_{k,K} \sigma(\mathsf{Q}(\cT))^{2k+1}$. Let $\psi=\frac{1}{20}$ (thus $\theta<\frac{\psi}{12}$) and let $\{c_i\}_{i\in I}\subset 10\mathsf{Q}(\cT)$ and $\{4B_i=B_{4\psi d_{\cT}(c_i)}(c_i)\}_{i\in I}$ be the covering constructed in Lemma~\ref{lem:coveringQS mini}.  Let $A'=A\cap N_4(\mathsf{Q}(\cT))$.  By Lemma~\ref{lem:covering error}, $A'\subset \bigcup_{i\in I} 4B_i\cup \Gamma$.

For $s\in \R$ and $F\subset \H^{2k+1}$, let $\mathsf{U}_sF\eqdef F\cup F Z^{2^{2s}}$.  For every $Q\in \Delta$ define
  $$\vpfprime(F;Q)(s)\eqdef\frac{\vol(N_4(Q)\cap \mathsf{U}_sF)}{2^s} \one_{\{\frac{\sigma(Q)}{2}< 2^s \le \sigma(Q)\}},$$
  and let $W'_{\cT}(F)\eqdef\sum_{Q\in \cT} \vpfprime(F;Q)$.  Then $\mathsf{D}_sF\subset \mathsf{U}_sF$ and therefore we have the point-wise inequality $ \vpfprime(F;Q)\le \vpf(F;Q)$.  Importantly, $\vpfprime(\cdot;Q)$ is subadditive, i.e., if $F,G\subset \H^{2k+1}$, then $\vpfprime(F\cup G;Q)\le \vpfprime(F;Q)+\vpfprime(G;Q)$.  Consequently, $W_{\cT}(A')\le \sum_{i\in I} W'_{\cT}(4B_i).$

  Let $y=c_i\in 10\mathsf{Q}(\cT)$ for some $i\in I$ and let $B=4B_i=B_{d_{\cT}(y)/5}(y)$.  As $Q$ ranges over $\Delta_j$ for some $j\in \Z$, the sets $N_4(Q)$ cover $\mathsf{U}_sB$ with  multiplicity bounded by a quantity that depends only on $k,K$. Hence, for any $s\in \R$, we have
  \begin{equation}\label{eq:poit-wise boud W'}
  0\le W'_{\cT}(B)(s)\lesssim_{k,K} \frac{\vol(\mathsf{U}_sB)}{2^s}\lesssim \frac{d_{\cT}(y)^{2k+2}}{2^s}.
  \end{equation}
  We claim that if $W'_{\cT}(B)(s)>0$, then $2^s> d_{\cT}(y)/24$.  Indeed, suppose that $Q\in \cT$ and $s\in \R$ are such that $\vpfprime(B;Q)(s)> 0$.  Then $\sigma(Q)/2< 2^s\le \sigma(Q)$ and $(B\cup BZ^{2^{2s}})\cap N_4(Q)\ne \emptyset$.  Let $x$ lie in this intersection.  Since $x\in N_4(Q)$, we have $d_{\cT}(x)\le 5\sigma(Q)$ and since $x\in B\cup BZ^{2^{2s}}$ we have
  $$d(x,y)\le \frac{d_{\cT}(y)}{5}+4\cdot 2^s\le\frac{d_{\cT}(y)}{5}+4\sigma(Q).$$
  Since $d_{\cS}$ is 1--Lipschitz, it follows that
  $$d_{\cT}(y)\le d_{\cT}(x)+d(x,y)\le \frac{d_{\cT}(y)}{5}+9\sigma(Q).$$
  Solving for $d_{\cT}(y)$, we find that $d_{\cT}(y) < 12\sigma(Q)$ and thus $2^s> \frac{d_{\cT}(y)}{24}$. Consequently,
  \begin{equation*}
    \|W'_{\cT}(B)\|_{L_p(\R)}
    =\bigg(\int_{\log_2 \frac{d_{\cT}(y)}{24}}^\infty W'_{\cT}(B)(s)^p\ud s\bigg)^{\frac{1}{p}}
    \stackrel{\mathclap{\eqref{eq:poit-wise boud W'}}}{\lesssim}_{k,K} \bigg(\int_{\log_2 \frac{d_{\cT}(y)}{24}}^\infty \frac{d_{\cT}(y)^{p(2k+2)}}{2^{ps}} \ud s\bigg)^{\frac{1}{p}}
    \lesssim d_{\cT}(y)^{2k+1}.
  \end{equation*}
This proves that
  $$\|W_{\cT}(A')\|_{L_p(\R)} \le \sum_{i\in I} \|W'_{\cT}(4B_i)\|_{L_p(\R)} \lesssim_{k,K} \sum_{i\in I} d_{\cT}(c_i)^{2k+1}.$$

  The sets $\{\partial E\cap B_i\}_{i\in I}$ are disjoint, and for each $i$, $\cH^{2k+1}(\partial E\cap B_i)\asymp d_{\cT}(c_i)^{2k+1}$.  Since $c_i\in 10\mathsf{Q}(\cT)$, we have $d_{\cT}(c_i)\le11 \sigma(\mathsf{Q}(\cT))$.  Therefore, $\partial E\cap B_i\subset 12\mathsf{Q}(\cT)$, and
  $$\|W_{\cT}(A')\|_{L_p(\R)}\lesssim_{k,K}\sum_{i\in I} d_{\cT}(c_i)^{2k+1}\asymp \cH^{2k+1}\Big(\partial E\cap \bigcup_{i\in I} B_i\Big)\le \cH^{2k+1}\big(12\mathsf{Q}(\cT)\big)\lesssim \sigma\big(\mathsf{Q}(\cT)\big)^{2k+1}.$$

  Now we correct for the difference between $W_{\cT}(A')$ and $W_{\cT}(A)$.  Suppose that $Q\in \cT$ and $s\in \R$ satisfies $\sigma(Q)/2< 2^s\le\sigma(Q)$. If in addition $Q\ne \mathsf{Q}(\cT)$, then $\sigma(Q)\le \sigma(\mathsf{Q}(\cT))/2$ and therefore $N_4(Q)\cup N_4(Q)Z^{-2^{2s}}\subset N_8(Q)\subset N_4(\mathsf{Q}(\cT))$. Consequently,  $N_4(Q)\cap \mathsf{D}_s A=N_4(Q)\cap \mathsf{D}_s A'$. Hence, $\vpf(A';Q)=\vpf(Q;A)$, and therefore $W_\cT(A)=W_{\cT}(A')+ \vpf(A;\mathsf{Q}(\cT))-\vpf(A';\mathsf{Q}(\cT))$. So,
  \begin{multline}\label{eq:WS error}
    \|W_\cT(A)\|_{L_p(\R)} = \big\|W_\cT(A') +\vpf\big(A;\mathsf{Q}(\cT)\big)-\vpf\big(A';\mathsf{Q}(\cT)\big)\big\|_2\\ \le \|W_\cT(A')\|_{L_p(\R)} + 2 \frac{\vol\big(N_4(\mathsf{Q}(\cT))\big)}{\sigma\big(\mathsf{Q}(\cT)\big)/2}\lesssim \sigma\big(\mathsf{Q}(\cT)\big)^{2k+1}.
  \end{multline}
  Therefore, for any $\cT\in \cF$,
  \begin{align*}
    \|W_\cT(E)\|_{L_p(\R)}
    \le \|W_\cT(A)\|_{L_p(\R)}+\|W_\cT(\Gamma^+)\|_{L_p(\R)}
    \stackrel{\eqref{eq:WS Gamma}\wedge \eqref{eq:WS error}}{\lesssim _{\mathrlap{k,K}} }(\Lambda+1)\sigma\big(\mathsf{Q}(S)\big)^{2k+1}.
  \end{align*}
This completes the justification of~\eqref{eq:point for ech cT}.

  By combining~\eqref{eq:point for ech cT} with the Carleson packing condition on $\cF_M$, \eqref{eq:barVdecomp}, and \eqref{eq:bad cube error estimate}, for any $M\in \Delta$,
  \begin{multline*}
    \|W_{\Delta_M}(E)\|_{L_p(\R)}
    \le \|W_{\cB_M}(E)\|_{L_p(\R)} + \sum_{\cT\in \cF_M} \|W_\cT(E)\|_{L_p(\R)}
    \\\lesssim_{k,K} C\sigma(M)^{2k+1}+(\Lambda+1)\sum_{\cT\in \cF_M} \sigma\big(\mathsf{Q}(\cT)\big)^{2k+1}
    \lesssim C(\Lambda+1)\sigma(M)^{2k+1},
  \end{multline*}
  thus completing the proof of~\eqref{eq:vert on M}.

  Finally, we bound $\|\vpfl{B_1}(E)\|_{L_p(\R)}$.  Let $M_1,\dots, M_m\in \Delta_0$ be the set of cubes such that $B_5\cap M_i\ne \emptyset$ for all $i\in \m$. The number of such cubes is bounded by a quantity that depends only on $k$ and $K$.  If $Q\in \Delta$ and $N_4(Q)\cap B_1\ne \emptyset$, then $Q\subset M_i$ for some $i\in \m$, so for every $s\le 0$ we have
  $\vpfl{B_1}(E)(s)\le \sum_{i=1}^mW_{\Delta_{M_i}}(E)(s)$
Hence by~\eqref{eq:vert on M} we have $\big\|\vpfl{B_1}(E)\big\|_{L_p(-\infty,0)}\lesssim_{k,K} C(\Lambda+1).$ On the other hand, if $s>0$, then
  $\vpfl{B_1}(E)(s)\le 2^{-s}\vol(B_1)\lesssim 2^{-s},$ so $\big\|\vpfl{B_1}(E)\big\|_{L_p(0,\infty)} \lesssim 1$, and we conclude that $\left\|\vpfl{B_1}(E)\right\|_{L_p(\R)} \lesssim_{k,K}C(\Lambda+1)$, as desired.
\end{proof}

\section{Decomposing surfaces into Ahlfors regular pieces}
\label{sec:decomposing cellular}
In Section~\ref{sec:isoperimetry of corona}, we established Theorem~\ref{thm:isoperimetric continuous} for sets admitting a corona decomposition.  In this section and the next section, we will prove it for cellular sets and show that this implies Theorem~\ref{thm:isoperimetric continuous} in full generality. The following theorem will be proven in Section~\ref{sec:constructing corona}.

\begin{thm}\label{thm:Ahlfors admits corona}
  Fix $C,r\in (0,\infty)$. Let $E\subset \H^{2k+1}$ be a set such that $(E,E^\cc,\partial E)$ is $(C,r)$-regular as in Definition~\ref{def:local Ahlfors regularity}.  Then there exists $K=K(k,C)\in (0,\infty)$ such that the  pair $(E,\partial E)$ admits a $(K,r)$-corona decomposition as in Definition~\ref{def:coronaDecomp}.
\end{thm}

Theorem~\ref{thm:Ahlfors admits corona} asserts that local Ahlfors-regularity implies the existence of local intrinsic corona decompositions. In order to apply this structural statement, we shall prove in this section a further structural result that provides a decomposition of a general finite-perimeter cellular set into pieces that are locally Ahlfors-regular.

\begin{lemma}\label{lem:ballDecomp}
  There exists $C=C_k>0$ such that if $E\subset \H^{2k+1}$ is a finite-perimeter cellular set, then one can find $n\in \N$, $\{x_i\}_{i=0}^{n-1}\subset \H^{2k+1}$, $\{s_i\}_{i=0}^{n-1}\subset [1,\infty)$ and a sequence  $E_0=E,E_1,\dots, E_n=\emptyset$ of subsets of $\H^{2k+1}$ such that if we denote $B_i=B_{s_i}(x_i)$ then the following assertions hold true.
 \begin{itemize}
 \item $(E_i,E^\cc_i,\partial E_i)$ is $(C,s_i)$-regular for every $i\in \n$,
 \item $E_i\symdiff E_{i+1}\subset B_i$ for every $i\in \{0,\ldots,n-1\}$,
 \item $\sum_{i=0}^{n-1} s_i^{2k+1} \lesssim \cH^{2k+1}(\partial E)$,
 \item For every $p\in [1,\infty)$ we have
   \begin{equation}\label{eq:ball decomp vper}
   \|\vpf(E)\|_{L_p(\R)}\le C\cH^{2k+1}(\partial E)+\sum_{i=0}^{n-1} \big\|\vpfl{B_i}(E_{i})\big\|_{L_p(\R)}.
   \end{equation}
 \end{itemize}
\end{lemma}

Before proving Lemma~\ref{lem:ballDecomp}, we now show that Theorem~\ref{thm:isoperimetric continuous} follows quickly from Theorem~\ref{thm:Ahlfors admits corona} and Lemma~\ref{lem:ballDecomp}, in combination with the results that we already proved in Sections~\ref{sec:reductions}, \ref{sec:intrinsic graphs} and~\ref{sec:isoperimetry of corona}.

\begin{proof}[{Proof of Theorem~\ref{thm:isoperimetric continuous} assuming Theorem~\ref{thm:Ahlfors admits corona} and Lemma~\ref{lem:ballDecomp}}] In Section~\ref{sec:reductions} we have already shown that Theorem~\ref{thm:isoperimetric continuous} follows from Theorem~\ref{thm:isoperimetric cellular}, i.e., its special case for cellular sets. So, suppose that $E\subset \H^{2k+1}$ is a cellular set with $\cH^{2k+1}(\partial E)<\infty$. Let $n\in \N$, $\{s_i\}_{i=0}^{n-1}$, $\{E_i\}_{i=0}^n$, $\{B_i\}_{i=0}^{n-1}$ be as in  Lemma~\ref{lem:ballDecomp}. By Theorem~\ref{thm:Ahlfors admits corona} we know that for every $i\in \n$ the pair $(E_i,\partial E_i)$ admits a $(K,s_i)$-local corona decomposition with $K\lesssim 1$. By combining Proposition~\ref{prop:lip graph general statement} and Proposition~\ref{prop:coronaBoundsVPerL2L1} we see that $\vpfl{B_i}(E_{i})\lesssim s_i^{2k+1}$ for every $i\in \{0,\ldots,n-1\}$. Hence,
$$
\|\vpf(E)\|_{L_2(\R)}\lesssim \cH^{2k+1}(\partial E)+\sum_{i=0}^{n-1} \big\|\vpfl{B_i}(E_{i})\big\|_{L_p(\R)}\lesssim  \cH^{2k+1}(\partial E)+\sum_{i=0}^{n-1} s_i^{2k+1}\lesssim \cH^{2k+1}(\partial E),
$$
where the first inequality is~\eqref{eq:ball decomp vper} and the final inequality uses the third assertion of Lemma~\ref{lem:ballDecomp}.
\end{proof}

The remainder of this section is devoted to the proof of Lemma~\ref{lem:ballDecomp}.  To do so, we will need some preparatory definitions and lemmas regarding cellular sets.  Recall that $E\subset \H^{2k+1}$ is cellular if it is a union of translates of the unit cube $C_0=[-\frac{1}{2},\frac{1}{2}]^{2k+1}$ by elements of $\H^{2k+1}_\Z$.  The set of all translates of $C_0$ by elements of $\H^{2k+1}_\Z$ gives rise to the unit lattice $\tau^{2k+1}$.  If $C\in \tau^{2k+1}$, then $C=hC_0$ for some $h\in \H^{2k+1}_\Z$, which we denote $h_C$.

For $x,y,z>0$, let $C_{x,y,z}$ be the cuboid $[-x,x]^k\times [-y,y]^k\times [-z,z]\subset \H^{2k+1}$.  If $h\in \H^{2k+1}_\Z$ and $r>0$ is an integer, we define the \emph{cellular ball} $\Bword_r(h)$ to be the set
$$\Bword_r(h)\eqdef h\cdot \Bword_r\eqdef h\big(\H^{2k+1}_\Z\cap C_{r,r,r^2}\big) C_0.$$

\begin{lemma}\label{lem:cellBallProps}
  There exists $c=c_k\in (0,\infty)$ such that for all $g\in \H^{2k+1}_\Z$ and all $r\in \N$ and $p\in (0,\infty)$,
  \begin{itemize}
\item    $B_{r/c}(g) \subset \Bword_r(g)\subset B_{cr}(g)$,
\item   $\cH^{2k+2}\big(\Bword_r(g)\big)\asymp r^{2k+2}$,
 \item    $\cH^{2k+1}\big(\partial\Bword_r(g)\big)\asymp r^{2k+1}$,
\item    $\big\|\vpf\big(\Bword_r(g)\big)\big\|_{L_p(\R)} \asymp r^{2k+1}$.
  \end{itemize}
\end{lemma}
\begin{proof}
  It suffices to consider the case that $g=\0$.  The first and second assertions follow from \eqref{eq:metric approximation} and the fact that $\cH^{2k+2}(B_r)\asymp r^{2k+2}$.

  If $D\in \tau^{2k+1}$ is adjacent to the cell $C_0$, then there is a $v\in C_0\cap D$.  Then $v, h_D^{-1}v\in C_0$, so $h_D\in C_0 \cdot C_0^{-1}=C_0^2$.  By~\eqref{eq:def group product omega}, $C_0^2\subset C_{1,1,1+k/4}$. Let $A\subset \tau^{2k+1}$ be the set of $(2k+1)$-cells contained in $\Bword_r$ that intersect $\partial \Bword_r$.  Each element of $A$ is adjacent to some cell $D$ such that $h_D\not \in C_{r,r,r^2}$, so if $C\in A$, then $h_C\not \in C_{r-1,r-1,r^2-2kr}$.  Then $A\subset C_{r,r,r^2}\setminus C_{r-1,r-1,r^2-2kr}$, so $|A|\lesssim r^{2k+1}$.  Since the boundary of $\Bword_r$ is contained in the boundaries of the cells of $A$, we have $\cH^{2k+1}(\partial\Bword_r(g))\lesssim r^{2k+1}$.  On the other hand, Lemma~\ref{lem:relativeIsoperimetric} implies that $\cH^{2k+1}(\partial\Bword_r)\gtrsim r^{2k+1}$. This proves the third assertion.

  Finally, to bound the vertical perimeter of $\Bword_r$, note that $\Bword_r$ consists of the union of $(2r+1)^{2k}$ ``stacks'' of parallelepipeds, i.e., regions of the form $S_{h,n}=\bigcup_{i=0}^n h Z^i C_0$ for some $h\in \H^{2k+1}_\Z$ and some height $n\lesssim r^2$.  Each such stack is a single, very tall parallelepiped, so
  $\vpf(S_{h,n})(s)=\frac{\min \{2^{2s}, n\}}{2^s}$ for every $s\in \R$. It follows that
  $$\vpf\big(\Bword_r\big)(s)\asymp (2r+1)^{2k}\cdot \frac{\min \{2^{2s}, r^2\}}{2^s}.$$
  By integrating these estimates we conclude that $\big\|\vpf\big(\Bword_r\big)\big\|_{L_p(\R)}\asymp r^{2k+1}$.
\end{proof}

We will need the following discrete version of Ahlfors regularity.

\begin{defn}
  Fix $C,r,s>0$.  A set $A\subset \H^{2k+1}$ is discretely $(C,r)$-Ahlfors $s$-regular if
  $$\frac{\rho^s}{C}\le \cH^s\big(A\cap \Bword_\rho(x)\big)\le C\rho^s,$$
  for all $1\le \rho \le r$ and all $x\in \H_\Z^{2k+1}$ such that $\cH^{s}(x C_0\cap A)>0$.

  We will use the same convention on dimension as in Definition~\ref{def:local Ahlfors regularity}.  That is, if $E\subset \H^{2k+1}$ and we say that $E$ is discretely $(C,r)$-regular, then it will be understood that $E$ is discretely $(C,r)$-Ahlfors $s$-regular with $s=2k+2$.  When we say that $\partial E$ is discretely $(C,r)$-regular we mean that $\partial E$ is assumed to be discretely $(C,r)$-Ahlfors $s$-regular with $s=2k+1$.
\end{defn}

By Lemma~\ref{lem:cellBallProps}, a $(C,r)$-regular set is discretely $(C',r)$-regular for some $C'>0$ depending on $C$ and $k$.  Conversely, any union of $s$-cells of $\tau^{2k+1}$ is locally Ahlfors regular on sufficiently small scales, so if $E$ is cellular and $(E,E^\cc,\partial E)$ is discretely $(C,r)$-regular, then it is also $(C',r)$-regular for some $C'>0$ depending on $C$ and $k$.

Recall that Lemma~\ref{lem:relativeIsoperimetric} is a relative isoperimetric inequality for subsets of balls in $\H^{2k+1}$.  By Lemma~\ref{lem:cellBallProps}, Lemma~\ref{lem:relativeIsoperimetric} also holds (with different constants) for subsets of cellular balls, i.e., there exists $S=S_k\in (1,\infty)$ such that if $r\ge 1$ and $h\in \H^{2k+1}$, then
\begin{equation}\label{eq:relative isoperimetric cellular}
  \left(\frac{\vol\big(E\cap h\Bword_r\big)\cdot \vol\big(E^{\cc}\cap h\Bword_r\big)}{\vol\big(\Bword_r\big)^2}\right)^{\frac{2k+1}{2k+2}} \lesssim \frac{r \area\big(\partial E\cap h\Bword_{Sr}\big)}{\vol\big(h\Bword_{Sr}\big)}.
\end{equation}

\begin{proof}[{Proof of Lemma~\ref{lem:ballDecomp}}]
  We will construct the $\{E_i\}_{i=0}^n$ inductively, by choosing appropriate cellular balls $H_i=\Bword_{r_i}(x_i)$ with $r_i\ge 1$ and letting
  \begin{equation}\label{eq:def Ei+1}E_{i+1}\eqdef
  \begin{cases}
    E_i\cup H_i & \text{if $\cH^{2k+2}(E_i\cap H_i)>\frac{\cH^{2k+2}(H_i)}{2}$},\\
    \overline{E_i \setminus  H_i} & \text{if $\cH^{2k+2}(E_i\cap H_i)\le\frac{\cH^{2k+2}(H_i)}{2}$.}
  \end{cases}\end{equation}
  This ensures that $E_i$ will be a cellular set for all $i\in \n$.  By the remark above, $(E_i,E_i^\cc,\partial E_i)$ is regular if and only if it is discretely regular.

  The most important factor in choosing the cellular ball $H_i$ is that the size of $\partial E_i$ must decrease quickly as $i$ increases.  Specifically, we claim for any finite-perimeter cellular set $E=E_0$, we can choose a cellular ball $H=H_0=\Bword_{r_0}(x_0)=\Bword_{r}(x)$ so that if $E_1$ is defined as above, then
  \begin{equation}\label{eq:area decrease inequality}
    \cH^{2k+1}(\partial E) - \cH^{2k+1}(\partial E_{1})=\cH^{2k+1}(H\cap \partial E) - \cH^{2k+1}(H\cap \partial E_{1})\gtrsim r^{2k+1}.
  \end{equation}

  Let $c_0>0$ be a large constant to be determined later and let $S$ be as in~\eqref{eq:relative isoperimetric cellular}.  Let $r\in \N$ be the largest integer such that $\partial E$ is discretely $(c_0,r-1)$-regular and $E,E^{\cc}$ are both discretely $(c_0^{(2k+2)/(2k+1)}, r-1)$-regular. Since $E$ is cellular, we can assume $r\ge 10S$ by taking $c_0$ to be sufficiently large.  Any ball of radius $r$ can be covered by boundedly (depending only on $k$) many balls of radius $r-1$, so $(E, E^{\cc}, \partial E)$ is discretely $(c_1,r)$--regular, where $c_1\asymp c_0^{(2k+2)/(2k+1)}$.

  If $r\ge \diam E$, then we can take $H$ to be a ball of radius $2r$ containing $E$, so that $E_1=\emptyset$.  Since $\partial E$ is discretely $(c_0,r-1)$--regular, we have $\cH^{2k+1}(\partial E)\gtrsim r^{2k+1}$ and thus \eqref{eq:area decrease inequality} holds.  Otherwise, suppose that $r<\diam E$.  Then one of the following conditions holds.
  \begin{enumerate}
  \item
    There is $x\in \partial E$ such that $\area (\partial E\cap \Bword_r(x))>c_0 r^{2k+1}$.
  \item
    There is  $x\in \partial E$ such that $\area (\partial E\cap \Bword_r(x))<c_0^{-1} r^{2k+1}$.
  \item
    There is  $x\in E$ such that $\vol(E\cap \Bword_r(x))<c_0^{-^{(2k+2)/(2k+1)}} r^{2k+2}.$
  \item
    There is  $x\in E^{\cc}$ such that $\vol(E^{\cc}\cap \Bword_r(x))<c_0^{-^{(2k+2)/(2k+1)}} r^{2k+2}.$
  \end{enumerate}
  We will handle these cases one by one.

  If Case (1) holds, take $H=\Bword_{r}(x)$.  So, $\cH^{2k+1}(H\cap \partial E_0) >c_0 r^{2k+1}$ and $\cH^{2k+1}(H\cap \partial E_{1})\lesssim r^{2k+1}$.  If $c_0$ is sufficiently large, then this implies~\eqref{eq:area decrease inequality}.

  Next, suppose that Case (2) holds.  Then $\area (\partial E\cap \Bword_r(x))<c_0^{-1} r^{2k+1}$, so by \eqref{eq:relative isoperimetric cellular} we have
  \begin{equation*}
    \frac{\min \Big\{\vol\Big(E\cap \Bword_{\lfloor r/S \rfloor}(x)\Big), \vol\Big(E^{\cc}\cap \Bword_{\lfloor r/S \rfloor}(x)\Big)\Big\}}{r^{2k+2}}
     \lesssim \left(\frac{\area\big(\partial E\cap \Bword_{r}(x)\big)}{r^{2k+1}}\right)^{\frac{2k+2}{2k+1}}<c_0^{-\frac{2k+2}{2k+1}}.
    \end{equation*}
  Hence,
  \begin{equation}\label{eq:use isoper box}
  \min \Big\{\vol\Big(E\cap \Bword_{\lfloor r/S \rfloor}(x)\Big), \vol\Big(E^{\cc}\cap \Bword_{\lfloor r/S \rfloor}(x)\Big)\Big\}
    \lesssim c_0^{-\frac{2k+2}{2k+1}} r^{2k+2}.
  \end{equation}
 Suppose that $\vol\big(E\cap \Bword_{\lfloor r/S \rfloor}(x)\big)\le \vol\big(E^{\cc}\cap \Bword_{\lfloor r/S \rfloor}(x)\big)$. Then by~\eqref{eq:use isoper box},
  \begin{equation}\label{eq:HvolUpperBound}
    \vol\Big(E\cap \Bword_{\lfloor r/S\rfloor}(x)\Big)\lesssim c_0^{-\frac{2k+2}{2k+1}}r^{2k+2}.
  \end{equation}
  By the pigeonhole principle, it follows that there is an integer $r_*\in [\frac{r}{2S},\frac{r}{S}]$ such that
  $$\area\Big(E \cap \partial \Bword_{r_*}(x)\Big)\lesssim c_0^{-\frac{2k+2}{2k+1}}r^{2k+1}.$$
  Let $H=\Bword_{r_*}(x)$; note that $r_*\ge 5$. Due to~\eqref{eq:HvolUpperBound}, if $c_0$ is large enough then $E_1=\overline{E\setminus  H}$ by~\eqref{eq:def Ei+1}. By the regularity of $\partial E$ we have
  \begin{equation}\label{eq:Hboundary intersection lower}
    \area(H\cap \partial E)\ge c_0^{-1} r_*^{2k+1}.
  \end{equation}
  On the other hand, $H \cap \partial E_1\subset E \cap \partial H$, so
  \begin{equation}
 \area(H\cap \partial E_1)
    \lesssim c_0^{-\frac{2k+2}{2k+1}}r^{2k+1}
    \label{eq:smallIntersectionBoundary}
    \le  c_0^{-\frac{2k+2}{2k+1}}r_*^{2k+1}.
  \end{equation}
  The implicit constants above are independent of $c_0$, so if $c_0$
  is sufficiently large, then
  $$\area(H\cap \partial E)-\area(H\cap \partial E_1)\gtrsim c_0^{-1} r_*^{2k+1},$$
  as desired.

 We argue similarly when $\vol\big(E\cap \Bword_{\lfloor r/S \rfloor}(x)\big)\ge \vol\big(E^{\cc}\cap \Bword_{\lfloor r/S \rfloor}(x)\big)$. In this case, by~\eqref{eq:use isoper box},
  $$\vol\Big(E^{\cc}\cap \Bword_{\lfloor r/S\rfloor}(x)\Big)\lesssim c_0^{-\frac{2k+2}{2k+1}}r^{2k+2}.$$
Choose $r_*\in [\frac{r}{2S},\frac{r}{S}]$ such that
  $$\area\Big(\overline{E^{\cc}} \cap \partial \Bword_{r^*}(x)\Big)\lesssim c_0^{-\frac{2k+2}{2k+1}}r^{2k+1},$$
  and let $H=\Bword_{r'}(x)$. Again, $r^*\ge 5$ and if $c_0$ is large enough then by~\eqref{eq:def Ei+1} we have $E_1=E\cup H$.   Then \eqref{eq:Hboundary intersection lower} holds as before and $H\cap \partial E_1\subset \overline{E^{\cc}} \cap \partial H$, so \eqref{eq:smallIntersectionBoundary} holds as above.  If $c_0$ is sufficiently large, then \eqref{eq:smallIntersectionBoundary} implies \eqref{eq:area decrease inequality}.

  Finally, Cases (3) and (4) are symmetric, so we consider only the case that
  \begin{equation}\label{eq:case 3 upper}
  \vol\big(E\cap \Bword_{r}(x)\big)< c_0^{-\frac{2k+2}{2k+1}}r^{2k+2}.
  \end{equation}
 Hence there exists   $\rho\asymp c_0^{-1/(2k+1)} r$ for which $\vol(\Bword_\rho(x)) > \vol (E\cap \Bword_{r}(x))$. This implies that $\Bword_\rho(x)$ intersects $\partial E$, say at $y\in \Bword_\rho(x)\cap \partial E$. If $c_0$ is sufficiently large, then $\rho<r/4$, so $\Bword_{r/4}(y)\subset \Bword_{r/2}(x)$. By the $(r-1)$--local regularity of $\partial E$,
   \begin{equation}\label{eq:r/2 volume lower}
   \area\Big(\Bword_{\frac{r}{2}}(x) \cap \partial E\Big)\ge \area\Big(\Bword_{\frac{r}{4}}(y)\cap \partial E\Big)\ge c_0^{-1} \left(\frac{r}{4}\right)^{2k+1}.\end{equation}
  On the other hand, by the coarea formula, by~\eqref{eq:case 3 upper} there is $r^*\in [\frac{r}{2},r]$ such that if $H=\Bword_{r}(x)$, then
  \begin{equation}\label{eq:H case 3}
  \area(H\cap \partial E_1)\lesssim c_0^{-\frac{2k+2}{2k+1}}r^{2k+1}.
  \end{equation}
  Comparing~\eqref{eq:r/2 volume lower} and~\eqref{eq:H case 3}, we see that \eqref{eq:area decrease inequality} holds when $c_0$ is sufficiently large.

  Applying the above procedure inductively, we obtain $\{E_i\}_{i=0}^n$ and $\{H_i=\Bword_{r_i}(x_i)\}_{i=0}^{n-1}$ such that
  \begin{equation}\label{eq:area decrease all i}
  \forall\, i\in \{0,\ldots,n-1\},\qquad   \cH^{2k+1}(\partial E_i) - \cH^{2k+1}(\partial E_{i+1})\gtrsim r_i^{2k+1}.
  \end{equation}
  Let $s_i=\diam H_i\asymp r_i\gtrsim 1$ and $B_i=B_{s_i}(x_i)$ for all $i\in \{0,\ldots,n-1\}$.  We claim that this sequence satisfies the conclusions of Lemma~\ref{lem:ballDecomp}. Firstly, each step of this iteration decreases $\cH^{2k+1}(\partial E_i)$ by a definite amount, so the construction eventually halts.  By~\eqref{eq:area decrease all i}, we have
  $$\cH^{2k+1}(\partial E)=\sum_{i=0}^{n-1} \big(\cH^{2k+1}(\partial E_i) - \cH^{2k+1}(\partial E_{i+1})\big)\gtrsim \sum_{i=0}^{n-1} r_i^{2k+1}\asymp \sum_{i=0}^{n-1} s_i^{2k+1}.$$
  By construction, there exists $C=C_k\in (0,\infty)$ such that $(E_i,E_i^\cc,\partial E_i)$ is $(C,s_i)$-regular and $E_i\symdiff E_{i+1}\subset B_i$ for all $i\in \{0,\ldots,n-1\}$.

  Finally, up to a set of measure zero, for every $i\in \{0,\ldots,n-1\}$ either $E_i\symdiff E_{i+1}=H_i\cap E_i^\cc$ or $E_i\symdiff E_{i+1}=H_i\cap E_i$.  Since $\vpfl{H_i}(E_i)=\vpfl{H_i}(E_i^\cc)$, Lemma~\ref{lem:vPerProps} implies that in both  cases we have
  \begin{align*}
    \big|\vpf(E_i)(s)-\vpf(E_{i+1})(s)\big|
    \le \vpf(E_i\symdiff E_{i+1})(s)
    \le \vpfl{H_i}(E_i)(s)+\vpf(H_i)(s)
    \le \vpfl{B_i}(E_i)(s)+\vpf(H_i)(s),
  \end{align*}
  for all $s\in \R$. Consequently,
  $$\forall\, s\in \R,\qquad |\vpf(E)(s)|\le \sum_{i=0}^{n-1} \big(\vpfl{B_i}(E_i)(s)+\vpf(H_i)(s)\big).$$
  Using  Lemma~\ref{lem:cellBallProps}, we therefore conclude that there exists $C'\in (0,\infty)$ such that
  \begin{equation*}
    \|\vpf(E)\|_{L_p(\R)}
    \le \sum_{i=0}^{n-1} \Big(\big\|\vpfl{B_i}(E_i)\big\|_{L_p(\R)}+C's_i^{2k+1}\Big)
    \le C'\cH^{2k+1}(\partial E)+\sum_{i=0}^{n-1} \big\|\vpfl{B_i}(E_i)\big\|_{L_p(\R)}.\tag*{\qedhere}
  \end{equation*}
\end{proof}

\section{Constructing corona decompositions} \label{sec:constructing corona}

In this section, we prove Theorem~\ref{thm:Ahlfors admits corona}. Recall that we are given $C,r\in (0,\infty)$ and $E\subset \R^{2k+1}$ such that $(E,E^\cc,\partial E)$ is $(C,r)$-regular. Our goal is, given $\lambda,\theta\in (0,\infty)$, to construct a $(K,\gamma,\lambda,\theta,r)$-corona decomposition for the pair $(E,\partial E)$, where $K$ is allowed to depend only on $k$ and $C$ and the Carleson packing constant $\gamma$ is allowed to depend only on $C,k,\lambda,\theta$.

To start, as we recalled in Section~\ref{sec:corona decompositions}, since $\partial E$ is $(C,r)$-regular we may apply the classical construction of Christ cubes~\cite{ChristTb,DavidWavelets} to obtain a $(K,r)$-cubical patchwork $\{\Delta_i\}_{i\le \lfloor \log_2 r\rfloor}$ of $\partial E$, where $K=K(C,k)$. The cubical patchwork $\Delta$ will be used below to build the desired coronization.

We will first give a brief outline of the argument, proceed to introduce the key concept of monotone and $\delta$-monotone sets, and then prove Theorem~\ref{thm:Ahlfors admits corona}.  As in Section~\ref{sec:corona decompositions}, for every $Q\in \Delta$ we define $N_{\rho}(Q)=\nbhd_{\rho\sigma(Q)}(Q)$ and $\rho Q=\partial E \cap N_\rho(Q).$ The proof follows these steps:
\begin{enumerate}
\item
  We show that for every $\epsilon>0$, there are sets $\cG_0,\cB_0\subset \Delta$ such that $\Delta=\cG_0\cup \cB_0$, $\cB_0$ is Carleson, and for every $Q\in \cG_0$, there is a half-space $P^+_Q$ with bounding plane $P_Q$ such that
  $$d_{N_{\frac{1}{\epsilon}}(Q)}\big(P_Q^+, E\big)\le \epsilon \sigma(Q).$$
To prove this, we adapt the results on the stability of monotone sets proved in \cite{CKN}.  In that paper, quantitative stability results were proved for monotone sets in $\H^3$.  Similar techniques would lead to a corresponding stability statement for $\H^{2k+1}$, but to keep this paper self-contained we will give a much simpler non-quantitative proof of the stability of monotone sets in $\H^{2k+1}$ that suffices for the present purposes.
\item
  We improve the previous result by showing that for all $\epsilon$, there are sets $\cG, \cB$ such that $\cB$ is Carleson and such that for all $Q\in \cG$, there is a vertical half-space $V^+_Q$ with bounding plane $V_Q$ such that
  $$d_{N_{\frac{1}{\e}}(Q)}\big(V_Q^+, E\big)\le \epsilon \sigma(Q).$$
This step uses the fact that a horizontal plane in $\H^{2k+1}$ is not homogeneous.  A horizontal plane is the union of the horizontal lines through a point, so it has a natural center.  The intersection of a horizontal plane with a ball far from its center is close to a vertical plane, so if $E$ is close to a horizontal plane $P$ on a ball, then $E$ is close to a vertical plane on most smaller balls.
\item
  For any $\eta>0$, we construct a coronization $(\cB,\cG,\cF)$ of $\partial E$ based on the sets of good and bad cubes from the previous step.  This coronization has the property that if $\cS\in \cF$ is a stopping-time region and $\mathsf{Q}(\cS)$ is the maximal cube of $\cS$, then all $Q\in \cF$ satisfy the angle estimate $\angle(V_Q,V_{\mathsf{Q}(\cS)})\le \eta$. As in Chapter 7 of \cite{DavidSemmesSingular}, we construct the stopping-time regions by starting with a maximal cube, then adding descendants of that cube as long as they satisfy certain conditions.  The main difficulty is proving that this construction does not produce too many stopping-time regions, i.e., that the maximal cubes of these regions satisfy a Carleson packing condition.
\item
  Finally, we show that this coronization is in fact a corona decomposition by showing that for all $\cS\in \cF$, there is an intrinsic Lipschitz graph satisfying Definition~\ref{def:coronaDecomp}.  We construct this graph by using a partition of unity to glue together the different planes $V_Q$ as $Q$ ranges over $\cS$.  This is similar to the argument in Chapter 8 of \cite{DavidSemmesSingular}, with some complications that arise from the Heisenberg group setting.
\end{enumerate}

\subsection{Monotone and $\delta$-monotone sets}\label{sec:monotone and delta monotone} One of the main tools in this section is the notion of nonmonotonicity used in~\cite{CKN}.  We will review the definition of nonmonotonicity and state some results relating the nonmonotonicity of $E$ to the intersections of $E$ with horizontal lines.  Note, however, that  our versions of $w_j$ and $\widehat{w}_j$ count intervals of length between $2^{j-1}$ and $2^j$, while the corresponding measures in~\cite{CKN} count intervals of length between $\delta^{j+1}$ and $\delta^j$ for some $0<\delta<1$.

\begin{defn}[monotone sets]
  A measurable subset $K\subset \R$ is said to be \emph{monotone} if the characteristic function $\one_K$ is monotone up to a set of measure zero.  Equivalently, up to a set of measure zero, $K$ and $\R\setminus K$ are both intervals.  A subset $F\subset \R^n$ is said to be monotone if for almost every line $L\subset \R^n$, the intersection $F\cap L$ is monotone.

  We say that a measurable set $E\subset \H^{2k+1}$ is monotone if for almost every \emph{horizontal} line $L$, the intersection $E\cap L$ is a monotone subset of $L$.  We say that $E$ is monotone on an open subset $U\subset \H^{2k+1}$ if for almost every horizontal line $L$, the intersections $E\cap L \cap U$ and $E^{\cc}\cap L\cap U$ are both, up to a set of measure zero, the intersections of $L\cap U$ with intervals.
\end{defn}
We will see in the next section that monotone subsets of $\R^n$ and $\H^{2k+1}$ are, up to a set of measure zero, half-spaces; this generalizes a result that Cheeger and Kleiner proved~\cite{CheegerKleinerMetricDiff} for $\H^3$.

Let $\cL$ be the space of horizontal lines in $\H^{2k+1}$ and let $\cL(B_r(x))$ be the set of horizontal lines that intersect $B_r(x)$.  Let $\cN$ be the unique (up to constants) measure on $\cL$ that is invariant under the action of the isometry group.  We normalize $\cN$ so that $\cN(\cL(B_r(x)))=r^{2k+1}$.

If $U\subset \H^{2k+1}$ and $L\in \cL$, define
\[\NM_U(E,L)= \inf\left\{\int_{L\cap U} |\one_I-\one_{E}|\ud\cH_L^1 \mid  \text{ $I$ is a monotone subset of $L$}\right\}.\]
The \emph{nonmonotonicity} of $E$ on $B_r(x)$ is defined by
\begin{equation}\label{eq:def NM Br}
  \NM_{B_r(x)}(E) = \frac{1}{r^{2k+2}} \int_{\cL(B_r(x))} \NM_{B_r(x)}(E,L) \ud\cN(L).
\end{equation}
The normalization by $r^{2k+2}$ makes $\NM_{B_r(x)}(E)$ be scale-invariant, i.e.,  we have
$$\NM_{B_{tr}(\0)}(\s_t(E))=\NM_{B_{r}(\0)}(E)$$
for all $t,r>0$.  A set $E\subset B_r(x)$ is said to be \emph{$\delta$--monotone} on $B_r(x)$ if $\NM_{B_r(x)}(E)< \delta$.

One can express the nonmonotonicity of $E$ using the kinematic formula.  This formula writes the perimeter of a subset of $\H^{2k+1}$ as an integral over the space $\cL$ of horizontal lines in $\H^{2k+1}$.  Recall that, by Lemma~\ref{lem:finite Hausdorff finite perimeter}, if $\cH^{2k+1}(\partial E)<\infty$, then $E$ has finite perimeter and $\Per_E(U)\lesssim \cH^{2k+1}(U\cap \partial E)$ for any open set $U\subset \H^{2k+1}$. If $E$ is a set with finite perimeter, then for almost every horizontal line $L$, the intersection $F=L\cap E$ is a subset of $L$ with finite perimeter. Hence there exists a unique collection of disjoint closed intervals $\cI(F)=\{I_1(F),I_2(F),\dots\}$ such that $F\symdiff \bigcup \cI(F)$ has measure zero.  The boundary $\partial \bigcup \cI(F)$ is a discrete set, and the perimeter measure $\Per(F)$ is the counting measure on $\partial \bigcup \cI(F)$; see~\cite{AmbFuscPall} or~\cite{CKN} for these basic statements. This description leads to the following kinematic formula (see~\cite{Mon05} or equation (6.1) in~\cite{CKN}).  There is $c>0$ such that for any finite-perimeter set $E\subset \H^{2k+1}$ and any open subset $U\subset \H^{2k+1}$,
\begin{equation}\label{eq:kinematic}
  \Per(E)(U)=c\int_{\cL} \Per(E\cap L)(U\cap L) \ud\cN(L).
\end{equation}

In~\cite{CKN} the perimeter measure is further decomposed as follows into measures based on the lengths of the intervals in $\cI(E\cap L)$.  For every $L\in \cL$ and $j\in \Z$, let
$$C_j(E,L)\eqdef\{I\in \cI(E\cap L)\mid 2^{j-1}\le \length(I)< 2^j\},$$
and
$$C_\infty(E,L)=\{I\in \cI(E\cap L)\mid \length(I)=\infty\}.$$
For every $j\in \Z\cup \{\infty\}$, let $\cE_j(E,L)$ denote the set of endpoints of the intervals in $C_j(E,L)$.  Let $c$ be as in~\eqref{eq:kinematic} and let $w_j(E,L)$ be $c$ times the counting measure on $\cE_j(E,L)$.  If we let
$$w_j(E)(A)\eqdef \int_{\cL} w_j(E,L)(A)\ud\cN(L)\qquad\mathrm{and}\qquad \widehat{w}_j(E)(A)\eqdef\frac{w_j(E)(A)+w_j(E^{\cc})(A)}{2},$$
then
\begin{equation}\label{eq:kinematicW}
  \Per(E)=\widehat{w}_\infty(E)+\sum_{j\in \Z} \widehat{w}_j(E).
\end{equation}

Furthermore, we can use the measures $\{\widehat{w}_j\}_{j\in \Z\cup\{\infty\}}$ to bound the nonmonotonicity of $E$.  If $L\in \cL$ and the intersection $E\cap L$ has finite perimeter, then after changing $E$ on a measure-zero subset, we can partition $L$ into intervals $J_1,\dots, J_n$, arranged in ascending order and each with positive length, so that the $J_i$'s are alternately contained in $E$ and disjoint from $E$.  Suppose that $J_1,J_3,\dots\subset E$.  Then $J_2,\dots, J_{n-1}$ all have finite length, and there is an interval $J=J_1$ or $J=L$ (depending on whether $n$ is even or odd) such that
\begin{align*}
  \NM_{\H^{2k+1}}(E,L)\le \int_{L\cap U} |\one_J-\one_{E}| \ud\cH_L^1
                      \le \sum_{i=2}^{n-1}\length(J_i)
                      \asymp \sum_{j\in \Z} 2^j \cdot \widehat{w}_j(E,L)(\H^{2k+1}).
\end{align*}
(If $J_2,J_4,\dots \subset E$, then the same inequality holds for either $J=J_n$ or $J=\emptyset$.)

A similar inequality holds for local nonmonotonicity.  Let $r>0$, $x\in \H^{2k+1}$, $a=\inf(B_r(x)\cap L)$, and $b=\sup(B_r(x)\cap L)$.  If $m\le M$ are such that $a\in J_m$, $b\in J_M$, then
$$\NM_{B_r(x)}(E,L)\le \sum_{i=m+1}^{M-1}\length(J_i).$$
For each of these intervals, we have $\length J_i\le b-a\le 2r$, so
\begin{equation}\label{eq:pre kine NM}
  \NM_{B_r(x)}(E,L)\lesssim \sum_{i=-\infty}^{\lceil 1+\log_2 r\rceil} 2^i \cdot \widehat{w}_i(E,L)(U).
\end{equation}
The following proposition is a simple consequence of the above discussion; see~\cite[Proposition~4.5]{CKN}.
\begin{prop}\label{prop:kinematicNM} Fix $C,r>0$. Suppose that $E\subset \H^{2k+1}$ and that $\partial E$ is $(C,r)$-regular.  For every $x\in \H^{2k+1}$, if we denote $B=B_r(x)$, then \begin{equation}\label{eq:NM bound broken on scales}
\NM_{B}(E)\lesssim C \sum_{n=-\infty}^{\lceil 1+\log_2 r\rceil} \frac{2^n}{r}\cdot \frac{\widehat{w}_n(E)(B)}{\Per(E)(B)}.
\end{equation}
\end{prop}
\begin{proof}
  Integrating \eqref{eq:pre kine NM} and using the regularity of $\partial E$, we find that
  \begin{equation*}
  \NM_{B}(E)
  \lesssim \frac{\sum_{n=-\infty}^{\lceil 1+\log_2 r\rceil}  2^{n} \cdot \widehat{w}_n(E)(B)}{\cH^{2k+2}(B)}
  \lesssim C\sum_{n=-\infty}^{\lceil 1+\log_2 r\rceil}  \frac{2^{n}}{r} \cdot \frac{\widehat{w}_n(E)(B)}{\Per(E)(B)}.\tag*{\qedhere}\end{equation*}
\end{proof}
Due to~\eqref{eq:kinematicW}, we have $\frac{\widehat{w}_n(E)(B)}{\Per(E)(B)}\le 1$ for every $n\in \Z$, so the sum in~\eqref{eq:NM bound broken on scales} is typically dominated by the terms where $2^n$ is close to $r$.

\subsection{$\delta$-monotone sets are close to planes}\label{sec:delta monotone is close to plane}

 By~\cite{CheegerKleinerMetricDiff}, every monotone subset of $\H^3$ is, up to a set of measure zero, a half-space bounded by a horizontal or vertical plane.  This statement was  quantified in~\cite{CKN}, showing that $\delta$-monotone sets are quantitatively close to half-spaces.  Specifically,
\begin{thm}[{Theorem~4.13 of~\cite{CKN}}]\label{thm:CKNstability}
  There exists $a>0$ such that for all $\epsilon>0$, if $E\subset B_1\subset \H^3$ is measurable and $\NM_{B_1}(E)<\epsilon^a$, then there is a half-space $P^+\subset\H^3$ such that
  $$\frac{\cH^4\big((E\symdiff P^+)\cap B_{\epsilon^3}\big)}{\cH^4(B_{\epsilon^3})}\lesssim \epsilon.$$
\end{thm}
While we were  motivated by this result, we will {\em not} use it in our proof.  Instead, we will use a higher-dimensional qualitative version (Proposition~\ref{prop:highDimStability} below) that states that as $\delta$ goes to 0, $\delta$--monotone sets approach half-spaces.  We will not need a quantitative estimate on the rate of convergence, so we can avoid the careful estimates necessary to prove Theorem~\ref{thm:CKNstability}.

We will first classify monotone subsets of $\R^n$ and $\H^{2k+1}$. Then, using a compactness result we will conclude that $\delta$-monotone sets are close to monotone sets.  Starting with the Euclidean setting, the following lemma asserts that monotone subsets of $\R^n$ are half-spaces.
\begin{lemma}\label{lem:Rnmonotone}
  If $F\subset \R^n$ is measurable and monotone, then either $\cH^n(F)=0$, $\cH^n(\R^n\setminus F)=0$, or there is a half-space $P^+\subset \R^n$ such that $\cH^n(F\symdiff P^+)=0$.
\end{lemma}
\begin{proof} Breaking (locally) from notational conventions in the rest of this paper, in the ensuing proof of Lemma~\ref{lem:Rnmonotone} all the distances and balls are with respect to the Euclidean metric on $\R^n$.

  Suppose that $x_1,x_2\in \R^n$ are density points of $F$ and that $y$ is on the line segment between them.  We claim that $y$ is a density point of $F$.  Let $\epsilon>0$ and let $0<r<\min\{d(x_1,y),d(x_2,y)\}$ satisfy $\cH^n(B_{r}(x_i)\cap F)> (1-\epsilon) \cH^n(B_{r}(x_i))$.  Let $\ell_0$ be a line passing through $B_{r/3}(x_1)$ and $B_{r/3}(x_2)$ such that for almost every line $\ell$ parallel to $\ell_0$, the set $F\cap \ell$ is monotone.

  Let $A$ be the set of lines parallel to $\ell_0$ that pass through $B_{r/3}(y)$.  Let $\mu$ be the Lebesgue measure on $A$, normalized so that $\mu(A)=1$.  Since $d(\ell_0,y)\le r/3$,  every line $\ell\in A$ intersects $B_{2r/3}(x_i)$ for $i=1,2$. It follows that every $\ell\in A$ intersects $B_{r}(x_i)$ in an interval of length at least $2r/3$.  Let
  $$X\eqdef \Big\{\ell\in A\mid \text{$F\cap \ell$ is monotone, } \cH^1\big(\ell \cap F\cap B_{r}(x_1)\big)> 0\text{, and }\cH^1\big(\ell \cap F\cap B_{r}(x_2)\big)>0\Big\}.$$
  If $\ell \in X$, then the interval $\ell \cap B_{r/3}(y)$ lies between $\ell \cap B_{r}(x_1)$ and $\ell \cap B_{r}(x_1)$, so the monotonicity of $F\cap \ell$ implies that $\cH^1(\ell \cap B_r(y)\cap F^\cc)=0$.  On the other hand, if $\ell \not \in X$, then either $\ell\cap B_{r}(x_1)$ or $\ell\cap B_{r}(x_2)$ is an interval disjoint from $F$, and therefore
  $$\cH^{1}\Big(\ell \cap F^\cc \cap \big(B_{r}(x_1)\cup B_{r}(x_2)\big)\Big)\ge \frac{2r}{3}.$$
  By integrating over $A$, we find that
  $$\cH^{2k+2}(B_{r}(x_1)\cap F^\cc)+\cH^{2k+2}(B_{r}(x_2)\cap F^\cc)\gtrsim \frac{2r}{3} \mu(A\setminus X) r^{n-1},$$
  so $\mu(A\setminus X)\lesssim \epsilon$.  It follows that
  $$\frac{\cH^n(B_{r/3}(y)\cap F^{\cc})}{\cH^n(B_{r/3}(y))}\lesssim \mu(A\setminus X)\lesssim \epsilon.$$

   We have shown that the set of density points of $F$ is convex.  By symmetry, the set of density points of $F^{\cc}$ is convex.  Since these sets are disjoint and their union is all of $\R^n$ except for a set of measure zero, the sets of density points of $F$ and of $F^{\cc}$ are either empty, $\R^n$, or half-spaces.
\end{proof}

Theorem~5.1 of~\cite{CheegerKleinerMetricDiff} asserts that monotone subsets of $\H^3$ are half-spaces.  By combining this with Lemma~\ref{lem:Rnmonotone}, we will show that the monotone subsets of $\H^{2k+1}$ are half-spaces.
\begin{prop}\label{prop:highDimMonotone}
  If $F\subset \H^{2k+1}$ is monotone, then, up to a set of measure zero, we have $F=\emptyset$ or $F=\H^{2k+1}$, or $F$ is a half-space bounded by a plane.
\end{prop}
\begin{proof}
  Recall that multiplication on $\H^{2k+1}$ is based on a symplectic form $\omega$ on $\R^{2k}$.  We identify $\R^{2k}$ with $\C^k$ so that $\omega(z,w)=\Imag(\sum_{i=1}^k \overline{z_i}w_i)$.  Let $\pi\from \H^{2k+1}\to\C^{2k}$ be the abelianization map.  We claim that up to measure zero, $F$ is a half-space of $\H^{2k+1}$.  (Recall that half-spaces of $\H^{2k+1}$ coincide with half-spaces of $\R^{2k+1}$.)

  If $v\in \C^{k}\setminus\{\0\}$, let $\H_v=\pi^{-1}(\C v)$.  This is a subgroup of $\H^{2k+1}$ that is isometrically isomorphic to $\H^3$.  If $h\in \H^{2k+1}$ and $v\in \C^{k}$ is a unit vector, let $\H(h,v)=h\H_v$. By the left-invariance of the Carnot--Carath\'eodory metric, each such coset is also isometric to $\H^3$.  Let $\mathscr{S}$ be the set of such cosets.  Any horizontal line $L\in \cL$ is contained in a unique element of $\mathscr{S}$, which can be constructed by letting $v$ be parallel to the line $\pi(L)$.  The set $\cL$ of horizontal lines in $\H^{2k+1}$ thus fibers over $\mathscr{S}$, and the fiber is the set $\cL_1$ of horizontal lines in $\H^3$.

  By Fubini's theorem, if $F\subset \H^{2k+1}$ is a monotone set, then for almost every $\H(h,v)\in \mathscr{S}$, the intersection $F\cap \H(h,v)$ is monotone.  If $F\cap \H(h,v)$ is monotone, then by~\cite[Theorem~5.1]{CheegerKleinerMetricDiff}, up to a measure zero set, $F\cap\H(h,v)\in \{\emptyset, \H(h,v)\}$, or $F\cap\H(h,v)$ is a half-space in $\H(h,v)$.  In each case, $\H(h,v)$ is a $3$--plane in $\H^{2k+1}$ and $F\cap\H(h,v)$ is monotone not just along horizontal lines but along all lines in $\H(h,v)$.

  For almost every line $\ell\subset \H^{2k+1}$ (not necessarily horizontal), the projection $\pi(\ell)$ is a line in $\C^{k}$, and if $v\in \C^k$ is parallel to $\pi(\ell)$, then $\ell\subset \H(h,v)$ for some $h\in \H^{2k+1}$.  For almost every such $\ell$, the intersection $F\cap \H(h,v)$ is monotone, so $F\cap \ell$ is monotone.  By Lemma~\ref{lem:Rnmonotone}, $F$ is, up to a measure-zero set, either empty, all of $\H^{2k+1}$, or a half-space, as desired.
\end{proof}

\begin{prop}\label{prop:highDimStability}
  For any $\epsilon,\rho,c>0$, there exists $\delta>0$ such that if $F\subset \H^{2k+1}$ is measurable, $\partial F$ is $(c,\rho)$-regular, and $\NM_{B_\rho}(F)<\delta^{2k+3}$, then there is a half-space $P^+\subset\H^{2k+1}$ such that
  $$\frac{\cH^{2k+2}\big((F\symdiff P^+)\cap B_{\delta \rho}\big)}{\cH^{2k+2}(B_{\delta \rho})}< \epsilon.$$
\end{prop}
\begin{proof}
  After rescaling, we may suppose that $\rho=1$.  If the assertion of Proposition~\ref{prop:highDimStability} does not hold true, then there is a sequence $\{F_i\}_{i=1}^\infty$ of subsets of $\H^{2k+1}$ such that $\partial F_i$ is $(c,\rho)$-regular and such that for every $i\in \N$ we have $F_i\subset B_1$, $\NM_{B_1}(F_i)<i^{-2k-3}$, and such that for any half-space $P^+\subset \H^{2k+1}$,
  \begin{equation}\label{eq:FawayFromPlanes}
    \frac{\cH^{2k+2}\big((F\symdiff P^+)\cap B_{1/i}\big)}{\cH^{2k+2}(B_{1/i})}\ge \epsilon.
  \end{equation}
  For every $i\in \N$ denote $S_i=s_i(F_i)$.  Then $S_i$ is $(c,i)$-regular, so if $r>0$, $x\in \H^{2k+1}$, and $i>r$, then $\cH^{2k+1}(B_r(x)\cap \partial S_i)\le cr^{2k+1}$.  By Lemma~\ref{lem:finite Hausdorff finite perimeter}, this implies that the sets $S_i$ have uniformly locally finite perimeter, so by Proposition~\ref{prop:compact lower semicontinuous}, we can pass to a subsequence such that $\one_{S_i}$ converges in $L_{1}^\text{loc}$ to the characteristic function of some $S\subset \H^{2k+1}$.  We claim that $S$ is monotone.

  We have $\NM_{B_i}(S_i)=\NM_{B_1}(F_i)<i^{-2k-3}$, so by~\eqref{eq:def NM Br},
  $$\int_{\cL(B_i)}\NM_{B_i}(S_i,L)\ud\cN(L)=i^{2k+2}\NM_{B_i}(S_i)<\frac{1}{i}.$$
  Again passing to a subsequence, we suppose that for almost every line $L\in \cL$, the intersection $L\cap S_i$ converges in $L_{1}^\text{loc}$ to $L\cap S$ and $\lim_{i\to \infty} \NM_{B_i}(S_i,L)=0$.  That is, there is a sequence $I_i\subset L$ of monotone sets such that
  \begin{equation}\label{eq:L1 diff Ii Si}
    \lim_{i\to \infty} \|\one_{B_i\cap I_i}-\one_{B_i\cap S_i}\|_{L_1(L)}=0.
  \end{equation}
  For any $r>0$, the sets $L\cap B_r\cap S_i$ converge (in $L_1$) to $L\cap B_r\cap S$.  By \eqref{eq:L1 diff Ii Si}, the sets $L\cap B_r\cap I_i$ also converge to $L\cap B_r\cap S$.  A limit of monotone sets is monotone, so $L\cap S$ is monotone.  Since this is true for almost every $L$, we conclude that $S$ is monotone and thus there is a plane $P$ such that $\cH^{2k+2}(S\symdiff P^+)=0$.  But by \eqref{eq:FawayFromPlanes}, we have
  $\cH^{2k+2}((S_i\symdiff P^+)\cap B_{1})\ge \epsilon \cH^{2k+2}(B_{1})$
  for all $i$. This contradiction completes the proof of Proposition~\ref{prop:highDimStability}.
\end{proof}

Lemma~\ref{lem:L1 distance bounds local distance} below shows that when $(F,F^\cc, \partial F)$ is regular and $F$ is $\delta$-monotone, Proposition~\ref{prop:highDimStability} also implies that $F$ and $P^+$ are close in the local distance (Definition~\ref{def:local distance}), which will be necessary to construct a corona decomposition.

\begin{lemma}\label{lem:L1 distance bounds local distance}
 Fix $C>0$.  For any $\theta>0$, there exists $\epsilon>0$ such that if $r>0$, $(F, F^\cc,\partial F)$ is $(C,r)$-regular, $x\in \H^{2k+1}$, and $P^+\subset \H^{2k+1}$ is a half-space satisfying
  $$\frac{\cH^{2k+2}\big((F\symdiff P^+)\cap B_{r}(x)\big)}{\cH^{2k+2}\big(B_{r}(x)\big)}\le \epsilon,$$
  then $d_{B_{r/2}(x)}(F,P^+)< \frac{\theta r}{2}$.
\end{lemma}
\begin{proof}
  Without loss of generality, we suppose that $\theta<\frac{1}{4}$.  Take any $\e\in \R$ that satisfies
  \begin{equation}\label{eq:choose epsilon}
  0<\epsilon<\min\left\{\frac{\theta^{2k+1}}{C\cH^{2k+2}\big(B_1(x)\big)}, \frac{\theta^{2k+1}}{2}\right\}.
  \end{equation}
  Suppose that $y\in (P^+\setminus F) \cap B_{r/2}(x)$.  If $d(y,P)>\theta r$, then $B_{\theta r}(y)\subset P^+$ and thus
  \begin{align*}
    \cH^{2k+2}\big((F\symdiff P^+)\cap B_{r}(x)\big)
    \ge \cH^{2k+2}(B_{\theta r}(y)\setminus F)
    =\frac{(\theta r)^{2k+2}}{C}
    \stackrel{\eqref{eq:choose epsilon}}{>} \epsilon \cH^{2k+2}\big(B_r(x)\big).
  \end{align*}
  This is a contradiction, so $d(y,P)\le \theta r$.  Replacing $F$ by $\H^{2k+1}\setminus F$ and $P^+$ by $P^-$, we find that $d(y,P)\le \theta r$ for all $y\in (F\setminus P^+) \cap B_{r/2}(x)$.

  Now suppose that $y\in (F\symdiff P^+)\cap B_{r/2}(x)$ and $d(y,\partial F)>\theta r$.  Then denote
 \begin{equation*}
 A=\left\{\begin{array}{ll}P^+\cap B_{\theta r}(y)& \mathrm{if}\ y\in P^+,\\
 \overline{P^-}\cap B_{\theta r}(y)&\mathrm{if}\ y\in \overline{P^-}.\end{array}\right.
 \end{equation*}
 By applying a translation, we may suppose that $y=\0$.  Let $R$ be the plane through $\0$ that is parallel (in the Euclidean sense) to $P$.  Orient $R$ so that $R^+\cap B_{\theta r}\subset A$.  Recall that for all $x\in \H^{2k+1}$, we have $x^{-1}=-x$, so $(R^+)^{-1}=R^-$.  By the left-invariance of the metric, we have $d(\0,g)=d(g^{-1},\0)$, so $(R^+\cap B_{\theta r})^{-1}=R^-\cap B_{\theta r}$.  The inverse map is measure-preserving, so $\cH^{2k+2}(R^+\cap B_{\theta r})=\cH^{2k+2}(R^-\cap B_{\theta r})=\frac{1}{2}\cH^{2k+2}(B_{\theta r})$ and thus
$$\cH^{2k+2}\big((F\symdiff P^+)\cap B_{r}(x)\big)\ge \cH^{2k+2}(A)\ge \cH^{2k+2}(R^+\cap B_{\theta r})\ge \frac{1}{2}\cH^{2k+2}(B_{\theta r})\stackrel{\eqref{eq:choose epsilon}}{>}\epsilon \cH^{2k+2}(B_r(x)).$$
Again, this is a contradiction, so $d(y,\partial F)\le \theta r$.
\end{proof}

\subsection{Most cubes are close to a plane}
Now we begin the proof of Theorem~\ref{thm:Ahlfors admits corona}.  The first step is to show that if $(E,E^\cc, \partial E)$ is regular, then most parts of $\partial E$ are close to a plane.  Specifically,
\begin{prop}\label{prop:goodCubes}
  For every $C, \epsilon>0$ there exists $K>0$ with the following property.  Suppose that $r>0$, and let $E\subset \H^{2k+1}$ be a set such that $(E,E^{\cc},\partial E)$ is $(C,r)$-regular. Let $\Delta=\{\Delta_i\}^{\lfloor \log_2 r\rfloor}_{i=-\infty}$ be a $(C,r)$-cubical patchwork for $\partial E$.  Denote
  $$\cG_0(\epsilon)\eqdef \Big\{Q\in \Delta \mid \exists \text{ a plane $P_Q$ such that } d_{N_{\frac{1}{\epsilon}}(Q)}\big(P_Q^+, E\big) \le \epsilon \sigma(Q)\Big\}.$$
  ($P_Q^+$ is one of the half-spaces bounded by $P_Q$.)  If $\cB_0(\epsilon)\eqdef \Delta\setminus \cG_0(\epsilon)$, then $\cB_0(\epsilon)$ is $K$-Carleson.
\end{prop}

Note that by rescaling, we may suppose that $r=1$.  We will prove the Proposition~\ref{prop:goodCubes} by showing that most cubes in $\Delta$ have small nonmonotonicity, then using Proposition~\ref{prop:highDimStability} to conclude that most cubes are close to planes.  For simplicity of notation, all the implicit constants in this section and the following sections will depend on $C$.

First, we decompose $\NM(E)$ into pieces based on cubes.  Namely, for every integer $j\le \lfloor \log_2 r\rfloor $ and  $Q\in \Delta_j$, denote
$w(Q)=\widehat{w}_{j}(E)(Q).$
This counts segments with an endpoint in $Q$ that have lengths between $\sigma(Q)/2$ and $\sigma(Q)$.
\begin{lemma}\label{lem:cB epsilon is Carleson}
  For any $Q\in \Delta$ we have
  $$\sum_{\substack{Q'\in \Delta\\ Q'\subset Q}} w(Q')\lesssim \sigma(Q)^{2k+1}.$$
  It follows that for any $\epsilon>0$, the set
  $\cB(\epsilon)\eqdef \{Q\in \Delta\mid w(Q)>\epsilon\sigma(Q)^{2k+1}\}$
 is $O(1/\e)$-Carleson.
\end{lemma}
\begin{proof}
  Suppose $Q\in \Delta_i$ for some integer $i\le \lfloor \log_2 r\rfloor$.  Then, using the regularity of $\partial E$,
  \begin{align*}
    \sum_{\substack{Q'\in \Delta\\ Q'\subset Q}} w(Q')
    =\sum_{j=-\infty}^i \sum_{\substack{Q'\in \Delta_j\\ Q'\subset Q}} \widehat{w}_j(E)(Q')
    = \sum_{j=-\infty}^i \widehat{w}_j(E)(Q)
    \stackrel{\eqref{eq:kinematicW}}{\le} \Per(E)(Q) \lesssim \sigma(Q)^{2k+1}.
  \end{align*}
  This implies that $\cB(\epsilon)$ is $O(1/\e)$-Carleson because
  \begin{align*}
    \sum_{\substack{Q'\in \cB(\epsilon)\\ Q'\subset Q}}\sigma(Q')^{2k+1}
    \le \sum_{\substack{Q'\in \cB(\epsilon)\\ Q'\subset Q}} \frac{w(Q')}{\e}
    \le \frac{1}{\e}\sum_{\substack{Q'\in \Delta\\ Q'\subset Q}} w(Q')
    \lesssim \frac{\sigma(Q)^{2k+1}}{\e}.\tag*{\qedhere}
  \end{align*}
\end{proof}

If $Q, Q'\in \Delta$ and $n>1$, we say that $Q$ and $Q'$ are
\emph{$n$-close} if
\begin{equation}\label{eq:NCloseSize}
  \frac{\sigma(Q)}{\sigma(Q')}\in [2^{-n},2^n]\qquad\mathrm{and}\qquad d(Q,Q')\le 2^n\max \{\sigma(Q),\sigma(Q')\}.
\end{equation}
If $\cS\subset \Delta$, denote
$$\cS^{(n)}\eqdef \{Q\in \Delta\mid \text{$Q$ is $n$-close to an element of $\cS$}\}.$$

Any cube that is $n$-close to $Q$ intersects $N_{2^{2n}}(Q)$ and satisfies~\eqref{eq:NCloseSize}.  By Definition~\ref{def:patchwork}, if $Q\in \Delta_i$ and $j<i$, then
\begin{equation}\label{eq: number of N close less}
  \big|\{Q'\in \Delta_j\mid Q'\cap N_{2^{2n}}(Q)\ne \emptyset\}\big|\lesssim \frac{\cH^{2k+1}\big(N_{2^{2n}}(Q)\big)}{2^{j(2k+1)}}\asymp 2^{(2k+1)(i-j+2n)}\lesssim 2^{3n(2k+1)}.
\end{equation}
If $j\ge i$, then
\begin{equation}\label{eq: number of N close more}
  \big|\{Q'\in \Delta_j\mid Q'\cap N_{2^{2n}}(Q)\ne \emptyset\}\big|\lesssim 2^{2n(2k+1)}.
\end{equation}
It follows that
\begin{equation}\label{eq: number of N close}
 \big|\{Q\}^{(n)}\big| \lesssim \sum_{j=i-n}^{i+n} 2^{3n(2k+1)} \lesssim n2^{3n(2k+1)};
\end{equation}
that is, the number of cubes of $\Delta_i$ that are $n$-close to $Q$ is bounded by a function of $n$. Furthermore,
\begin{lemma}\label{lem:cube expansion}
  For any $K, n>0$, there exists $L>0$ such that if $\cS\subset \Delta$ is $K$-Carleson, then $\cS^{(n)}$ is $L$-Carleson.
\end{lemma}
\begin{proof}
  If $D_1, D_2\in \Delta$ are $n$-close and if $A_1,A_2\in \Delta$ are ancestors of $D_1,D_2$, respectively,  such that $\sigma(A_1)/\sigma(A_2)\in [2^{-n},2^n]$, then $A_1$ and $A_2$ are also $n$-close.  It follows that if $M\in \Delta$ and $Q\subset M$ is such that $Q\in \cS^{(n)}$, then there is a cube $R\in \cS$ such that some ancestor $A$ of $R$ is $n$-close to $M$.

   Consequently, we can write
  \begin{equation}\label{eq:breas S(n)}
  \sum_{\substack{Q\in \cS^{(n)}\\Q\subset M}}\sigma(Q)^{2k+1}\le \sum_{A\in \{M\}^{(n)}} \sum_{\substack {R\in \cS\\R\subset A}} \sum_{Q\in \{R\}^{(n)}} \sigma(Q)^{2k+1}.
  \end{equation}
  But if $R\in \Delta$, then
  \begin{equation}\label{eq:on R(n)}
  \sum_{Q\in \{R\}^{(n)}} \sigma(Q)^{2k+1}\stackrel{\eqref{eq: number of N close}}{\lesssim}n2^{4n(2k+1)} \sigma(R)^{2k+1}.
  \end{equation}
  Since $\cS$ is $K$-Carleson,
  \begin{multline*}
    \sum_{\substack{ Q\in \cS^{(n)}\\Q\subset M}}\sigma(Q)^{2k+1}
    \stackrel{\eqref{eq:breas S(n)}\wedge \eqref{eq:on R(n)}}{\lesssim} \sum_{A\subset \{M\}^{(n)}} \sum_{\substack {R\in \cS\\R\subset A}} n2^{4n(2k+1)} \sigma(R)^{2k+1}
   \\ \le n2^{4n(2k+1)} \sum_{A\subset \{M\}^{(n)}} K\sigma(A)^{2k+1}
    \stackrel{\eqref{eq: number of N close}}{\lesssim} n^22^{(2k+1)8N}K \sigma(M)^{2k+1}.\tag*{\qedhere}
  \end{multline*}
\end{proof}
By Lemma~\ref{lem:cB epsilon is Carleson} and Lemma~\ref{lem:cube expansion}, $\cB(\epsilon)^{(n)}$ satisfies a Carleson packing condition.  Let
$$\cG(\epsilon,n)\eqdef \big\{Q\in \Delta \setminus \cB(\epsilon)^{(n)}\mid \sigma(Q)\le 2^{-n}\big\}.$$
Then $\Delta\setminus \cG(\epsilon,n)=\bigcup_{i=-n}^0\Delta_i\cup \cB(\epsilon)^{(n)}$ is $L$-Carleson (with $L$ allowed to depend on $n$), and every cube in $\cG(\epsilon,n)$ has small nonmonotonicity. The following lemma is inspired by~\cite[Proposition~4.5]{CKN}.

\begin{lemma}\label{lem:goodCubesNM}
  For all $0<\epsilon<1$, there are $\delta>0$ and $n\in \N$ such that if $Q\in \cG(\delta,n)$ and $x\in Q$,
  $$\NM_{B_{\frac{\sigma(Q)}{\e}}(x)}(E)<\epsilon.$$
\end{lemma}
\begin{proof}
  We take $n\ge \lceil \log_2(1/\e)\rceil+2$, so that $2^{-n}\le \epsilon/4$.  In this case, we shall bound $\NM_{B_{\sigma(Q)/\e}(x)}(E)$ in terms of $n$ and $\delta$. Suppose that $Q\in \cG(\delta,n)$ and $x\in Q$, and denote $B=B_{\rho}(x)$, where $\rho = \sigma(Q)/\e\le 1/4$.  Let $m=\lceil \log_2 \rho\rceil + 1\le -1$.  Then $2^m\le 2^n\sigma(Q)$. Proposition~\ref{prop:kinematicNM} states that
  \begin{equation}\label{eq:NMBE}
  \NM_B(E)\lesssim \sum_{j=-\infty}^{m}  \frac{2^{j}}{\rho} \cdot q_j,
  \end{equation}
  where $q_j\eqdef \frac{\widehat{w}_n(E)(B)}{\Per(E)(B)}$.  Note that for any $j\le m$, \eqref{eq:kinematicW} implies that $q_j\le 1$.

  Suppose that $j\in [m-n,m]$.  Then
  $$\widehat{w}_j(E)(B)\le \mathop{\sum_{Q'\in \Delta_j}}_{Q'\cap B\ne \emptyset}\widehat{w}_j(E)(Q').$$
  If $Q'\in \Delta_j$ and $Q'\cap B\ne \emptyset$, then $\diam Q'\le C \sigma(Q')\le C2^m\le 4C\rho$ (where $C$ is the patchwork constant for $\Delta$). Thus $Q'\subset B_{(4C+1)\rho}(x)$.  Furthermore, $Q'$ intersects $N_{1/\e}(Q)$, so $Q'$ is $n$-close to $Q$.  Since $Q\in \cG(\delta,N)$, it follows that $\widehat{w}_j(E)(Q')\le \delta \sigma(Q')^{2k+1}\asymp \delta\Per(E)(Q')$.  Consequently,
  $$\widehat{w}_j(E)(B)\lesssim \mathop{\sum_{Q'\in \Delta_j}}_{Q'\subset B_{(4C+1)\rho}(x)} \delta\Per(E)(Q')\le \delta \Per(E)\big(B_{(4C+1)\rho}(x)\big)\asymp \delta\Per(E)(B).$$
  We have thus shown that $q_j\lesssim \delta$ for all $j\in [m-n,m]$. By~\eqref{eq:NMBE} we therefore have,
  $$\NM_{B}(E) \lesssim \sum_{j=-\infty}^{m-n-1}  \frac{2^{j}}{\rho} \cdot q_j+\sum_{j=m-n}^{m}  \frac{2^{j}}{\rho} \cdot q_j
  \lesssim \frac{2^{m-n}}{2^m} +  \frac{\delta 2^m}{2^m}=\frac{1}{2^n}+\delta.$$
  Hence,  if $n$ is sufficiently large and $\delta$ is sufficiently small, then $\NM_{B}(E)\le \epsilon$.
\end{proof}

Proposition~\ref{prop:goodCubes} now follows from Lemma~\ref{lem:goodCubesNM} and the results on monotone sets in Section~\ref{sec:delta monotone is close to plane}.
\begin{proof}[{Proof of Proposition~\ref{prop:goodCubes}}]
   Set $\theta=\frac{\epsilon}{C+1/\e}$.  By Proposition~\ref{prop:highDimStability} and Lemma~\ref{lem:L1 distance bounds local distance}, there is  $\alpha>0$ such that if $x\in \partial E$, $0<r<1$, and $\NM_{B_r(x)}(E)<\alpha^{2k+3}$, then there is a half-space $P^+\subset\H^{2k+1}$ with
  $$d_{B_{\frac{\alpha r}{2}}(x)}(F,P^+)< \frac{\theta\alpha r}{2}.$$
Denote $\gamma=\frac{2(C+1/\e)}{\alpha}$.  By
  Lemma~\ref{lem:goodCubesNM}, there are $\delta>0$ and $n>0$ such
  that if $Q\in \cG(\delta,n)$ and $x\in Q$, then
  $\NM_{B_{\gamma \sigma(Q)}(x)}(E)<\alpha^{2k+3}.$ Thus there is a plane $P$ such that
  \begin{align*}
    d_{B_{\frac{\alpha \gamma \sigma(Q)}{2}}(x)}(E,P^+)
    =d_{B_{(C+\frac{1}{\e})\sigma(Q)}(x)}(E,P^+)
    \le \Big(C+\frac{1}{\e}\Big) \theta  \sigma(Q) =\epsilon\sigma(Q).
  \end{align*}
  Since $N_{1/\e}(Q)\subset B_{(C+1/\e)\sigma(Q)}(x)$, this implies that $Q\in \cG_0(\epsilon)$.

  Let $\cG=\cG(\delta,n)$, $\cB=\Delta\setminus \cG$.  By Lemma~\ref{lem:cB epsilon is Carleson} and Lemma~\ref{lem:cube expansion}, $\cB$ is $K$-Carleson for some $K$ depending on $C$.  We have seen that $\cG\subset \cG_0(\epsilon)$, so $\cB_0(\epsilon)\subset \cB$ and thus $\cB_0(\epsilon)$ is $K$-Carleson.
\end{proof}

\subsection{Most cubes are close to a vertical plane}
\subsubsection{Planes and angles}\label{sec:planes and angles}
Before the next step, we recall some facts about planes and angles in the Heisenberg group.  Recall that $\mathsf{H}\subset\H^{2k+1}$, the subspace of horizontal vectors, is the subspace spanned by $\{X_i\}_{i=1}^k$ and $\{Y_i\}_{i=1}^k$.  Recall also that a plane in $\H^{2k+1}$ is either a vertical plane or the union $P_y=y\mathsf{H}$ of all of the horizontal lines passing through a point $y\in \H^{2k+1}$.  The planes in $\H^{2k+1}$ are exactly the planes in $\R^{2k+1}$ under the identification with $\H^{2k+1}$.

If $P,Q\subset \H^{2k+1}$ are planes, we define $\angle(P,Q)$ to be the angle between $P$ and $Q$ considered as planes in $\R^{2k+1}$ (i.e., the Euclidean angle between their normals).  If $P$ is a plane and $L$ is a line (a coset of a one-parameter subgroup, or equivalently, a line in $\R^{2k+1}$) that intersects $P$ at a point $y$, we define $\angle(P,L)$ to be the minimum Euclidean angle between $L$ and a line in $P$ that passes through $y$; this quantity varies between $0$ and $\frac{\pi}{2}$.

Importantly, the angle between planes is neither translation-invariant nor scale-invariant.  For example, if $y\in \mathsf{H}\setminus \{\0\}$, then $P_y=y\mathsf{H}$ is a plane through $\0$.  If $L_Z=\langle Z\rangle$ is a vertical line, then $\angle(s_t(P_y),s_t(L_Z))=\angle(s_t(P_y),L_Z)$ converges to zero as $t\to \infty$.

Since horizontal planes  play an important role in this section, we give two formulas expressing a horizontal plane in coordinate form.  Let $\pi\from \H^{2k+1}\to H$ be the abelianization map.  If $u,p\in \H^{2k+1}$, then $p\in P_u$ if and only if $p=u (\pi(p)-\pi(u))$.  That is, we can write $P_u$ as a graph over $\mathsf{H}$:
\begin{equation}\label{eq:abelianization graph}
  P_u=\left\{p\in \H^{2k+1}\mid z(p)=z(u)+\frac{\omega(u,p)}{2}\right\}=\left\{hZ^{z(u)+\frac{\omega(u,h)}{2}}\mid h\in \mathsf{H}\right\}.
\end{equation}
If $r>0$ and $u=Y_1^r$, then we can write
\begin{align}
  P_u =\left\{p \mid z(p)=\frac{\omega(u,p)}{2}\right\}
  \label{eq:planeGraph}
      =\left\{p \mid x_1(p)= \frac{-2 z(p)}{r}\right\}.
\end{align}

If $P$ is a plane and $x\in \H^{2k+1}$, denote
$$\alpha_x(P)\eqdef \begin{cases}
  \dEuc(\pi(x),\pi(y)) & \text{if $P=P_y$}, \\
  \infty & \text{if $P$ is vertical.}
\end{cases}$$
By \eqref{eq:abelianization graph}, $\alpha_{\0}(P)$ is determined by the angle between $P$ and the vertical line $L_Z\eqdef \langle Z\rangle$; specifically,
$$\angle(L_Z,P)=\arctan \left(\frac{2}{\alpha_{\0}(P)}\right).$$

\begin{lemma}\label{lem:distanceToCenter}
  Let $P^+\subset \H^{2k+1}$ be a half-space with boundary $P$ and let $x\in \H^{2k+1}$.  There is a half-space $V^+$ with vertical boundary such that
  $$d_{B_1(x)}(P^+, V^+) \lesssim \frac{1}{\alpha_x(P)}.$$
\end{lemma}
\begin{proof}
  If $P$ is vertical, then the lemma is satisfied by letting $V^+=P^+$.  If $P$ does not intersect $B_1(x)$, then either $B_1(x)\subset P^+$ or $B_1(x)\cap P^+=\emptyset$.  In this case, we can choose $V$ to be a vertical plane disjoint from $B_1(x)$ and choose $V^+$ to be one of the sides of $V$.

  We thus suppose that $P=P_y$ for some $y\in \H^{2k+1}$ and that $P$ intersects $B_1(x)$, say at $w$.  By translating and rotating, we may suppose that $w=\0$, that $x\in B_1$, and $y=Y_1^r$, where $r=\alpha_w(P)$.

  If $d(x,y)\le 2$, choose $V$ to be an arbitrary vertical plane through $\0$. Then the lemma is satisfied because $d_{B_1(x)}(P^+, V^+) \le \diam (B_1(x))\lesssim 1$.  So, suppose that $d(x,y)>2$, thus $r\asymp \alpha_x(P)$.  Let $V=\{h\in \H^{2k+1}:\ x_1(h)=0\}$.  By \eqref{eq:planeGraph} we have
  $P=\{p \mid x_1(p)= -2 z(p)/r\}$.

  If $P^+=\{p\in \H^{2k+1} \mid x_1(p) \ge -2z(p)/r\}$, let $V^+=\{p\in \H^{2k+1} \mid x_1(p)\ge 0\}$.  We claim that $d_{B_2}(P^+,V^+)$ is small.  If $q\in (P^+\symdiff V^+)\cap B_2$, then $|z(q)|\lesssim 1$ and $|x_1(q)|\le 2|z(q)|/r\lesssim 1/r$.  It follows that $d(q,V)\lesssim 1/r$ and $d(q,P)\lesssim 1/r$, so
  $d_{B_1(x)}(P^+,V^+)\le d_{B_2}(P^+,V^+)\lesssim 1/r$,
  as desired.  Likewise, if $P^+=\{p\in \H^{2k+1} \mid x_1(p) \le -2z(p)/r\}$, then choose $V^+=\{p\in \H^{2k+1} \mid x_1(p)\le 0\}$ and again $d_{B_1(x)}(P^+,V^+)\lesssim 1/r$.
\end{proof}

By applying a scaling to Lemma~\ref{lem:distanceToCenter}, we find that
\begin{cor}\label{cor:scaledDistanceToCenter}
  For any $\sigma>0$, if $P^+\subset \H^{2k+1}$ is a half-space and $x\in \H^{2k+1}$, then there is a half-space $V^+$ with vertical boundary such that
  $$\frac{d_{B_\sigma(x)}(P^+, V^+)}{\sigma} \lesssim \frac{\sigma}{\alpha_x(P)}.$$
\end{cor}

Finally, we will need the following simple lemma
\begin{lemma}\label{lem:tiltingPlanes}
  For any $\rho>1$, there exists $0<\delta<\frac{1}{2}$ such that for any
  $x\in \H^{2k+1}$ and any $\sigma>0$, if $P,R\subset \H^{2k+1}$ are
  planes such that $d(x,P)< \delta \sigma$,
  $d_{B_\sigma(x)}(P^+,R^+)<\delta\sigma$, and $\alpha_x(P)<\rho\sigma$,
  then $\alpha_x(R)<2\rho\sigma$.
\end{lemma}
\begin{proof}
  After rescaling and translating, we may suppose that $\sigma=1$ and $x=\0$.  Then $\alpha_x(P)<\rho$ and $\angle(P,L_Z)>\arctan(2/\rho)$.  If $\delta$ is sufficiently small, then $\angle(P,R)$ is small, so $\angle(R,L_Z)>\arctan(1/\rho)$ and thus $\alpha_x(R)<2\rho$, as desired.
\end{proof}

\subsubsection{Approximating cubes by vertical planes}
The goal here is to prove the following proposition.
\begin{prop}\label{prop:goodVerticalCubes}
  For every $0<\epsilon<1$, there is a set $\cG(\epsilon)\subset \Delta$ such that for every $Q\in \cG(\epsilon)$, there is a vertical half-space $V^+_Q$ with bounding plane $V_Q$ such that
  \begin{equation}\label{eq:goodVerticalCubes}
    d_{N_{\frac{1}{\e}}(Q)}(V_Q^+, E)\le \epsilon\sigma(Q).
  \end{equation}
  Moreover, $\cB\eqdef\Delta\setminus \cG(\epsilon)$ satisfies a Carleson packing condition with constant that depends on $\e$.
\end{prop}

Let $\eta>0$ be a small number to be determined later.  By Proposition~\ref{prop:goodCubes}, there is a set of cubes $\cG_0 = \cG_0(\eta) \subset \Delta$ and a collection of planes $\{P_Q\}_{Q\in \Delta}$ such that $d_{N_{1/\eta}(Q)}(P_Q^+, E) \le \eta \sigma(Q)$ for all $Q\in \cG_0$.  The set $\cG=\cG(\epsilon)$ will consist of the cubes $Q\in\cG_0$ for which $P_Q$ is close to vertical.  Specifically, if $Q\in \cG_0$, denote
$$\alpha(Q)\eqdef \frac{\min_{x\in Q} \alpha_x(P_Q)}{\sigma(Q)}.$$
Corollary~\ref{cor:scaledDistanceToCenter} implies that if $x\in Q$, then there is a vertical plane $V$ such that
\begin{equation}\label{eq:alphaImpliesVerticalApprox}
  d_{N_{\frac{2}{\e}}(Q)}(P_Q^+,V^+)\le d_{(\frac{2}{\epsilon}+C)\sigma(Q)}(P_Q^+,V^+) \lesssim \frac{(\frac{2}{\e}+C)^2\sigma(Q)^2}{\alpha_x(P_Q)} \lesssim
  \frac{\sigma(Q)}{\e^2\alpha(Q)}.
\end{equation}
We will define $\cG=\cG_0\setminus \cB_1$, where $\cB_1\eqdef \{Q\in \cG_0\mid \alpha(Q)<r\}$ are $r>0$ will be a quantity depending on $\e$ that will be determined later.

The main difficulty is to prove that $\cB_1$ is Carleson.  If $Q,R\in \cG_0$, we say that $R$ is a \emph{good descendant} of $Q$ if $R\subset Q$ and $A\in \cG_0$ for all $A\in \Delta$ such that $R\subset A\subset Q$.  Let $G(Q)$ be the set of good descendants of $Q$ and define for every $n\in \N\cup\{0\}$,
$$G(Q,r)\eqdef \{R\in G(Q)\mid \alpha(R)<r\}\qquad\mathrm{and}\qquad G_n(Q,r)\eqdef\{R\in G(Q,r)\mid \sigma(R)=2^{-n}\sigma(Q)\}.$$
Our goal is to bound the size of $G(Q,r)$.
\begin{lemma}\label{lem:verticalCubeChildren}
  Fix $r>0$.  If $\eta$ is sufficiently small, then for all $Q\in \cG_0$ and all $R\in G(Q,r)$, every $A\in \Delta$ such that $R\subset A\subset Q$ is an element of $G(Q,r)$.  Consequently, for all integers $m,n\ge 0$,
  \begin{equation}\label{eq:vertical cube children}
  G_{m+n}(Q,r)=\bigcup_{R\in G_{m}(Q,r)}G_{n}(R,r).
\end{equation}
  In fact, there exists $0<\lambda<1$ depending on $r$ such that for every integer $n\ge 0$ we have
  \begin{equation}\label{eq:descendents decay lambda}
    \sum_{R\in G_n(Q,r)} \sigma(R)^{2k+1}\lesssim \lambda^n\sigma(Q)^{2k+1},
  \end{equation}
  and thus
  $$\sum_{R\in G(Q,r)} \sigma(R)^{2k+1} \lesssim_r \sigma(Q)^{2k+1}.$$
\end{lemma}
\begin{proof}
  To prove the first part of the lemma, let $R\in G(Q,r)$ and let $A\in G(Q)$ be the parent of $R$.  We need to show that if $\eta$ is sufficiently small, then $A\in G(Q,r)$, i.e., $\alpha(A)<r$.

  By part (4) of Lemma~\ref{lem:localHausProps}, if $x\in R$ and $\eta<\frac{1}{2}$, then
  \begin{align*}
    d_{B_{\sigma(R)}(x)}(P_A^+,P_{R}^+)
    \le
      d_{B_{2\sigma(R)}(x)}(P_A^+,E)+d_{B_{2\sigma(R)}(x)}(P_{R}^+,E)
    \le
      \eta \sigma(A)+\eta \sigma(R)
    =3\eta \sigma(R).
  \end{align*}
Since $\alpha(R)<r$, there exists $x\in R$ such that $\alpha_x(P_R)<r \sigma(R)$.  We have $d(x,P_R)\le \eta \sigma(R)$.  By Lemma~\ref{lem:tiltingPlanes}, if $\eta$ is sufficiently small, then $\alpha_x(P_A)< 2r\sigma(R)$ and $\alpha(A)\le \alpha_x(P_A)/\sigma(A)<r.$

  For every $Q\in \cG_0$ and $n\ge 0$, denote $g_{Q,n}\eqdef \sum_{R\in G_{n}(Q,r)} \sigma(R)^{2k+1}$.  To prove the second part, we will first show that there is some $n_0$ and some $\eta$ such that for all $Q\in \cG_0$, we have
  \begin{equation}\label{eq:verticalDescendantsDecay}
    g_{Q,n_0}\le \frac{\sigma(Q)^{2k+1}}{2}.
  \end{equation}
 This estimate implies the desired conclusion~\eqref{eq:descendents decay lambda} because for any $n\ge 0$ we have
  $$g_{Q,n+n_0}\stackrel{\eqref{eq:vertical cube children}}{=} \sum_{A\in G_{n}(Q,r)} g_{A,n_0} \stackrel{\eqref{eq:verticalDescendantsDecay}}{\le} \sum_{A\in G_{n}(Q,r)}\frac{\sigma(A)^{2k+1}}{2}=\frac{g_{Q,n}}{2}.$$
  Furthermore, for all $n\ge 0$,
  $$g_{Q,n}=\sum_{R\in G_{n}(Q,r)}\sigma(R)^{2k+1}\le C \sum_{R\in G_{n}(Q,r)}\cH^{2k+1}(R)\le C \cH^{2k+1}(Q)\le C^2 \sigma(Q)^{2k+1}.$$
  For every integer $n\ge 0$ write  $n=mn_0 + i$, where $m=\lfloor n/n_0\rfloor$. Then $i\ge 0$ and it follows from the above estimates that $g_{Q,n}\le 2^{-m}g_{Q,i}\le 2C^22^{-n/n_0}\sigma(Q)^{2k+1}$. Hence~\eqref{eq:descendents decay lambda} holds with $\lambda=2^{-1/n_0}$. It therefore remains to justify~\eqref{eq:verticalDescendantsDecay}.

  Let $Q\in \cG_0$.  If $\alpha(Q)\ge r$, then $G_0(Q,r)=\emptyset$ and therefore~\eqref{eq:vertical cube children} implies that $G_n(Q,r)=\emptyset$.  So, suppose that $\alpha(Q)< r$.  Let $y$ be the center of $P_Q$, i.e., $P_y=P_Q$. Suppose that $\delta>0$ is such that Lemma~\ref{lem:tiltingPlanes} holds true for $\rho=r$.  Let $n\in \N$ and $R\in G_n(Q,r)$.  We claim that $R$ is close to $y$ if $\eta<2^{-n-1}\delta$.  Indeed, let $x\in R$ be such that $\alpha_x(R) < r \sigma(R)$.  Then $d(x,P_Q)<\eta \sigma(Q)<\delta \sigma(R)$ and
  \begin{align*}
    d_{B_{\sigma(R)}(x)}(P_Q^+,P_{R}^+)
    \le d_{B_{2\sigma(R)}(x)}(P_{Q}^+,E)+d_{B_{2\sigma(R)}(x)}(P_{R}^+,E)
    \le \eta\big(\sigma(Q)+\sigma(R)\big)< \delta \sigma(R).
  \end{align*}
  Since $\alpha_x(P_R) < r \sigma(R)$, Lemma~\ref{lem:tiltingPlanes} implies that $\alpha_x(P_Q) < 2r \sigma(R)$.  Let $z\in P_Q$ be a point such that $d(x,z)\le 2 d(x,P_Q)\le 2\eta \sigma(Q)$.  Then $d(z,y)=d(\pi(z),\pi(y))\le \alpha_x(P_Q)+2\eta \sigma(Q)$ and
  $$d(x,y)\le d(x,z)+d(z,y)\le \alpha_x(P_Q)+4\eta \sigma(Q)\le (4\delta+2r) \sigma(R)=2^{-n}(4\delta+2r)\sigma(Q).$$
 Consequently, if $R\in G_{n_0}(Q,r)$, then $R\subset B_{2^{-n_0}(4\delta+2r+C)\sigma(Q)}(y),$
  where $\delta$ depends on $r$.  Choose $n_0$ (depending on $r$ and $C$) sufficiently large that $2^{-n_0}(4\delta+2r+C)<1/(2C)$.  Then
  \begin{equation*}
    g_{Q,n_0} \le C \sum_{R\in G_{n}(Q,r)} \cH^{2k+1}(R)\le C \cH^{2k+1}\Big(Q\cap B_{\frac{\sigma(Q)}{2C}}(y)\Big) \le C^2\cdot \frac{\sigma(Q)^{2k+1}}{(2C)^{2k+1}}
    \le \frac{\sigma(Q)^{2k+1}}{2},
  \end{equation*}
  which is the desired estimate~\eqref{eq:verticalDescendantsDecay}.
\end{proof}

\begin{proof}[{Proof of Proposition~\ref{prop:goodVerticalCubes}}]
   By \eqref{eq:alphaImpliesVerticalApprox}, there exists $r\asymp 1/\e^3$ so that if $Q\in \Delta$, $x\in Q$, and $P$ is a plane such that $\alpha_x(P)\ge r\sigma(Q)$, then there is a vertical plane $V$ such that
  $$d_{N_{\frac{2}{\epsilon}}(Q)}(P^+,V^+)\le \frac{\epsilon\sigma(Q)}{2}.$$
  Let $0<\eta<\e/2$ be such that Lemma~\ref{lem:verticalCubeChildren} holds for this choice of $r$.  Let $\cG_0=\cG_0(\eta)$ and $\cB_0=\cB_0(\eta)$ be as in Proposition~\ref{prop:goodCubes}. Denote
  $\cB_1=\{Q\in \cG_0\mid \alpha(Q)<r\}$ and $\cG=\cG_0\setminus \cB_1$.

  If $Q\in \cG$, then by part (4) of Lemma~\ref{lem:localHausProps} we have
  $$d_{N_{\frac{1}{\epsilon}}}(E,V_Q^+)\le d_{N_{\frac{2}{\epsilon}}(Q)}(E, P_Q^+) + d_{N_{\frac{2}{\epsilon}}(Q)}(P_Q^+, V_Q^+)\le \frac{\epsilon\sigma(Q)}{2} +\frac{\epsilon\sigma(Q)}{2}=\e\sigma(Q).$$
  We claim that $\Delta\setminus \cG=\cB_0\cup \cB_1$ satisfies a Carleson packing condition with constant that depends on $\e$.  Since this is true for $\cB_0$ by  Proposition~\ref{prop:goodCubes}, it remains only to prove that $\cB_1$ is Carleson.

  For all $R\in \Delta$ such that $\sigma(R)<1$, let $F(R)$ be the parent of $R$.  Recall that $\Delta_0\subset \cB_0$, so if $R\in \cB_1$, then there is a smallest natural number $i$ such that $F^i(R)\in \cB_0$.   Lemma~\ref{lem:verticalCubeChildren} implies that if $F^j(R)\in \cB_1$, then either $F^{j+1}(R)\in \cB_0$ or $F^{j+1}(R)\in \cB_1$, so, by induction, $F^j(R)\in \cB_1$ for all $j=0,\dots, i-1$.  That is, if $A=F^{i-1}(R)$, then $F(A)\in \cB_0$ and $R\in G(A,r)$.

  Let $M=\{A\in \Delta \mid F(A) \in \cB_0\}$.  Then $\cB_1=\bigcup_{A\in M} G(A,r)$. By Lemma~\ref{lem:cube expansion}, $M$ is $L$-Carleson for some $L$ depending on $\epsilon$.   If $Q\in \Delta$, $R\subset Q$, and $R\in \cB_1$, then either $R\in G(Q,r)$ (where we set $G(Q,r)=\emptyset$ if $Q\in \cB_0$) or $R\in G(A,r)$ for some $A\in M$ such that $A\subset Q$. Hence,  by Lemma~\ref{lem:verticalCubeChildren},
  \begin{multline*}
    \sum_{\substack{R\in \cB_1\\ R\subset Q}} \sigma(R)^{2k+1}
    =\sum_{R\in G(Q,r)} \sigma(R)^{2k+1}+\sum_{\substack{A\in M\\ A\subset Q}} \sum_{R\in G(A,r)} \sigma(R)^{2k+1}\\
    \lesssim_\epsilon \sigma(Q)^{2k+1}+\sum_{\substack{A\in M\\ A\subset Q}} \sigma(A)^{2k+1}
    \le (L+1) \sigma(Q)^{2k+1}.\tag*{\qedhere}
  \end{multline*}
\end{proof}

\subsection{Constructing stopping-time regions}
In this section, we group the cubes in $\cG$ into stopping-time regions.  We will prove the following proposition.  The arguments in this section are based on arguments in Chapter 7 of \cite{DavidSemmesSingular}, except for Lemma~\ref{lem:cF1 causes nonmonotonicty}, which uses nonmonotonicity.
\begin{prop}\label{prop:smallAngleSTRs}
  For every $0<\lambda<\frac{\pi}{2}$, if $\epsilon>0$ is sufficiently small (depending on $\lambda$) and if $\cG=\cG(\epsilon)\subset \Delta$ and $\cB=\Delta\setminus \cG$ are as in Proposition~\ref{prop:goodVerticalCubes}, then we can partition $\cG$ into a set $\cF$ of stopping-time regions such that $(\cB,\cG,\cF)$ is a coronization (Definition~\ref{def:coronization}) of $\partial E$ and
  \begin{equation}\label{eq:small angle condition}
    \forall \cS\in \cF,\  \forall Q\in \cS \qquad \angle(V_Q,V_{\mathsf{Q}(\cS)})<\lambda.
  \end{equation}
  The Carleson constants of the coronization will depend only on $\lambda$, $\epsilon$, $C$, and $k$.
\end{prop}
For every $Q\in \Delta$, let $\cS_Q$ be the maximal coherent set such that $Q$ is the maximal cube in $\cS_Q$ and such that for all $R\in \cS_Q$, we have $\angle(V_{R},V_Q)<\lambda$.  We construct such a set by the usual inductive stopping-time argument; we start by adding $Q$, then for each $R\in \cS_Q$, we add the children of $R$ to $\cS_Q$ if and only if they are all in $\cG$ and all of their approximating vertical planes are $\lambda$-close to $V_Q$.

We partition $\cG$ into such regions by applying this construction repeatedly.  First, let $Q$ be a maximal element of $\cG$.  We add $\cS_Q$ to $\cF$ and remove the elements of $\cS_Q$ from $\cG$, then take a maximal element from the remaining cubes and repeat.  By design, this results in a partition of $\cG$ into coherent regions that satisfies \eqref{eq:small angle condition}.

The main difficulty is to show that the maximal elements of the stopping-time regions satisfy a Carleson packing condition.  To prove this, we divide $\cF$ into three parts.  For $\cS\in \cF$, let $m(\cS)\subset \cS$ be the set of minimal cubes in $\cS$.  These cubes are disjoint, but need not cover $\mathsf{Q}(\cS)$ because there may be points that are contained in arbitrarily small cubes.  Each element of $m(\cS)$ has a child $Q$ such that either $Q\in \cB$ or $\angle(V_Q,V_{\mathsf{Q}(\cS)})>\lambda$.  Let $m_0(\cS)$ be the set of minimal cubes with at least one child in $\cB$ and let $m_1(\cS)$ be the set of minimal cubes with at least one child such that $\angle(V_Q,V_{\mathsf{Q}(\cS)})>\lambda$. Then $m(\cS)=m_0(\cS)\cup m_1(\cS)$.  Let
\begin{align}
\nonumber  \cF_0&=
         \bigg\{\cS\in \cF\mid \cH^{2k+1}\bigg(\bigcup_{Q\in  m_0(\cS)} Q\bigg)\ge \frac{\cH^{2k+1}\big(\mathsf{Q}(\cS)\big)}{4}\bigg\},\\ \label{def cF1}
  \cF_1&=
         \bigg\{\cS\in \cF\mid \cH^{2k+1}\bigg(\bigcup_{Q\in m_1(\cS)} Q \bigg)\ge \frac{\cH^{2k+1}\big(\mathsf{Q}(\cS)\big)}{2}\bigg\},\\
  \cF_2&=
  \nonumber       \bigg\{\cS\in \cF\mid \cH^{2k+1}\bigg(\mathsf{Q}(\cS)\setminus \bigcup_{Q\in m(\cS)} Q\bigg)\ge \frac{\cH^{2k+1}\big(\mathsf{Q}(\cS)\big)}{4}\bigg\}.
\end{align}
Then $\cF=\cF_0\cup \cF_1\cup \cF_2$. For every $i\in \{0,1,2\}$ let $M_i=\{\mathsf{Q}(\cS)\mid \cS\in \cF_i\}$, and if $R\in \Delta$, let $M_i(R)=\{Q\in M_i\mid Q\subset R\}$ and $\cF_i(R)=\{\cS\in \cF_i\mid \mathsf{Q}(\cS)\subset R\}$.  We claim that there are $K_0,K_1,K_2>0$ depending on $\lambda$ such that $M_i$ is $K_i$-Carleson for each $i\in \{0,1,2\}$.

Firstly, we consider $M_0$.  For every $Q\in \Delta$, let $F(Q)$ be the parent of $Q$.  For each $\cS\in \cF$, let $Y_{\cS}=\{Q\in \cB\mid F(Q)\in \cS\}$ be the set of bad children of elements of $m_0(\cS)$.  If $Q\in m_0(\cS)$, then $Q$ has a child $D\in Y_\cS$, so
$$\sum_{Q\in m_0(\cS)}\sigma(Q)^{2k+1}\asymp \sum_{Q\in Y_\cS}\sigma(Q)^{2k+1}.$$
Consequently, if $\cS\in \cF_0$, then we have
$$\sum_{Q\in Y_\cS}\sigma(Q)^{2k+1}\asymp \sigma\big(\mathsf{Q}(\cS)\big)^{2k+1}.$$
Because the elements of $\cF$ are disjoint, the sets $\{Y_\cS\}_{\cS\in \cF}$ are disjoint subsets of $\cB$, so for $R\in \Delta$,
\begin{align*}
  \sum_{Q\in M_0(R)} \sigma(Q)^{2k+1}
  \asymp \sum_{\cS\in\cF_0(R)} \sum_{Q\in Y_\cS}\sigma(Q)^{2k+1}
  \le \sum_{\substack{Q\in \cB\\ Q\subset R}} \sigma(Q)^{2k+1}
   \lesssim \sigma(R)^{2k+1}.
\end{align*}
That is, the maximal cubes of $\cF_0$ satisfy the desired Carleson packing condition.

Next, we consider $M_2$.  The sets $\{U_\cS=\mathsf{Q}(\cS)\setminus \bigcup_{Q\in m(\cS)} Q\}_{\cS\in \cF}$ are pairwise disjoint and satisfy $\cH^{2k+1}(U_\cS)\gtrsim \sigma(\mathsf{Q}(\cS))^{2k+1}$ for all $\cS\in \cF_2$.  If $Q\in \Delta$, then
$$\sum_{\cS\in \cF_2(Q)} \sigma\big(\mathsf{Q}(\cS)\big)^{2k+1} \lesssim \sum_{\cS\in \cF_2(Q)} \cH^{2k+1}(U_\cS)=\cH^{2k+1}\biggl(\bigcup_{\cS\in \cF_2(Q)}U_\cS\biggr)\le \cH^{2k+1}(Q)\asymp \sigma(Q)^{2k+1},$$
so $M_2$ satisfies the desired  Carleson packing condition as well.

It remains to show that $M_1$ is Carleson.  Recall that for every integer $j\le 0$ and  $Q\in \Delta_j$, we defined the nonmonotonicity of $Q$ by $w(Q)=\widehat{w}_{j}(E)(Q).$ If $Q\in \cS$ and $\angle(V_Q,V_{\mathsf{Q}(\cS)})>\lambda$, then there are horizontal lines that cross $V_Q$ in the positive direction (i.e., from $V_Q^-$ to $V_Q^+$) but cross $V_{\mathsf{Q}(\cS)}$ in the negative direction.  We will show that these lines contribute to the {\em total nonmonotonicity}  of the stopping time region $\cS$, which is defined to be the quantity $w(\cS)=\sum_{Q\in \cS} w(Q)$. We will then bound the number of cubes in $M_1$ in terms of the nonmonotonicity.

Our first goal will be to prove that most lines that cross $V_Q$ also cross $\partial E$.  
As in Section~\ref{sec:planes and angles}, we identify $\R^{2k}$ with the subspace $\mathsf{H}=\{v\in \H^{2k+1}:\ z(v)=0\}$ of horizontal vectors.  We will use $\angle_{\mathsf{H}}$ to denote Euclidean angles between vectors and planes in $\mathsf{H}$.
\begin{lemma}\label{lem:lineColoring}
  For any $\lambda_0>0$ and any $0<a<b$, if $\epsilon>0$ is sufficiently small, then the following property holds.  Let $\cG=\cG(\epsilon)$ be a set satisfying Proposition~\ref{prop:goodVerticalCubes}.  Fix $Q\in \cG$ and let $x\in N_1(Q)$ be a point such that $d(x,V_Q)<3\epsilon \sigma(Q)$. Also, let $u\in \mathsf{H}$ be a horizontal vector such that $\angle_{\mathsf{H}}(u,\pi(V_Q))>\lambda_0$ and $u$ points toward $V_Q^+$.  Let $L=x u^\R$ be the horizontal line through $x$ in the direction of $u$ and let $\gamma(t)$ be its unit-speed parameterization, with $\gamma(0)=x$.  Then we have
  $$\gamma\Big(\big(-b\sigma(Q),-a\sigma(Q)\big)\Big)\subset E^{\cc}\qquad\mathrm{and}\qquad \gamma\Big(\big(a\sigma(Q),b\sigma(Q)\big)\Big)\subset E.$$
\end{lemma}
\begin{proof}
  suppose that $\e\in (0,\infty)$ satisfies
  \begin{equation}\label{eq:epsilon choice angle lambda 0}
  0<\epsilon<\min\left\{\frac{1}{4},\frac{a\sin \lambda_0}{4}, \frac{1}{b+1}\right\}.
  \end{equation}
  Let $f\from \H^{2k+1}\to \R$ be the signed distance to $V_Q$, i.e.,
  \[f(p)=\begin{cases}
      d(p,V_Q) & p\in V_Q^+, \\
      -d(p,V_Q) & p\in V_Q^-.
    \end{cases}\]
  This is a linear function on $\H^{2k+1}$.  If $f(p)>\epsilon\sigma(Q)$, then $p\in V_Q^+$ and $d(p, V_Q)>\epsilon\sigma(Q)$. Hence, if also  $p\in N_{1/\e}(Q)$, then $p\in E$.  Similarly, if $f(p)<-\epsilon\sigma(Q)$ and $p\in N_{1/\e}(Q)$, then $p\in E^\cc$.

  For all $t\in \R$, we have $\pi(\gamma(t))=\pi(x)+tu$, so
  $$f\big(\gamma(t)\big)=f(x)+t\sin \Big(\angle_{\mathsf{H}}\big(u,\pi(V_Q)\big)\Big).$$
  By hypothesis, $|f(x)|<3\epsilon \sigma(Q)$. Consequently,
  if $t\in (a\sigma(Q),b\sigma(Q))$, then
  $$f\big(\gamma(t)\big)>-3\epsilon\sigma(Q)+a\sigma(Q) \sin \Big(\angle_{\mathsf{H}}\big(u,\pi(V_Q)\big)\Big)\stackrel{\eqref{eq:epsilon choice angle lambda 0}}{>}\epsilon \sigma(Q).$$
Since $\gamma(t)\in N_{1+t/\sigma(Q)}(Q)\subset N_{1/\e}(Q)$, it follows that $\gamma(t)\in E$. Likewise, if $t\in (-b\sigma(Q),-a\sigma(Q))$, then
  $f(\gamma(t))<-\epsilon \sigma(Q)$ and
  $\gamma(t)\in N_{1/\e}(Q)$, so $\gamma(t)\in E^{\cc}$.
\end{proof}

Consequently, each stopping-time region in $\cF_1$ contributes to the nonmonotonicity of $E$.
\begin{lemma}\label{lem:cF1 causes nonmonotonicty}
  For any $0<\lambda<\frac{\pi}{2}$, there exists $\epsilon >0$ such that if $\cG=\cG(\epsilon)$ satisfies Proposition~\ref{prop:goodVerticalCubes} and $\cS\in \cF_1$, then
  $w(\cS)\gtrsim_\lambda \sigma(\mathsf{Q}(\cS))^{2k+1}.$
\end{lemma}
\begin{proof}
  Let $0<c<1$ be as in Lemma~\ref{lem:cubeCenters}, so that for all $Q\in \Delta$, there is a point $x_{Q}\in Q$ such that $B_{c\sigma(Q)}(x_{Q})\cap \partial E \subset Q$.  Let $0<\epsilon<c/4$ be a number satisfying Lemma~\ref{lem:lineColoring} for $\lambda_0=\lambda/4$, $a=c/4$ and $b=2$.  Note that $\epsilon$ can be taken to be a function of $\lambda$ and $C$.

  For each $Q\in \Delta$, let $\widehat{w}_Q$ be the restriction of $\widehat{w}_{\log_2\sigma(Q)}(E)$ to $Q$, i.e., the measure defined by $\widehat{w}_Q(A)=\widehat{w}_{\log_2\sigma(Q)}(E)(A\cap Q)$ for every measurable $A\subset \H^{2k+1}$.  For every $\cS\in \cF$ denote $\widehat{w}_\cS=\sum_{Q\in \cS} \widehat{w}_Q$.  Then
  \begin{equation}\label{eq:w hat of H}
  \widehat{w}_\cS(\H^{2k+1})=\sum_{Q\in \cS}\widehat{w}_{\log_2\sigma(Q)}(E)(Q)=\sum_{Q\in \cS} w(Q)=w(\cS).
  \end{equation}
Suppose that $A\in m_1(S)$ and let $Q$ be a child of $A$ such that $Q\in \cG$ and $\theta\eqdef \angle(V_{Q},V_{\mathsf{Q}(S)})>\lambda$.  We claim that $\widehat{w}_\cS(A)\gtrsim_\lambda \sigma(A)^{2k+1}$. To see this, write $\sigma=\sigma(Q)$ and $\sigma_\cS=\sigma(\mathsf{Q}(\cS))$.  Let $S^{2k-1}\subset \mathsf{H}$ be the unit sphere in the space of horizontal vectors.  Let $u_0\in S^{2k-1}$ bisect $V_Q$ and $V_{\mathsf{Q}(S)}$; i.e., let $u_0\in \mathsf{H}$ be the horizontal unit vector that points into $\pi(V_Q^+)\cap \pi(V_{\mathsf{Q}(S)}^-)$, and satisfies
  $$\angle_{\mathsf{H}}\big(\pi(V_{Q}), \pi(u_0)\big)=\angle_{\mathsf{H}}\big(\pi(V_{\mathsf{Q}(S)}), \pi(u_0)\big)=\frac{\theta}{2}.$$
  Denote $U=\{u\in S^{2k-1}: \angle_{\mathsf{H}}(u,u_0)<\lambda/4\}$ 
  and $W=B_{\epsilon \sigma}(x_Q)$.  For $x\in W$ and $u\in U$, let $L(x,u)$ be the horizontal line through $x$ in direction $u$ and let $\gamma(t)=\gamma_{x,u}(t)=x u^t$ be its unit-speed parameterization.  Then $\angle_\mathsf{H}(u,\pi(V_Q))>\frac{\lambda}{4}$ and $d(x,V_Q)<\epsilon \sigma+d(x_Q,V_Q)<2\epsilon \sigma$, so $\gamma((-b\sigma,-a\sigma))\subset E^{\cc}$ and $\gamma((a\sigma,b\sigma))\subset E$ by Lemma~\ref{lem:lineColoring}.  Similarly, since
  $$d\big(x,V_{\mathsf{Q}(\cS)}\big)\le d(x,x_Q)+d\big(x_Q,V_{\mathsf{Q}(\cS)}\big)<\epsilon \sigma +\epsilon \sigma_\cS< 2\epsilon \sigma_\cS,$$
  Lemma~\ref{lem:lineColoring} implies that
  $\gamma((-b\sigma_\cS,-a\sigma_\cS))\subset E$ and
  $\gamma((a\sigma_\cS,b\sigma_\cS))\subset E^{\cc}$. Let $I=[p,q]\subset \R$ be the maximal interval such that $(a\sigma,b\sigma)\subset I$ and $\gamma(I)\subset E$.  Then $I\subset [-a\sigma,a\sigma_\cS]$, and therefore
  $\sigma<(b-a)\sigma\le \ell(I)\le 2a\sigma_S<\sigma_S.$
  Furthermore, $p\in [-a\sigma,a\sigma]$, so $\gamma(p)\in \partial E$ and
  $$d(\gamma(p),x_Q)\le a\sigma+d(x,x_Q) \le \frac{c\sigma}{4}+\epsilon\sigma\le \frac{c\sigma}{2}.$$
  By our choice of $x_Q$, this implies $\gamma(p)\in Q$.

 Denote $L_{W,U}=\{L(x,u)\mid x\in W, u\in U\}$.  For each $L\in L_{W,U}$, the intersection $E\cap L(x,u)$ contains a maximal interval $I$ of length $\ell(I)\in (\sigma, \sigma_\cS)$ with one endpoint in $A$.  It follows that
  \begin{equation}\label{eq:w hat lower on A}
  \widehat{w}_\cS(A)\ge \sum_{i=\log_2 \sigma}^{\log_2 \sigma_\cS}\widehat{w}_i(A)\ge \frac{\cL(L_{W,U})}{2} \asymp \epsilon^{2k+1} \sigma^{2k+1} \lambda^{2k-1}\asymp_\lambda \sigma(A)^{2k+1}.
  \end{equation}
  Since $\cS\in \cF_1$ and all the elements of $m_1(\cS)$ are disjoint, we have
  \begin{equation*}
  w(\cS)\stackrel{\eqref{eq:w hat of H}}{\ge} \widehat{w}_\cS\bigg(\bigcup_{A\in m_1(\cS)} A \bigg) \stackrel{\eqref{eq:w hat lower on A}}{\gtrsim}_{\!\!\!\lambda}\sum_{A\in m_1(\cS)} \sigma(A)^{2k+1}\gtrsim \cH^{2k+1}\bigg(\bigcup_{A\in m_1(\cS)} A \bigg) \stackrel{\ref{def cF1}}{\gtrsim} \sigma\big(\mathsf{Q}(\cS)\big)^{2k+1}.\qedhere\end{equation*}
\end{proof}

Finally, for all $R\in \Delta$ by Lemma~\ref{lem:cF1 causes nonmonotonicty} we have
$$\sum_{\cS\in \cF_1(R)} \sigma\big(\mathsf{Q}(\cS)\big)^{2k+1}\lesssim_\lambda \sum_{\cS\in \cF_1(R)} w(\cS)\stackrel{\eqref{eq:kinematicW}}{\le} \Per(E)(R) \asymp \sigma(R)^{2k+1},$$
so $\cF_1$ satisfies a Carleson packing condition.  This concludes the proof of Proposition~\ref{prop:smallAngleSTRs}.\qed

\subsection{Stopping-time regions are close to Lipschitz graphs}\label{sec:STRLipschitzGraphs}
Finally, we prove that the coronization constructed in Proposition~\ref{prop:smallAngleSTRs} is in fact a corona decomposition $\partial E$ by constructing intrinsic Lipschitz graphs that approximate the stopping-time regions.

\begin{prop}\label{prop:STRLipschitzGraphs}
  For every $\eta,\theta\in (0,1)$, there are $\epsilon>0$ and $\lambda\in (0,1)$ with the following property.  If $\cS\subset \Delta$ is a coherent set such that for all $Q\in \cS$, there is a vertical half-space $V^+_Q$ with bounding plane $V_Q$ such that
  \begin{equation}\label{eq:STR lip close to planes}
    d_{N_{\frac{1}{\epsilon}}(Q)}(V_Q^+, E)\le \epsilon\sigma(Q)\qquad \mathrm{and}\qquad \angle(V_Q,V_{\mathsf{Q}(\cS)})<\lambda
  \end{equation}
 (e.g., a stopping-time region in the coronization constructed in Prop.~\ref{prop:smallAngleSTRs}), then there is an intrinsic Lipschitz graph $\Gamma\subset \H^{2k+1}$ with Lipschitz constant at most $\eta$ such that for all $Q\in \cS$,
  $$d_{N_4(Q)}(\Gamma^+, E)\le \theta \sigma(Q).$$
\end{prop}

We choose $\lambda=\frac{\eta}{10}$.  The constant $\epsilon$ will depend on $\eta$ and $\theta$, but it will satisfy $\epsilon<10^{-3}$.  Let $M=\mathsf{Q}(\cS)$.  We can apply an isometry of $\H^{2k+1}$ so that $V_M=\{h:\H^{2k+1}:\ x_k(h)=0\}$ and $V^+_{M}=\{h\in \H^{2k+1}:\ x_k(h)>0\}$.  Then the one-parameter subgroup $\langle X_k\rangle$ is a horizontal line orthogonal to $V_{M}$. For $y\in \H^{2k+1}$ we denote $L_y=y\langle X_k\rangle$.  Let $\cL^\perp=\{L_x\mid x\in \H^{2k+1}\}$ be the set of cosets of $\langle X_k\rangle$ and, for any $U\subset \H^{2k+1}$, let $\cL^\perp(U)=\{L_u\mid u\in U\}$ be the set of cosets that pass through $U$.  If $g\from \cL^\perp(U)\to \R$ is a function, we let $\bar{g}\from U\to \R$ denote the corresponding function $\bar{g}(y)=g(L_y)$ that is constant on cosets.

If $V$ is a vertical plane such that $\angle(V,V_M)<\frac{\pi}{2}$, then it intersects cosets of $\langle X_k\rangle$ transversely, and we define $\Pi_V\from \H^{2k+1}\to V$ to be the projection whose fibers are the $L_y$'s.  That is, we let $\Pi_{V}(y)$ be the intersection point of $L_y$ and $V$.  Let $\Pi=\Pi_{V_M}$ be the projection to $V_M$, i.e., $\Pi(y)=y X_k^{-x_k(y)}$.

One computes that if $V\subset \H^{2k+1}$ is a vertical plane then for every $v\in V$, $t\in \R$ and $y\in \H^{2k+1}$,
$$
d(v X_k^t,V)=|t|\cos \big(\angle(V,V_M)\big)\qquad\mathrm{and}\qquad d\big(y,\Pi_V(y)\big)= \frac{d(y,V)}{\cos \big(\angle(V,V_M)\big)}.
$$

In particular, if $V=V_Q$ for some $Q\in S$, then
$\angle(V,V_S)\le \lambda$, and therefore
\begin{equation}\label{eq:distToFlatPlane}
  d\big(y,\Pi_V(y)\big)\le 2d(y,V).
\end{equation}

Recall that for $\rho>0$, we define $\rho Q=\partial E\cap N_\rho(Q).$
\begin{lemma}\label{lem:N4 in 10Q}
  $\cL^\perp(N_{9/2}(M))\subset \cL^\perp(10M)$
\end{lemma}
\begin{proof}
   $|x_k(h)|\le \epsilon \sigma(M)$ for all $h\in M$, so every $v\in N_{9/2}(M)$ satisfies $|x_k(v)|\le (9/2+\epsilon)\sigma(M)$.  Therefore $d(v,\Pi(v))\le 5\sigma(M)$ and $\Pi(v)\in N_{19/2}(M)$.  Since $\Pi(v)\in V_M$, by the first inequality in~\eqref{eq:STR lip close to planes} we have $\Pi(v) X_k^{2 \epsilon \sigma(M)}\in E$ and $\Pi(v) X_k^{-2 \epsilon \sigma(M)}\not\in E$.  It follows that there is $q\in \partial E \cap L_{\Pi(v)}$ such that $d(\Pi(v),q)\le 2\epsilon \sigma(M)$.  Then $q\in 10M$ and $L_v=L_q\in \cL^\perp(10M)$, as desired.
\end{proof}

As in Chapter 8 of \cite{DavidSemmesSingular}, we will construct
$\Gamma$ using the function $d_{\cS}$
that we defined in Section~\ref{sec:corona decompositions}.  The quantity $d_{\cS}(v)$ measures the amount of control we have over $E$ near $v$:
\begin{lemma}\label{lem:ds balls}
  If $0<\alpha<1$, and if $\epsilon>0$ is sufficiently small (depending on $\alpha$ and $C$), then for every point $v\in \H^{2k+1}$, and every $r\in (\alpha d_{\cS}(v),\sigma(M)/\alpha]$, there exists $R\in \cS$ and a vertical half-space $V^+=V_R^+$ such that
  $d_{B_r(v)}(E,V^+)\le \epsilon r/\alpha.$
\end{lemma}
\begin{proof}
 Suppose that $0<\epsilon<\alpha^2/4$.  Note that $\alpha d_{\cS}(v)<\sigma(M)/\alpha$ by our assumption on $r$, so $d(v,M)\le d_{\cS}(v)<\sigma(M)/\alpha^2$ and thus $B_r(v)\in N_{\alpha^{-2}+\alpha^{-1}}(M)\subset N_{\epsilon^{-1}}(M)$.

  We first consider the simpler case that $r\in [\alpha\sigma(M), \sigma(M)/\alpha]$.  Let $V^+=V_M^+$. Then
  $$d_{B_r(v)}(E,V^+)\le d_{N_{\frac{1}{\e}}(M)}(E,V^+)\stackrel{\eqref{eq:STR lip close to planes}}{\le} \epsilon \sigma(M)\le \frac{\epsilon r}{\alpha}.$$

  Otherwise, if $r<\alpha\sigma(M)$, then by the definition of $d_{\cS}(v)$, there is a cube $Q\in \cS$ such that $d(v,Q)+\sigma(Q)<r/\alpha$.  Since $\sigma(Q)< r/\alpha<\sigma(M)$, $Q$ has an ancestor $A$ (possibly $A=Q$) such that $r/(2\alpha)\le \sigma(A)<r/\alpha$.  Then
  $d(v,A)\le d(v,Q)<r/\alpha\le 2\sigma(A),$
  so $v\in N_2(A)$ and $B_r(v)\subset N_{2+2\alpha}(A)\subset N_{1/\e}(A)$. If $V^+=V_A^+$, then
  \begin{equation*}d_{B_r(v)}(E,V^+)\le d_{N_{\frac{1}{\e}}(A)}(E,V^+)\stackrel{\eqref{eq:STR lip close to planes}}{\le} \epsilon \sigma(A)<\frac{\epsilon r}{\alpha}.\qedhere \end{equation*}
\end{proof}

Lemma~\ref{lem:ds balls} implies the following lemma, which is a Heisenberg-version of Lemma~8.4 of \cite{DavidSemmesSingular}, showing that $M$ satisfies the intrinsic Lipschitz condition on scales above $d_{\cS}$.
\begin{lemma}\label{lem:DavidSemmesLemma8.4}
  If $0<\beta<1$ and if $\epsilon>0$ is sufficiently small (depending on $\beta$ and $\lambda$, and $C$), then for all $v,w\in 20M$ such that $d(v,w)\ge \beta \min \{d_{\cS}(v),d_{\cS}(w)\}$, we have
  \begin{equation}\label{eq:Lemma8.4}
    |x_k(v)-x_k(w)|\le 2\lambda d(v,w).
  \end{equation}
\end{lemma}
\begin{proof}
  Let $\alpha=\min\{\frac{1}{40+C},\beta\}$ and let $0<\epsilon<\frac{\lambda\alpha}{4}$ be a number satisfying Lemma~\ref{lem:ds balls}.

  If necessary, switch $v$ and $w$ so that $d_{\cS}(v)\le d_{\cS}(w)$.  Let $r=d(v,w)$.  Then $r\le \diam (N_{20}(M))\le (40+C)\sigma(M)$ and $r\ge\beta d_{\cS}(v)$, so $r\in (\alpha d_{\cS}(v),\sigma(M)/\alpha]$ and by Lemma~\ref{lem:ds balls}, there is a vertical half-space $V^+$ such that $d_{B_r(v)}(E,V^+)\le \epsilon r/\alpha$.  Consequently, there are points $v', w'\in V$ such that $d(v,v')\le \lambda r/4$ and $d(w,w')\le \lambda r/4$.  In particular, $d(v',w')\le 3r/2$.

  Since $V^+=V_R^+$ for some $R\in \cS$, we have $\angle(V,V_S)<\lambda$, so $|x_k(v')-x_k(w')|\le \lambda d(v',w')\le 3\lambda r/2$
  and
  $|x_k(v)-x_k(w)|\le \frac{\lambda}{2} r+|x_k(v')-x_k(w')|\le 2\lambda d(v,w).$
\end{proof}

In particular, if $\Gamma_0=d_{\cS}^{-1}(0)$, then $\Gamma_0$ is an intrinsic Lipschitz graph over $\Pi(\Gamma_0)$.
\begin{cor}\label{cor:Gamma0Lipschitz}
  If $w\in \Gamma_0$ and $v\in 20M$, then
  $|x_k(v)-x_k(w)|\le 2\lambda d(v,w)$
  and $L_{w}\cap 20M=\{w\}$.
\end{cor}
\begin{proof}
  If $v\in 20M$, then Lemma~\ref{lem:DavidSemmesLemma8.4} implies that
  $|x_k(v)-x_k(w)|\le 2\lambda d(v,w)$.  If $v\in L_w\cap 20M$,
  then $d(v,w)=|x_k(v)-x_k(w)|\le 2\lambda d(v,w)$, so $d(v,w)=0$ and
  $v=w$.
\end{proof}

We will extend $\Gamma_0$ and construct the rest of $\Gamma$ by smoothing $M$ using a partition of unity.
For the rest of this section, we will take $\alpha=\beta=10^{-3}$ and assume that $\epsilon$ is small enough that Lemma~\ref{lem:ds balls} and Lemma~\ref{lem:DavidSemmesLemma8.4} hold.

Let $\psi=1/50$ and let $\{c_i\}_{i\in I}\subset 10M$ be as in Lemma~\ref{lem:coveringQS mini}, so that the balls $\{B_i=B_{\psi d_{\cS}(c_i)}(c_i)\}_{i\in I}$ are disjoint, and so that if $4\le \rho<50$, then the balls $\{\rho B_i=B_{\rho\psi d_{\cS}(c_i)}(c_i)\}_{i\in I}$ cover $10M\setminus \Gamma_0$ with bounded multiplicity. Since $d_{\cS}$ is 1-Lipschitz, we have
$$\forall\, i\in I,\qquad q\in \rho B_i\implies \left(1-\frac{\rho}{50}\right)d_{\cS}(c_i)\le d_{\cS}(q) \le \left(1+\frac{\rho}{50}\right) d_{\cS}(c_i).$$
For each $i\in I$ we have $d_{\cS}(c_i)\in (\alpha d_{\cS}(c_i),\sigma(M)/\alpha]$, so by Lemma~\ref{lem:ds balls} there is a vertical plane $V_i$ such that
\begin{equation}\label{eq:STRcoverCloseness}
  d_{B_{d_{\cS}(c_i)}(c_i)}(V_i^+,E)  = d_{50B_i}(V_i^+,E) \le\frac{\epsilon d_{\cS}(c_i)}{\alpha}.
\end{equation}

\begin{lemma}\label{lem:boundedLineIntersections}
  For all $L\in \cL^\perp(10 M)$, let
  $I_L=\{i\in I\mid 8 B_i\cap L\ne \emptyset\}$
  and let $\delta_L=\inf_{y\in L\cap 20M} d_{\cS}(y)$.  If $\epsilon>0$ is sufficiently small, then
  \begin{enumerate}
  \item \label{it:boundedLineIntersections:interval}
    There is an interval $U_L\subset L$ such that $\ell(U_L)\lesssim \epsilon \delta_L$ and such that $20M\cap L\subset U_L$ and $\bigcup_{i\in I_L} L\cap V_i \subset U_L$.
  \item \label{it:boundedLineIntersections:deltaSimD}
    If $L\in \cL^\perp(10M)$, $q\in L\cap 20M$ and $i\in I_L$, then $\delta_L\asymp d_{\cS}(q) \asymp d_{\cS}(c_i)$.
  \item \label{it:boundedLineIntersections:boundedlyMany} $|I_L| \asymp 1.$
  \end{enumerate}
\end{lemma}
\begin{proof}
  Let $q\in L\cap 20 M$.  We claim that $L\cap 20M$ is contained in an interval of radius $\beta d_{\cS}(q)$ centered at $q$.  If $p\in L\cap 20 M$, then Lemma~\ref{lem:DavidSemmesLemma8.4} implies that either $d(p,q)< \beta\min \{d_{\cS}(p),d_{\cS}(q)\}$ or
  $|x_k(p)-x_k(q)|\le 2\lambda d(p,q)<d(p,q).$
  But $p$ and $q$ both lie on $L$, so $d(p,q)=|x_k(p)-x_k(q)|$, and thus $p\in B_{\beta d_{\cS}(q)}(q)$.  Consequently, $d_{\cS}(p)\ge d_{\cS}(q)-\beta d_{\cS}(q)$; since this holds for all $p,q\in L\cap 20 M$, it implies that $\delta_L\asymp d_{\cS}(q)$ for all $q\in L\cap 20 M$. In fact, we can prove that $L\cap 20 M$ actually lies in a smaller interval.  Suppose that $p,q\in L\cap 20M$.  Let $i\in I$ be such that $q\in 8B_i$.   Since $d_{\cS}$ is 1-Lipschitz, we have $d_{\cS}(q)<2 d_{\cS}(c_i)$ and thus $p\in 9B_i$. Denoting  $\{v\}=\Pi_{V_i}(L)=L\cap V_i$, we have
    $$d(p, v)\stackrel{\eqref{eq:distToFlatPlane}}{\le} 2 d(p,V_i) \stackrel{\eqref{eq:STR lip close to planes}}{\lesssim} \epsilon d_{\cS}(c_i)\asymp \epsilon d_{\cS}(q)\asymp \epsilon \delta_L.$$
    Similarly, $d(q, v)\lesssim \epsilon \delta_L$.  Hence, $\diam (L\cap 20 M)\asymp \epsilon \delta_L$.  This proves the first part of condition~(\ref{it:boundedLineIntersections:interval}).

  To prove the second part of condition~(\ref{it:boundedLineIntersections:interval}), suppose that $L\in \cL^\perp(10 M)$ and $i\in I_L$.  Let $y_i\in 8B_i\cap L$ and let $\{v_i\}=\Pi_{V_i}(L)$ as before.  Then $d(c_i,V_i)\lesssim \epsilon d_{\cS}(c_i)$ and
  $$d(v_i,y_i)\le 2d(y_i,V_i) \le 2d(y_i,c_i) + 2d(c_i,V_i) \stackrel{\eqref{eq:STRcoverCloseness}}{\le} 2d(y_i,c_i) + 2\epsilon d_{\cS}(c_i).$$
  If $\epsilon$ is sufficiently small, then it follows that $v_i\in 17B_i$, and by \eqref{eq:STRcoverCloseness}, there is some $q_i\in L\cap \partial E\cap 20B_i$ such that $d(q_i,v_i)\lesssim \epsilon d_{\cS}(c_i)$.  It follows that $\delta_L\asymp d_{\cS}(q_i)\asymp d_{\cS}(c_i)$ and $q_i\in 20M$, so we have $d(v_i, L\cap 20M)\lesssim \epsilon d_{\cS}(c_i)$ and the second part of condition (\ref{it:boundedLineIntersections:interval}) holds, and also condition (\ref{it:boundedLineIntersections:deltaSimD}) holds. Furthermore if also $j\in I_L$ then $d(q_i,q_j)\lesssim \e\delta_L\asymp \e d_{\cS}(c_i)$ by conditions (\ref{it:boundedLineIntersections:interval}) and (\ref{it:boundedLineIntersections:deltaSimD}). So, $q_i\in 21B_j$, and since the balls $\{20B_j\}_{j\in I_L}$ form a cover with bounded multiplicity, this implies that $|I_L|\asymp 1$.
\end{proof}

Next, we define some auxiliary bump functions.  Let $\cb\from \H^{2k+1}\to \R$ be a bump function such that $\cb(B_{5\psi})=1$, $\cb=0$ outside $B_{6\psi}$ and such that $\cb$ is smooth and Lipschitz \emph{with respect to the Euclidean metric}.  That is, $\cb$ is a smooth function considered as a function on $\R^{2k+1}$ and $\|\cb\|_{\Lip(\ell_2^{2k+1})}\asymp 1$.  This condition is needed below because non-horizontal derivatives of $\cb$ will appear in some bounds.  We denote the Euclidean metric on $\H^{2k+1}$ by $\dEuc$, that is, $\dEuc(v,v')=\|v-v'\|$ for all $v,v'\in \R^{2k+1}$. Note that if $\rho>0$ and $\gamma$ is a horizontal path inside $B_\rho$, then $\ell(\gamma)\asymp_\rho \ell_{\text{Euc}}(\gamma)$.  Therefore, if $v,v'\in B_\rho$, then $\dEuc(v,v')\lesssim_\rho d(v,v').$
 Hence $\cb$ is also Lipschitz with respect to the Carnot--Carath\'eodory  metric and $\|\cb\|_{\Lip(\H^{2k+1})}\asymp 1$.

For each $i\in I$, define $\tau_i\from \H^{2k+1}\to \R$ by
$$\forall\, y\in \H^{2k+1},\qquad \tau_i(y)\eqdef d_{\cS}(c_i) \cb\Big(\s_{\frac{1}{d_{\cS}(c_i)}}\big(c_i^{-1}y\big)\Big).$$
The scaling and translation in the above definition sends $B_i$ to $B_\psi$, so $\tau_i=d_{\cS}(c_i)$ on $5B_i$, $\tau_i$ vanishes  outside the ball $6B_i$, and $\|\tau_i\|_{\Lip(\H^{2k+1})}\asymp 1$.

We will use the $\{\tau_i\}_{i\in I}$ to construct a partition of unity on $\cL^\perp$.  For each $i\in I$, we define $\tau'_i\from \cL^\perp \to \R$ by $\tau'_i(L)=\tau_i(L\cap V_i)$.  Equivalently, $\bar{\tau}'_i(p)=\tau_i(\Pi_{V_i}(p))$.  Then $\tau'_i$ is a continuous function that is zero outside $\cL^\perp(V_i\cap 6B_i)$.
\begin{lemma}\label{lem:tauLipschitz}
  Let $T\eqdef \sum_{i\in I} \tau'_i$.  Then the following assertions hold true.
  \begin{enumerate}
  \item For all $q\in 10M$ we have $\bar{T}(q)\asymp d_{\cS}(q)$.
  \item For all $i\in I$ we have $\|\bar{\tau}'_i\|_{\Lip(8B_i)}\lesssim 1$.
  \end{enumerate}
  Hence,  the functions $\big\{\phi_i\eqdef \frac{\tau'_i}{T}\big\}_{i\in I}$ form a partition of unity on $\cL^\perp(10M)$.
\end{lemma}
\begin{proof}
  First, we show that $\bar{T}(q)\gtrsim d_{\cS}(q)$ for all $q\in 10M$.  If $q\in \Gamma_0$ then $d_\cS(q)=0$ and there is nothing to prove. So, suppose that $q\in 10M\setminus \Gamma_0$.  Let $i\in I$ be such that $q\in 4B_i$.  Using~\eqref{eq:STRcoverCloseness}, if $\epsilon<\alpha/2$, then $d(q, V_i)\le d_{\cS}(c_i)/2$.  By~\eqref{eq:distToFlatPlane},
  $d(\Pi_{V_i}(q),c_i)\le 2 d(q,V_i)+ d(q,c_i)<5\psi d_{\cS}(c_i).$
  Thus $\Pi_{V_i}(q)\in 5B_i$, and
  $\bar{T}(q)\ge \tau_i(\Pi_{V_i}(q)) = d_{\cS}(c_i) \asymp d_{\cS}(q).$ On the other hand, Lemma~\ref{lem:boundedLineIntersections} implies that for any $L\in \cL^\perp(10M)$, the set $I_L=\{i\mid 8 B_i\cap L\ne \emptyset\}$ has boundedly many elements, and if $q\in L\cap 10M$, then $d_{\cS}(q)\asymp d_{\cS}(c_i)$ for each $i\in I_L$.  Thus
  $\bar{T}(q)\lesssim d_{\cS}(q)  |I_{L}|\lesssim d_{\cS}(q)$. This proves the first part of the Lemma~\ref{lem:tauLipschitz}.

  It remains to show that if $w_1,w_2\in 8B_i$, then
  $|\tau_i(\Pi_{V_i}(w_1))-\tau_i(\Pi_{V_i}(w_2))|\lesssim d(w_1,w_2).$
  The main difficulty is that the map $\Pi_{V_i}$ is not Lipschitz with respect to the Carnot--Carath\'eodory  metric; it is, however, smooth with respect to the Euclidean structure on $\R^{2k+1}$.

  Let $\xi\from \H^{2k+1}\to \H^{2k+1}$ be the map $\xi(y)=\s_{d_{\cS}(c_i)^{-1}}(c_i^{-1}y)$ that rescales and translates $B_i$ to $B_\psi$.  This map sends lines in $\cL^\perp$ to lines in $\cL^\perp$, so $\xi\circ \Pi_{V_i}=\Pi_{\xi(V_i)}\circ \xi$ and
  \begin{align}\label{eq:rewrite tau prime}
 \forall\, w\in \H^{2k+1},\qquad    \bar{\tau}'_i(w)=d_{\cS}(c_i) \cb\big(\xi(\Pi_{V_i}(w))\big)
                =d_{\cS}(c_i) \cb\big(\Pi_{\xi(V_i)}(\xi(w))\big).
  \end{align}
  The projection $\Pi_{\xi(V_i)}$ is a smooth map in Euclidean coordinates.  In fact, since $\xi(V_i)$ is a vertical plane with $\angle(V_{M},\xi(V_i))\le \lambda$ and
  $$d\big(\0,\xi(V_i)\big)=\frac{d(c_i,V_i)}{d_{\cS}(c_i)}\lesssim \epsilon,$$
  the Euclidean derivatives of $\Pi_{\xi(V_i)}$ are uniformly bounded inside the ball $B_1$, and hence the Lipschitz constant with respect to the Euclidean metric of the restriction of $\Pi_{\xi(V_i)}$ to $B_1$   is bounded.  That is, for $v_1,v_2\in B_1$ we have
  $\dEuc(\Pi_{\xi(V_i)}(v_1),\Pi_{\xi(V_i)}(v_2))\lesssim \dEuc(v_1,v_2).$ Because $w_1,w_2\in 8B_i$, we may apply this with $v_1=\xi(w_1),v_2=\xi(w_2)\in B_1$ to deduce that
  \begin{equation}\label{eq:euc lip}
  d_{\text{Euc}}\big(\Pi_{\xi(V_i)}(\xi(w_1)),\Pi_{\xi(V_i)}(\xi(w_2))\big) \lesssim \dEuc\big(\xi(w_1),\xi(w_2)\big)\lesssim d\big(\xi(w_1),\xi(w_2)\big)=\frac{d(w_1,w_2)}{d_{\cS}(c_i)}.\end{equation}
Consequently,
  \begin{equation*}\big|\bar{\tau}'_i(w_1)-\bar{\tau}'_i(w_2)\big|\stackrel{\eqref{eq:rewrite tau prime}\wedge\eqref{eq:euc lip}}{\le} d_{\cS}(c_i) \|\cb\|_{\Lip(\ell_2^{2k+1})}\cdot  \frac{d(w_1,w_2)}{d_{\cS}(c_i)}\lesssim d(w_1,w_2).\qedhere\end{equation*}
\end{proof}

For each $i\in I$, let $f_i\from \cL^\perp \to \R$ be the function $f_i(L)=x_k(L\cap V_i)$.  Then $\bar{f}_i(p)=x_k(\Pi_{V_i}(p))$, so $\bar{f}_i$ is constant on vertical lines and
\begin{equation}\label{eq:Lip fi circ Pi}
  \big\|\bar{f}_i\big\|_{\Lip(\H^{2k+1})}\le 2\lambda.
\end{equation}
We define $f$ on $\cL^\perp(10M)$ by
$$f(L)=\begin{cases}
  x_k(L\cap \Gamma_0) & \mathrm{if}\ L\in \cL^\perp(\Gamma_0),\\
  \sum_{i\in I} \phi_i(L) f_i(L) & \mathrm{if}\ L\in \cL^\perp(10M\setminus \Gamma_0).
\end{cases}$$
By Corollary~\ref{cor:Gamma0Lipschitz}, this function is well-defined.  Let $$\Gamma\eqdef \big\{w\in \H^{2k+1}\mid x_k(w)=\bar{f}(w)\big\}=\Big\{v X_k^{\bar{f}(v)}\mid v\in V_\cS\Big\}.$$
This is an intrinsic graph over $\Pi_{V_\cS}(10M)$.

\begin{lemma}\label{lem:glued lip}
  For all sufficiently small $\epsilon>0$ the following assertions hold true.
  \begin{enumerate}
  \item If $i\in I$, $q\in 10M\cap 8B_i$ and $w\in L_q\cap \Gamma$, then
    \begin{align}
      \label{eq:fi close to Q} |x_k(q)-\bar{f}_i(q)|&\lesssim \epsilon d_{\cS}(q)\\
      \label{eq:f close to Q} |x_k(q)-\bar{f}(q)|&=d(q,w)\lesssim \epsilon d_{\cS}(q).\\
      \label{eq:dsw dsq} d_{\cS}(w) & \asymp d_{\cS}(q).
    \end{align}
  \item $\Gamma$ is an intrinsic Lipschitz graph over $10M$ with Lipschitz constant at most $10\lambda$.
  \end{enumerate}
\end{lemma}
\begin{proof}
  If $q\in \Gamma_0$, then $w=q$, so \eqref{eq:fi close to Q} holds vacuously and \eqref{eq:f close to Q} and \eqref{eq:dsw dsq} are trivial. So, suppose that $q\in 10M\setminus \Gamma_0$ and that $w$ and $i$ are as above.  By part (\ref{it:boundedLineIntersections:interval}) of Lemma~\ref{lem:boundedLineIntersections}, we have
  $$|x_k(q)-\bar{f}_i(q)|=d\big(q,\Pi_{V_i}(q)\big)\lesssim \epsilon d_{\cS}(q).$$
  Since $f$ is an average of these $f_i$'s, this implies \eqref{eq:f close to Q}.  If $\epsilon$ is small enough, then $d(q,w)\le d_{\cS}(q)/2$, so $d_{\cS}(w)\asymp d_{\cS}(q).$ This proves the first part of Lemma~\ref{lem:glued lip}.

  To prove that $\Gamma$ is a Lipschitz graph, let $q_1,q_2\in 10 M$ and let $\{w_j\}=L_{q_j}\cap \Gamma$.  By the first part of the lemma, for $j=1,2$, we have $d(q_j, w_j)\lesssim \epsilon d_{\cS}(q_j)$.  We suppose that $\epsilon$ is small enough that $d(q_j, w_j)\le 10^{-3}\lambda d_{\cS}(q_j)$.    Since $d_{\cS}$ is 1-Lipschitz, this implies $d_{\cS}(q_j)/2\le d_{\cS}(w_j)\le 3d_{\cS}(q_j)/2$.

  We claim that $w_1$ and $w_2$ satisfy
  $|\bar{f}(w_1)-\bar{f}(w_2)|=|x_k(w_1)-x_k(w_2)|\le 10 \lambda d(w_1,w_2).$
  We consider two cases depending on whether $d(w_1,w_2)$ is greater or less than $10^{-2} \min \{d_{\cS}(w_1),d_{\cS}(w_2)\}$.

  Suppose that $d_{\cS}(w_1)\le d_{\cS}(w_2)$ and that $10^{-2} d_{\cS}(w_1) < d(w_1,w_2)$.  Then
  $$d_{\cS}(w_1)\le d_{\cS}(w_2)\le d(w_1,w_2)+d_{\cS}(w_1)\le 101 d(w_1,w_2),$$
  and also $d(q_j,w_j)\le 2\lambda\cdot 10^{-3}d_{\cS}(w_j) \le \lambda d(w_1,w_2)/4.$
  It follows that
  \begin{equation}\label{eq:w1 w2 close to q1 q2}
    \frac{d(w_1,w_2)}{2} \le d(q_1,q_2)\le 2d(w_1,w_2),
  \end{equation}
  and therefore $d(q_1,q_2)>\frac{d_{\cS}(w_1)}{200}\ge \frac{d_{\cS}(q_1)}{400}\ge \beta d_{\cS}(q_1).$    By Lemma~\ref{lem:DavidSemmesLemma8.4} applied to $q_1$ and $q_2$, we find
  \begin{multline*}
    |x_k(w_1)-x_k(w_2)|\le |x_k(q_1)-x_k(q_2)|+ d(q_1,w_1)+d(q_2,w_2)
                       \\\stackrel{\eqref{eq:Lemma8.4}}{\le} 2\lambda d(q_1,q_2) + \lambda d(w_1,w_2)
                       \stackrel{\eqref{eq:w1 w2 close to q1 q2}}{\le} 5\lambda d(w_1,w_2),
  \end{multline*}
  as desired.

  Now suppose $d(w_1,w_2) < 10^{-2} \min\{d_{\cS}(w_1),d_{\cS}(w_2)\}.$
  Then $w_1,w_2\not\in \Gamma_0$ and $d_{\cS}(w_1)\asymp d_{\cS}(w_2)$; indeed, $d_{\cS}(w_1)\in [0.99 d_{\cS}(w_2),1.01 d_{\cS}(w_2)]$.  Let $\delta=d_{\cS}(w_1)$, so that $\delta \asymp d_{\cS}(w_j)\asymp d_{\cS}(q_j)$.

  We bound $|\bar{f}(w_1)-\bar{f}(w_2)|$ by bounding the behavior of the partition of unity.  Suppose $\bar{\tau}'_i(w_1)>0$.  We claim that $w_1,w_2\in 8B_i$.  Since $\bar{\tau}'_i(w_1)=\tau_i(\Pi_{V_i}(w_1))\ne 0$, we have $\Pi_{V_i}(w_1)\in 6B_i$.  Then
  $$d\big(w_1,\Pi_{V_i}(w_1)\big)=|\bar{f}(w_1)-\bar{f}_i(w_1)|\stackrel{\eqref{eq:fi close to Q} \wedge \eqref{eq:f close to Q}}{\lesssim} \epsilon d_{\cS}(q_1).$$
  By part (\ref{it:boundedLineIntersections:deltaSimD}) of Lemma~\ref{lem:boundedLineIntersections}, we have $d_{\cS}(q_1)\asymp d_{\cS}(c_i)$, so if $\epsilon$ is small enough, then $w_1\in 7B_i$.  Consequently,
  $d_{\cS}(w_1)\le 2 d_{\cS}(c_i)$ and
  $$d(w_1,w_2) \le \frac{d_{\cS}(w_1)}{100}\le \frac{d_{\cS}(c_i)}{50},$$
  so $w_2\in 8B_i$.  Likewise, if $\bar{\tau}'_i(w_2)>0$, then $w_1,w_2\in 8 B_i$.

  Therefore, there are only boundedly many $i\in I$ such that
  $\bar{\tau}'_i(w_1)\ne\bar{\tau}'_i(w_2)$, and by Lemma~\ref{lem:tauLipschitz},
  $$|\bar{\tau}'_i(w_1)-\bar{\tau}'_i(w_2)|\lesssim d(w_1,w_2)$$
  for any such $i$.  It follows that $|\bar{T}(w_1)-\bar{T}(w_2)|\lesssim d(w_1,w_2)$.

  Let $\Delta\bar{\phi}_i=\bar{\phi}_i(w_1)-\bar{\phi}_i(w_2)$
  and $\Delta \bar{f}_i=\bar{f}_i(w_1)-\bar{f}_i(w_2)$.  Since $\bar{T}(w_i)\asymp d_{\cS}(w_i)\asymp \delta$,
  \begin{align*}
    \left|\Delta\bar{\phi}_i\right|
    &=\frac{|\bar{\tau}'_i(w_1) \bar{T}(w_2)-\bar{\tau}'_i(w_2) \bar{T}(w_1)|}{\bar{T}(w_1)\bar{T}(w_2)}\\
    &\lesssim \delta^{-2}\big(|\bar{\tau}'_i(w_1) \bar{T}(w_2)-\bar{\tau}'_i(w_1) \bar{T}(w_1)|+|\bar{\tau}'_i(w_1) \bar{T}(w_1)-\bar{\tau}'_i(w_2) \bar{T}(w_1)|\big)\\
    &\lesssim \delta^{-2}\big(\bar{\tau}'_i(w_1) d(w_1,w_2)+\bar{T}(w_1)d(w_1,w_2)\big) \\
    &\lesssim \frac{d(w_1,w_2)}{\delta}.
  \end{align*}
  By \eqref{eq:Lip fi circ Pi}, we have $|\Delta \bar{f}_i|\le 2 \lambda
  d(w_1,w_2)$.  Hence,
  \begin{align}
   \nonumber |\bar{f}(w_2)-\bar{f}(w_1)|
    &=\biggl|\sum_{i\in I} \bar{\phi}_i(w_1) \bar{f}_i(w_1)- \sum_{i\in I}\bar{\phi}_i(w_2) \bar{f}_i(w_2)\biggr|\\ \nonumber
    &\le\biggl|\sum_{i\in I} \Delta\bar{\phi}_i \cdot \bar{f}_i(w_1)\biggr| +
      \biggl|\sum_{i \in I} \bar{\phi}_i(w_2)\cdot \Delta \bar{f}_i\biggr|\\ \nonumber
    &\le \bar{f}(w_1) \biggl|\sum_{i\in I}\Delta\bar{\phi}_i\biggr|
      +\sum_{i\in I}|\Delta\bar{\phi}_i|\cdot|\bar{f}_i(w_1)-\bar{f}(w_1)| +\sum_{i\in I} \bar{\phi}_i(w_2)|\Delta \bar{f}_i|\\
    &\le 0
      +\sum_{i\in I}O\left(\frac{ d(w_1,w_2)}{\delta}\cdot  \epsilon \delta\right)
      +2\lambda d(w_1,w_2) \sum_{i\in I} \bar{\phi}_i(w_2) \label{eq:bounf fi-f}\\
    &\le O\big(\epsilon d(w_1,w_2)\big)+\lambda d(w_1,w_2).\nonumber
  \end{align}
  where in~\eqref{eq:bounf fi-f}  we used \eqref{eq:fi close to Q} and \eqref{eq:f close to Q} to bound $|\bar{f}_i-\bar{f}|$.  If $\epsilon$ is sufficiently small, then this implies the desired estimate $|\bar{f}(w_1)-\bar{f}(w_2)|\le 2 \lambda d(w_1,w_2)$.
\end{proof}

Since $\Gamma$ is an intrinsic Lipschitz graph over $\Pi(10M)$, Theorem~\ref{thm:intrinsic nonlinear hahn banach} implies that $\Gamma$ can be extended to an intrinsic Lipschitz graph over $V_{M}$ with no increase in the Lipschitz constant.  We also denote this extension by $\Gamma$, and we let $\Gamma^+$ be the half-space
$\Gamma=\{w\in \H^{2k+1}\mid x_k(w)\ge \bar{f}(w)\}.$

Finally, we claim that this satisfies \eqref{eq:defCoronaDecomp}.  Let $Q\in \cS$.  By \eqref{eq:goodVerticalCubes}, we have
$d_{N_{1/\e}(Q)}(V_Q^+, E)\le \epsilon \sigma(Q).$
We claim that
\begin{equation}\label{eq:VQ close to Gamma}
  d_{N_{\frac92}(Q)}(V_Q^+, \Gamma^+)\lesssim \epsilon \sigma(Q).
\end{equation}

Let $L\in \cL^\perp(N_{9/2}(Q))$. Suppose that $v\in L\cap V_Q$ and $w\in L\cap \Gamma$. By Lemma~\ref{lem:N4 in 10Q}, there exists $q\in L\cap 10Q$.  By \eqref{eq:goodVerticalCubes} and \eqref{eq:distToFlatPlane}, we have
$$|x_k(q)-x_k(v)|=d(q,v)\le 2 d(q,V_Q)\le 2\epsilon \sigma(Q).$$
By \eqref{eq:f close to Q}, we have
$$|x_k(w)-x_k(q)|\lesssim\epsilon d_{\cS}(q) \le \epsilon \sigma(Q),$$
so $|x_k(w)-x_k(v)|\lesssim \epsilon \sigma(Q)$.

We have $V_Q^+\cap L=[v,\infty)$ and $\Gamma^+\cap L=[w,\infty)$, so if $a\in L\cap (V_Q^+\symdiff \Gamma^+)$, then either $a\in [v,w]$ or $a\in [w,v]$, depending on the order of $v$ and $w$.  In both cases,  $d(a,V_Q)\le |x_k(w)-x_k(v)|\lesssim \epsilon\sigma(Q)$ and $d(a,\Gamma)\lesssim \epsilon\sigma(Q)$, so \eqref{eq:VQ close to Gamma} holds.  If $\epsilon$ is sufficiently small, then Lemma~\ref{lem:localHausProps} implies
\begin{align*}
  d_{N_4(Q)}(\Gamma^+, E)\le d_{N_{\frac92}(Q)}(\Gamma^+, V_Q^+) + d_{N_{\frac{1}{\e}}(Q)}(V_Q^+, E)
                         \lesssim \epsilon \sigma(Q).
\end{align*}
If $\epsilon$ is sufficiently small, then
$$d_{N_4(Q)}(\Gamma^+, E)\le \theta \sigma(Q).$$
This completes the proof of Proposition~\ref{prop:STRLipschitzGraphs}.

\section{Historical comments on Sparsest Cut and directions for further research}\label{sec:previous}

Among the well-established deep and multifaceted connections between theoretical computer science and pure mathematics, the Sparsest Cut Problem stands out for its profound and often unexpected impact on a variety of areas. Previous research on  this question came hand-in-hand with the development of remarkable mathematical and algorithmic ideas that spurred many further works of importance in their own right. Because the present work belongs to this tradition, here we will put it into context by elaborating further on the history of these investigations. We will also  describe interesting directions for further research and open problems.

The first polynomial-time algorithm for Sparsest Cut with $O(\log n)$ approximation ratio was obtained in the important work~\cite{LR99}, which studied the notable special case of {\em Sparsest Cut with Uniform Demands} (see Section~\ref{sec:uniform}  below).  This work introduced a linear programming relaxation and developed influential techniques for its analysis, and it has led to a myriad of algorithmic applications. The seminal contributions~\cite{LLR95,AR98} obtained the upper bound $\rho_{\mathsf{GL}}(n)\lesssim \log n$ in full generality by incorporating a classical embedding theorem of Bourgain~\cite{Bou85}, thus heralding the transformative use of metric embeddings in algorithm design. The matching lower bound on the integrality gap of this linear program was proven in~\cite{LR99,LLR95}.  This showed for the first time that Bourgain's embedding theorem is asymptotically sharp and was the first (and very influential)  demonstration of the power of expander graphs in the study of metric embeddings.

A $O(\sqrt{\log n})$ upper bound for the approximation ratio of the Goemans--Linial algorithm in the special case of uniform demands was obtained in the important work~\cite{ARV09}.  This work relied on a clever use of the concentration of measure phenomenon and introduced influential techniques such as a ``chaining argument'' for metrics of negative type and the use of expander flows. \cite{ARV09} also had direct impact on results in pure mathematics, including combinatorics and metric geometry; see e.g.~the ``edge replacement theorem" and the estimates on the observable diameter of doubling metric measure spaces in~\cite{NRS05}. The best-known general upper bound $\rho_{\mathsf{GL}}(n)\lesssim (\log n)^{\frac12+o(1)}$ of~\cite{ALN08} built on the (then very recent) development of two techniques: The chaining argument of~\cite{ARV09} (through its more careful analysis in~\cite{Lee05}) and the {\em measured descent} embedding method of~\cite{KLMN05} (through its phrasing as a gluing technique for Lipschitz maps in~\cite{Lee05}). Another important input to~\cite{ALN08} was a re-weighting argument of~\cite{CGR08} that allowed for the construction of an appropriate ``random zero set" from the argument of~\cite{ARV09,Lee05} (see~\cite{Nao10,Nao14} for more on this notion and its significance).

The impossibility result~\cite{KV15} that refuted the Goemans--Linial conjecture relied on a striking link to complexity theory through the Unique Games Conjecture (UGC), as well as an interesting use of discrete harmonic analysis (through~\cite{Bou02}) in this context; see also~\cite{KR09} for an incorporation of a different tool from discrete harmonic analysis (namely~\cite{KKL88}, following~\cite{KN06}) for the same purpose, as well as~\cite{CKKRS06,CK07-hardness} for computational hardness. The best impossibility result currently known~\cite{KM13} for Sparsest Cut with Uniform Demands relies on the development of new  pseudorandom generators.

The idea of using the  geometry of the Heisenberg group to bound $\rho_{\mathsf{GL}}(n)$ from below originated in~\cite{LN06}, where the relevant metric of negative type was constructed through a complex-analytic argument, and initial (qualitative) impossibility results were presented through the use of Pansu's differentiation theorem~\cite{Pan89} and  the Radon--Nikod\'ym Property  from functional analysis (see e.g.~\cite{BL00}). In~\cite{CK10}, it was shown that the Heisenberg group indeed provides a proof that $\lim_{n\to \infty} \rho_{\mathsf{GL}}(n)=\infty$.  This proof introduced a remarkable new notion of differentiation for  mappings into $L_1(\R)$, which led to the use of tools from geometric measure theory~\cite{FSSCRectifiability,FSSC03} to study the problem. A different proof that $\H^3$ fails to admit a bi-Lipschitz embedding into $L_1(\R)$ was found in~\cite{CheegerKleinerMetricDiff}, where a classical notion of metric differentiation~\cite{Kir94} was used in conjunction with the novel idea to consider monotonicity of sets in this context, combined with a sub-Riemannian-geometric argument that yielded a classification of monotone subsets of $\H^3$.  The main result of~\cite{CKN} finds a quantitative lower estimate for the scale at which this differentiation argument can be applied, leading to a lower bound of $(\log n)^{\Omega(1)}$ on $\rho_{\mathsf{GL}}(n)$.  This result relies on a mixture of the methods of~\cite{CK10} and~\cite{CheegerKleinerMetricDiff} and  overcomes obstacles that are not present in the original qualitative investigations. In particular, \cite{CKN} introduced quantitative measures of non-monotonicity that we use in the present work. The quantitative differentiation bound of~\cite{CKN} remains the best bound currently known, and it would be very interesting to discover the sharp behavior in this more subtle question.

The desire to avoid the (often difficult) need to obtain sharp bounds for quantitative differentiation motivated the investigations~\cite{AusNaoTes,LafforgueNaor}. In particular, \cite{AusNaoTes} devised a method to prove sharp (up to lower order factors) nonembeddability statements for the Heisenberg group based on a cohomological argument and a quantitative ergodic theorem. For Hilbert space-valued mappings, \cite{AusNaoTes} used a cohomological argument in combination with representation theory to prove the following quadratic inequality for every finitely supported function $\upphi:\H_{\ms{\Z}}^{2k+1}\to L_2(\R)$.
\begin{equation}\label{eq:discrete global quadratic}
\bigg(\sum_{t=1}^\infty \frac{1}{t^2}\sum_{h\in \H_{\ms{\Z}}^{2k+1}} \big\|\upphi\big(hZ^t\big)-\upphi(h)\big\|_{L_2(\R)}^2\bigg)^{\frac12}\lesssim \bigg(\sum_{h\in \H_{\ms{\Z}}^{2k+1}}\sum_{\sigma\in \mathfrak{S}_2}\big\|\upphi(h\sigma)-\upphi(h)\big\|_{L_2(\R)}^2\bigg)^{\frac12}.
\end{equation}
In~\cite{LafforgueNaor} a different approach based on Littlewood--Paley theory was devised, leading to the following generalization of~\eqref{eq:discrete global quadratic} that holds  for every $p\in (1,2]$ and every finitely supported $\upphi:\H_{\ms{Z}}^{2k+1}\to L_p(\R)$.
\begin{equation}\label{eq:discrete global p}
\left(\sum_{t=1}^\infty \frac{1}{t^2}\bigg(\sum_{h\in \H_{\ms{\Z}}^{2k+1}} \big\|\upphi\big(hZ^t\big)-\upphi(h)\big\|_{L_p(\R)}^p\bigg)^\frac{2}{p}\right)^{\frac12}\le  C(p)\bigg(\sum_{h\in \H_{\ms{\Z}}^{2k+1}}\sum_{\sigma\in \mathfrak{S}_2}\big\|\upphi(h\sigma)-\upphi(h)\big\|_{L_p(\R)}^p\bigg)^{\frac{1}{p}},
\end{equation}
for some $C(p)\in (0,\infty)$. See~\cite{LafforgueNaor} for a strengthening of~\eqref{eq:discrete global p} that holds for general uniformly convex targets (using the recently established~\cite{MTX06} vector-valued Littlewood--Paley--Stein theory for the Poisson semigroup). These functional inequalities yield sharp non-embeddability estimates for balls in $\H^{2k+1}_{\ms{\Z}}$, but the method of~\eqref{eq:discrete global p} inherently yields a constant $C(p)$ in~\eqref{eq:discrete global p} that satisfies $\lim_{p\to 1} C(p)=\infty$.  The estimate~\eqref{eq:discrete local intro} that we prove here for $L_1(\R)$-valued mappings is an endpoint estimate corresponding to~\eqref{eq:discrete global p}, showing that $C(p)$ actually remains bounded as $p\to 1$. This answers positively a conjecture of~\cite{LafforgueNaor} and is crucial for the results that we obtain here.

As explained in Section~\ref{sec:intrinsic graphs}, our proof of~\eqref{eq:discrete global quadratic} uses the $\H^3$-analogue of~\eqref{eq:discrete global quadratic}. It should be mentioned at this juncture that the proofs of~\eqref{eq:discrete global quadratic}  and~\eqref{eq:discrete global p} in~\cite{AusNaoTes,LafforgueNaor} were oblivious to the dimension of the underlying Heisenberg group. An unexpected aspect of the present work is that the underlying dimension does play a role at the endpoint $p=1$, with the analogue of~\eqref{eq:discrete local intro}  (or Theorem~\ref{thm:isoperimetric discrete}) for $\H_{\ms{\Z}}^3$ being in fact {\em incorrect} (recall Remark~\ref{eq:finite p}).

As we recalled above, past progress on the Sparsest Cut Problem   came hand-in-hand with substantial mathematical developments. The present work is a culmination of a long-term project that is rooted in mathematical phenomena that are interesting not just for their relevance to approximation algorithms but also for their connections to the broader mathematical world.  In the ensuing subsections we shall describe some further results and questions related to this general direction.

\subsection{The Sparsest Cut Problem with  Uniform Demands}\label{sec:uniform} An especially important special case of the Sparsest Cut Problem is when the demand matrix $D$ is the matrix  all of whose entries equal $1$, i.e., $D=\1_{\n\times\n}\in M_n(\R)$, and the capacity matrix $C$ lies in $M_n(\{0,1\})$, i.e., all its entries are either $0$ or $1$. This is known as the {\em Sparsest Cut Problem with  Uniform Demands}. In this case $C$ can also be described as the adjacency matrix of a graph $G$ whose vertex set is $\n$ and whose edge set consists of those unordered pairs $\{i,j\}\subset \n$ for which $C_{ij}=1$. With this interpretation, given $A\subset \n$ the numerator in~\eqref{eq:def opt} equals twice the number of edges that are incident to $A$ in $G$. And, since  $D=\1_{\n\times\n}$, the  denominator in~\eqref{eq:def opt} is equal to $2|A|(n-|A|)\asymp n\min\{|A|,|\n\setminus A|\}$. So, the Sparsest Cut Problem with Uniform Demands asks for an algorithm that takes as input a finite graph and outputs a quantity which is bounded above and below by universal constant multiples of its {\em conductance}~\cite{JS89} divided by $n$. The Goemans--Linial integrality gap corresponding to this special case is
$$
\uprho_{\mathsf{GL}}^{\mathrm{unif}}(n)\eqdef \sup_{\substack{C\in M_n(\{0,1\})\\ C\ \mathrm{symmetric}}} \frac{\mathsf{OPT}(C,\1_{\n\times\n})}{\mathsf{SDP}(C,\1_{\n\times \n})}.
$$
The Goemans--Linial algorithm furnishes the best-known approximation ratio also in the  case of uniform demands. We have $\uprho_{\mathsf{GL}}^{\mathrm{unif}}(n)\lesssim \sqrt{\log n}$ by the important work~\cite{ARV09}, improving over the previous bound $\uprho_{\mathsf{GL}}^{\mathrm{unif}}(n)\lesssim \log n$ of~\cite{LR99}. As explained in~\cite{CKN09}, the present approach based on (fixed dimensional) Heisenberg groups cannot yield a lower bound on $\uprho_{\mathsf{GL}}^{\mathrm{unif}}(n)$ that tends to $\infty$ as $n\to \infty$. The currently best-known lower bound~\cite{KM13} is $\uprho_{\mathsf{GL}}^{\mathrm{unif}}(n)\ge \exp(c\sqrt{\log\log n})$ for some universal constant $c\in (0,\infty)$, improving over the previous bound $\uprho_{\mathsf{GL}}^{\mathrm{unif}}(n)\gtrsim \log\log n$ of~\cite{DKSV06}.   Determining the asymptotic behavior of $\uprho_{\mathsf{GL}}^{\mathrm{unif}}(n)$ remains an intriguing open problem.

 \subsection{A sharp upper bound on $\rho_{\mathsf{GL}}(n)$} The $o(1)$ term in the bound $\rho_{\GL}(n)\lesssim (\log n)^{\frac12+o(1)}$ that is stated in~\cite{ALN08} is $\frac{\log\log\log n}{\log\log n}$, i.e., $\rho_{\GL}(n)\lesssim \sqrt{\log n}\log\log n$. This specific lower-order factor can be reduced with more technical work (see the discussion in Section~6 of~\cite{ALN08}). To the best of our knowledge, the question of removing the $o(1)$ term altogether was not previously investigated, since the  pressing issue  was to determine the correct power of $\log n$. Now that Theorem~\ref{thm:main GL lower intro} is established, it is worthwhile to revisit this matter to obtain the sharp upper bound $\rho_{\GL}(n)\lesssim \sqrt{\log n}$. We conjecture that this estimate is indeed valid, and we expect that its proof should follow from a more careful application of available methods. We postpone this investigation to future work.

 \subsection{Sharp stability of monotone sets} We already asked above what is the sharp dependence in the quantitative differentiation statement (stability of monotone sets) of~\cite{CKN}. This is a challenging open question whose need was circumvented in the present context by adopting the strategy of aiming for the functional inequality~\eqref{eq:discrete local intro} rather than for a sharp  structure theorem for sets of finite perimeter in $\H^{2k+1}$ that shows that they are approximately a half-space on a macroscopically large scale. The advantage of our approach is that it relies on a  corona decomposition in which the approximation is by an intrinsic Lipschitz graph rather than a half-space, i.e., we avoid the need to ``zoom in" all the way to the scale at which the boundary resembles a half-space. Intrinsic Lipschitz graphs can be much less structured than half-spaces, but we control them using cancellations (in a lower dimension) which can be captured through a decomposition of each $\H^3$-slice into irreducible representations. Regardless of any application to embedding theory, sharpening the bound of~\cite{CKN} on the stability of monotone sets is an interesting question in geometric measure theory.

 \subsection{Hardness} Proving $\mathsf{NP}$-hardness of approximation for the Sparsest Cut Problem remains elusive if one wishes to prove inapproximability within any constant factor (such UGC-based hardness of approximation is available due to~\cite{KV15,CKKRS06}). This is a longstanding open question that seems to be very challenging. We speculate that the hardness of approximation threshold for Sparsest Cut is $\Theta(\sqrt{\log n})$, but we do not see a path towards such a statement even under the UGC.

 \subsection{Further questions related to quantitative rectifiability and singular integrals} It would be interesting to explore the impact of the methods of the present article on other questions related to quantitative rectifiability and singular integral operators.  In future work we will investigate how an application of these ideas to $\R^n$ rather than the Heisenberg group yields a new perspective on classical results such as the Jones Travelling Salesman Theorem~\cite{Jon90}, and we believe that there is scope for additional applications of the use of quantitative non-monotonicity to govern hierarchical partitioning schemes as well as the type of ``induction on dimension" that is used here to pass from $L_2$ bounds to their $L_1$ counterparts (in a higher dimension).

\bibliographystyle{abbrv}
\bibliography{corona,appendix}
\end{document}